\documentclass[reqno,11pt,a4paper]{amsart}

\usepackage{amssymb,amsmath,amsfonts,amssymb}
\usepackage{hyperref}
\usepackage{enumerate}
\usepackage{dsfont}
\usepackage{graphicx}
\usepackage{verbatim}
\usepackage{array}
\usepackage{booktabs}
\usepackage{tabularx}
-\marginparwidth \oddsidemargin -9mm \evensidemargin \oddsidemargin
\usepackage{latexsym}
\usepackage{color}
\advance\hoffset by 5mm

\def\@abssec#1{\vspace{.05in}\footnotesize \parindent .2in
{\bf #1. }\ignorespaces}

\newtheorem{theorem}{Theorem}[section]

\newtheorem{lemma}[theorem]{Lemma}
\newtheorem{proposition}[theorem]{Proposition}
\newtheorem{corollary}[theorem]{Corollary}
\newtheorem{definition}[theorem]{Definition}
\newtheorem{remark}[theorem]{Remark}

\newcommand{\R}{\ensuremath{\mathbb{R}}}
\newcommand{\RR}{\ensuremath{\mathcal{R}}}
\newcommand{\WW}{\ensuremath{\mathcal{W}}}

\newcommand{\Z}{\ensuremath{\mathbb{Z}}}
\newcommand{\NN}{\ensuremath{\mathbb{N}}}

\newcommand{\TT}{\ensuremath{\mathbb{T}}}

\newcommand{\dd}{\mathrm{d}}

\renewcommand{\div}{{\rm div\,}}
\def \p {\partial}

\allowdisplaybreaks \numberwithin{equation}{section}

\begin{document}
\title[$W^{k,\infty}$ and $C^{k,\gamma}$ regularity of temperature patches for 2D Boussinesq system]{Global regularity and infinite Prandtl-number limit for 2D Boussinesq temperature patches in $W^{k,\infty}$ and $C^{k,\gamma}$: lower-order case}
\author{Omar Lazar}
\address{Department of mathematics and statistics, American University of Sharjah, Sharjah, UAE}
\email{olazar@aus.edu}
\author{Liutang Xue}
\address{School of Mathematical Sciences, Laboratory of Mathematics and Complex Systems (MOE),  Beijing Normal University, Beijing 100875, P.R. China}
\email{xuelt@bnu.edu.cn}
\author{Jiakun Yang}
\address{School of Mathematical Sciences, Laboratory of Mathematics and Complex Systems (MOE),  Beijing Normal University, Beijing 100875, P.R. China}
\email{yangjk@mail.bnu.edu.cn}

\thanks{\textit{Acknowledgment.} 
L. Xue is supported by the National Natural Science
Foundation of China (Grant Nos. 12271045, 12571245, and 12671270).
J. Yang is supported by the China Scholarship Council
(Scholarship No. 202506040134).
}


\renewcommand{\subjclassname}{\textup{2020} Mathematics Subject Classification}


\begin{abstract}
We study  global regularity and the infinite Prandtl number
limit of temperature patches for the two-dimensional
non-diffusive Boussinesq system on the whole space,
with velocity dissipation $\Lambda^{2\alpha}$,
$\frac12<\alpha\leq1$.
For bounded and integrable initial temperatures
and divergence-free initial velocities in $W^{1,p}$,
$2\leq p<\infty$, we establish global well-posedness.
Under suitable profile and striated regularity assumptions,
we prove global persistence of $C^{1,\gamma}$ and
$C^{2,\gamma}$ boundary regularity for temperature patches
with possibly nonconstant profiles, for every $0<\gamma\leq1$.
These results extend the previously established fractional
range beyond $\gamma=2\alpha-1$ and include the Lipschitz
endpoints $W^{2,\infty}$ and $W^{3,\infty}$.
Persistence at these endpoints contrasts with
the $W^{2,\infty}$ ill-posedness established
by Kiselev and Luo~\cite{KL23}
for two-dimensional Euler vortex patches.
For $\frac12<\alpha<1$, our $W^{2,\infty}$ persistence
result also answers the curvature question raised
by Khor and Xu~\cite{KX22} 
for initial boundaries
with only $W^{2,\infty}$ regularity.

The velocity gradient and boundary regularity estimates
are uniform in $\mathrm{Pr}\in[1,\infty)$
on each fixed finite time interval,
including $\alpha=1$ on the whole space.
We also obtain uniform global $L^r$ velocity bounds
for suitable finite exponents $r$ when $\alpha<1$,
and at $\alpha=1$ for zero-mean initial temperatures
with finite first absolute moment.
At $\alpha=1$, such uniform bounds fail for compactly supported
initial temperatures with nonzero mean.
For temperature patches, we justify the infinite
Prandtl number limit in the original frame
when $\alpha<1$ or the initial temperature has zero mean,
and in a moving frame for nonzero-mean data at $\alpha=1$.
In each case, the solutions converge
to the unique normalized global weak solution
of the (fractional) Stokes-transport system.
The limiting temperature retains its patch structure
and full boundary regularity, including both
Lipschitz endpoints.
For $\alpha=1$, this limiting argument provides
an indirect proof of Grayer II's H\"older boundary
persistence results~\cite{Gray23}
and extends them to the $W^{3,\infty}$ endpoint.
\end{abstract}

\maketitle
\tableofcontents
\vspace{-1cm}
\addtocontents{toc}{\protect\setcounter{tocdepth}{1}}

\section{Introduction}
\subsection{The problem and the main objectives}
We study the Cauchy problem for the 2D
non-diffusive Boussinesq system with velocity dissipation:
\begin{equation}\tag{B$_\alpha$}\label{BoussEq}
  \left\{
  \begin{aligned}
    \partial_t\theta+u\cdot\nabla\theta
    &=0,
    \qquad (t,x)\in\mathbb{R}_+\times\mathbb{R}^2,\\
    \frac{1}{\mathrm{Pr}}
    \bigl(\partial_tu+u\cdot\nabla u\bigr)
    +\nu\Lambda^{2\alpha}u
    &=-\nabla p+\theta e_2,\\
    \nabla\cdot u&=0,\\
    (u,\theta)|_{t=0}&=(u_0,\theta_0),
  \end{aligned}
  \right.
\end{equation}
where $u=(u_1,u_2)^T$, $p$, and $\theta$ denote the velocity,
pressure, and temperature, respectively;
$e_2=(0,1)^T$, $\nu\geq 0$ is the kinematic viscosity,
and $\mathrm{Pr}>0$ is the dimensionless Prandtl number.
The dissipation operator
$\Lambda^{2\alpha}=(-\Delta)^\alpha$, with $0<\alpha\leq1$,
is defined by the Fourier multiplier $\widehat{\Lambda^{2\alpha}f}(\xi)
  =|\xi|^{2\alpha}\widehat f(\xi)$. 

For $\nu>0$ and $\alpha=1$, the system is a classical model
of natural convection in geophysical fluid
dynamics~\cite{AJM03,JP1987}.
In the absence of thermal diffusion, the temperature
is transported by the flow and drives the velocity
through buoyancy.
The large Prandtl number regime is also motivated by
the study of Rayleigh-B\'enard convection~\cite{CD1999}.
For $\nu>0$ and $0<\alpha<1$, fractional dissipation
provides weaker velocity regularization than
the classical Laplacian.

The inviscid system, obtained by setting $\nu=0$,
shares a close structural analogy with the 3D axisymmetric Euler equations with swirl, particularly away from the symmetry axis~\cite{MB02}.
Finite-time singularity formation has been established
under various regularity and domain assumptions
in~\cite{EJ20,CH21,CH23a,CH25,EP23};
see also~\cite{CLMZ25} for a forced model.

In this paper, we study the persistence of temperature
interfaces and the infinite Prandtl number limit.
We fix $\nu=1$ and work in the subcritical dissipative range
\[
  \frac12<\alpha\leq1,
  \qquad
  \varepsilon:=\frac{1}{\mathrm{Pr}}\in(0,1].
\]

More precisely, we consider temperature patches
with possibly nonconstant profiles,
also referred to as temperature fronts~\cite{Gil82,AJM03}.
Their initial data take the form 
\[
  \theta_0(x)=\bar\theta_0(x)\mathds{1}_{D_0}(x),
\]
where $D_0\subset\mathbb R^2$ is a bounded simply connected
domain and $\bar\theta_0$ is a bounded profile on
$\overline{D_0}$.
Since the temperature is transported without diffusion,
a sufficiently regular velocity preserves the patch structure,
transporting both the profile and the domain by the fluid flow.
The Boussinesq temperature patch problem asks whether
the initial regularity of the patch boundary persists
for all time.
In this paper, we establish the global persistence of
$C^{1,\gamma}$ and $C^{2,\gamma}$ boundary regularity
for every $0<\gamma\leq1$, under suitable assumptions
on the initial profile and velocity.
For boundary parametrizations, $C^{k,1}$ is equivalently described by $W^{k+1,\infty}$; 
thus these endpoints correspond to $W^{2,\infty}$ and $W^{3,\infty}$ boundary regularity.
These endpoint persistence results contrast with
the behavior of 2D Euler vortex patches:
Kiselev and Luo~\cite{KL23} constructed
$C^2$ initial patch boundaries whose curvature
becomes unbounded immediately.
See Table~\ref{tab:Euler-Bouss-subcritical}
for a comparison.

A second objective is to justify the infinite Prandtl
number limit on $\mathbb R^2$.
A central issue is the control of the low-frequency velocity.
Uniform global $L^Q$ bounds allow us to work
in the original frame when $\alpha<1$
and for zero-mean temperature patches at $\alpha=1$.
For nonzero-mean patches at $\alpha=1$,
we instead use time-dependent translations
to obtain uniform local velocity bounds
and fix the normalization of the limiting Stokes velocity.
The uniform striated estimates then allow us
to transfer the boundary regularity to the limit.

As a foundation for these results, we establish
global well-posedness for $\theta_0\in L^1\cap L^\infty$
and divergence-free $u_0\in W^{1,p}$,
$2\leq p<\infty$, together with velocity gradient
and striated estimates uniform in $\mathrm{Pr}$.
Throughout, this uniformity is understood
on each fixed interval $[0,T]$,
with constants allowed to depend on $T$, $\alpha$,
the regularity parameters,
and the prescribed bounds on the initial data.

\subsection{Global well-posedness and uniform estimates}
In the full Laplacian case $\alpha=1$ with $\mathrm{Pr}=1$,
Chae~\cite{DC06} and, independently, Hou and Li~\cite{HL05}
established global well-posedness for smooth initial data,
answering Problem~3 posed by Moffatt~\cite{Moff01}.
The regularity assumptions were subsequently relaxed.
Abidi and Hmidi~\cite{AH07} proved global well-posedness for
$\theta_0\in B^0_{2,1}$ and
$u_0\in L^2\cap B^{-1}_{\infty,1}$.
Hmidi and Keraani~\cite{HK07} constructed global weak
solutions for $\theta_0\in L^2$ and $u_0\in H^s$,
$0\leq s<2$, whose uniqueness was later established
by Danchin and Pa\"icu~\cite{DP08}.
Further results on the persistance of Sobolev regularity
can be found in~\cite{HKZ15}.

For the critical fractional dissipation
$\alpha=\frac12$ with $\mathrm{Pr}=1$,
Hmidi, Keraani, and Rousset~\cite{HKR10}
proved global well-posedness for
$\theta_0\in L^2\cap B^0_{\infty,1}$ and
$u_0\in H^1\cap\dot W^{1,p}$ for some $p>2$.

For $\frac12<\alpha<1$ and $\mathrm{Pr}=1$,
Khor and Xu~\cite{KX22} established global well-posedness
for $\theta_0\in L^1\cap L^\infty$ and
$u_0\in H^1\cap W^{1,p}$ for some $p>2$.

Our first theorem covers the full subcritical range
$\frac12<\alpha\leq1$ and allows
$u_0\in W^{1,p}$ for any $2\leq p<\infty$,
including $p=2$ and requiring no additional $H^1$
assumption when $p>2$.
Moreover, it provides estimates for the velocity gradient
that are uniform in $\mathrm{Pr}\in[1,\infty)$,
as well as uniform global $L^Q$ velocity bounds
for suitable finite exponents $Q$ when $\alpha<1$,
and at $\alpha=1$ when $\theta_0$ has zero integral
and a finite first absolute moment.
The Besov and Chemin-Lerner notation used below
is recalled in Section~\ref{sec:tools}.
\begin{theorem}\label{thm:main1}
Let $\varepsilon =\frac{1}{\mathrm{Pr}} \in (0,1]$, $\alpha\in (\frac{1}{2},1]$.
Suppose that 
\begin{align*}
  \theta_0\in L^1\cap L^\infty(\mathbb{R}^2),\qquad u_0\in W^{1,p}(\mathbb{R}^2)\;\; \textrm{for some}
  \;\;2\leq p<\infty,\qquad \mathrm{div}\,u_0=0.
\end{align*}
Then there exists a unique global weak solution $(u,\theta)$ in the sense of Definition \ref{defn:weak-Bouss} 
to the 2D Boussinesq-Navier-Stokes system \eqref{BoussEq} 
such that, for every $T>0$,
\begin{align*}
  \theta\in L^\infty([0,T],L^1\cap L^\infty),\quad u\in L^\infty (0,T;W^{1,p})\cap \widetilde{L}^1(0,T;B^{2\alpha}_{\infty,\infty}).
\end{align*}
Moreover, $\nabla u$ belongs to the following spaces,
with norms bounded independently of $\varepsilon$:
\[
\nabla u\in
\begin{cases}
  L^\infty([0,T],L^2)
  \cap\widetilde L^1([0,T],B^{2\alpha-1}_{\infty,\infty}),
  & \text{if }\frac12<\alpha<1,\quad p=2,\\
  \widetilde L^1([0,T],B^1_{\infty,\infty}),
  & \text{if }\alpha=1,\quad p=2,\\
  L^\infty([0,T],L^p)
  \cap L^1([0,T],B^{2\alpha-1}_{\infty,\infty}),
  & \text{if }\frac12<\alpha\leq1,\quad 2<p<\infty.
\end{cases}
\]
In addition, we also have, with bounds independent of
$\varepsilon\in(0,1]$,
\begin{equation}\label{eq:u-in-LQ-unif}
  u\in L^\infty([0,T],L^Q),
\end{equation}
provided that either
\[
  \tfrac12<\alpha<1,
  \qquad
  \max\left\{p,\tfrac1{1-\alpha}\right\}<Q<\infty,
\]
or
\[
  \alpha=1,\qquad p<Q<\infty,\qquad
  \int_{\mathbb R^2}\theta_0\,\dd x=0,\qquad
  \int_{\mathbb R^2}|x|\,|\theta_0(x)|\,\dd x<\infty.
\]
In particular, if $p>2$, then $u\in L^1([0,T],B^{2\alpha}_{\infty,\infty})$.
\end{theorem}

For families of initial data uniformly bounded in the spaces
specified in Theorem~\ref{thm:main1}, the estimates imply,
for every fixed $T>0$,
\[
  \sup_{0<\varepsilon\leq1}
  \int_0^T
  \Vert \nabla u^\varepsilon(t)\Vert _{L^\infty}\,\dd t
  <\infty.
\]
This bound controls the Lipschitz constants of the flow maps
and their inverses uniformly on $[0,T]$,
providing a starting point for the patch regularity analysis.

For $\frac12<\alpha<1$, \eqref{eq:u-in-LQ-unif}
provides uniform velocity bounds in
$L^\infty([0,T],L^Q)$
for every finite $Q>\max\{p,\frac1{1-\alpha}\}$.
At $\alpha=1$, the same conclusion holds for every finite
$Q>p$ if each $\theta_0^\varepsilon$ has zero mean
and the first absolute moments are uniformly bounded.

In contrast, Remark~\ref{rmk:nonuniform-Lp-alpha1}
shows that, at $\alpha=1$, if the initial temperatures
are supported
in a fixed ball and
$\int_{\mathbb R^2}\theta_0^\varepsilon\,\dd x\to M\ne0$,
then, for every $T>0$ and $1<r<\infty$,
\[
  \Vert u^\varepsilon\Vert _{L^\infty([0,T],L^r)}
  \geq c \ \varepsilon^{-1/r}
\]
for all sufficiently small $\varepsilon$ and where $c$ depends on $T$ and $r$.
The growth rate is optimal for every finite $r\geq p$.
This obstruction to uniform global velocity bounds
is closely related to the two-dimensional Stokes paradox:
the logarithmic growth of the Stokes kernel $G_1$
in \eqref{def:G-alp=1} produces a logarithmic far field
for the associated Stokes velocity when the total
temperature is nonzero;
see~\cite[Chapter~V]{Galdi11}.

This distinction guides our choice of frame
for the infinite Prandtl number limit of patches:
we use the original frame for $\frac12<\alpha<1$
and for zero-mean data at $\alpha=1$,
and a moving frame for nonzero-mean data at $\alpha=1$.


\subsection{Persistence of the patch boundary}\label{subsec:intr-patch-bdr}
The vortex patch problem for the 2D Euler equations is
a classical free-boundary problem in incompressible fluid
dynamics that has attracted the attention of many researchers.
A central question, raised in the 1980s, is whether
an initially smooth patch boundaries remain regular for all time.
Chemin~\cite{Chem93} resolved this question using
paradifferential calculus and striated estimates,
proving global persistence of $C^{k,\gamma}$ boundaries
for every integer $k\geq1$ and $0<\gamma<1$.
Bertozzi and Constantin~\cite{BerC} gave an alternative
proof based on geometric cancellation in singular integrals.
A presentation of this approach for $C^{k,\gamma}$ patches
can be found in Majda and Bertozzi~\cite[Section~8.3]{MB02};
see also~\cite{Serf94,Radu22} for related approaches.
At the integer endpoints, Kiselev and Luo~\cite{KL23}
constructed initially $C^2$ patches whose boundaries
instantaneously lose $W^{2,\infty}$ regularity.

For the 2D Boussinesq system \eqref{BoussEq} with $\alpha=1$ and
$\mathrm{Pr}=1$, Danchin and Zhang~\cite{DZ17}
proved global persistence of $C^{1,\gamma}$ boundary
regularity for $0<\gamma<1$.
Gancedo and Garc\'ia-Ju\'arez~\cite{GGJ17}
established persistence of $C^{1,\gamma}$, $W^{2,\infty}$ and
$C^{2,\gamma}$ boundaries by exploiting cancellations
in time-dependent singular integrals.
For non constante temperature fronts with $\alpha=1$,
Chae, Miao, and Xue~\cite{CMX22} developed higher-order
striated estimates allowing them to prove $C^{k,\gamma}$ persistence
for every integer $k\geq1$ and $0<\gamma<1$,
they also dealt with the $W^{2,\infty}$ endpoint.
For the same range of $k$ and $\gamma$,
our previous work~\cite{LXY25} established
$C^{k,\gamma}$ persistence on $\mathbb T^2$,
with bounds uniform in $\mathrm{Pr}\in[1,\infty)$
on each fixed finite time interval.
The corresponding uniform estimates on $\mathbb R^2$
in the case $\alpha=1$ were not obtained there.
For the 3D case, we refer the reader to
\cite{GGJ20,LLX24,MY26}.
Global persistence is compatible with growth of boundary
norms as $t\to\infty$. As a matter of fact, Park~\cite{Park24} constructed
a configuration of two disjoint temperature patches
and established algebraic lower bounds for the growth
of their curvature and perimeter.

For $\frac12<\alpha<1$ and $\mathrm{Pr}=1$,
Khor and Xu~\cite{KX22} proved global persistence
of $C^{1,\gamma}$ boundaries for
$0<\gamma\leq2\alpha-1$ and initial temperature
$\theta_0=\mathds{1}_{D_0}$.
They also asked whether curvature can be controlled
in this fractional regime, as in the classical case
$\alpha=1$.
Our previous work~\cite{LXY25} established global
$C^{k,\gamma}$ persistence for every integer $k\geq1$
and $0<\gamma\leq2\alpha-1$, 
with estimates uniform in
$\mathrm{Pr}$ on both $\mathbb R^2$ and $\mathbb T^2$.
Taking $k=2$ gives bounded curvature on every finite
time interval and thus answers this question
for $C^{2,\gamma}$ initial boundaries.
These results do not, however, establish global persistence
for initial boundaries with only $W^{2,\infty}$ regularity.

Here, we establish global $C^{k,\gamma}$ boundary persistence
for $k=1,2$, $\frac12<\alpha\leq1$, and $0<\gamma\leq1$,
under the corresponding profile assumptions
and weaker assumptions on the initial velocity.
These results extend the previously established fractional
range beyond $\gamma=2\alpha-1$ and include the
$W^{2,\infty}$ and $W^{3,\infty}$ endpoints.
In particular, the $W^{2,\infty}$ result answers
the curvature question raised in~\cite{KX22}
for initial boundaries with only $W^{2,\infty}$ regularity.
Crucially, all boundary estimates on $\mathbb R^2$ are uniform
in $\mathrm{Pr}$ on each fixed finite time interval.

We now specify the geometric setting for the results
on the persistence of boundary regularity.
By Theorem~\ref{thm:main1}, the velocity generates a unique flow
$X_t\colon\mathbb R^2\to\mathbb R^2$ satisfying
\begin{equation}\label{eq:flowmap}
  \frac{\dd}{\dd t}X_t(y)=u(t,X_t(y)),
  \qquad X_0(y)=y.
\end{equation}
For every $t\geq0$, $X_t$ is a bi-Lipschitz homeomorphism,
and the transport equation preserves the patch form,
with $D(t)=X_t(D_0)$.

We write $\partial D(t)\in C^{k,\gamma}$ if it admits
a one-to-one parametrization
$z(\cdot,t)\in C^{k,\gamma}(\mathbb S^1)$ satisfying
$\inf_{s\in\mathbb S^1}|\partial_s z(s,t)|>0$ with
$\mathbb S^1=\mathbb R/\mathbb Z$.
Boundary regularity bounds are understood modulo translations:
the notation
$\partial D\in L^\infty([0,T],C^{k,\gamma})$
means that the parametrizations can be chosen so that
\[
  \operatorname*{ess\,sup}_{0<t<T}
  \left(
    \Vert z(\cdot,t)-z(0,t)\Vert _{C^{k,\gamma}(\mathbb S^1)}
    +\frac{1}{
      \inf_{s\in\mathbb S^1}|\partial_s z(s,t)|}
  \right)<\infty.
\]
This convention separates boundary regularity
from the translational motion of the patch.

For $k\in\{1,2\}$, choose a level-set function
$\varphi_0\in C^{k,\gamma}(\mathbb R^2)$
if $0<\gamma<1$, or
$\varphi_0\in W^{k+1,\infty}(\mathbb R^2)$
if $\gamma=1$, such that
\[
  D_0=\{\varphi_0>0\},\qquad
  \partial D_0=\{\varphi_0=0\},\qquad
  \inf_{\partial D_0}|\nabla\varphi_0|>0.
\]
See, for example,~\cite[p.~919]{KRYZ16}.
We choose $\varphi_0$ to be a negative constant
outside a sufficiently large ball, so that
$W_0:=\nabla^\perp\varphi_0$ has compact support.
Throughout, we use $\nabla^\perp=(\partial_2,-\partial_1)$,
consistent with $\omega=\partial_1u_2-\partial_2u_1$ and
$u=\nabla^\perp\Lambda^{-2}\omega$.
We fix
$x_*\in\partial D_0$.
Since $W_0$ is continuous, tangent to $\partial D_0$,
and nowhere zero there, reparametrizing arclength
in the direction of $W_0$ gives a periodic curve
$\widetilde z_0$ satisfying
\[
  \frac{\dd}{\dd a}\widetilde z_0(a)
  =W_0(\widetilde z_0(a)),
  \qquad \widetilde z_0(0)=x_*.
\]
Its minimal period is
$P_0=\int_{\partial D_0}\frac{\dd\ell}{|W_0|}\in(0,\infty)$,
where $\dd\ell$ denotes arclength.
Thus $z_0(s):=\widetilde z_0(P_0s)$ is one-periodic
and satisfies $\partial_s z_0=P_0W_0\circ z_0$.
Replacing $\varphi_0$ by $P_0\varphi_0$
and retaining the notation $\varphi_0$ and $W_0$,
we obtain the normalized parametrization
\[
  z_0\colon\mathbb S^1\longrightarrow\partial D_0,
  \qquad \partial_s z_0=W_0\circ z_0.
\]

Our first result on the persistence of temperature patch
boundary regularity is the following.
\begin{theorem}\label{thm:C1}
Let $\varepsilon =\frac{1}{\mathrm{Pr}} \in (0,1]$, $\alpha\in (\frac{1}{2},1]$ and $\gamma\in(0,1]$.
Assume that the initial data satisfy the following conditions.
\begin{itemize}
\item 
The initial temperature has the patch form
\begin{align}\label{eq:the-data-patch}
  \theta_0(x) = \bar{\theta}_0(x) \mathbf{1}_{D_0}(x),
\end{align}
where 
$D_0\subset \mathbb{R}^2$ is a bounded simply connected domain with boundary 
$\partial D_0\in C^{1,\gamma}$, $\gamma\in(0,1]$. 
For $\alpha\in(\frac{1}{2},1)$, assume that
\begin{align*}
  \bar{\theta}_0(x) \in
  \begin{cases}
    L^\infty(\overline{D_0}),\quad & \textrm{if}\;\;\gamma\in(0,2\alpha-1],\\
    C^{\gamma+1-2\alpha}(\overline{D_0}),
    & \textrm{if}\;\;  \gamma\in(2\alpha-1,1),\\
    C^{\mu}(\overline{D_0}),\,\;\textrm{for some}\;\mu \in (2-2\alpha,1),
    & \textrm{if}\;\; \,\gamma=1,
  \end{cases}
\end{align*}
while for $\alpha=1$, assume that
\begin{align*}
  \bar{\theta}_0(x) \in
  \begin{cases}
    L^\infty(\overline{D_0}),\quad & \textrm{if}\;\;\gamma\in(0,1),\\
    C^{\mu}(\overline{D_0}),\,\;\textrm{for some}\;\mu \in (0,1),
    & \textrm{if}\;\; \,\gamma=1.
  \end{cases}
\end{align*}
\item 
The initial velocity $u_0$ satisfies $\operatorname{div}u_0=0$ and
\begin{align*}
   u_0\in 
   \begin{cases}
    W^{1,p}(\mathbb{R}^2), \, \textrm{for some}\,\,2\leq p<\infty,
      \quad & \textrm{if}\;\gamma\in(0,2\alpha-1), \\
     W^{1,p}(\mathbb{R}^2), \, \textrm{for some}\,\,\frac{2}{2\alpha-\gamma}<p<\infty,
     \quad & \textrm{if}\,\,\gamma\in[2\alpha-1,1].
   \end{cases}
 \end{align*}
\end{itemize}
Then the unique global solution $(u,\theta)$
to the 2D Boussinesq-Navier-Stokes system \eqref{BoussEq} satisfies
\begin{equation}\label{eq:the-patch-exp}
  \theta(t, x) = \bar{\theta}_0(X_t^{-1}(x)) \mathbf{1}_{D(t)}(x),\qquad D(t)=X_t(D_0),
\end{equation}
where $X_t$ is the flow map generated by $u$
and $X^{-1}_t$ is its inverse. Moreover, for every $T>0$,
\begin{equation}\label{eq:reg-parD(t)-1-0-main}
  \partial D(t) \in L^{\infty}\left(0, T ; C^{1,\gamma}\right)\;\; \textrm{uniformly in $\varepsilon \in (0,1]$}.
\end{equation}
\end{theorem}

For the next derivative of the patch boundary, we impose additional
regularity on the profile and on the initial velocity in the tangent
direction $W_0$.
\begin{theorem}\label{thm:main2}
Let $\varepsilon=\frac{1}{\mathrm{Pr}}\in(0,1],\,\alpha\in (\frac{1}{2},1]$ and $\gamma\in(0,1]$.
Assume that the initial data satisfy the following conditions.
\begin{itemize}
\item The initial temperature has the patch form \eqref{eq:the-data-patch}, 
where $D_0\subset \mathbb{R}^2$ is a bounded simply connected domain with boundary 
$\partial D_0\in C^{2,\gamma}$. 
For $\alpha\in(\frac{1}{2},1)$, assume that
\begin{equation}\label{eq:bar-the-cd3}
  \bar{\theta}_0(x) \in
  \begin{cases}
    C^{2 + \gamma-2\alpha}(\overline{D_0}),\quad & \textrm{if}\;\;\gamma\in(0,2\alpha-1) \cup (2\alpha -1,1),\\
    C^{1+\widetilde{\mu}}(\overline{D_0}),\,\;\textrm{for some}\;\widetilde{\mu}\in(0,1), \quad
    & \textrm{if}\;\; \gamma = 2\alpha-1,\\
    C^{1+\mu}(\overline{D_0}),\,\;\textrm{for some}\;\mu\in(2-2\alpha,1), \quad
    & \textrm{if}\;\; \gamma = 1,
  \end{cases}
\end{equation}
while for $\alpha=1$, assume that
\begin{equation}\label{eq:bar-the-cd4}
  \bar{\theta}_0(x) \in
  \begin{cases}
    C^{\gamma}(\overline{D_0}),\quad & \textrm{if}\;\;\gamma\in(0,1),\\
    C^{1+\mu}(\overline{D_0}),\,\;\textrm{for some}\;\mu\in(0,1), \quad
    & \textrm{if}\;\; \gamma = 1.
  \end{cases}
\end{equation}
\item The initial velocity satisfies $\operatorname{div}u_0=0$ and
\begin{align*}
  \textrm{$u_0 \in W^{1,q}(\mathbb{R}^2)$\;\; for some $\tfrac{2}{2\alpha-1}<q<\infty$, \qquad
  $\partial_{W_0}u_0\in \dot{W}^{1,p}(\mathbb{R}^2)$\;\; for some $\tfrac{2}{2\alpha-\gamma}<p<\infty$}.
\end{align*}
\end{itemize}
Then the unique global solution $(u,\theta)$
to the 2D Boussinesq system \eqref{BoussEq} satisfies \eqref{eq:the-patch-exp}
and, for every $T>0$,
\begin{equation}\label{eq:reg-parD(t)-1-main}
  \partial D(t) \in L^{\infty}\big(0, T ; C^{2,\gamma}\big)\quad \textrm{uniformly in $\varepsilon \in (0,1]$}.
\end{equation}
\end{theorem}

For $\frac12<\alpha<1$, Theorem~\ref{thm:C1}
improves on~\cite{KX22} by allowing $u_0\in H^1$
for $0<\gamma<2\alpha-1$ and removing the additional
$H^1$ condition for $u_0\in W^{1,p}$, $2<p<\infty$,
throughout $0<\gamma\leq2\alpha-1$.
The additional $H^1$ condition is also removed
at the classical $W^{2,\infty}$ endpoint
while it was a required conditon in~\cite{CMX22}.
Both theorems also improve on our previous
work~\cite{LXY25}; see Table~\ref{tab:comparison-LXY25}
for a detailled comparison.

\medskip
\noindent\textbf{Comparison with Euler vortex patches.}
For $k=1,2$ and $0<\gamma<1$, both Euler vortex patches
and the Boussinesq temperature patches considered here preserve
$C^{k,\gamma}$ boundary regularity.
Their behavior differs at the Lipschitz endpoints:
Kiselev and Luo~\cite{KL23} constructed initially $C^2$
Euler patches whose boundaries instantaneously lose
$W^{2,\infty}$ regularity, whereas the Boussinesq
temperature patches considered here preserve both
$W^{2,\infty}$ and $W^{3,\infty}$ boundary regularity.

The velocity laws help explain this difference.
Let $\omega_E$ and $\omega_B$ denote the Euler
and Boussinesq vorticities, respectively, and set $\mathcal R_{1-2\alpha}:=\partial_1\Lambda^{-2\alpha},\,\Gamma:=\omega_B-\mathcal R_{1-2\alpha}\theta$. 
Then
\begin{align}
  \nabla u_E
  &=\nabla\nabla^\perp\Lambda^{-2}\omega_E,
  \notag\\
  \nabla u_B
  &=\nabla\nabla^\perp\Lambda^{-2}\Gamma
    +\nabla\nabla^\perp\partial_1
      \Lambda^{-2-2\alpha}\theta.
  \label{eq:subcritical}
\end{align}
The Euler  vorticity patch enters the velocity gradient
through an operator of order zero, whereas the temperature
contribution to $\nabla u_B$ has order $1-2\alpha<0$.
The remaining Boussinesq term acts on the good unknown $\Gamma$,
whose evolution has a dissipatif structure.

For $C^{1,\gamma}$ Euler vortex patches, $0<\gamma<1$,
one has $\nabla u_E\in L^\infty([0,T],L^\infty)$.
For Boussinesq temperature patches, under the corresponding
data assumptions,
$\nabla u_B\in L^1([0,T],C^{2\alpha-1})$
when $\frac12<\alpha<1$~\cite{KX22,LXY25},
and $\nabla u_B\in L^1([0,T],W^{1,\infty})$
when $\alpha=1$~\cite{GGJ17,CMX22},
for every finite $T>0$.
The latter estimate relies on the geometric cancellation
lemma of Bertozzi and Constantin~\cite{BerC}
and gives $W^{2,\infty}$ boundary persistence.

Our endpoint proofs combine dissipative estimates
for $\Gamma$ with tangential commutator bounds,
using the gain of $2\alpha-1$ derivatives
to close the boundary estimates throughout
$\frac12<\alpha\leq1$.
This positive gain disappears at $\alpha=\frac12$,
which is not covered by the present argument.

Table~\ref{tab:Euler-Bouss-subcritical} summarizes
the comparison.

\begin{center}
\refstepcounter{table}\label{tab:Euler-Bouss-subcritical}
\begin{minipage}{\textwidth}
\centering
\small
\textbf{Table~\thetable.}
2D Euler vortex patches and Boussinesq temperature patches
with $\frac12<\alpha\leq1$.
\par\medskip
\renewcommand{\arraystretch}{1.22}
\begin{tabular}{@{}
  >{\raggedright\arraybackslash}m{0.16\textwidth}|
  >{\raggedright\arraybackslash}m{0.35\textwidth}|
  >{\raggedright\arraybackslash}m{0.37\textwidth}@{}}
\hline
& \textbf{Euler vortex patches}
& \textbf{Boussinesq temperature patches}
\\ \hline
Velocity gradient operator
& Order zero on the patch vorticity $\omega_E$.
& Order zero on $\Gamma$, and order $1-2\alpha<0$
  on the temperature $\theta$.
\\ \hline
Velocity gradient regularity
& $\nabla u_E\in L^\infty_T(L^\infty)$
  for $C^{1,\gamma}$ patches,
  $0<\gamma<1$~\cite{Chem93,BerC}.
& $\nabla u_B\in L^1_T(W^{1,\infty})$
  for $\alpha=1$~\cite{GGJ17,CMX22};
  $\nabla u_B\in L^1_T(C^{2\alpha-1})$
  for $\frac12<\alpha<1$ \cite{KX22,LXY25}.
\\ \hline
$0<\gamma<1$
& Global $C^{1,\gamma}$ and $C^{2,\gamma}$
  boundary persistence~\cite{Chem93,BerC,Serf94,Radu22}.
& Global $C^{1,\gamma}$ and $C^{2,\gamma}$
  boundary persistence.
\\ \hline
Lipschitz endpoints
& Some $C^2$ boundaries instantaneously lose
  $W^{2,\infty}$ regularity~\cite{KL23}.
& Global $W^{2,\infty}$ and $W^{3,\infty}$
  boundary persistence.
\\ \hline
\end{tabular}

\par\smallskip
{\footnotesize\raggedright
The Boussinesq boundary persistence results listed here
follow from Theorems~\ref{thm:C1} and~\ref{thm:main2},
with bounds uniform in $\mathrm{Pr}\in[1,\infty)$
on each fixed finite time interval.
For earlier results, see~\cite{DZ17,GGJ17,CMX22,KX22,LXY25}.
\par}
\end{minipage}
\end{center}

\subsection{The infinite Prandtl number limit on the whole space}

Formally letting $\varepsilon\to0$ in \eqref{BoussEq}
eliminates the inertial term and gives the (fractional)
Stokes-transport system
\begin{equation}\tag{ST$_\alpha$}\label{eq:ST}
  \left\{
  \begin{aligned}
    \partial_t\theta+u\cdot\nabla\theta
      &=0,
      \qquad (t,x)\in\mathbb R_+\times\mathbb R^2,\\
    \Lambda^{2\alpha}u
      &=-\nabla p+\theta e_2,\\
    \nabla\cdot u&=0,\\
    \left.\theta\right|_{t=0}&=\theta_0.
  \end{aligned}
  \right.
\end{equation}

For $\alpha=1$, $\Lambda^{2\alpha}=-\Delta$, and \eqref{eq:ST} is a two-dimensional analogue of continuum
models for particle sedimentation in viscous fluids~\cite{Hofer18,Mech19,HS25}.
In these models, the particle density supplies
the gravitational forcing in the Stokes equations
and, after removal of a constant settling drift
by a change of reference frame,
is transported by the resulting velocity.
Grayer II~\cite{Gray23} motivated \eqref{eq:ST} with $\alpha=1$
by the formal infinite Prandtl number limit.
He proved global well-posedness on $\mathbb R^2$
for bounded, compactly supported initial temperatures
and global persistence of $W^{2,\infty}$ and $C^{k,\gamma}$
boundary regularity for patches with constant profiles,
where $k=1,2$ and $0<\gamma<1$.
Leblond~\cite{Leb22} established global well-posedness
for bounded initial densities in bounded domains
of $\mathbb R^2$ and $\mathbb R^3$,
and in the infinite strip $\mathbb R\times(0,1)$.
Asymptotic stability of stratified states for the
Stokes-transport system in a periodic channel
was studied in~\cite{DGL23,Park25}.
Related results on the regularity and long-time dynamics
of density interfaces in gravity-driven Stokes flows
appear in~\cite{AMY00,GGBS23,GGBS24}.
More recently, Gancedo et al.~\cite{GGHSY26}
established infinite-time growth of interface length
or curvature in unstable configurations.

For the fractional Stokes-transport system \eqref{eq:ST},
Cobb~\cite{Cobb23} established well-posedness results
in critical function spaces in dimensions $d\geq2$.
In dimension two, his global existence and uniqueness
theorem applies to
$\frac12<\alpha<1$ and
$\theta_0\in L^p\cap L^{2/(2\alpha-1)}$,
where $1\leq p<1/\alpha$.

Our previous work~\cite{LXY25} justified
the infinite Prandtl number limit on $\mathbb T^2$
for $\frac12<\alpha\leq1$, with global persistence
of $C^{k,\gamma}$ boundary regularity
for the limiting temperature patches governed by
\eqref{eq:ST}.
Under the stated assumptions on the initial data,
this result holds for every integer $k\geq1$,
with $0<\gamma\leq2\alpha-1$ when $\frac12<\alpha<1$
and $0<\gamma<1$ when $\alpha=1$.

The present paper establishes the corresponding limit
on $\mathbb R^2$ throughout the subcritical range
$\frac12<\alpha\leq1$, preserving
$C^{k,\gamma}$ patch boundary regularity
for $k=1,2$ and $0<\gamma\leq1$,
under the stated initial data assumptions.

On $\mathbb T^2$, Poincar\'e's inequality controls
the velocity through its gradient and spatial mean.
On $\mathbb R^2$, however, uniform gradient estimates
alone do not yield uniform global $L^p$ velocity bounds.
The uniform $L^Q$ bound in \eqref{eq:u-in-LQ-unif}
allows convergence in the original frame for
$\frac12<\alpha<1$ and for zero-mean temperature patches
at $\alpha=1$.
In both cases, global $L^Q$ integrability fixes
the normalization of the limiting (fractional) Stokes velocity.
For nonzero-mean temperature patches at $\alpha=1$,
we use a time-dependent change of frame, corresponding
to a generalized Galilean transformation, to obtain
uniform local velocity bounds and fix the normalization
of the limit.
This use of time-dependent translations is reminiscent
of the accelerated change of reference frame in
Kelliher's characterization of bounded Euler flows
at infinity~\cite[Eq.~(1.5) and Lemma~6.1]{Kel15}.
Related transformations appear in the study of weak
Navier-Stokes solutions by Kukavica and Vicol~\cite{KV08},
while Cobb~\cite{Cobb24} characterizes far-field conditions
ensuring equivalence between the original and projected
Euler equations.

To state the limit theorem, we fix the velocity normalization $u=G_\alpha*\theta$. 
For $\frac12<\alpha<1$, $G_\alpha$ is the decaying fractional
Stokes kernel associated with the Fourier multiplier
\[
  \widehat{G_\alpha*\theta}(\xi)
  =
  |\xi|^{-2\alpha}
  \left(\mathrm{Id}-\frac{\xi\otimes\xi}{|\xi|^2}\right)
  e_2\,\widehat\theta(\xi),
  \qquad \xi\ne0.
\]
Its physical-space expression is
\begin{equation}\label{def:G-alp-frac}
  G_\alpha(x)
  =
  \frac{c_\alpha}{|x|^{2-2\alpha}}
  \begin{pmatrix}
    (2-2\alpha)\dfrac{x_1x_2}{|x|^2}\\[2mm]
    1-(2-2\alpha)\dfrac{x_1^2}{|x|^2}
  \end{pmatrix},
  \qquad x\ne0,
\end{equation}
where
$c_\alpha = \frac{\Gamma(1-\alpha)}{\alpha\,2^{2\alpha+1}\pi\,\Gamma(\alpha)}$.
For $\frac12<\alpha<1$ and $\theta\in L^1\cap L^\infty$,
the convolution $G_\alpha*\theta$ is the unique
fractional Stokes velocity in $\mathcal S'_h$,~\cite[Proposition~2.1]{Cobb23}.
At the endpoint $\alpha=1$, we use the logarithmic Stokes kernel
with its additive constant fixed as
in Grayer II~\cite[Section~2]{Gray23}:
\begin{equation}\label{def:G-alp=1}
  G_1(x)
  =
  \frac{1}{4\pi}
  \begin{pmatrix}
    \dfrac{x_1x_2}{|x|^2}\\[2mm]
    \dfrac12-\log|x|-\dfrac{x_1^2}{|x|^2}
  \end{pmatrix},
  \qquad x\ne0.
\end{equation}
For the bounded, compactly supported patch temperatures
considered here, $G_1*\theta$ is well defined.
In both cases, the normalization $u=G_\alpha*\theta$
fixes the reference frame of the limiting dynamics;
see Remark~\ref{rmk:norm-velocity}.

With this normalization, we now state our result
on the infinite Prandtl number limit
for temperature patches on $\mathbb R^2$.
\begin{theorem}\label{thm:infinite-Prandtl}
Let $\frac12<\alpha\leq1$, $k\in\{1,2\}$,
and $0<\gamma\leq1$.
Assume that $\theta_0(x)=\bar\theta_0(x)\mathds{1}_{D_0}(x)$ 
satisfies the assumptions
of Theorem~\ref{thm:C1} when $k=1$,
and of Theorem~\ref{thm:main2} when $k=2$.
In particular, $D_0\subset\mathbb R^2$ is a bounded,
simply connected domain with
$\partial D_0\in C^{k,\gamma}$.

For each $\varepsilon=\mathrm{Pr}^{-1}\in(0,1]$,
let $u_0^\varepsilon$ be divergence-free.
Assume that the following bounds hold for exponents
independent of $\varepsilon$:
\begin{itemize}
\item
If $k=1$, there exists
$\max\left\{2,\frac{2}{2\alpha-\gamma}\right\}<p<\infty$
such that
\[
  \sup_{0<\varepsilon\leq1}
  \Vert u_0^\varepsilon\Vert _{W^{1,p}(\mathbb R^2)}
  <\infty.
\]

\item
If $k=2$, there exist
$\frac{2}{2\alpha-\gamma}<p<\infty$ and $\frac{2}{2\alpha-1}<q<\infty$
such that
\[
  \sup_{0<\varepsilon\leq1}
  \left(
    \Vert u_0^\varepsilon\Vert _{W^{1,q}(\mathbb R^2)}
    +
    \Vert \partial_{W_0}u_0^\varepsilon
      \Vert _{\dot W^{1,p}(\mathbb R^2)}
  \right)
  <\infty,
\]
where $W_0=\nabla^\perp\varphi_0$
is the tangent vector field 
fixed above.
\end{itemize}

Let $(u^\varepsilon,\theta^\varepsilon)$
be the corresponding unique global solution
of \eqref{BoussEq} with initial data
$(u_0^\varepsilon,\theta_0)$.
Then, for every $T>0$, the following statements hold for the whole family as
$\varepsilon\to0$.
\begin{enumerate}[(i)]
\item
If $\frac12<\alpha<1$, or both $\alpha=1$ and
$\int_{\mathbb R^2}\theta_0\,\dd x=0$, the original solutions
$(u^\varepsilon,\theta^\varepsilon)$ converge to $(u,\theta)$.
Set $a^\varepsilon\equiv0$ and
$(U^\varepsilon,\Theta^\varepsilon)
:=(u^\varepsilon,\theta^\varepsilon)$.

\item
If $\alpha=1$ and $\int_{\mathbb R^2}\theta_0\,\dd x\ne0$,
there exist absolutely continuous paths
$a^\varepsilon:[0,T]\to\mathbb R^2$,
with $a^\varepsilon(0)=0$, such that the transformed fields
\begin{equation}\label{def:Theta-U-eps}
  \begin{aligned}
    \Theta^\varepsilon(t,x)
    &:=\theta^\varepsilon(t,x+a^\varepsilon(t)),\\
    U^\varepsilon(t,x)
    &:=u^\varepsilon(t,x+a^\varepsilon(t))
       -\dot a^\varepsilon(t)
  \end{aligned}
\end{equation}
converge to $(u,\theta)$.
\end{enumerate}
In both cases, $(u,\theta)$ is the unique normalized global
weak solution of \eqref{eq:ST} with initial temperature
$\theta_0$ and
\begin{equation}\label{eq:norm-u-exp-ST}
  u(t,x)=G_\alpha*\theta(t,x).
\end{equation}
Here $G_\alpha$ is given by
\eqref{def:G-alp-frac} and \eqref{def:G-alp=1}.
More precisely, for every $1\leq r<\infty$,
\[
  \Theta^\varepsilon\longrightarrow\theta
  \qquad\text{strongly in }
  C\bigl([0,T];L^r(\mathbb R^2)\bigr),
\]
and, for every $R>0$,
\begin{align*}
  U^\varepsilon
  &\rightharpoonup^\ast u
  \qquad\text{in }
  L^\infty\bigl(0,T;W^{1,p}(B_R)\bigr),
  \qquad k=1,\\
  U^\varepsilon
  &\rightharpoonup^\ast u
  \qquad \text{in }
  L^\infty\bigl(0,T;W^{1,q}(B_R)\bigr),
  \qquad k=2.
\end{align*}

Moreover, if $X_t$ denotes the flow generated by $u$,
then the limiting temperature retains the patch form
\[
  \theta(t,x)=\theta_0\bigl(X_t^{-1}(x)\bigr),
  \qquad D(t)=X_t(D_0),
\]
and its boundary satisfies $\partial D\in L^\infty\bigl(0,T;C^{k,\gamma}\bigr)$. 
\end{theorem}
 
The limiting patch retains the full $C^{1,\gamma}$ and $C^{2,\gamma}$
boundary regularity, including $\gamma=1$.
In particular, for $\alpha=1$, we obtain global persistence
of $W^{3,\infty}$ boundaries, extending
the $C^{2,\gamma}$ result of Grayer II~\cite{Gray23},
$0<\gamma<1$, to the Lipschitz endpoint.

\begin{remark}\label{rmk:norm-velocity}
The normalization fixes the reference frame
of the limiting dynamics.
For any $a\in C^1([0,T];\mathbb R^2)$ with $a(0)=0$,
the transformation
\[
  (u,\theta,p)(t,x)\longmapsto
  \bigl(
    u(t,x+a(t))-\dot a(t),\,
    \theta(t,x+a(t)),\,
    p(t,x+a(t))
  \bigr)
\]
preserves \eqref{eq:ST} and its initial temperature:
the translation terms cancel in the transport equation,
and $\Lambda^{2\alpha}$ annihilates spatial constants.
Since convolution with $G_\alpha$ commutes with translations,
the normalization $u=G_\alpha*\theta$ is preserved
only if $\dot a=0$, hence $a\equiv0$.

For $a\in C^2([0,T];\mathbb R^2)$,
the same change of frame preserves
the Boussinesq equations in \eqref{BoussEq}
with the transformed pressure
\[
  P^\varepsilon(t,x)
  =
  p^\varepsilon(t,x+a(t))
  +\varepsilon\ddot a(t)\cdot x.
\]
Thus the frame acceleration is absorbed into the pressure.
For the patch data in Theorem~\ref{thm:infinite-Prandtl},
we take $a^\varepsilon\equiv0$ when $\frac12<\alpha<1$,
or when $\alpha=1$ and
$\int_{\mathbb R^2}\theta_0\,\dd x=0$.
In the latter case, zero total temperature cancels
the logarithmic far field, giving
$G_1*\theta=O(|x|^{-1})$ for compactly supported temperatures.
The uniform bound \eqref{eq:u-in-LQ-unif}
and the kernel estimates therefore place both
the limiting velocity $u$ and $G_\alpha*\theta$
in $L^Q$ for a suitable finite $Q$,
excluding a nonzero spatially constant difference.
For nonzero total temperature at $\alpha=1$,
the averaged normalization in Section~\ref{sec:limit}
provides uniform local velocity bounds in the moving frame
and fixes the normalization of the limiting velocity.

For related well-posedness results for the limiting
system \eqref{eq:ST}, see Cobb~\cite{Cobb23}
for $\frac12<\alpha<1$ and
Grayer II~\cite[Theorem~3]{Gray23} for $\alpha=1$.
Under the normalization $u=G_\alpha*\theta$,
Lemma~\ref{lem:ST-uniqueness} provides a unified,
self-contained Lagrangian proof of uniqueness
in the solution class considered here.
At $\alpha=1$, this class requires the temperature
support to remain bounded on every finite time interval.
\end{remark}

The present paper extends the results of our previous work
\cite{LXY25} in several directions.
For the reader's convenience, Table~\ref{tab:comparison-LXY25}
compares the two works and highlights the main extensions
obtained here.

\begin{center}
\refstepcounter{table}\label{tab:comparison-LXY25}
\begin{minipage}{\textwidth}
\centering
\small
\textbf{Table~\thetable.}
Comparison with our previous work \cite{LXY25}.
\par\medskip
\renewcommand{\arraystretch}{1.22}
\begin{tabular}{@{}
  >{\raggedright\arraybackslash}m{0.19\textwidth}|
  >{\raggedright\arraybackslash}m{0.33\textwidth}|
  >{\raggedright\arraybackslash}m{0.36\textwidth}@{}}
\hline
& \textbf{Previous work \cite{LXY25}}
& \textbf{This paper}
\\ \hline
Spatial setting
& $\mathbb{R}^2$ and $\mathbb{T}^2$.
& Results proved on $\mathbb{R}^2$;
  the arguments can also extend to $\mathbb{T}^2$.
\\ \hline
Boundary regularity
& $C^{k,\gamma}$ for every integer $k\geq1$, with
  $0<\gamma\leq2\alpha-1$ if $\frac12<\alpha<1$,
  and $0<\gamma<1$ if $\alpha=1$.
& $C^{k,\gamma}$ for $0<\gamma\leq1$ and $\frac12<\alpha\leq1$.
The present paper treats $k=1,2$; the general case $k\geq3$
will be addressed in a forthcoming paper.
\\ \hline
Initial velocity assumptions
for $C^{2,\gamma}$ persistence
& $u_0\in H^1\cap\dot W^{1,p},\,\partial_{W_0}u_0\in W^{1,p}$ for some $p>\frac{2}{2\alpha-1}$.
& $u_0\in W^{1,q}$ for some $q>\frac{2}{2\alpha-1}$, $\partial_{W_0}u_0\in\dot W^{1,p}$ for some $p>\frac{2}{2\alpha-\gamma}$.
\\ \hline
Boundary estimates uniform in $\mathrm{Pr}$
& On both $\mathbb{R}^2$ and $\mathbb{T}^2$
  for $\frac12<\alpha<1$;
  only on $\mathbb{T}^2$ for $\alpha=1$.
& On $\mathbb{R}^2$ throughout
  $\frac12<\alpha\leq1$.
\\ \hline
Infinite Prandtl number limit
& Established only on $\mathbb{T}^2$ in the original frame.
& 
Established on $\mathbb R^2$ in the original frame
for $\frac12<\alpha<1$ and for zero-mean patches when $\alpha=1$;
in a moving frame for nonzero-mean patches when $\alpha=1$.
\\ \hline
\end{tabular}
\end{minipage}
\end{center}

\subsection{Main ideas of the proof and organization}

Following~\cite{HKR10,HKR11}
and subsequent applications to temperature patches
in~\cite{CMX22,KX22,LXY25},
we introduce the so-called good unknown
\[
  \Gamma:=\omega-\mathcal R_{1-2\alpha}\theta,
  \qquad
  \mathcal R_{1-2\alpha}:=\partial_1\Lambda^{-2\alpha}.
\]
where $\omega=\partial_1u_2-\partial_2u_1$
is the vorticity.
It satisfies the following nice PDE
\begin{equation}\label{eq:Gamma-intr}
  \varepsilon\bigl(\partial_t\Gamma+u\cdot\nabla\Gamma\bigr)
  +\Lambda^{2\alpha}\Gamma
  =
  \varepsilon
  [\mathcal R_{1-2\alpha},u\cdot\nabla]\theta,
\end{equation}
which replaces the forcing $\partial_1\theta$
in the vorticity equation with a commutator.

A important improvement in Theorem~\ref{thm:main1}
in~\cite{LXY25} is the uniformity of the gradient estimates
in $\varepsilon$ (recall that $\varepsilon=\frac{1}{Pr}$) at $\alpha=1$ on $\mathbb R^2$.
In~\cite{LXY25}, the treatment of low frequencies
in the commutator estimates introduces zero-order
velocity norms, whose available bounds depend on $\varepsilon$.
The commutator estimates \eqref{eq:comm-es1}
and~\eqref{eq:R-1cm-es1} control
$[\mathcal R_{-1},u\cdot\nabla]\theta$
using only norms of $\nabla u$ and $\theta$,
allowing the gradient estimates to close
uniformly in $\varepsilon$.
There is a specific difficulty when $p=2$. Indeed, 
$\omega_0\in L^2$ does not imply $\Gamma_0\in L^2$,
since $\mathcal R_{-1}\theta_0$ need not belong to  $L^2$ necessairly.
We therefore decompose
\[
  \Gamma=\Gamma^{(1)}+\Gamma^{(2)},\qquad
  \Gamma^{(1)}|_{t=0}=\omega_0,\qquad
  \Gamma^{(2)}|_{t=0}=-\mathcal R_{-1}\theta_0,
\]
as in \eqref{eq:Gamma1}-\eqref{eq:Gamma2},
and estimate the components in $L^2$ and $L^q$,
respectively, for some $q>2$.
For $p>2$, direct $L^p$ estimates for $\Gamma$
avoid the additional $H^1$ assumption on $u_0$.
Together with frequency decomposition and dissipatif
smoothing, these arguments yield the wanted \textit{{\it{a priori}}} estimates
needed to prove existence in Theorem~\ref{thm:main1}.
Uniqueness follows from the smoothing estimate
for the linearized velocity equation,
combined with Besov estimates and Osgood's lemma (see Proposition~\ref{prop:uniq}).

The boundary estimates are expressed through
the transported tangent field
\[
  \varphi:=\varphi_0\circ X_t^{-1},
  \qquad W:=\nabla^\perp\varphi,
  \qquad \partial_W:=W\cdot\nabla,
\]
where $\varphi_0$ is the initial level-set function.
The tangent field $W$ satisfies
\begin{equation}\label{eq:W-eq}
  \partial_tW+u\cdot\nabla W=\partial_Wu,
  \qquad W|_{t=0}=W_0.
\end{equation}
For $z(a,t):=X_t(z_0(a))$, with
$\partial_a z_0=W_0\circ z_0$, we have
\begin{equation}\label{eq:boundary-identities}
  \partial_a z=W(t,z),\qquad
  \partial_a^2z=(\partial_WW)(t,z),\qquad
  \partial_a^3z=(\partial_W^2W)(t,z).
\end{equation}
See also~\cite[Equations~(4.2)-(4.6)]{CMX22}.
The uniform integral bound on $\Vert \nabla u\Vert _{L^\infty}$
preserves tangent nondegeneracy.
Thus $C^{1,\gamma}$ regularity persistence follows from controlling
$W$ in $L^\infty([0,T],C^{0,\gamma})$,
while the $C^{2,\gamma}$ regularity persistence additionally requires
the same control of $\partial_WW$.
We identify $C^{0,1}$ with $W^{1,\infty}$,
so this formulation include both Lipschitz endpoints.
The approximation arguments in the appendices
justify these identities for the limiting solutions.

When $\frac12<\alpha<1$ and $0<\gamma\leq2\alpha-1$,
or when $\alpha=1$ and $0<\gamma<1$,
Theorem~\ref{thm:main1} gives
$\nabla u\in L^1([0,T],C^\gamma)$
uniformly in $\varepsilon$,
under the stated velocity assumptions.
Standard flow estimates then give
$C^{1,\gamma}$ boundary persistence
without using the patch structure. For $C^{2,\gamma}$ persistence in this range,
we adapt the striated estimates of~\cite{LXY25}
to the weaker assumptions on the initial velocity.
Lemma~\ref{lem:pGamma} controls $\partial_W\Gamma$,
providing the estimate of $\partial_W^2u$
needed to propagate $\partial_WW$.

The new boundary arguments deals with the case
$\gamma>2\alpha-1$ and the endpoint $\gamma=1$.
Tangency ensures that $\partial_W\theta$
contains no boundary measure from the temperature jump,
while
$[\partial_t+u\cdot\nabla,\partial_W]=0$
shows that $\partial_W\theta$ is transported by $u$.
We exploit this structure through the (striated) commutator
estimates in Lemmas~\ref{lem:pW}
and~\ref{lem:pW-Cgam}, inspired by the kernel
representation of tangential derivatives
in Bertozzi and Constantin~\cite{BerC}.

For Theorem~\ref{thm:C1}, these lemmas control
the temperature contributions
\begin{align*}
  [\nabla^2\Lambda^{-2-2\alpha},\partial_W]\theta
  &\quad\text{in }C^\gamma,
  \,\,\,\,\,2\alpha-1<\gamma<1,\\
  [\nabla^3\Lambda^{-2-2\alpha},\partial_W]\theta
  &\quad\text{in }L^\infty,
   \,\,\,\,\,\gamma=1.
\end{align*}
Together with the bounds for $\Gamma$
and the transported striated regularity of $\theta$,
they yield
$\partial_Wu\in L^1([0,T],C^\gamma)$
and $\partial_W\nabla u\in L^1([0,T],L^\infty)$,
respectively.
The equations for $W$ and $\nabla W$
then close the required estimates uniformly
in $\varepsilon$.

For Theorem~\ref{thm:main2}, the remaining cases require stronger Besov estimates
for $\partial_W\Gamma$.
Tangential differentiation of \eqref{eq:Gamma-intr}
produces
\[
  [\Lambda^{2\alpha},\partial_W]\Gamma,
  \qquad
  \varepsilon\partial_W
  \bigl([\mathcal R_{1-2\alpha},u\cdot\nabla]\theta\bigr).
\]
Lemma~\ref{lem:CMX2.5} handles the first commutator,
while Lemmas~\ref{lem:pW} and~\ref{lem:pW-Cgam}
control the second.
Together with dissipative smoothing, these estimates yield the uniform Besov bounds
in Lemma~\ref{lem:pGamma-high}.
For $2\alpha-1<\gamma<1$, they provide
$\partial_W^2u\in L^1([0,T],C^\gamma)$,
which closes the transport estimate
for $\partial_WW$ in \eqref{pWW}, with a bound independant of $\varepsilon$.
At $\gamma=1$, the bound
$\partial_W\Gamma\in L^1([0,T],B^{\gamma'}_{\infty,1})$
for some $\gamma'>1$, together with the temperature
and lower order estimates, controls $\partial_W^2\nabla u$
in $L^1([0,T],L^\infty)$.
The equation for $\partial_W\nabla W$ then gives
$\partial_WW\in L^\infty([0,T],W^{1,\infty})$ uniformly in $\varepsilon$,
and hence $W^{3,\infty}$ boundary persistence.

For the infinite Prandtl number limit,
Proposition~\ref{prop:weak-limit-R2}
and Corollary~\ref{cor:strong-limit-R2} first establish convergence without assuming a patch structure.
Under the stated uniform bounds on the initial velocities,
they apply to initial temperatures uniformly bounded
in $L^1\cap L^\infty$ and converging strongly in $L^2$.
At $\alpha=1$, the initial temperature supports
must additionally lie in a common bounded set.

Estimate \eqref{eq:u-LQ-es} gives uniform global $L^Q$
velocity bounds for finite
$Q>\max\{p,1/(1-\alpha)\}$ when $\alpha<1$,
and for finite $Q>p$ when $\alpha=1$ and
$\int_{\mathbb R^2}\theta_0^\varepsilon\,\dd x=0$
for every $\varepsilon$.
In the latter case, the support and $L^1$ bounds
control the first absolute moments.
We therefore take $a^\varepsilon\equiv0$ in these cases.
For the remaining families at $\alpha=1$,
we choose $a^\varepsilon$ by \eqref{def:a-eps}
to enforce \eqref{eq:averaged-normalization}.
Together with the gradient and kernel estimates,
these choices yield uniform local bounds for
$(U^\varepsilon,\Theta^\varepsilon)$
defined in \eqref{def:Theta-U-eps}.
If the initial temperatures are supported in a common bounded set,
the supports of $\Theta^\varepsilon$ remain uniformly bounded
on each finite time interval.

These bounds, together with compactness from the transport
equation, allows us to pass to the limite.
The limiting (fractional) Stokes relation and
the sublinear growth bound determine the velocity
up to a spatially constant vector:
\[
  u=G_\alpha*\theta+c(t),
  \qquad c(t)\in\mathbb R^2.
\]
Global $L^Q$ integrability forces $c(t)=0$
in the original-frame cases,
while the averaged normalization gives the same conclusion
in the moving-frame case.
The flow representation, conservation of the $L^2$ norms,
and strong convergence of the initial data then upgrade
the local distributional convergence to strong convergence
in $C([0,T];L^2(\mathbb R^2))$.
Interpolation gives convergence in $C([0,T];L^r(\mathbb R^2))$
for every $1<r<\infty$.

Finally, we pass the uniform boundary estimates to the limit
using the translated tangent fields
\[
  \widetilde W^\varepsilon(t,x)
  :=W^\varepsilon(t,x+a^\varepsilon(t)).
\]
Their evolution equations and uniform spatial bounds
yield strong local compactness at lower regularity,
allowing us to identify the limiting tangent field
and its striated derivatives.
The H\"older bounds pass to the limit by lower semicontinuity,
while the Lipschitz bounds follow from weak-$*$ convergence
of the corresponding spatial derivatives.
Together with the flow representation
and uniform tangent nondegeneracy,
these bounds preserve the full boundary regularity,
including the endpoint $\gamma=1$,
and complete the proof of
Theorem~\ref{thm:infinite-Prandtl}.

The rest of the paper is organized as follows. 
Section~\ref{sec:tools} recalls the Besov and 
Chemin-Lerner notation and presents the commutator 
and striated estimates, together with estimates 
for the transport, transport-diffusion, 
and linearized velocity equations. 
Section~\ref{sec:well-pose} proves 
Theorem~\ref{thm:main1}, establishing global well-posedness 
and gradient bounds 
uniform in $\mathrm{Pr}$, together with uniform global $L^Q$ velocity bounds
under the stated conditions.
Sections~\ref{sec:C1} and~\ref{sec:C2} 
prove Theorems~\ref{thm:C1} and~\ref{thm:main2}, respectively, 
concerning the persistence of $C^{1,\gamma}$ and $C^{2,\gamma}$ 
patch boundary regularity, including the Lipschitz endpoints. 
Section~\ref{sec:limit} justifies 
the infinite Prandtl number limit on $\mathbb R^2$ 
and contains the proof of Theorem~\ref{thm:infinite-Prandtl}. 
Appendix~\ref{sec:A} constructs compatible smooth 
approximations of the initial data, 
and Appendix~\ref{sec:B} justifies passage to the limit 
in the regularized equations and preserves 
the uniform striated bounds through compactness.

\vskip1.5mm
\textbf{Notation.}
The following notation will be used throughout the paper.
\begin{itemize}

\item
For an open set $\Omega\subset\mathbb{R}^{d}$, we denote by
$\mathcal{D}(\Omega):=C_{c}^{\infty}(\Omega)$ the space of test
functions and by $\mathcal{D}'(\Omega)$ its topological dual, namely,
the space of distributions on $\Omega$. We denote by
$\mathcal{S}(\mathbb{R}^{d})$ the Schwartz class and by
$\mathcal{S}'(\mathbb{R}^{d})$ its topological dual, namely, the space
of tempered distributions.

\item For an integer $k\geq0$ and $1\leq p\leq\infty$, we denote by
$W^{k,p}(\mathbb{R}^d)$ the Sobolev space of functions whose
distributional derivatives up to order $k$ belong to
$L^p(\mathbb{R}^d)$, equipped with the norm $\Vert f\Vert _{W^{k,p}}
  :=\sum_{|\beta|\leq k}\Vert \partial^\beta f\Vert _{L^p}$, 
where $\beta\in\mathbb{N}_0^d$ is a multi-index.
In particular, $W^{0,p}=L^p$ and $H^k=W^{k,2}$.

For an integer $k\geq0$ and $0<\gamma\leq1$, we denote by
$C^{k,\gamma}(\mathbb{R}^d)$ the space of $k$ times continuously
differentiable functions with finite norm
\begin{align*}
  \Vert f\Vert _{C^{k,\gamma}}
  :=\sum_{|\beta|\leq k}\Vert \partial^\beta f\Vert _{L^\infty}
    +\sum_{|\beta|=k}[\partial^\beta f]_{C^\gamma},\quad
  [g]_{C^\gamma}
  :=\sup_{\substack{x,y\in\mathbb{R}^d\\x\neq y}}
    \frac{|g(x)-g(y)|}{|x-y|^\gamma}.
\end{align*}
We write $C^\gamma=C^{0,\gamma}$.
At $\gamma=1$, the seminorm above is the Lipschitz seminorm.

For $k\geq0$ and $0<\gamma<1$, the H\"older space
$C^{k,\gamma}(\mathbb{R}^d)$ coincides with the inhomogeneous
Besov space $B_{\infty,\infty}^{k+\gamma}(\mathbb{R}^d)$,
with equivalent norms: $\Vert f\Vert _{C^{k,\gamma}}
  \approx \Vert f\Vert _{B_{\infty,\infty}^{k+\gamma}}$, the definition of the inhomogeneous Besov spaces is stated in Section~\ref{sec:tools}. 
For negative exponents $s<0$, we use the notation
$C^s(\mathbb{R}^d):=B_{\infty,\infty}^s(\mathbb{R}^d)$.

    \item
    For a Banach
    space $X$ and $p\in[1,\infty]$, we use the abbreviation $L_ {T}^{p}(X):=L^{p}([0,T],X)$. 
    We also write
    $\Vert f\Vert _{X\cap Y}:=\Vert f\Vert _{X}+\Vert f\Vert _{Y}$ and
    $\Vert (f_{1},\ldots,f_{n})\Vert _{X}:=\sum_{i=1}^{n}\Vert f_{i}\Vert _{X}$.
\item $X \Subset Y$ means that the embedding is compact.

    \item
    Given a vector field $W$, we set
    $\partial_{W}f:=W\cdot\nabla f$.

    \item
    For two operators $T$ and $P$,
    $[T,P]:=T\circ P-P\circ T$ denotes their commutator.
    For two $d\times d$ matrices $A=(a_{ij})$ and $B=(b_{ij})$, we denote by
    $A:B$ the usual product of two matrices.
    \item
    For $R>0$, $B_{R}:=\{x\in\mathbb{R}^{d}:|x|<R\}$.
    For a set $E\subset\mathbb{R}^{d}$, $\mathbf{1}_{E}$ denotes its
    characteristic function, and
    $E+B_{R}:=\{x+y:x\in E,\ y\in B_{R}\}$.
\end{itemize}

\section{Tools}\label{sec:tools}
This section collects the analytic tools used to establish
global well-posedness, uniform velocity gradient bounds,
and patch boundary persistence.
We recall the Littlewood-Paley notation and basic
Besov estimates, then present commutator estimates
for Riesz-type and negative-order operators
and the associated striated estimates.
We conclude with estimates for the transport,
transport-diffusion, and linearized velocity equations.
\subsection{Littlewood-Paley theory and basic estimates}
\label{subsec:basic-tools}

Choose two nonnegative radial functions
$\chi,\varphi\in C_c^\infty(\R^d)$,
supported respectively in the ball
$\{|\xi|\leq4/3\}$ and the annulus
$\{3/4\leq|\xi|\leq8/3\}$,
such that (see~\cite{BCD11,LR02})
\begin{equation}\label{eq:POU}
  \chi(\xi)+\sum_{j\geq0}\varphi(2^{-j}\xi)=1,
  \qquad \xi\in\R^d.
\end{equation}
For every $f\in\mathcal S'(\R^d)$,
define the dyadic blocks and low-frequency cutoffs by
\begin{align}\label{exp:Del-Sj}
  \Delta_{-1}f
  &:=\chi(D)f=h'*f,
  \nonumber\\
  \Delta_jf
  &:=\varphi(2^{-j}D)f
  =2^{jd}h(2^j\cdot)*f,
  \qquad j\geq0,
  \nonumber\\
  S_jf
  &:=\sum_{-1\leq\ell\leq j-1}\Delta_\ell f
  =\chi(2^{-j}D)f
  =2^{jd}h'(2^j\cdot)*f,
  \qquad j\geq0,
\end{align}
where $h=\mathcal F^{-1}\varphi$,
$h'=\mathcal F^{-1}\chi$,
and $\mathcal F^{-1}$ denotes the inverse Fourier transform.
We also set $\Delta_j=0$ for $j\leq-2$ and $S_{-1}=0$.

For $f,g\in\mathcal S'(\R^d)$,
whenever the products and series below are well defined,
Bony's decomposition reads
\[
  fg=T_fg+T_gf+R(f,g),
\]
where
\begin{align*}
  T_f g:= \sum_{q\in \mathbb{N}} S_{q-1}f \Delta_q g,\quad R(f,g)=\sum_{q\geq -1}\Delta_q f \widetilde{\Delta}_q g,\quad
  \widetilde{\Delta}_q := \Delta_{q-1} + \Delta_q + \Delta_{q+1}.
\end{align*}

For $s\in\R$ and $p,r\in[1,\infty]$,
the nonhomogeneous Besov space $B^s_{p,r}(\R^d)$
consists of all $f\in\mathcal S'(\R^d)$ such that
\[
  \Vert f\Vert _{B^s_{p,r}}
  :=
  \left\Vert 
    \left(2^{js}\Vert \Delta_jf\Vert _{L^p}\right)_{j\geq-1}
  \right\Vert _{\ell^r}
  <\infty.
\]

For $\rho\in[1,\infty]$ and $T>0$,
the corresponding Chemin-Lerner space
$\widetilde L^\rho_T(B^s_{p,r})$
consists of time-dependent tempered distributions $g$
such that
\[
  \Vert g\Vert _{\widetilde L^\rho_T(B^s_{p,r})}
  :=
  \left\Vert 
    \left(
      2^{js}\Vert \Delta_jg\Vert _{L^\rho(0,T;L^p)}
    \right)_{j\geq-1}
  \right\Vert _{\ell^r}
  <\infty.
\]
Minkowski's inequality and dyadic summation give
\[
  \widetilde L^\rho_T(B^{s+\delta}_{p,r})
  \hookrightarrow L^\rho_T(B^s_{p,r}),
  \qquad \delta>0.
\]
The embedding also holds with $\delta=0$
when $r\leq\rho$.

We recall the following localized multiplier estimates;
see~\cite[Proposition~3.1]{HKR11}.
\begin{lemma}\label{lem:m(D)}
Let $m(D)=\Lambda^\sigma m_0(D)$ with $\sigma>0$,
where $m_0(D)$ is a Fourier multiplier
whose symbol $m_0\in C^\infty(\R^d\setminus\{0\})$
is homogeneous of degree zero.
Then, for every $p\in[1,\infty]$,
$j\in\mathbb N$, and $q\in\{-1\}\cup\mathbb N$,
\begin{equation}\label{es:mD-Sj-Lp}
\begin{aligned}
  \Vert m(D)S_ju\Vert _{L^p}
  \leq C2^{j\sigma}\Vert S_ju\Vert _{L^p},
  \qquad
  \Vert m(D)\Delta_qu\Vert _{L^p}
  \leq C2^{q\sigma}\Vert \Delta_qu\Vert _{L^p}.
\end{aligned}
\end{equation}
The constant $C$ depends only on $d$, $\sigma$,
and finitely many seminorms of $m_0$ on $\mathbb S^{d-1}$.
\end{lemma}

The next lemma collects product, commutator,
and tangential derivative estimates in Besov spaces.
\begin{lemma}\label{lem:CMX2.5}
Let $d\geq2$, and let $u$ and
$\mathcal W=\{W_i\}_{i=1}^N$
be smooth divergence-free vector fields on $\mathbb R^d$.
Let $\phi$ be a smooth scalar function.
Let $m(D)$ be the Fourier multiplier with symbol
\[
  m(\xi)=|\xi|^\sigma m_0(\xi),
  \qquad \sigma>-1,
\]
where $m_0\in C^\infty(\mathbb R^d\setminus\{0\})$
is homogeneous of degree zero.
For any function space $X$, write
$\Vert \mathcal W\Vert _X:=\sum_{i=1}^N\Vert W_i\Vert _X$,
and $\Vert \partial_{\mathcal W}f\Vert _X :=\sum_{i=1}^N\Vert W_i\cdot\nabla f\Vert _X$.
Then the following estimates hold.

\begin{enumerate}[(1)]
\item
For every $\delta\in(0,1)$ and $p,r\in[1,\infty]$,
\begin{equation}\label{eq:prodBes}
\begin{aligned}
  \Vert u\cdot\nabla\phi\Vert _{B^{-\delta}_{p,r}}
  \leq C\min\Bigl\{
    &\Vert u\Vert _{B^{-\delta}_{p,r}}\Vert \nabla\phi\Vert _{L^\infty},
    \Vert u\Vert _{L^\infty}
      \Vert \nabla\phi\Vert _{B^{-\delta}_{p,r}}
  \Bigr\}.
\end{aligned}
\end{equation}
At the endpoint, for every $p\in[1,\infty]$,
\begin{equation}\label{eq:prodBes-endp}
\begin{aligned}
  \Vert u\cdot\nabla\phi\Vert _{B^{-1}_{p,\infty}}
  \leq C\min\Bigl\{
    &\Vert u\Vert _{B^{-1}_{p,\infty}}
      \Vert \nabla\phi\Vert _{B^0_{\infty,1}},
    \Vert u\Vert _{B^0_{\infty,1}}
      \Vert \nabla\phi\Vert _{B^{-1}_{p,\infty}}
  \Bigr\}.
\end{aligned}
\end{equation}

\item
Let $s\in(-1,1)$ and $p,r\in[1,\infty]$.
If $-1<s+\sigma<1$, then
\begin{equation}\label{eq:commEs}
  \Vert [m(D),u\cdot\nabla]\phi\Vert _{B^s_{p,r}}
  \leq C\Vert u\Vert _{W^{1,\infty}}
  \Vert \phi\Vert _{B^{s+\sigma}_{p,r}}.
\end{equation}
If $s+\sigma\geq1$, then
\begin{equation}\label{eq:commEs2}
\begin{aligned}
  \Vert [m(D),u\cdot\nabla]\phi\Vert _{B^s_{p,r}}
  \leq {}
  C\Vert u\Vert _{W^{1,\infty}}
  \Vert \phi\Vert _{B^{s+\sigma}_{p,r}}
  +C\Vert u\Vert _{B^{s+\sigma}_{p,r}}
  \Vert \nabla\phi\Vert _{L^\infty}.
\end{aligned}
\end{equation}

\item
Let $s\in(-1,1)$ and $p,r\in[1,\infty]$.
If $-1<s+\sigma<1$, then
\begin{equation}\label{eq:paWmD-es}
\begin{aligned}
  \Vert \partial_{\mathcal W}(m(D)\phi)\Vert _{B^s_{p,r}}
  \leq {}
  C\Vert \partial_{\mathcal W}\phi\Vert _{B^{s+\sigma}_{p,r}}
  +C\Vert \mathcal W\Vert _{W^{1,\infty}}
  \Vert \phi\Vert _{B^{s+\sigma}_{p,r}}.
\end{aligned}
\end{equation}
If $s+\sigma\geq1$, then
\begin{equation}\label{eq:paWmD-es2}
\begin{aligned}
  \Vert \partial_{\mathcal W}(m(D)\phi)\Vert _{B^s_{p,r}}
  \leq {}&
  C\Vert \partial_{\mathcal W}\phi\Vert _{B^{s+\sigma}_{p,r}}
  +C\Vert \mathcal W\Vert _{W^{1,\infty}}
  \Vert \phi\Vert _{B^{s+\sigma}_{p,r}}
  +C\Vert \mathcal W\Vert _{B^{s+\sigma}_{p,r}}
  \Vert \nabla\phi\Vert _{L^\infty}.
\end{aligned}
\end{equation}
\end{enumerate}

The constants depend only on the stated parameters and,
for multiplier estimates, finitely many seminorms
of $m_0$ on $\mathbb S^{d-1}$.
\end{lemma}

\begin{proof}[Proof of Lemma~\ref{lem:CMX2.5}]
Estimates \eqref{eq:prodBes}, \eqref{eq:commEs},
and \eqref{eq:paWmD-es} follow from
\cite[Lemma~2.3]{LXY25}.
We prove the remaining three estimates.

We prove \eqref{eq:prodBes-endp} using Bony's decomposition: $u\cdot\nabla\phi
  =
  T_u\cdot\nabla\phi
  +T_{\nabla\phi}\cdot u
  +R(u\cdot,\nabla\phi)$. 
For the paraproduct terms, the spectral localization
and the bound
$\sup\limits_{j\in\mathbb{N}} 2^{-j}\Vert S_{j-1}f\Vert _{L^p} \leq C\Vert f\Vert _{B^{-1}_{p,\infty}}$
yield
\begin{align*}
  \Vert T_u\cdot\nabla\phi\Vert _{B^{-1}_{p,\infty}}
  &\leq \sup_{q\geq -1}2^{-q}\sum_{\substack{j\in\mathbb N; \,|j-q|\leq 4}}
  \Vert \Delta_q(S_{j-1}u\cdot \nabla \Delta_j\phi)\Vert _{L^p} \\
  &\leq C \sup_{j\in\mathbb{N}} 2^{-j}
  \min\Bigl\{ \Vert S_{j-1}u\Vert _{L^p}\Vert \nabla \Delta_j\phi\Vert _{L^\infty},
  \ \Vert S_{j-1}u\Vert _{L^\infty}\Vert \nabla \Delta_j\phi\Vert _{L^p} \Bigr\}\\
  &\leq C \min\Bigl\{\Vert u\Vert _{B^{-1}_{p,\infty}}\Vert \nabla \phi\Vert _{L^\infty},
  \ \Vert u\Vert _{L^\infty}\Vert \nabla\phi\Vert _{B^{-1}_{p,\infty}}\Bigr\},
\end{align*}
and
\begin{align*}
  \Vert T_{\nabla\phi}\cdot u\Vert _{B^{-1}_{p,\infty}}
  & \leq \sup_{q\geq -1} 2^{-q}\sum_{\substack{j\in\mathbb N;\, |j-q|\leq 4}}
  \Vert \Delta_q(\Delta_j u\cdot \nabla S_{j-1}\phi)\Vert _{L^p} \\
  &\leq C \min\Bigl\{\Vert u\Vert _{B^{-1}_{p,\infty}} \Vert \nabla \phi\Vert _{L^\infty},
  \ \Vert u\Vert _{L^\infty}\Vert \nabla\phi\Vert _{B^{-1}_{p,\infty}}\Bigr\}.
\end{align*}

For the remainder, since $\div u=0$, we have $R(u\cdot,\nabla \phi) = \div R(u,\phi) $, thus
\begin{align*}
  \Vert R(u\cdot,\nabla \phi)\Vert _{B^{-1}_{p,\infty}} 
  & = \sup_{q\geq -1} 2^{-q}\Vert \Delta_q\, R(u\cdot,\nabla \phi)\Vert _{L^p}  \\
  &\leq  C \sup_{q\geq -1} \sum_{\substack{j\ge q-3;\,j\ge 2 }}
  \Vert \Delta_q(\Delta_j u\,\widetilde{\Delta}_j\phi)\Vert _{L^p}
  + \sup_{q\geq -1} 2^{-q}\sum_{\substack{j\ge q-3;\,j\leq 2 }}
  \Vert \Delta_q(\Delta_j u\cdot \nabla \widetilde{\Delta}_j\phi)\Vert _{L^p} \\
  &\leq C  \sum_{\substack{ j\ge 2}} 2^{-j}
  \min\Bigl\{\Vert \Delta_j u\Vert _{L^p}\Vert \nabla \widetilde{\Delta}_j\phi\Vert _{L^\infty},
  \ \Vert \Delta_j u\Vert _{L^\infty}\Vert \nabla \widetilde{\Delta}_j\phi\Vert _{L^p}\Bigr\}\\
  &\quad+ C  \sum_{\substack{-1\leq j\leq 2}} 2^{-j}
  \min\Bigl\{\Vert \Delta_j u\Vert _{L^p}\Vert \nabla \widetilde{\Delta}_j\phi\Vert _{L^\infty},
  \ \Vert \Delta_j u\Vert _{L^\infty}\Vert \nabla \widetilde{\Delta}_j\phi\Vert _{L^p}\Bigr\}\\
  &\leq C \min\Bigl\{\Vert u\Vert _{B^{-1}_{p,\infty}}\Vert \nabla \phi\Vert _{B^0_{\infty,1}},
  \ \Vert u\Vert _{B^0_{\infty,1}}\Vert \nabla\phi\Vert _{B^{-1}_{p,\infty}}\Bigr\}.
\end{align*}
Combining these estimates and using
$B^0_{\infty,1}\hookrightarrow L^\infty$
proves \eqref{eq:prodBes-endp}.

To prove \eqref{eq:commEs2}, Bony's decomposition gives
\begin{align*}
  [m(D),u\cdot\nabla]\phi
  &=
  \sum_{j\geq0}
  [m(D),S_{j-1}u\cdot\nabla]\Delta_j\phi
  + \sum_{j\geq0}
  [m(D),\Delta_j u\cdot\nabla]S_{j-1}\phi
  \\
  &\quad+
  \sum_{j\geq-1}
  [m(D),\Delta_j u\cdot\nabla]\widetilde\Delta_j\phi
  :=I_1+I_2+I_3.
\end{align*}
The estimates for $I_1$ and $I_3$ proceed as in
the proof of \eqref{eq:commEs} and yield
\[
  \Vert I_1\Vert _{B^s_{p,r}}+\Vert I_3\Vert _{B^s_{p,r}}
  \leq
  C\Vert u\Vert _{W^{1,\infty}}\Vert \phi\Vert _{B^{s+\sigma}_{p,r}}.
\]
We therefore detail the estimate of $I_2$,
which produces the additional term in \eqref{eq:commEs2}.
By Lemma~\ref{lem:m(D)},
\begin{equation}\label{eq:fact1a}
  \Vert \nabla m(D)\Delta_j\phi\Vert _{L^p}
  \leq
  C2^{j(1+\sigma)}\Vert \Delta_j\phi\Vert _{L^p},
  \qquad
  j\geq-1,\quad 1\leq p\leq\infty,\quad \sigma>-1.
\end{equation}
Using spectral localization and \eqref{eq:fact1a},
we obtain, for $q\geq-1$,
\begin{align*}
  2^{qs}\Vert \Delta_q I_2\Vert _{L^p}
  &\leq
  C2^{qs}
  \sum_{\substack{j\geq0;\,|j-q|\leq4}}
  \Bigl(
    2^{j\sigma}\Vert \Delta_j u\Vert _{L^p}
    \Vert \nabla S_{j-1}\phi\Vert _{L^\infty}
    +\Vert \Delta_j u\Vert _{L^\infty}
    \Vert \nabla m(D)S_{j-1}\phi\Vert _{L^p}
  \Bigr)
  \\
  & \leq
  C\sum_{\substack{j\geq0;\,|j-q|\leq4}}
  \biggl(
    2^{j(s+\sigma)}\Vert \Delta_j u\Vert _{L^p}
    \Vert \nabla\phi\Vert _{L^\infty}
    +\Vert \nabla u\Vert _{L^\infty}
    \sum_{\ell=-1}^{j-2}
    2^{-(j-\ell)(1-s)}
    2^{\ell(s+\sigma)}\Vert \Delta_\ell\phi\Vert _{L^p}
  \biggr).
\end{align*}
Taking the $\ell^r$ norm in $q$
and applying the discrete Young inequality gives
\[
  \Vert I_2\Vert _{B^s_{p,r}}
  \leq
  C\Vert u\Vert _{B^{s+\sigma}_{p,r}}\Vert \nabla\phi\Vert _{L^\infty}
  +
  C\Vert \nabla u\Vert _{L^\infty}\Vert \phi\Vert _{B^{s+\sigma}_{p,r}}.
\]
Combining the estimates for $I_1$, $I_2$, and $I_3$
proves \eqref{eq:commEs2}.

We next prove \eqref{eq:paWmD-es2}.
Since $s<1$ and $s+\sigma\geq1$, we have $\sigma>0$.
The dyadic multiplier estimates, including the
low-frequency block, therefore give
$\Vert m(D)f\Vert _{B^s_{p,r}}\leq C\Vert f\Vert _{B^{s+\sigma}_{p,r}}$.
For each $W\in\mathcal W$, we use the identity
\[
  \partial_W(m(D)\phi)
  =
  m(D)\partial_W\phi-[m(D),W\cdot\nabla]\phi.
\]
Applying \eqref{eq:commEs2} with $u=W$
and summing over the family gives the desired inequality \eqref{eq:paWmD-es2}.
\end{proof}

We recall the following commutator estimate
from~\cite[Lemma~6.10(1)]{HKR10}.
\begin{lemma}\label{lem:HKR}
Let $v$ be a smooth divergence-free vector field
on $\R^d$, $d\geq2$, and let $f$ be a smooth scalar function.
Then, for every $p\in[1,\infty]$
and $q\in\{-1\}\cup\mathbb N$,
\[
  \Vert [\Delta_q,v\cdot\nabla]f\Vert _{L^p}
  \leq C\Vert \nabla v\Vert _{L^p}
  \Vert f\Vert _{B^0_{\infty,\infty}},
\]
where $C$ is independent of $q$.
\end{lemma}

The following product estimates follow from Bony's decomposition,
by an argument similar to that in
\cite[Lemma~6.10(2)]{HKR10}.
We omit the proof.
\begin{lemma}\label{lem-610-2}
Let $v$ be a smooth divergence-free vector field
on $\R^d$, $d\geq2$, and let $f$ be a smooth scalar function.
For every $p\in[1,\infty]$ and $s\in[-1,0)$,
\[
  \Vert v\cdot\nabla f\Vert _{B^s_{p,\infty}}
  \leq C\Vert f\Vert _{B^{1+s}_{\infty,\infty}}
  \Vert v\Vert _{B^0_{p,1}}.
\]
\end{lemma}

The following lemma gives the striated regularity
of the initial temperature patch.
\begin{lemma}\label{lem:pWk-theta}
Let $k\geq2$ be an integer and $0<\mu<1$.
Let $D_0\subset\R^2$ be a bounded simply connected
Lipschitz domain, and let
$\varphi_0\in W^{k-1,\infty}(\R^2)$ satisfy
$D_0=\{\varphi_0>0\}$.
Set $W_0:=\nabla^\perp\varphi_0$ and
\[
  \theta_0=\bar\theta_0\mathds{1}_{D_0},
  \qquad
  \bar\theta_0\in C^{k-2,\mu}(\overline{D_0}).
\]
Then
\begin{equation}\label{pWj-theta}
  \partial_{W_0}^j\theta_0
  \in L^1\cap L^\infty(\R^2),
  \qquad 0\leq j\leq k-2,
\end{equation}
and
\begin{equation}\label{pWk-theta}
  \partial_{W_0}^{k-1}\theta_0
  \in B^{\mu-1}_{\infty,\infty}(\R^2).
\end{equation}
The tangential derivatives are understood
in the distributional sense.
\end{lemma}

\begin{proof}[Proof of Lemma~\ref{lem:pWk-theta}]
We argue as \cite[Proposition 3.1]{Sue15} and \cite[Lemma 2.6]{CMX22}.
By Rychkov's extension theorem~\cite{VSR1999},
there exists
$\widetilde\theta_0\in C^{k-2,\mu}(\R^2)$
such that
$\widetilde\theta_0|_{D_0}=\bar\theta_0$.

Since $D_0=\{\varphi_0>0\}$ and
$W_0=\nabla^\perp\varphi_0$,
the Sobolev chain rule gives
\[
  W_0\mathds{1}_{D_0}
  =\nabla^\perp(\varphi_0)_+,
  \qquad
  \div(W_0\mathds{1}_{D_0})=0
  \quad\text{in }\mathcal D'(\R^2),
\]
where $(\varphi_0)_+:=\max\{\varphi_0,0\}$.
Using this identity and the product rule inductively,
we obtain
\[
  \partial_{W_0}^j\theta_0
  =
  \bigl(\partial_{W_0}^j\widetilde\theta_0\bigr)
  \mathds{1}_{D_0},
  \qquad 0\leq j\leq k-2.
\]
The assumptions
$\widetilde\theta_0\in C^{k-2,\mu}$
and $W_0\in W^{k-2,\infty}$
ensure that
$\partial_{W_0}^j\widetilde\theta_0\in L^\infty$
for these values of $j$.
Since $D_0$ is bounded,
this immediately proves \eqref{pWj-theta}.

For the highest-order derivative, we have
$\partial_{W_0}^{k-1}\theta_0
  =
  \div\bigl(
    W_0\mathds{1}_{D_0}\,
    \partial_{W_0}^{k-2}\widetilde\theta_0
  \bigr)$.
Since $\div(W_0\mathds{1}_{D_0})=0$,
we apply \eqref{eq:prodBes} with $p=r=\infty$ and $\epsilon=1-\mu$ to obtain
\begin{align*}
  \Vert \partial_{W_0}^{k-1}\theta_0\Vert _{B^{\mu-1}_{\infty,\infty}}
  \leq
  C\Vert W_0\mathds{1}_{D_0}\Vert _{L^\infty}
  \Vert \nabla\partial_{W_0}^{k-2}\widetilde\theta_0
    \Vert _{B^{\mu-1}_{\infty,\infty}}
  \leq
  C\Vert W_0\Vert _{L^\infty}
  \Vert \partial_{W_0}^{k-2}\widetilde\theta_0\Vert _{C^\mu}
  <\infty.
\end{align*}
This proves \eqref{pWk-theta}. 
\end{proof}

\subsection{Commutator estimates for Riesz-type operators}\label{subsec:Riesz-commutators}

We recall two classical commutator estimates
for second-order Riesz transforms.
Set $\mathcal R_j:=\frac{1}{i}\partial_j\Lambda^{-1},\, j=1,2$. 
The first is a Calder\'on commutator estimate;
see~\cite{Calderon1965}
and~\cite[Lemma~2.2]{LZ24}.

\begin{lemma}\label{lem:riesz-comm}
Let $p,p_1\in(1,\infty)$ and $p_2\in(1,\infty]$
satisfy
$\frac1p=\frac1{p_1}+\frac1{p_2}$.
For every $X\in C_c^1(\R^2;\R^2)$
and $g\in L^{p_1}(\R^2)$, we have
\[
  \Vert [\mathcal R_j\mathcal R_\ell,\partial_X]g\Vert _{L^p}
  \leq
  C\Vert \nabla X\Vert _{L^{p_2}}\Vert g\Vert _{L^{p_1}},
  \qquad j,\ell\in\{1,2\},
\]
where $\partial_X:=X\cdot\nabla$
and $C$ depends only on $p,p_1,p_2$.
\end{lemma}

The second follows from the
Coifman-Rochberg-Weiss BMO commutator theorem;
see~\cite{CRW76} or \cite[Theorem 7.5.6]{Gra08}.

\begin{lemma}\label{lem:gra}
Let $1<p<\infty$ and $b\in\mathrm{BMO}(\R^2)$.
For $j,\ell\in\{1,2\}$,
the commutator $[\mathcal R_j\mathcal R_\ell,b]$,
initially defined on $C_c^\infty(\R^2)$,
extends to a bounded operator on $L^p(\R^2)$, with
\[
  \Vert [\mathcal R_j\mathcal R_\ell,b]f\Vert _{L^p}
  \leq
  C_p\Vert b\Vert _{\mathrm{BMO}}\Vert f\Vert _{L^p}.
\]
\end{lemma}

We next establish commutator estimates for
$\mathcal R_{1-\beta}:=\partial_1\Lambda^{-\beta}$.

\begin{lemma}\label{lem:Rbeta-cm}
Let $d\geq2$, $p\in[2,\infty]$, and $r\in[1,\infty]$.
Assume that $u:\R^d\to\R^d$ is a smooth
divergence-free vector field and
$\phi:\R^d\to\R$ is a smooth function.
The following estimates hold.
\begin{enumerate}[(1)]
\item
For $\beta\in(1,2)$ and $s\in(\beta-2,1)$,
\begin{equation}\label{eq:Rbeta-cm-es}
  \Vert [\mathcal R_{1-\beta},u\cdot\nabla]\phi\Vert _{B^s_{p,r}}
  \leq
  C\Vert \nabla u\Vert _{L^p}
  \Bigl(
    \Vert \phi\Vert _{B^{s+1-\beta}_{\infty,r}}
    +\Vert \phi\Vert _{L^2}
  \Bigr).
\end{equation}

\item
If $d=2$ and $\beta=2$, then for every $q\in[2,\infty)$,
\begin{equation}\label{eq:comm-es1}
  \Vert [\mathcal R_{-1},u\cdot\nabla]\phi\Vert _{L^q}
  \leq
  C_q\Vert \nabla u\Vert _{L^q}\Vert \phi\Vert _{L^2}.
\end{equation}
Moreover, for every $s\in(0,1)$,
\begin{equation}\label{eq:R-1cm-es1}
\begin{aligned}
  \Vert [\mathcal R_{-1},u\cdot\nabla]\phi\Vert _{B^s_{p,r}}
  \leq
  \begin{cases}
    C\Vert \nabla u\Vert _{L^p}
    \Bigl(
      \Vert \phi\Vert _{B^{s-1}_{\infty,r}}
      +\Vert \phi\Vert _{L^2}
    \Bigr),
    & 2\leq p<\infty,
    \\[2mm]
    C\Vert \nabla u\Vert _{L^\infty}
    \left(\Vert \phi\Vert _{B^{s-1}_{\infty,r}}+\Vert \phi\Vert _{L^2}\right)
    +C\Vert \nabla u\Vert _{L^q}\Vert \phi\Vert _{L^2},
    & p=\infty,
  \end{cases}
\end{aligned}
\end{equation}
where $q\in[2,\infty)$ is arbitrary in the case $p=\infty$.
\end{enumerate}
The constants depend only on $d$ and the displayed exponents.
\end{lemma}

\begin{proof}[Proof of Lemma \ref{lem:Rbeta-cm}] 
For \eqref{eq:Rbeta-cm-es},
we refer to~\cite[Lemma~2.6]{LXY25}.
We prove \eqref{eq:comm-es1}
and \eqref{eq:R-1cm-es1} below.

For $2<q<\infty,$ thanks to the Sobolev embedding 
$\dot{W}^{1,\frac{2q}{q+2}}(\mathbb{R}^2) \hookrightarrow L^q(\mathbb{R}^2)$, H\"older's inequality and Lemma \ref{lem:riesz-comm}, 
we find
\begin{align}
 \Vert [\mathcal{R}_{-1}, u\cdot\nabla]\phi\Vert _{L^q(\mathbb{R}^2)}
  &\leq C \Vert \nabla[\mathcal{R}_{-1}, u\cdot\nabla]\phi \Vert _{L^{\frac{2q}{q+2}}(\mathbb{R}^2)} \nonumber  \\
  &\leq C \Vert [\nabla\mathcal{R}_{-1}, u\cdot\nabla]\phi\Vert _{L^{\frac{2q}{q+2}}(\mathbb{R}^2)} 
  + C \Vert \nabla u\cdot\nabla\mathcal{R}_{-1}\phi \Vert _{L^{\frac{2q}{q+2}}(\mathbb{R}^2)} \nonumber  \\
  &\leq C \Vert \nabla u\Vert _{L^q(\mathbb{R}^2)}\Vert \phi\Vert _{L^2(\mathbb{R}^2)}.\nonumber
\end{align}
For $q=2$, applying Lemma \ref{lem:gra} and the continuous embedding 
$\dot{H}^1(\mathbb{R}^2)\hookrightarrow BMO(\mathbb{R}^2)$, we deduce
\begin{align*}
  \Vert [\mathcal{R}_{-1}, u\cdot\nabla]\phi\Vert _{L^2(\mathbb{R}^2)} 
  = \Vert [\partial_i\mathcal{R}_{-1}, u_i]\phi\Vert _{L^2(\mathbb{R}^2)} 
  \leq C \Vert u\Vert _{BMO(\mathbb{R}^2)}\Vert \phi\Vert _{L^2(\mathbb{R}^2)}  
  \leq C \Vert \nabla u\Vert _{L^2(\mathbb{R}^2)} \Vert \phi\Vert _{L^2(\mathbb{R}^2)},
\end{align*}
where the summation convention on repeated indices is adopted.

We now prove \eqref{eq:R-1cm-es1}.
Bony's decomposition gives
\begin{align*}
  [\mathcal{R}_{-1}, u\cdot \nabla ] \phi
  & = \sum_{j\in \NN}[\mathcal{R}_{-1}, S_{j-1}u\cdot\nabla] \Delta_j \phi
  + \sum_{j\in \NN} [\mathcal{R}_{-1}, \Delta_j u \cdot\nabla] S_{j-1}\phi
  + \sum_{j\geq 3} [\mathcal{R}_{-1}, \Delta_j u\cdot\nabla] \widetilde{\Delta}_j \phi \nonumber \\
  & \quad + \sum_{-1\leq j\leq 2} [\mathcal{R}_{-1},\Delta_j u\cdot \nabla] \widetilde{\Delta}_j\phi  := {I}_1 + {I}_2 + {I}_3 + I_4.
\end{align*}
Arguing as \cite[Lemma 2.6]{LXY25}, we have
\begin{align*}
  \Vert I_1 + I_2+I_3\Vert _{B^s_{p,r}(\mathbb{R}^2)} \leq C_s \Vert \nabla u\Vert _{L^p(\mathbb{R}^2)} \left(\Vert \phi\Vert _{B^{s-1}_{\infty,r}(\mathbb{R}^2)} + \Vert \phi\Vert _{L^2(\mathbb{R}^2)}\right) .
\end{align*}
For $I_4$, applying \eqref{eq:comm-es1} and Bernstein's inequality gives
\begin{align*}
  \Vert I_4\Vert _{B^s_{p,r}(\mathbb{R}^2)} & \leq C \sum_{-1\leq l\leq 6} 
  \sum_{-1\leq j\leq 2} \Vert \Delta_l [\mathcal{R}_{-1}, \Delta_j u 
  \cdot \nabla]\widetilde{\Delta}_j \phi\Vert _{L^p(\mathbb{R}^2)} \\
  & \leq 
  \begin{cases}
    C \sum_{-1\leq j \leq 2} \Vert [\mathcal{R}_{-1},\Delta_j u\cdot \nabla]\widetilde{\Delta}_j\phi\Vert _{L^q(\mathbb{R}^2)},
    \quad & \textrm{for}\;\; p=\infty, \\
    C \sum_{-1\leq j \leq 2} \Vert [\mathcal{R}_{-1},\Delta_j u\cdot \nabla]\widetilde{\Delta}_j\phi\Vert _{L^p(\mathbb{R}^2)},
    \quad & \textrm{for}\;\; 2\leq p<\infty, 
  \end{cases} \\
  & \leq
  \begin{cases}
    C_q \Vert \nabla u\Vert _{L^q(\mathbb{R}^2)} \Vert \phi\Vert _{L^2(\mathbb{R}^2)}, \quad & \textrm{for}\;\; p=\infty, \\
    C_p \Vert \nabla u\Vert _{L^p(\mathbb{R}^2)} \Vert \phi\Vert _{L^2(\mathbb{R}^2)}, \quad & \textrm{for}\;\; 2\leq p<\infty.
  \end{cases} 
\end{align*}
Combining the above two estimates leads to the desired estimate \eqref{eq:R-1cm-es1}.
\end{proof}

\subsection{Commutator estimates for negative-order operators}
\label{subsec:tangential-tools}

Using kernel representations, we estimate commutators
involving negative-order Fourier multipliers
and their tangential derivatives.

The following kernel representation of homogeneous
Fourier multipliers of negative order
will be used in the commutator estimates below.
\begin{lemma}\label{lem:mDphi}
Let $m(D)=\Lambda^\sigma m_0(D)$, where $-d<\sigma<0$
and $m_0(D)$ is a Fourier multiplier with symbol
$m_0\in C^\infty(\mathbb R^d\setminus\{0\})$
homogeneous of degree zero.
Then there exists $\Omega\in C^\infty(\mathbb S^{d-1})$
such that, for every $\phi\in\mathcal S(\mathbb R^d)$,
\begin{equation}\label{mDphi}
  m(D)\phi(x)
  =K_\sigma*\phi(x)
  =\int_{\mathbb R^d}K_\sigma(x-y)\phi(y)\,\dd y,
  \qquad
  K_\sigma(x)
  =\frac{\Omega(x/|x|)}{|x|^{d+\sigma}},
  \quad x\neq0.
\end{equation}
\end{lemma}

\begin{proof}
Since $\sigma>-d$, the symbol
$m(\xi)=|\xi|^\sigma m_0(\xi)$ is locally integrable
and defines a tempered distribution.
Its inverse Fourier transform $K:=\mathcal F^{-1}m$
is homogeneous of degree $-d-\sigma$.
Moreover, by \cite[Proposition~2.4.8]{Gra14},
$K$ is smooth away from the origin.
Setting $\Omega:=K|_{\mathbb S^{d-1}}$, we therefore obtain
\[
  K(x)=\frac{\Omega(x/|x|)}{|x|^{d+\sigma}}
  =K_\sigma(x),
  \qquad x\neq0.
\]

Because $-d<-d-\sigma<0$, the function $K_\sigma$
is locally integrable and decays at infinity,
so it also defines a tempered distribution.
The difference $K-K_\sigma$ is supported at the origin
and is homogeneous of degree $-d-\sigma>-d$.
By \cite[Proposition~2.4.1]{Gra14}, it is a finite
linear combination of derivatives of $\delta_0$.
Since $\partial^\beta\delta_0$ is homogeneous
of degree $-d-|\beta|$, this forces $K-K_\sigma=0$.

The Fourier convolution identity now gives
$m(D)\phi=K_\sigma*\phi$ in $\mathcal S'$.
Local integrability and decay of $K_\sigma$,
together with the rapid decay of $\phi$,
ensure absolute convergence of the integral
in \eqref{mDphi}.
Differentiation under the integral sign gives
$\partial^\beta(K_\sigma*\phi) =K_\sigma*(\partial^\beta\phi)$
for every multi-index $\beta$.
Thus the convolution is smooth, and
\eqref{mDphi} holds pointwise.
\end{proof}

We next establish $L^\infty$ estimates for commutators
of negative-order Fourier multipliers with transport operators,
together with estimates for their tangential derivatives.

\begin{lemma}\label{lem:pW}
Let $m(D)=\Lambda^\sigma m_0(D)$ with $-d<\sigma<0$,
where $m_0\in C^\infty(\R^d\setminus\{0\})$
is homogeneous of degree zero.
Let $u\in W^{1,\infty}_{\mathrm{loc}}(\R^d;\R^d)$
be divergence-free with $\nabla u\in L^\infty(\R^d)$,
and let $\phi\in L^1\cap L^\infty(\R^d)$.
Then the following estimates hold
with a constant $C=C(d,\sigma,m_0)>0$.
\begin{enumerate}[(1)]
\item
We have
\begin{equation}\label{es:prod-es5}
  \Vert [m(D),u\cdot\nabla]\phi\Vert _{L^\infty}
  \leq C\Vert \nabla u\Vert _{L^\infty}
  \Vert \phi\Vert _{L^1\cap L^\infty}.
\end{equation}

\item
Let $\WW=\{W_i\}_{i=1}^N$ be a finite familly
of divergence-free vector fields
in $W^{1,\infty}_{\mathrm{loc}}(\R^d;\R^d)$.
Assume additionally that, for each $i$,
$\nabla W_i,\ \nabla(\partial_{W_i}u)\in L^\infty$
and $\partial_{W_i}\phi\in L^1\cap L^\infty$.
Then
\begin{equation}\label{es:prod-es4}
\begin{aligned}
  \Vert \partial_{\WW}([m(D),u\cdot\nabla]\phi)\Vert _{L^\infty}
  &\leq
  C\bigl(
    \Vert \nabla\WW\Vert _{L^\infty}\Vert \nabla u\Vert _{L^\infty}
    +\Vert \nabla(\partial_{\WW}u)\Vert _{L^\infty}
  \bigr)\Vert \phi\Vert _{L^1\cap L^\infty}
  \\
  &\quad+
  C\Vert \nabla u\Vert _{L^\infty}
  \Vert \partial_{\WW}\phi\Vert _{L^1\cap L^\infty}.
\end{aligned}
\end{equation}
\end{enumerate}
\end{lemma}

\begin{proof}[Proof of Lemma~\ref{lem:pW}]
Let $K_\sigma$ be the kernel given by
Lemma~\ref{lem:mDphi}. Its homogeneity and smoothness
away from the origin imply
\[
  |K_\sigma(z)|\leq C|z|^{-d-\sigma},
  \qquad
  |\nabla K_\sigma(z)|
  \leq C|z|^{-d-\sigma-1},
  \qquad z\ne0.
\]
We first take $\phi\in C_c^\infty(\R^d)$.
The convolution formula gives
\[
  [m(D),u\cdot\nabla]\phi(x)
  =
  \int_{\R^d}K_\sigma(x-y)
  \bigl(u(y)-u(x)\bigr)\cdot\nabla\phi(y)\,\dd y.
\]
Integrating by parts over $\{|x-y|>\delta\}$
and using $\div u=0$, we obtain
\begin{equation}\label{eq:comm-u-phi}
  [m(D),u\cdot\nabla]\phi(x)
  =
  \int_{\R^d}
  \bigl(u(y)-u(x)\bigr)\cdot
  \nabla K_\sigma(x-y)\phi(y)\,\dd y.
\end{equation}
Indeed, the boundary term on $\{|x-y|=\delta\}$
is bounded by
$C\delta^{-\sigma} \Vert \nabla u\Vert _{L^\infty}\Vert \phi\Vert _{L^\infty}$,
which tends to zero as $\delta$ tends to 0, this is
because $\sigma<0$.

The Lipschitz bound on $u$ gives
\begin{equation}\label{eq:integrad-es1}
  \bigl|
  \bigl(u(y)-u(x)\bigr)\cdot
  \nabla K_\sigma(x-y)\phi(y)
  \bigr|
  \leq
  C\Vert \nabla u\Vert _{L^\infty}
  |x-y|^{-d-\sigma}|\phi(y)|.
\end{equation}
Since $-d<\sigma<0$, splitting the integral in two parts
at $|x-y|=1$ gives
\begin{equation}\label{conv-phi}
\begin{aligned}
  \sup_{x\in\R^d}
  \int_{\R^d}|x-y|^{-d-\sigma}|\phi(y)|\,\dd y
  \leq
  \Vert \phi\Vert _{L^\infty}
  \int_{|z|\leq1}|z|^{-d-\sigma}\,\dd z
  +\Vert \phi\Vert _{L^1}
  \leq C\Vert \phi\Vert _{L^1\cap L^\infty}.
\end{aligned}
\end{equation}
Thus the integral in \eqref{eq:comm-u-phi}
is absolutely convergent for $\phi\in L^1\cap L^\infty$
and defines the commutator under the stated assumptions.
Combining \eqref{eq:comm-u-phi},
\eqref{eq:integrad-es1}, and \eqref{conv-phi}
proves \eqref{es:prod-es5}.

We next prove \eqref{es:prod-es4}.
It suffices to consider a single vector field
$W\in\mathcal W$.
We claim that, in the distributional sense,
\begin{equation}\label{eq:pWcomm-u-phi}
\begin{aligned}
  \partial_W[m(D),u\cdot\nabla]\phi(x)
  &=
  \int_{\R^d}
  (u(y)-u(x))\cdot\nabla^2K_\sigma(x-y)
  \cdot(W(x)-W(y))\phi(y)\,\dd y
  \\
  &\quad+
  \int_{\R^d}
  (\partial_Wu(y)-\partial_Wu(x))
  \cdot\nabla K_\sigma(x-y)\phi(y)\,\dd y
  \\
  &\quad+
  \int_{\R^d}
  (u(y)-u(x))\cdot\nabla K_\sigma(x-y)
  \partial_W\phi(y)\,\dd y.
\end{aligned}
\end{equation}

Let $\varphi\in C_c^\infty(\R^d)$.
By $\div W=0$ and \eqref{eq:comm-u-phi},
\begin{align*}
  \bigl\langle
  \partial_W[m(D),u\cdot\nabla]\phi,\varphi
  \bigr\rangle
  &=
  \int_{\R^d}\int_{\R^d}
  (u(x)-u(y))\cdot\nabla K_\sigma(x-y)
  \partial_W\varphi(x)\phi(y)\,\dd x\,\dd y
  \\
  &=
  \lim_{\delta\rightarrow0}
  \int_{\R^d}\int_{|x-y|>\delta}
  (u(x)-u(y))\cdot\nabla K_\sigma(x-y)
  \partial_W\varphi(x)\phi(y)\,\dd x\,\dd y
  \\
  &:=\lim_{\delta\rightarrow0}\mathcal I^\delta.
\end{align*}
The limit follows from \eqref{eq:integrad-es1}
and \eqref{conv-phi}.
Integrating by parts in $x$ and splitting
$W(x)$ as follows $W(y)+(W(x)-W(y))$, we obtain
\begin{align*}
  \mathcal I^\delta
  &=
  \int_{\R^d}\int_{|x-y|=\delta}
  \bigl[(u(x)-u(y))\cdot\nabla K_\sigma(x-y)\bigr]
  W(y)\cdot\frac{y-x}{|y-x|}
  \varphi(x)\phi(y)\,\dd S(x)\,\dd y
  \\
  &\quad-
  \int_{\R^d}\int_{|x-y|>\delta}
  W(x)\cdot\nabla_x
  \bigl[(u(x)-u(y))\cdot\nabla K_\sigma(x-y)\bigr]
  \varphi(x)\phi(y)\,\dd x\,\dd y
  \\
  &\quad+
  \int_{\R^d}\int_{|x-y|=\delta}
  \bigl[(u(x)-u(y))\cdot\nabla K_\sigma(x-y)\bigr]
  (W(x)-W(y))\cdot\frac{y-x}{|y-x|}
  \varphi(x)\phi(y)\,\dd S(x)\,\dd y
  \\
  &:=\mathcal I_1^\delta+\mathcal I_2^\delta
     +\mathcal I_3^\delta.
\end{align*}
For the stated regularity, the integrations by parts
are justified by smooth cutoffs;
the surface identities hold for almost every $\delta>0$.

Expanding the derivative in $\mathcal I_2^\delta$ gives
\begin{align*}
  \mathcal I_2^\delta
  &=
  \int_{\R^d}\int_{|x-y|>\delta}
  (u(y)-u(x))\cdot\nabla^2K_\sigma(x-y)\cdot W(x)
  \varphi(x)\phi(y)\,\dd x\,\dd y
  \\
  &\quad-
  \int_{\R^d}\int_{|x-y|>\delta}
  \partial_Wu(x)\cdot\nabla K_\sigma(x-y)
  \varphi(x)\phi(y)\,\dd x\,\dd y.
\end{align*}

For $\mathcal I_1^\delta$, we interchange the surface
integrations and then integrate by parts in $y$:
\begin{align*}
  \mathcal I_1^\delta
  &=
  -\int_{\R^d}\int_{|y-x|=\delta}
  \bigl[(u(x)-u(y))\cdot\nabla K_\sigma(x-y)\bigr]
  W(y)\cdot\frac{x-y}{|x-y|}
  \phi(y)\,\dd S(y)\,\varphi(x)\,\dd x
  \\
  &=
  -\int_{\R^d}\int_{|y-x|>\delta}
  \div_y\Bigl(
    W(y)
    \bigl[(u(x)-u(y))\cdot\nabla K_\sigma(x-y)\bigr]
    \phi(y)
  \Bigr)\,\dd y\,\varphi(x)\,\dd x
  \\
  &=
  \int_{\R^d}\int_{|y-x|>\delta}
  (u(x)-u(y))\cdot\nabla^2K_\sigma(x-y)\cdot W(y)
  \phi(y)\,\dd y\,\varphi(x)\,\dd x
  \\
  &\quad+
  \int_{\R^d}\int_{|y-x|>\delta}
  \partial_Wu(y)\cdot\nabla K_\sigma(x-y)
  \phi(y)\,\dd y\,\varphi(x)\,\dd x
  \\
  &\quad-
  \int_{\R^d}\int_{|y-x|>\delta}
  (u(x)-u(y))\cdot\nabla K_\sigma(x-y)
  \partial_W\phi(y)\,\dd y\,\varphi(x)\,\dd x.
\end{align*}
Integration by parts at infinity is justified
by a standard cutoff argument.
The cutoff error is bounded by
$C_{\varphi,u,W}R^{-d-\sigma}\Vert \phi\Vert _{L^1}$,
which tends to zero as $R\to\infty$
because $d+\sigma>0$.

Finally, the Lipschitz bounds on $u$ and $W$,
together with
$|\nabla K_\sigma(z)|\leq C|z|^{-d-\sigma-1}$,
give
\[
  |\mathcal I_3^\delta|
  \leq
  C\delta^{-\sigma}
  \Vert \nabla u\Vert _{L^\infty}
  \Vert \nabla W\Vert _{L^\infty}
  \Vert \phi\Vert _{L^\infty}
  \Vert \varphi\Vert _{L^1}
  \longrightarrow0
  \qquad\text{as }\delta\rightarrow0.
\]
Combining $\mathcal I_1^\delta$ and $\mathcal I_2^\delta$
produces precisely the three truncated integrals
in \eqref{eq:pWcomm-u-phi}, paired with $\varphi$.
The differences of $u$, $W$, and $\partial_Wu$,
together with the kernel bounds and \eqref{conv-phi},
make all three integrals absolutely convergent.
Dominated convergence therefore gives
\eqref{eq:pWcomm-u-phi}.

Using the kernel bounds
\[
  |\nabla K_\sigma(z)|
  \leq C|z|^{-d-\sigma-1},
  \qquad
  |\nabla^2K_\sigma(z)|
  \leq C|z|^{-d-\sigma-2},
\]
and the Lipschitz bounds on $u$, $W$, and $\partial_Wu$,
we deduce from \eqref{eq:pWcomm-u-phi} that
\begin{align*}
  \Vert \partial_W[m(D),u\cdot\nabla]\phi\Vert _{L^\infty}
  &\leq
  C\Bigl(
    \Vert \nabla W\Vert _{L^\infty}\Vert \nabla u\Vert _{L^\infty}
    +\Vert \nabla(\partial_Wu)\Vert _{L^\infty}
  \Bigr)
  \sup_{x\in\R^d}
  \int_{\R^d}|x-y|^{-d-\sigma}|\phi(y)|\,\dd y
  \\
  &\quad+
  C\Vert \nabla u\Vert _{L^\infty}
  \sup_{x\in\R^d}
  \int_{\R^d}|x-y|^{-d-\sigma}
  |\partial_W\phi(y)|\,\dd y
  \\
  &\leq
  C\Bigl(
    \Vert \nabla W\Vert _{L^\infty}\Vert \nabla u\Vert _{L^\infty}
    +\Vert \nabla(\partial_Wu)\Vert _{L^\infty}
  \Bigr)\Vert \phi\Vert _{L^1\cap L^\infty}
  \\
  &\quad+
  C\Vert \nabla u\Vert _{L^\infty}
  \Vert \partial_W\phi\Vert _{L^1\cap L^\infty},
\end{align*}
where the last inequality follows from \eqref{conv-phi}.
Summing over the vector fields in $\mathcal W$
proves \eqref{es:prod-es4}.
\end{proof}

We next establish H\"older bounds for
$[m_\sigma(D),u\cdot\nabla]\phi$
and its tangential derivatives when $-d<\sigma<-1$.

\begin{lemma}\label{lem:pW-Cgam}
Let $d\geq2$, and let $u$ and
$\mathcal W=\{W_i\}_{i=1}^N$
be smooth divergence-free vector fields on $\R^d$.
Let $\phi$ be a smooth scalar function.
Assume that
\[
  m_\sigma(D)=\Lambda^\sigma m_0(D),
  \qquad -d<\sigma<-1,
\]
where 
$m_0\in C^\infty(\R^d\setminus\{0\})$
is homogeneous of degree zero.

For every $0<s<1$, there exists a constant
$C=C(d,s,\sigma,m_0)>0$ such that
\begin{equation}\label{es:prod-es6}
  \Vert [m_\sigma(D),u\cdot\nabla]\phi\Vert _{C^s}
  \leq
  C\Vert u\Vert _{C^s}\Vert \phi\Vert _{L^1\cap L^\infty},
\end{equation}
and
\begin{equation}\label{es:prod-es7}
\begin{aligned}
  \Vert \partial_{\mathcal W}
  ([m_\sigma(D),u\cdot\nabla]\phi)\Vert _{C^s}
  &\leq
  C\Vert \nabla u\Vert _{L^\infty}
  \Bigl(
    \Vert \nabla\mathcal W\Vert _{L^\infty}
    \Vert \phi\Vert _{L^1\cap L^\infty}
    +
    \Vert \partial_{\mathcal W}\phi\Vert _{L^1\cap L^\infty}
  \Bigr)
  \\
  &\quad+
  C\Vert \partial_{\mathcal W}u\Vert _{C^s}
  \Vert \phi\Vert _{L^1\cap L^\infty}.
\end{aligned}
\end{equation}
\end{lemma}

\begin{proof}[Proof of Lemma~\ref{lem:pW-Cgam}]
We first prove \eqref{es:prod-es6}.
Let $K_\sigma$ be the kernel of $m_\sigma(D)$
given by Lemma~\ref{lem:mDphi}.
Since $-d<\sigma+1<0$, the same argument
as in \eqref{conv-phi} gives
\[
  \sup_{x\in\R^d}
  \int_{\R^d}|x-y|^{-d-1-\sigma}|\phi(y)|\,\dd y
  \leq C\Vert \phi\Vert _{L^1\cap L^\infty}.
\]
Thus, by \eqref{eq:comm-u-phi} and the bound
$|\nabla K_\sigma(z)|\leq C|z|^{-d-1-\sigma}$,
\begin{equation}\label{es:convu-phi}
\begin{aligned}
  \Vert [m_\sigma(D),u\cdot\nabla]\phi\Vert _{L^\infty}
  &\leq
  C\Vert u\Vert _{L^\infty}
  \sup_{x\in\R^d}
  \int_{\R^d}|x-y|^{-d-1-\sigma}|\phi(y)|\,\dd y
  \\
  &\leq
  C\Vert u\Vert _{L^\infty}\Vert \phi\Vert _{L^1\cap L^\infty}.
\end{aligned}
\end{equation}

To estimate the H\"older seminorm, let $h\in\R^d$
with $0<|h|<1$. We write
\begin{align*}
  &([m_\sigma(D),u\cdot\nabla]\phi)(x)
  -([m_\sigma(D),u\cdot\nabla]\phi)(x+h)
  \\
  &\quad=
  \int_{|x-y|<2|h|}
  (u(x+h)-u(y))\cdot\nabla K_\sigma(x+h-y)
  \phi(y)\,\dd y
  \\
  &\qquad-
  \int_{|x-y|<2|h|}
  (u(x)-u(y))\cdot\nabla K_\sigma(x-y)
  \phi(y)\,\dd y
  \\
  &\qquad+
  \int_{|x-y|\geq2|h|}
  (u(x+h)-u(x))\cdot\nabla K_\sigma(x+h-y)
  \phi(y)\,\dd y
  \\
  &\qquad+
  \int_{|x-y|\geq2|h|}
  (u(x)-u(y))\cdot
  \bigl(\nabla K_\sigma(x+h-y)-\nabla K_\sigma(x-y)\bigr)
  \phi(y)\,\dd y
  \\
  &\quad:=\mathcal I_1+\mathcal I_2+\mathcal I_3+\mathcal I_4.
\end{align*}

For the first two terms, the H\"older continuity of $u$
and the kernel bound yield
\begin{align*}
  |\mathcal I_1|+|\mathcal I_2|
  &\leq
  C[u]_{C^s}\Vert \phi\Vert _{L^\infty}
  \int_{|z|<3|h|}|z|^{-d-1-\sigma+s}\,\dd z
  \\
  &\leq
  C[u]_{C^s}\Vert \phi\Vert _{L^\infty}
  |h|^{s-1-\sigma}
  \leq
  C|h|^s[u]_{C^s}\Vert \phi\Vert _{L^\infty},
\end{align*}
where the last inequality uses $\sigma<-1$ and $|h|<1$.

On $|x-y|\geq2|h|$, we have
$|x+h-y|\geq|x-y|/2$.
Consequently,
\begin{align*}
  |\mathcal I_3|
  &\leq
  C|h|^s[u]_{C^s}
  \int_{|x-y|\geq2|h|}
  |x-y|^{-d-1-\sigma}|\phi(y)|\,\dd y
  \\
  &\leq
  C|h|^s[u]_{C^s}
  \Vert \phi\Vert _{L^1\cap L^\infty}.
\end{align*}

For $\mathcal I_4$, the mean value formula and
$|\nabla^2K_\sigma(z)|\leq C|z|^{-d-2-\sigma}$
give
\[
  |\nabla K_\sigma(x+h-y)-\nabla K_\sigma(x-y)|
  \leq C|h||x-y|^{-d-2-\sigma}
  \qquad\text{for }|x-y|\geq2|h|.
\]
Since $s<1$, on this region
$|h||x-y|^{s-1}\leq C|h|^s$. 
It follows that
\begin{align*}
  |\mathcal I_4|
  &\leq
  C|h|[u]_{C^s}
  \int_{|x-y|\geq2|h|}
  |x-y|^{-d-2-\sigma+s}|\phi(y)|\,\dd y
  \\
  &\leq
  C|h|^s[u]_{C^s}
  \int_{\R^d}
  |x-y|^{-d-1-\sigma}|\phi(y)|\,\dd y
  \leq
  C|h|^s[u]_{C^s}
  \Vert \phi\Vert _{L^1\cap L^\infty}.
\end{align*}

Combining these estimates controls all increments
with $0<|h|<1$, while those with $|h|\geq1$
are controlled by \eqref{es:convu-phi}.
This proves \eqref{es:prod-es6}.

We next prove \eqref{es:prod-es7}.
Fix $W\in\mathcal W$.
By \eqref{eq:pWcomm-u-phi}, we write $\partial_W[m_\sigma(D),u\cdot\nabla]\phi
  =F_1+F_2$  
where
\begin{align*}
  F_1(x)
  &:=
  \int_{\R^d}
  (\partial_Wu(y)-\partial_Wu(x))
  \cdot\nabla K_\sigma(x-y)\phi(y)\,\dd y,
  \\
  F_2(x)
  &:=
  \int_{\R^d}
  (u(y)-u(x))\cdot\nabla^2K_\sigma(x-y)
  \cdot(W(x)-W(y))\phi(y)\,\dd y
  \\
  &\quad+
  \int_{\R^d}
  (u(y)-u(x))\cdot\nabla K_\sigma(x-y)
  \partial_W\phi(y)\,\dd y.
\end{align*}

The kernel estimates used to prove \eqref{es:prod-es6}
apply to $F_1$ with $u$ replaced by $\partial_Wu$,
and yield
\[
  \Vert F_1\Vert _{C^s}
  \leq
  C\Vert \partial_Wu\Vert _{C^s}\Vert \phi\Vert _{L^1\cap L^\infty}.
\]

For $F_2$, differentiation gives, for $1\leq i\leq d$,
\begin{equation}\label{eq:pWcomm-u-phi3}
\begin{aligned}
  \partial_iF_2(x)
  &=
  -\int_{\R^d}
  \partial_i u(x)\cdot\nabla^2K_\sigma(x-y)
  \cdot(W(x)-W(y))\phi(y)\,\dd y
  \\
  &\quad+
  \int_{\R^d}
  (u(y)-u(x))\cdot\nabla^2\partial_iK_\sigma(x-y)
  \cdot(W(x)-W(y))\phi(y)\,\dd y
  \\
  &\quad+
  \int_{\R^d}
  (u(y)-u(x))\cdot\nabla^2K_\sigma(x-y)
  \cdot\partial_iW(x)\phi(y)\,\dd y
  \\
  &\quad-
  \int_{\R^d}
  \partial_i u(x)\cdot\nabla K_\sigma(x-y)
  \partial_W\phi(y)\,\dd y
  \\
  &\quad+
  \int_{\R^d}
  (u(y)-u(x))\cdot\nabla\partial_iK_\sigma(x-y)
  \partial_W\phi(y)\,\dd y.
\end{aligned}
\end{equation}
This identity is justified by differentiating truncated
integrals and passing to the limit.
The cutoff contribution near $x=y$ is
$O(\delta^{-1-\sigma})$ and therefore vanishes
as $\delta\rightarrow0$, since $\sigma<-1$.

The kernel bounds
\[
  |\nabla^jK_\sigma(z)|
  \leq C|z|^{-d-\sigma-j},
  \qquad j=1,2,3,
\]
and the Lipschitz differences of $u$ and $W$
reduce the estimates of $F_2$ and $\nabla F_2$
to convolution with
$|z|^{-d-\sigma}$ and $|z|^{-d-\sigma-1}$,
respectively.
Both kernels are locally integrable and bounded
on $\{|z|\geq1\}$.
Using \eqref{conv-phi} and its analogue
with $\sigma$ replaced by $\sigma+1$, we obtain that, for $0<s<1$
\begin{align*}
  \Vert F_2\Vert _{C^s}
  &\leq
  C\bigl(\Vert F_2\Vert _{L^\infty}
       +\Vert \nabla F_2\Vert _{L^\infty}\bigr)
  \\
  &\leq
  C\Vert \nabla u\Vert _{L^\infty}
  \Bigl(
    \Vert \nabla W\Vert _{L^\infty}
    \Vert \phi\Vert _{L^1\cap L^\infty}
    +
    \Vert \partial_W\phi\Vert _{L^1\cap L^\infty}
  \Bigr).
\end{align*}
Combining the bounds for $F_1$ and $F_2$
and summing over $W\in\mathcal W$
proves \eqref{es:prod-es7}.
\end{proof}

\subsection{Estimates for transport, transport-diffusion,
and linearized velocity equations}
\label{subsec:evolution-estimates}

We first recall the Besov estimates
for the transport equation;
see~\cite{BCD11}
and~\cite[Proposition~4.1]{HKR10}.
\begin{lemma}\label{lem:TE-prog-es}
Let $p\in[1,\infty]$.
Assume that $u$ is a smooth divergence-free vector field
on $\R^d$, $d\geq2$, and that $\phi$ is a smooth solution of
\[
  \partial_t\phi+u\cdot\nabla\phi=f,
  \qquad
  \phi|_{t=0}=\phi_0.
\]
Then the following estimates hold.
\begin{enumerate}[(1)]
\item
For every $r\in[1,\infty]$ and $s\in(-1,1)$,
there exists a constant $C=C(d,s)>0$ such that
\begin{equation}\label{eq:T-BesEs0}
\begin{aligned}
  \Vert \phi\Vert _{L^\infty_t(B^s_{p,r})}
  \leq C\Bigl(
    \Vert \phi_0\Vert _{B^s_{p,r}}
    +\Vert f\Vert _{\widetilde L^1_t(B^s_{p,r})}
    +\int_0^t
    \Vert \nabla u(\tau)\Vert _{L^\infty}
    \Vert \phi(\tau)\Vert _{B^s_{p,r}}\,\dd\tau
  \Bigr).
\end{aligned}
\end{equation}
Consequently,
\begin{equation}\label{eq:T-BesEs}
  \Vert \phi\Vert _{L^\infty_t(B^s_{p,r})}
  \leq
  Ce^{C\int_0^t\Vert \nabla u(\tau)\Vert _{L^\infty}\,\dd\tau}
  \left(
    \Vert \phi_0\Vert _{B^s_{p,r}}
    +\Vert f\Vert _{\widetilde L^1_t(B^s_{p,r})}
  \right).
\end{equation}

\item
At the endpoint $s=-1$, there exists a constant $C>0$
such that
\begin{equation}\label{eq:TE-prog-es3}
  \Vert \phi(t)\Vert _{B^{-1}_{p,\infty}}
  \leq
  Ce^{C\int_0^t\Vert u(\tau)\Vert _{B^1_{\infty,1}}\,\dd\tau}
  \left(
    \Vert \phi_0\Vert _{B^{-1}_{p,\infty}}
    +\int_0^t
    \Vert f(\tau)\Vert _{B^{-1}_{p,\infty}}\,\dd\tau
  \right).
\end{equation}
\end{enumerate}
\end{lemma}

The following smoothing estimates for the transport-diffusion
equation will be used to propagate regularity.
\begin{lemma}\label{lem:TD-sm}
Let $\nu>0$, $\rho,p,r\in[1,\infty]$, and $0<\alpha\leq1$.
Assume that $u$ is a smooth divergence-free vector field
on $\R^d$, $d\geq2$, and that $\phi$ is a smooth solution of
\begin{equation}\label{eq:TD}
  \partial_t\phi+u\cdot\nabla\phi+\nu\Lambda^{2\alpha}\phi=f,
  \qquad
  \phi|_{t=0}=\phi_0.
\end{equation}
For every $-1<s<1$, there exists
$C=C(d,s,\alpha)>0$, independent of $\nu$, such that
for all $t>0$,
\begin{equation}\label{eq:TD-sm4}
\begin{aligned}
  \nu^{1/\rho}
  \left\Vert 
    \left(
      2^{q(s+2\alpha/\rho)}
      \Vert \Delta_q\phi\Vert _{L^\rho_t(L^p)}
    \right)_{q\in\mathbb N}
  \right\Vert _{\ell^r}
  \leq
  Ce^{C\int_0^t\Vert \nabla u(\tau)\Vert _{L^\infty}\,\dd\tau}
  \left(
    \Vert \phi_0\Vert _{B^s_{p,r}}
    +\Vert f\Vert _{L^1_t(B^s_{p,r})}
  \right).
\end{aligned}
\end{equation}
At the endpoint $s=-1$, $r=\infty$,
there exists $C=C(d,\alpha)>0$, independent of $\nu$,
such that
\begin{equation}\label{eq:TD-sm5}
\begin{aligned}
  \nu^{1/\rho}
  \sup_{q\in\mathbb N}
  2^{q(-1+2\alpha/\rho)}
  \Vert \Delta_q\phi\Vert _{L^\rho_t(L^p)}
  \leq
  Ce^{C\int_0^t
    \Vert \nabla u(\tau)\Vert _{B^0_{\infty,1}}\,\dd\tau}
  \left(
    \Vert \phi_0\Vert _{B^{-1}_{p,\infty}}
    +\Vert f\Vert _{L^1_t(B^{-1}_{p,\infty})}
  \right).
\end{aligned}
\end{equation}
\end{lemma}

\begin{proof}[Proof of Lemma~\ref{lem:TD-sm}]
Estimate \eqref{eq:TD-sm4} follows from
\cite[Lemma~2.9]{LXY25}.
The proof of \eqref{eq:TD-sm5} is similar to that of
\eqref{eq:TD-sm4}, with the main difference lying
in the treatment of the commutator term.
It therefore suffices to establish the endpoint estimate
\begin{equation}\label{eq:comm-es-end}
  \sup_{q\geq-1}
  2^{-q}\Vert [\Delta_q,u\cdot\nabla]\phi\Vert _{L^p}
  \leq
  C\Vert \nabla u\Vert _{B^0_{\infty,1}}
  \Vert \phi\Vert _{B^{-1}_{p,\infty}}.
\end{equation}
We treat $q\geq4$ and $-1\leq q\leq3$ separately.

For $q\geq4$, Bony's decomposition and $\div u=0$ give
\begin{equation}\label{eq:comm-decomp-u-phi}
\begin{aligned}
  \bigl[\Delta_q,u\cdot\nabla\bigr]\phi
  &=
  \Delta_q\partial_jR(u^j,\phi)
  -R(u^j,\partial_j\Delta_q\phi)
  \\
  &\quad+
  \Delta_qT_{\partial_j\phi}u^j
  -T_{\Delta_q\partial_j\phi}u^j
  -[T_{u^j},\Delta_q]\partial_j\phi
  :=\sum_{i=1}^5R_q^i.
\end{aligned}
\end{equation}
Here and below, repeated spatial indices are summed.
By Bernstein's inequality,
\begin{equation}\label{eq:R1-u-phi}
\begin{aligned}
  2^{-q}\Vert R_q^1\Vert _{L^p}
  &\leq
  C\sum_{k\geq q-3}
  \Vert \Delta_k\phi\Vert _{L^p}
  \Vert \widetilde\Delta_ku\Vert _{L^\infty}
  \\
  &\leq
  C\Vert \phi\Vert _{B^{-1}_{p,\infty}}
  \sum_{k\geq q-3}
  \sum_{|k'-k|\leq1}
  \Vert \Delta_{k'}\nabla u\Vert _{L^\infty}
  \leq
  C\Vert \nabla u\Vert _{B^0_{\infty,1}}
  \Vert \phi\Vert _{B^{-1}_{p,\infty}}
\end{aligned}
\end{equation}
where $k\geq1$ and $k'\geq0$ because $q\geq4$.
For the next three terms, spectral localization, Bernstein's inequality, 
and $\Vert S_{k-1} \nabla \phi\Vert _{L^p} \leq C 2^{2k} \Vert \phi\Vert _{B^{-1}_{p,\infty}}$, give
\begin{align*}
  2^{-q}\sum_{i=2}^4\Vert R_q^i\Vert _{L^p}
  &\leq C\Vert \phi\Vert _{B^{-1}_{p,\infty}}
  \bigg(
    \sum_{\substack{k\in\mathbb{N};\,|k-q|\leq4}}
      2^{|k-q|}\Vert \Delta_k\nabla u\Vert _{L^\infty}
    +\sum_{k\geq q-2}
      2^{q-k}\Vert \Delta_k\nabla u\Vert _{L^\infty}
  \bigg)
  \\
  &\leq C\Vert \nabla u\Vert _{B^0_{\infty,\infty}}
    \Vert \phi\Vert _{B^{-1}_{p,\infty}}.
\end{align*}

Finally,
$R_q^5 =  \sum_{\substack{k\in\mathbb N;\,|k-q|\leq4}}
  [\Delta_q,S_{k-1}u^j]\Delta_k\partial_j\phi$.
The standard kernel commutator estimate
(see~\cite[Lemma~2.97]{BCD11}) gives
\begin{equation}\label{eq:R4-u-phi}
  \sup_{q\geq4}2^{-q}\Vert R_q^5\Vert _{L^p}
  \leq C \Vert \nabla u\Vert _{L^\infty}\sup_{q\ge -1}2^{-q}\Vert \Delta_q \phi\Vert _{L^p} \leq C \Vert \nabla u\Vert _{L^\infty}\;\Vert \phi\Vert _{B^{-1}_{p,\infty}}.
\end{equation}
Combining the five estimates and using
$B^0_{\infty,1}\hookrightarrow L^\infty$ gives
\begin{equation}\label{es:qge4}
  \sup_{q\geq4}
  2^{-q}\Vert [\Delta_q,u\cdot\nabla]\phi\Vert _{L^p}
  \leq
  C\Vert \nabla u\Vert _{B^0_{\infty,1}}
  \Vert \phi\Vert _{B^{-1}_{p,\infty}}.
\end{equation}

For $-1\leq q\leq3$, decompose $\phi=S_8\phi+\phi^h,\,\phi^h:=(\mathrm{Id}-S_8)\phi$.

We first estimate the low-frequency contribution.
Let $h_q$ denote the convolution kernel of $\Delta_q$.
Since $\div u=0$, integration by parts gives
\[
  [\Delta_q,u\cdot\nabla]S_8\phi(x)
  =
  \int_{\R^d}
  \nabla h_q(x-y)\cdot
  \bigl(u(y)-u(x)\bigr)S_8\phi(y)\,\dd y.
\]
Using
$|u(y)-u(x)|\leq |x-y|\Vert \nabla u\Vert _{L^\infty}$ and $\int_{\R^d}|z|\,|\nabla h_q(z)|\,\dd z\leq C$,
we obtain
\[
  \Vert [\Delta_q,u\cdot\nabla]S_8\phi\Vert _{L^p}
  \leq
  C\Vert \nabla u\Vert _{L^\infty}\Vert S_8\phi\Vert _{L^p}.
\]
Moreover,
$\Vert S_8\phi\Vert _{L^p}
  \leq  C\Vert \phi\Vert _{B^{-1}_{p,\infty}}$.
Therefore, for $-1\leq q\leq3$,
\[
  2^{-q}\Vert [\Delta_q,u\cdot\nabla]S_8\phi\Vert _{L^p}
  \leq
  C\Vert \nabla u\Vert _{L^\infty}
  \Vert \phi\Vert _{B^{-1}_{p,\infty}}.
\]

For the high-frequency contribution,
$\Delta_q\phi^h=0$ for $-1\leq q\leq 3$, and hence
$$[\Delta_q,u\cdot\nabla]\phi^h = \Delta_q\div(u\phi^h).$$.
By spectral localization, the two paraproduct terms
vanish after applying $\Delta_q$. Thus
\[
  \Delta_q\div(u\phi^h)
  =
  \sum_{k\geq6}
  \Delta_q\div
  \bigl(\Delta_ku\,\widetilde\Delta_k\phi^h\bigr).
\]
Using
\[
  k\geq 6,\qquad\Vert \Delta_ku\Vert _{L^\infty}
  \leq C2^{-k}\Vert \Delta_k\nabla u\Vert _{L^\infty},
  \qquad
  \Vert \widetilde\Delta_k\phi^h\Vert _{L^p}
  \leq C2^k\Vert \phi\Vert _{B^{-1}_{p,\infty}},
\]
we infer that
\begin{align*}
  \sup_{-1\leq q\leq 3} 2^{-q}\Vert [\Delta_q,u\cdot\nabla]\phi^h\Vert _{L^p}
  & \leq
  C\sum_{k\geq6}
  \Vert \Delta_ku\Vert _{L^\infty}
  \Vert \widetilde\Delta_k\phi^h\Vert _{L^p} \\
  & \leq
  C\Vert \phi\Vert _{B^{-1}_{p,\infty}}
  \sum_{k\geq6}\Vert \Delta_k\nabla u\Vert _{L^\infty}
  \leq
  C\Vert \nabla u\Vert _{B^0_{\infty,1}}
  \Vert \phi\Vert _{B^{-1}_{p,\infty}}.
\end{align*}
Together with the low-frequency estimate
and \eqref{es:qge4}, this proves
\eqref{eq:comm-es-end}.
\end{proof}

We next establish an analogue of
\cite[Proposition~4.3]{HKR10}
for the linearized velocity equation with
$\alpha\in(\tfrac12,1]$ and $p\in[2,\infty)$.

\begin{proposition}\label{prop-linear-velocity-eqn}
Let $d\geq2$, $\alpha\in(\tfrac12,1]$, $s\in(-1,1)$,
$p\in[2,\infty)$, and $\rho\in[1,\infty]$.
Let $v$ be a smooth divergence-free vector field on
$\mathbb R^d$, and let $u$ be a smooth solution of
\begin{equation}\label{eq:LB}
  \partial_t u+v\cdot\nabla u+\Lambda^{2\alpha}u
  +\nabla \pi=f,
  \qquad
  \div u=0,
  \qquad
  u|_{t=0}=u_0.
\end{equation}
Then, for every $t>0$,
\begin{equation}\label{es:linear-velocity-eqn}
\begin{aligned}
  \Vert u\Vert _{L^\infty_t(B^s_{p,\infty})}
  \leq
  Ce^{C\int_0^t\Vert v(\tau)\Vert _{W^{1,\infty}}\,\dd\tau}
  \bigg(
    \Vert u_0\Vert _{B^s_{p,\infty}}
    +(1+t^{1-\frac1\rho})
    \Vert f\Vert _{\widetilde L^\rho_t
      (B^{s-2\alpha(1-\frac1\rho)}_{p,\infty})}
  \bigg).
\end{aligned}
\end{equation}
More generally, if $f=f_1+f_2$, then, for every
$\rho_1,\rho_2\in[1,\infty]$ and $t>0$,
\begin{equation}\label{es:linear-velocity-eqn-2}
\begin{aligned}
  \Vert u\Vert _{L^\infty_t(B^s_{p,\infty})}
  \leq
  Ce^{C\int_0^t\Vert v(\tau)\Vert _{W^{1,\infty}}\,\dd\tau}
  \bigg(
    \Vert u_0\Vert _{B^s_{p,\infty}}
    +\sum_{i=1}^2
    (1+t^{1-\frac1{\rho_i}})
    \Vert f_i\Vert _{\widetilde L^{\rho_i}_t
      (B^{s-2\alpha(1-\frac1{\rho_i})}_{p,\infty})}
  \bigg).
\end{aligned}
\end{equation}
The constants depend only on $d,\alpha,s,p$
and the indicated time integrability exponents.
\end{proposition}

\begin{proof}[Proof of Proposition~\ref{prop-linear-velocity-eqn}]
We prove \eqref{es:linear-velocity-eqn}.
Estimate \eqref{es:linear-velocity-eqn-2}
follows by estimating the two forcing terms separately
in the same argument.

For $q\in\mathbb N$, applying $\mathbb P\Delta_q$
to \eqref{eq:LB}, where
$\mathbb P:=\mathrm{Id}-\nabla\Delta^{-1}\div$
is the Leray projector, gives
\begin{align*}
  \partial_t\Delta_q u
  +S_{q-1}v\cdot\nabla\Delta_q u
  +\Lambda^{2\alpha}\Delta_q u
  =
  \mathbb P\Delta_q f+\mathbb P R_q
  -[\mathbb P,S_{q-1}v\cdot\nabla]\Delta_q u,
\end{align*}
where
\[
  R_q:=(S_{q-1}v-v)\cdot\nabla\Delta_q u
       -[\Delta_q,v\cdot\nabla]u.
\]
Testing the localized equation with $|\Delta_q u|^{p-2}\Delta_q u$,
using $\div S_{q-1}v=0$, the $L^p$ boundedness of $\mathbb P$, 
and the frequency-localized Bernstein inequality
in~\cite{CMZ07}, we obtain
\begin{align*}
  \frac{\dd}{\dd t}\Vert \Delta_q u(t)\Vert _{L^p}
  +c2^{2\alpha q}\Vert \Delta_q u(t)\Vert _{L^p}
  \leq
  C\Big(
    \Vert R_q\Vert _{L^p}
    +\Vert \Delta_q f\Vert _{L^p}
    +\Vert [\mathbb P,S_{q-1}v\cdot\nabla]
      \Delta_q u\Vert _{L^p}
  \Big).
\end{align*}
where $c>0$ is independent of $q$.
Integrating this differential inequality and H\"older's inequality in time yield 
\begin{align*}
  \Vert \Delta_q u\Vert _{L^\infty_t(L^p)}
  &\leq
  \Vert \Delta_q u_0\Vert _{L^p}
  +C2^{-2\alpha q(1-\frac1\rho)}
    \Vert \Delta_q f\Vert _{L^\rho_t(L^p)}
  \\
  &\quad+
  C\int_0^t
  \Big(
    \Vert R_q(\tau)\Vert _{L^p}
    +\Vert [\mathbb P,S_{q-1}v(\tau)\cdot\nabla]
      \Delta_q u(\tau)\Vert _{L^p}
  \Big)\,\dd\tau.
\end{align*}
Multiplying by $2^{qs}$ and taking the supremum over
$q\in\mathbb N$, we find
\begin{equation}\label{rrr}
\begin{aligned}
  \sup_{q\in\mathbb N}
  2^{qs}\Vert \Delta_q u\Vert _{L^\infty_t(L^p)}
  &\leq
  C\Vert u_0\Vert _{B^s_{p,\infty}}
  +C\Vert f\Vert _{\widetilde L^\rho_t(B^{s'}_{p,\infty})}
  \\
  &\quad+
  C\int_0^t
  \sup_{q\in\mathbb N}2^{qs}
  \Big(
    \Vert R_q\Vert _{L^p}
    +\Vert [\mathbb P,S_{q-1}v\cdot\nabla]
      \Delta_q u\Vert _{L^p}
  \Big)\,\dd\tau,
\end{aligned}
\end{equation}
where $s':=s-2\alpha(1-\frac1\rho)$.

We next estimate the two remainder terms.
By \cite[Lemma~2.100]{BCD11}, for $-1<s<1$,
\[
  \sup_{q\in\mathbb N}
  2^{qs}\Vert [\Delta_q,v\cdot\nabla]u\Vert _{L^p}
  \leq
  C\Vert \nabla v\Vert _{L^\infty}\Vert u\Vert _{B^s_{p,\infty}}.
\]
For $q\geq1$, Bernstein's inequality gives
\begin{align*}
  2^{qs}
  \Vert (S_{q-1}v-v)\cdot\nabla\Delta_q u\Vert _{L^p}
  \leq
  C2^{qs}
  \sum_{k\geq q-1}
  \Vert \Delta_k v\Vert _{L^\infty}
  2^q\Vert \Delta_q u\Vert _{L^p}
  \leq
  C\Vert \nabla v\Vert _{L^\infty}
  2^{qs}\Vert \Delta_q u\Vert _{L^p}.
\end{align*}
For $q=0$, the convention $S_{-1}=0$ instead gives
$\Vert (S_{-1}v-v)\cdot\nabla\Delta_0 u\Vert _{L^p}\leq C\Vert v\Vert _{L^\infty}\Vert \Delta_0 u\Vert _{L^p}$.
Thus
\[
  \sup_{q\in\mathbb N}2^{qs}\Vert R_q\Vert _{L^p}
  \leq
  C\Vert v\Vert _{W^{1,\infty}}\Vert u\Vert _{B^s_{p,\infty}}.
\]
To estimate the projection commutator, first note that
it vanishes for $q=0$, since $S_{-1}=0$.
For $q\geq1$, the Fourier supports of both
$\Delta_q u$ and
$S_{q-1}v\cdot\nabla\Delta_q u$
lie in a fixed annulus dilated by $2^q$.
Choose
$\widetilde\psi\in
C_c^\infty(\mathbb R^d\setminus\{0\})$
equal to one on this annulus.
Then
\[
  [\mathbb P,S_{q-1}v\cdot\nabla]\Delta_q u
  =
  [\mathbb P\widetilde\psi(2^{-q}D),
    S_{q-1}v\cdot\nabla]\Delta_q u.
\]
Since the symbol of $\mathbb P$ is homogeneous of degree zero,
the localized operator has the convolution kernel
$2^{qd}\overline h(2^q\cdot)$,
where $\overline h$ is a matrix-valued Schwartz function.
Hence
\begin{align*}
  [\mathbb P,S_{q-1}v\cdot\nabla]\Delta_q u(x)
  =
  2^{qd}\int_{\mathbb R^d}
  \overline h(2^q y)
  \Big(
    \big(S_{q-1}v(x-y)-S_{q-1}v(x)\big)
    \cdot\nabla\Delta_q u(x-y)
  \Big)\,\dd y.
\end{align*}
Using the Lipschitz bound on $S_{q-1}v$
and Bernstein's inequality, we obtain
\begin{align*}
  2^{qs}
  \Vert [\mathbb P,S_{q-1}v\cdot\nabla]\Delta_q u\Vert _{L^p}
  &\leq
  C2^{q(s-1)}
  \Vert |y|\overline h(y)\Vert _{L^1}
  \Vert \nabla S_{q-1}v\Vert _{L^\infty}
  \Vert \nabla\Delta_q u\Vert _{L^p}
  \\
  &\leq
  C\Vert \nabla v\Vert _{L^\infty}
  2^{qs}\Vert \Delta_q u\Vert _{L^p} \leq
  C\Vert \nabla v\Vert _{L^\infty}\Vert u\Vert _{B^s_{p,\infty}}.
\end{align*}
Inserting these estimates into \eqref{rrr} gives
\begin{align*}
  \sup_{q\in\mathbb N}
  2^{qs}\Vert \Delta_q u\Vert _{L^\infty_t(L^p)}
  \leq
  C\Vert u_0\Vert _{B^s_{p,\infty}}
  +C\Vert f\Vert _{\widetilde L^\rho_t(B^{s'}_{p,\infty})}
  +
  C\int_0^t
  \Vert v(\tau)\Vert _{W^{1,\infty}}
  \Vert u(\tau)\Vert _{B^s_{p,\infty}}\,\dd\tau.
\end{align*}

For the low-frequency block, applying
$\mathbb P\Delta_{-1}$ to \eqref{eq:LB} gives
\begin{align*}
  \partial_t\Delta_{-1}u
  +\Lambda^{2\alpha}\Delta_{-1}u
  &=
  -\mathbb P([\Delta_{-1},v\cdot\nabla]u)
  +\mathbb P\Delta_{-1}f
  -\mathbb P(v\cdot\nabla\Delta_{-1}u).
\end{align*}
The nonnegativity of the dissipation term
(see~\cite{CorC04}), together with the commutator estimate
in~\cite[Lemma~2.100]{BCD11}
and Bernstein's inequality, gives
\begin{align*}
  \frac{\dd}{\dd t}\Vert \Delta_{-1}u(t)\Vert _{L^p}
  &\leq
  C\Big(
    \Vert \nabla v(t)\Vert _{L^\infty}
    \Vert u(t)\Vert _{B^s_{p,\infty}}
    +\Vert \Delta_{-1}f(t)\Vert _{L^p}
    +\Vert v(t)\Vert _{L^\infty}\Vert \Delta_{-1}u(t)\Vert _{L^p}
  \Big)
  \\
  &\leq
  C\Big(
    \Vert v(t)\Vert _{W^{1,\infty}}
    \Vert u(t)\Vert _{B^s_{p,\infty}}
    +\Vert \Delta_{-1}f(t)\Vert _{L^p}
  \Big).
\end{align*}
Integrating in time and applying H\"older's inequality,
we conclude that
\begin{align*}
  \Vert \Delta_{-1}u\Vert _{L^\infty_t(L^p)}
  \leq
  C\Vert u_0\Vert _{B^s_{p,\infty}}
  +C\int_0^t
    \Vert v(\tau)\Vert _{W^{1,\infty}}
    \Vert u(\tau)\Vert _{B^s_{p,\infty}}\,\dd\tau
  +
  Ct^{1-\frac1\rho}
  \Vert f\Vert _{\widetilde L^\rho_t(B^{s'}_{p,\infty})}.
\end{align*}

Combining the high- and low-frequency estimates gives
\begin{align*}
  \Vert u\Vert _{L^\infty_t(B^s_{p,\infty})}
  \leq
  C\Vert u_0\Vert _{B^s_{p,\infty}}
  +C\int_0^t
    \Vert v(\tau)\Vert _{W^{1,\infty}}
    \Vert u(\tau)\Vert _{B^s_{p,\infty}}\,\dd\tau
  +
  C(1+t^{1-\frac1\rho})
  \Vert f\Vert _{\widetilde L^\rho_t(B^{s'}_{p,\infty})}.
\end{align*}
Gronwall's inequality now proves
\eqref{es:linear-velocity-eqn}.
\end{proof}

\addtocontents{toc}{\protect\setcounter{tocdepth}{2}}
\section{Global well-posedness for the 2D Boussinesq system}
\label{sec:well-pose}
The purpose of this section is to give the proof of Theorem \ref{thm:main1}. 

We begin with the definition of weak solutions.

\begin{definition}\label{defn:weak-Bouss}
Let $T>0$ be fixed and let $(u_0,\theta_0)\in L_{\mathrm{loc}}^2 (\mathbb R^2)\times L_{\mathrm{loc}}^2 (\mathbb R^2)$ satisfies $\operatorname{div}u_0=0$. We say that
\begin{align*}
    (u,\theta)
    \in
    L^2([0,T], L_{\mathrm{loc}}^2 (\mathbb R^2))
    \times
    L^2([0,T], L_{\mathrm{loc}}^2 (\mathbb R^2))
\end{align*}
is a weak solution to \eqref{BoussEq} on $[0,T]$ with initial data
$(u_0,\theta_0)$ if, in addition, $u\in L^1([0,T],L^p(\mathbb R^2))$ for some $1\leq p\leq\infty$, and the following identities hold.

For every divergence-free
$\phi\in C_c^\infty([0,T)\times\mathbb R^2;\mathbb R^2)$,
\begin{align*}
 \int_0^T\int_{\mathbb R^2}
 \left(
   \varepsilon u\cdot\partial_t\phi
   +\varepsilon u\otimes u:\nabla\phi
   -u\cdot\Lambda^{2\alpha}\phi
   +\theta\phi\cdot e_2
 \right)\,\dd x\,\dd t +\varepsilon\int_{\mathbb R^2}
    u_0(x)\cdot\phi(0,x)\,\dd x
 =0,
\end{align*}
and, for every
$\psi\in C_c^\infty([0,T)\times\mathbb R^2)$,
\begin{align*}
 \int_0^T\int_{\mathbb R^2}
 \left(
   \theta\partial_t\psi
   +\theta u\cdot\nabla\psi
 \right)\,\dd x\,\dd t
 +\int_{\mathbb R^2}
    \theta_0(x)\psi(0,x)\,\dd x
 &=0,\\
 \int_0^T\int_{\mathbb R^2}
 u\cdot\nabla\psi\,\dd x\,\dd t
 &=0.
\end{align*}
If $(u,\theta)$ is a weak solution on $[0,T]$ for every $T>0$,
then we say that $(u,\theta)$ is a global weak solution to \eqref{BoussEq}.
\end{definition}

We next have the following basic \textit{{\it{a priori}}} estimates.
\begin{proposition}\label{prop:ap-es-basic}
  The smooth solution $(u,\theta)$ to the 2D Boussinesq-Navier-Stokes system \eqref{BoussEq} satisfies the following estimates.
\begin{enumerate}[(1)]
\item
If $\alpha\in(\frac{1}{2},1),\,u_0\in H^1(\mathbb{R}^2)$, $\theta_0\in L^1\cap L^\infty(\mathbb{R}^2)$, then 
\begin{align}\label{eq:u-Gam-es2}
  \Vert \nabla u\Vert _{L^\infty_T(L^2)} +  \Vert \Gamma\Vert _{L^\infty_T(L^2)}+ \Vert \Gamma\Vert _{L^1_T(B^1_{2,1})}
  + \left\Vert \Gamma\right\Vert _{\widetilde{L}_T^1 (B_{2, \infty}^{2\alpha})}
  + \Vert \nabla u\Vert _{\widetilde{L}^1_T(C^{2\alpha-1})} 
  \leq C e^{C T} .
\end{align}
\item 
If $\alpha=1,\,u_0\in H^1(\mathbb{R}^2)$, $\theta_0\in L^1\cap L^\infty(\mathbb{R}^2)$, then 
\begin{align}\label{es:nabu-Gam2-2b}
  \Vert \nabla u\Vert _{\widetilde{L}^1_T( B^1_{\infty,\infty})} \leq C e^{CT}. \quad   
\end{align}
\item
 If $\alpha\in(\frac{1}{2},1]$, $u_0\in W^{1,p}(\mathbb{R}^2)$ for some $2<p<\infty$,  
$\theta_0\in L^1\cap L^\infty(\mathbb{R}^2)$, 
then
\begin{align}\label{es:nabu-Gam2}
  \Vert \nabla u\Vert _{L^\infty_T(L^p)}+\Vert \Gamma\Vert _{L^\infty_T(L^p)}
  + \Vert \Gamma\Vert _{\widetilde{L}^1_T(B^{2\alpha}_{p,\infty})}
  +  \Vert \nabla u\Vert _{L^1_T(B^{2\alpha-1}_{\infty,\infty})}
  \leq C e^{CT} .
\end{align}
\item
Let $u_0\in W^{1,p}(\mathbb R^2)$ for some
$2\leq p<\infty$, and let
$\theta_0\in L^1\cap L^\infty(\mathbb R^2)$.
Assume that either
\[
  \tfrac12<\alpha<1,
  \qquad
  \max\left\{p,\tfrac1{1-\alpha}\right\}<Q<\infty,
\]
or
\[
  \begin{gathered}
    \alpha=1,\qquad p<Q<\infty,\qquad
    \int_{\mathbb R^2}\theta_0(x)\,\dd x=0,
    \qquad
    \int_{\mathbb R^2}|x|\,|\theta_0(x)|\,\dd x<\infty.
  \end{gathered}
\]
Then
\begin{equation}\label{eq:u-LQ-es}
  \Vert u\Vert _{L^\infty_T(L^Q)}
  \leq Ce^{\exp(CT)}.
\end{equation}
\end{enumerate}
In the above, the constant $C>0$ is independent of the coefficient $\varepsilon=\frac{1}{\mathrm{Pr}}$.
\end{proposition}

\begin{remark}
\label{rmk:nonuniform-Lp-alpha1}
The endpoint $\alpha=1$ is qualitatively different.
Global $L^r$ velocity bounds (which are uniform in $\varepsilon$)  cannot hold
in general for any $1<r<\infty$.
This failure occurs even for fixed, compactly supported
initial temperatures with nonzero total integral,
as the quantitative lower bounds below show.

Let $\alpha=1$ and consider the solutions of
Theorem~\ref{thm:main1} with uniformly bounded prescribed
initial norms, for a fixed velocity exponent
$2\leq p<\infty$.
Suppose that, for some $R_0>0$ independent of $\varepsilon$,
\begin{align}\label{eq:velocity-bdd-assum}
 \operatorname{supp}\theta_0^\varepsilon\subset B_{R_0},
 \qquad
 M_\varepsilon:=\int_{\mathbb R^2}\theta_0^\varepsilon\,\dd x
 \longrightarrow M\ne0.
\end{align}
For every $T>0$ and $1<r<\infty$,
there exist
$c_{T,r}>0$ and $\varepsilon_0>0$ such that
\begin{equation}\label{eq:u-Lp-nonuniform-alpha1}
 A_\varepsilon:=\Vert u^\varepsilon\Vert _{L^\infty([0,T],L^r)}
 \geq c_{T,r}\varepsilon^{-1/r},
 \qquad 0<\varepsilon<\varepsilon_0.
\end{equation}
For every finite $r\geq p$, this lower bound matches
the upper bound from \eqref{eq:u-Lp-es-refined}:
\begin{equation}\label{eq:u-Lr-sharp-alpha1}
  c_{T,r}\,\varepsilon^{-1/r}
  \leq
  \Vert u^\varepsilon\Vert _{L^\infty([0,T],L^r)}
  \leq
  C_{T,r}\,\varepsilon^{-1/r},
  \qquad 0<\varepsilon<\varepsilon_0.
\end{equation}
To prove the lower bound \eqref{eq:u-Lp-nonuniform-alpha1}, we may assume
$A_\varepsilon<\infty$.
The interpolation inequality
\[
  \Vert u^\varepsilon(t)\Vert _{L^\infty}
  \leq C_r
  \Vert u^\varepsilon(t)\Vert _{L^r}^{\frac r{r+2}}
  \Vert \nabla u^\varepsilon(t)\Vert _{L^\infty}^{\frac2{r+2}},
\]
together with the uniform bound on
$\int_0^T\Vert \nabla u^\varepsilon\Vert _{L^\infty}\,\dd t$,
gives
\begin{align}\label{eq:u-eps-L1Linf}
  \Vert u^\varepsilon\Vert _{L^1([0,T],L^\infty)}
  \leq C_TA_\varepsilon^{\frac r{r+2}}.
\end{align}
The flow equation \eqref{eq:flowmap} therefore ensures
that the temperature supports remain in
$B_{R_0+C_TA_\varepsilon^{r/(r+2)}}$
for $0\leq t\leq T$.

Choose $\Psi\in C_c^\infty(B_2;\mathbb R^2)$ with
$\div\Psi=0$ and $\Psi=e_2$ on $B_1$,
and set $\Psi_R(x):=\Psi(x/R)$.
Let $\eta\in C_c^\infty(0,T)$ be nonnegative with
$\int_0^T\eta\,\dd t=1$.
For $R>R_0+C_TA_\varepsilon^{r/(r+2)}$,
choose $\phi(t,x)=\eta(t)\Psi_R(x)$
in the velocity weak formulation of
Definition~\ref{defn:weak-Bouss}.
Since $\eta(0)=0$ and $\Psi_R=e_2$
on the temperature supports,
conservation of the total temperature 
$\int_{\mathbb{R}^2} \theta^\varepsilon(t,x)\dd x = \int_{\mathbb{R}^2} \theta^\varepsilon_0(x) \dd x =M_\varepsilon$ gives
\begin{align*}
  M_\varepsilon
  ={}
  -\varepsilon\int_0^T\!\int_{\mathbb R^2}
  \Bigl(
    \eta' u^\varepsilon\cdot\Psi_R
    +\eta\,u^\varepsilon\otimes u^\varepsilon:
      \nabla\Psi_R
  \Bigr)\,\dd x\,\dd t
  -\int_0^T\!\int_{\mathbb R^2}
    \eta\,u^\varepsilon\cdot\Delta\Psi_R
    \,\dd x\,\dd t.
\end{align*}
The inertial terms are controlled using \eqref{eq:u-eps-L1Linf} and 
$\Vert u^\varepsilon\Vert _{L^\infty([0,T],L^1(B_{2R}))}
\leq C_rR^{2-2/r}A_\varepsilon$. 
Moreover, scaling gives
\[
  \Vert \nabla\Psi_R\Vert _{L^\infty}\leq C/R,
  \quad
  \Vert \Delta\Psi_R\Vert _{L^{\frac{r}{r-1}}}
  =R^{-2/r}\Vert \Delta\Psi\Vert _{L^{\frac{r}{r-1}}}.
\]
Consequently, H\"older's inequality gives
\[
  |M_\varepsilon|
  \leq C_{T,r}\,\varepsilon
  \left(
    A_\varepsilon R^{2-2/r}
    +A_\varepsilon^{1+\frac r{r+2}}R^{1-2/r}
  \right)
  +C_rA_\varepsilon R^{-2/r}.
\]

Set $R=L(1+A_\varepsilon)^{r/2}$.
Since $r/(r+2)<r/2$, we may fix $L$ independently
of $\varepsilon$ so that $B_R$ contains all
temperature supports and $C_rA_\varepsilon R^{-2/r}
  \leq C_rL^{-2/r}\leq |M|/4$. 
The remaining terms satisfy
\begin{align*}
  A_\varepsilon R^{2-2/r}
  \leq C_{r,L}(1+A_\varepsilon)^r,\quad
  A_\varepsilon^{1+\frac r{r+2}}R^{1-2/r}
  \leq
  C_{r,L}(1+A_\varepsilon)^{\frac r2+\frac r{r+2}}
  \leq C_{r,L}(1+A_\varepsilon)^r.
\end{align*}
By \eqref{eq:velocity-bdd-assum}, for sufficiently small $\varepsilon$,
$|M_\varepsilon|\geq|M|/2$, and hence $\varepsilon(1+A_\varepsilon)^r\geq c_{T,r}$. 
Thus
$A_\varepsilon\geq c_{T,r}\varepsilon^{-1/r}$
for sufficiently small $\varepsilon$,
which in particular proves
\eqref{eq:u-Lp-nonuniform-alpha1}.
\end{remark}

\begin{proof}[Proof of Proposition~\ref{prop:ap-es-basic}]
For $\frac12<\alpha<1$, estimate \eqref{eq:u-Gam-es2}
follows from~\cite[Proposition~3.2]{LXY25}.

For $\frac12<\alpha<1$ and $2<p<\infty$,
we adapt the proof of~\cite[Proposition~3.3]{LXY25}
by estimating $\Gamma$ and $\nabla u$ directly in $L^p$,
without imposing an additional $H^1$ assumption.
This gives
\[
  \Vert \Gamma\Vert _{L^\infty_T(L^p)}
  +\Vert \nabla u\Vert _{L^\infty_T(L^p)}
  \leq Ce^{CT}, \qquad
  \Vert \Gamma\Vert _{\widetilde L^1_T(B^{2\alpha}_{p,\infty})}
  \leq Ce^{CT}
  \bigl(1+\Vert \Gamma\Vert _{L^1_T(B^0_{\infty,\infty})}\bigr),
\]
with constants independent of $\varepsilon$.
Since $2\alpha>1$, Bernstein's inequality
and interpolation yield
\[
  \Vert \Gamma\Vert _{L^1_T(B^0_{\infty,\infty})}
  \leq
  C\Vert \Gamma\Vert _{\widetilde L^1_T(B^{2\alpha}_{p,\infty})}^{2/p}
  \Vert \Gamma\Vert _{L^1_T(L^p)}^{1-2/p}.
\]
As $2/p<1$, Young's inequality and the preceding
$L^p$ bound therefore imply
$\Vert \Gamma\Vert _{\widetilde L^1_T(B^{2\alpha}_{p,\infty})}\leq Ce^{CT}$.
Similarly as in~\cite[Eq.~(75)]{LXY25},
with $\Vert \nabla u\Vert _{L^1_T(L^2)}$ replaced by
$\Vert \nabla u\Vert _{L^1_T(L^p)}$, we obtain
$\Vert \nabla u\Vert _{L^1_T(C^{2\alpha-1})}\leq Ce^{CT}$.
Together with the preceding estimates,
this proves \eqref{es:nabu-Gam2} in this case.

It remains to treat the case $\alpha=1$,
for which the whole-space estimates
in~\cite[Section~3]{LXY25} depend on $\varepsilon$.
The commutator estimate \eqref{eq:comm-es1}
in Lemma~\ref{lem:Rbeta-cm}
controls $[\mathcal R_{-1},u\cdot\nabla]\theta$
using only norms of $\nabla u$ and $\theta$.
This allows us to prove
\eqref{es:nabu-Gam2-2b} and \eqref{es:nabu-Gam2}
with constants independent of $\varepsilon$.
  
\underline{\textit{Proof of \eqref{es:nabu-Gam2-2b} with $\alpha=1$}.}  

In view of the relation 
\begin{align}\label{eq:u-Gam-rela}
  u=\nabla^\perp \Lambda^{-2}\omega = \nabla^\perp \Lambda^{-2} \Gamma + \nabla^\perp \partial_1 \Lambda^{-2-2\alpha} \theta,
\end{align}
the equation of $\Gamma$ recasts
\begin{equation}\label{eq:Gamma}
  \partial_t \Gamma +  u \cdot \nabla \Gamma  - \frac{1}{\varepsilon}\Delta \Gamma
   = \left[\mathcal{R}_{-1}, u\cdot \nabla\right] \theta  = \big[\mathcal{R}_{-1}, (\nabla^{\bot}\Lambda^{-2}\Gamma+\nabla^{\bot}\partial_1\Lambda^{-4}\theta)\cdot \nabla\big] \theta,
\end{equation}
with initial data $\Gamma|_{t=0}=\Gamma_0=\omega_0 - \partial_1\Lambda^{-2}\theta_0$.
We decompose $\Gamma$ into $\Gamma\triangleq \Gamma^{(1)}+\Gamma^{(2)}$, where
\begin{equation}\label{eq:Gamma1}
\left\{\begin{aligned}
  \partial_t \Gamma^{(1)} +  u \cdot \nabla \Gamma^{(1)}  - \frac{1}{\varepsilon}\Delta \Gamma^{(1)}
  &= \big[\mathcal{R}_{-1}, \nabla^{\bot}\Lambda^{-2}\Gamma^{(1)}\cdot \nabla\big] \theta,\\
  \Gamma^{(1)}|_{t=0}&=\omega_0 ,
\end{aligned}\right.
\end{equation}
and
\begin{equation}\label{eq:Gamma2}
  \left\{\begin{aligned}
   \partial_t \Gamma^{(2)} +  u \cdot \nabla \Gamma^{(2)}  - \frac{1}{\varepsilon}\Delta \Gamma^{(2)}
  &= \big[\mathcal{R}_{-1}, \nabla^{\bot}\Lambda^{-2}(\Gamma^{(2)}+\partial_1\Lambda^{-2}\theta)\cdot \nabla\big] \theta,\\
  \Gamma^{(2)}|_{t=0}&=-\partial_1\Lambda^{-2}\theta_0.
  \end{aligned}\right.
\end{equation}

Taking advantage of \eqref{eq:comm-es1} in Lemma~\ref{lem:Rbeta-cm}, we get
\begin{align*}
  \frac{1}{2} \frac{\dd }{\dd t} \Vert \Gamma^{(1)}(t)\Vert _{L^2}^2 +& \frac{1}{\varepsilon} \Vert \nabla \Gamma^{(1)}(t)\Vert _{L^2}^2
   \leq \Vert [\mathcal{R}_{-1}, \nabla^{\bot}\Lambda^{-2}\Gamma^{(1)}\cdot \nabla]\theta\Vert _{L^2}(t) \Vert \Gamma^{(1)}\Vert _{L^2} \nonumber \\
 & \leq C \Vert \nabla\nabla^{\perp}\Lambda^{-2}\Gamma^{(1)}(t)\Vert _{L^2} \Vert \theta(t)\Vert _{L^2}  \Vert \Gamma^{(1)}(t)\Vert _{L^2}  \leq C \Vert \theta_0\Vert _{L^2}  \Vert \Gamma^{(1)}(t)\Vert ^2_{L^2}. \nonumber 
\end{align*}
Gronwall's inequality leads to
\begin{equation}\label{es:GamL2}
\begin{split}
  \Vert \Gamma^{(1)}\Vert _{L^\infty_T(L^2)}^2 + \frac{1}{\varepsilon} \Vert \nabla \Gamma^{(1)}\Vert _{L^2_T(L^2)}^2 
  \leq C  \Vert \omega_0\Vert _{L^2}^2  e^{C \Vert \theta_0\Vert _{L^2}T} \leq \, C e^{CT}.
\end{split}
\end{equation}

Multiplying both sides of the $\Gamma^{(2)}$-equation \eqref{eq:Gamma2} with $|\Gamma^{(2)}|^{p-2}\Gamma^{(2)}(\cdot,t)$ 
($2<p<\infty$) and integrating over the spatial variable, we use the integration by parts and H\"older's inequality to get
\begin{equation}
\begin{split}
  &\frac{1}{p} \frac{\dd}{\dd t} \Vert \Gamma^{(2)}(t)\Vert _{L^p}^p + \frac{p-1}{\varepsilon} \int_{\mathbb{R}^2} |\Gamma^{(2)}(x,t)|^{p-2}
  |\nabla\Gamma^{(2)}(x,t)|^2 \dd x\\
   &\leq \big\Vert  \big[\mathcal{R}_{-1}, \nabla^{\bot}\Lambda^{-2}(\Gamma^{(2)}+\partial_1\Lambda^{-2}\theta)\cdot \nabla\big] \theta\big\Vert _{L^p} \Vert \Gamma^{(2)}(t)\Vert _{L^p}^{p-1},
\end{split}
\end{equation}
which together with the fact 
$\frac{1}{p} \frac{\dd}{\dd t} \Vert \Gamma^{(2)}\Vert _{L^p}^p=\Vert \Gamma^{(2)}\Vert _{L^p}^{p-1} 
\frac{\dd}{\dd t}\Vert \Gamma^{(2)}\Vert _{L^p}$ 
implies that
\begin{align*}
  \frac{\dd}{\dd t} \Vert \Gamma^{(2)}(t)\Vert _{L^p} \leq \big\Vert \big[\mathcal{R}_{-1}, \nabla^{\bot}\Lambda^{-2}
  (\Gamma^{(2)}+\partial_1\Lambda^{-2}\theta)\cdot \nabla \big] \theta \big\Vert _{L^p} .
\end{align*}
Thanks to \eqref{eq:comm-es1}, we have 
\begin{align}
  \frac{\dd}{\dd t}\Vert \Gamma^{(2)}(t)\Vert _{L^p} 
  & \leq C \Vert \nabla \nabla^{\bot}\Lambda^{-2}(\Gamma^{(2)} 
  + \partial_1\Lambda^{-2}\theta)(t)\Vert _{L^p}\Vert \theta(t)\Vert _{L^2} \nonumber\\
  & \leq C \big(\Vert \Gamma^{(2)}(t)\Vert _{L^p} + \Vert \RR_{-1}\theta(t)\Vert _{L^p} \big)\Vert \theta(t)\Vert _{L^2}\nonumber\\
  & \leq C \big(\Vert \Gamma^{(2)}(t)\Vert _{L^p} + \Vert \theta_0\Vert _{L^1\cap L^\infty}\big) \Vert \theta_0\Vert _{L^1\cap L^\infty},
\end{align}
where in the last line we have used the fact that 
\begin{align}\label{eq:fact-R-1the}
  \Vert \mathcal{R}_{-1}\theta(t)\Vert _{L^p(\mathbb{R}^2)} \leq C \Vert \theta(t)\Vert _{L^{\frac{2p}{p+2}}(\mathbb{R}^2)} 
  \leq C \Vert \theta_0\Vert _{L^{\frac{2p}{p+2}}(\mathbb{R}^2)}.
\end{align}
Taking advantage of Gronwall's inequality gives $\Vert \Gamma^{(2)}\Vert _{L^\infty_T(L^p)} \leq C e^{CT}$ for every $2<p<\infty$. 

Next, we intend to show the estimate of $\Vert \omega\Vert _{L^1_T(B^0_{\infty,1})}$. 
Using the identity $\omega= \Gamma^{(1)} + \Gamma^{(2)} +  \mathcal{R}_{-1}\theta$ 
together with the above estimates, we have that for some $N\in \mathbb{N}$ (to be chosen later),
\begin{align*}
  \Vert \omega\Vert _{L^1_T(B^0_{\infty,1})} & \leq \Vert \mathcal{R}_{-1}\theta\Vert _{L^1_T(B^0_{\infty,1})} 
  + \sum_{-1\leq q\leq N} \Vert \Delta_q (\Gamma^{(1)} + \Gamma^{(2)})\Vert _{L^1_T(L^\infty)} 
  + \sum_{q\geq N} \Vert \Delta_q \Gamma\Vert _{L^1_T(L^\infty)} \\
  & \leq C T\Vert \theta_0\Vert _{L^1\cap L^\infty} + 
  C \sum_{-1\leq q\leq N}\Big(2^q \Vert \Gamma^{(1)}\Vert _{L^1_T(L^2)} 
  + 2^{q \frac{2}{p}} \Vert \Gamma^{(2)}\Vert _{L^1_T(L^p)} \Big) 
  + \sum_{q\geq N} \Vert \Delta_q \Gamma\Vert _{L^1_T(L^\infty)} \\
  & \leq CT + C e^{CT} 2^N +  \sum_{q\geq N} 2^q\Vert \Delta_q \Gamma\Vert _{L^1_T(L^2)},
\end{align*}
where in the second line we have used the fact that 
\begin{equation}\label{eq:R-1-the-bdd}
\begin{split}
  \Vert \mathcal{R}_{-1}\theta(t)\Vert _{B^0_{\infty,1}} & = \sum_{q\geq -1} \Vert \Delta_q \mathcal{R}_{-1}\theta(t)\Vert _{L^\infty} \\
  & \leq C \Vert \Delta_{-1} \mathcal{R}_{-1}\theta(t)\Vert _{L^p} + C \sum_{q\geq 0} 2^{-q} 
  \Vert \Delta_q \theta(t)\Vert _{L^\infty}
  \leq C\Vert \theta_0\Vert _{L^1\cap L^\infty}.
\end{split}
\end{equation}
For every $q\in \mathbb{N}$, applying the frequency localization operator $\Delta_q$ to the $\Gamma$-equation \eqref{eq:Gamma} gives
\begin{equation}\label{eq:Delta-q-Gamm-2}
\begin{aligned}
  \partial_t \Delta_q\Gamma + u \cdot \nabla \Delta_q\Gamma - \frac{1}{\varepsilon}\Delta \Delta_q\Gamma
  = - \left[\Delta_q, u \cdot \nabla\right] \Gamma +
  \Delta_q\big(\left[\mathcal{R}_{-1}, u \cdot \nabla\right] \theta\big)  := f_q .
\end{aligned}
\end{equation}
Taking the scalar product of the above equation with $\Delta_q \Gamma$, we use integration by parts and Bernstein's inequality to obtain
\begin{align*}
  \frac{1}{2}\frac{\dd}{\dd t} \Vert \Delta_q\Gamma(t)\Vert _{L^2}^2 
  + \frac{c_0}{\varepsilon} 2^{2 q} \Vert  \Delta_q \Gamma(t)\Vert _{L^2}^2
  \leq \Vert \Delta_q \Gamma(t)\Vert _{L^2} \Vert f_q(t)\Vert _{L^2},
\end{align*}
for some fixed $c_0>0$ independent of $q\in\mathbb{N}$, this implies that
\begin{align*}
  \Vert \Delta_q \Gamma(t)\Vert _{L^2} \leq e^{-\frac{c_0}{\varepsilon}2^{2q}t} \Vert \Delta_q \Gamma_0\Vert _{L^2}
  + \int_0^t e^{-\frac{c_0}{\varepsilon}2^{2q}(t-\tau)} \Vert f_q(\tau)\Vert _{L^2} \dd \tau,  
\end{align*}
and
\begin{align}\label{eq:Gamma-est1}
  \Vert \Delta_q\Gamma\Vert _{L^1_t(L^2)} \leq \frac{\varepsilon}{c_0} 2^{-2 q} \Vert \Delta_q \Gamma_0\Vert _{L^2}
  + \frac{\varepsilon}{c_0} 2^{-2 q} \Vert f_q\Vert _{L^1_t(L^2)},
  \quad \forall q\in\mathbb{N}.
\end{align}
According to Lemmas \ref{lem:HKR}, \eqref{eq:comm-es1} and \eqref{eq:R-1-the-bdd},
we infer that for all $q\in \mathbb{N}$,
\begin{align}
  \Vert [\Delta_q, u\cdot\nabla]\Gamma(t)\Vert _{L^2} & \leq \Vert \nabla u(t)\Vert _{L^2} \Vert \Gamma(t)\Vert _{B^0_{\infty,\infty}}\nonumber\\
  & \leq \Vert \nabla u(t)\Vert _{L^2} \big(\Vert \omega(t)\Vert _{B^0_{\infty,\infty}} 
  + \Vert \mathcal{R}_{-1}\theta(t)\Vert _{B^0_{\infty,\infty}}\big)\label{es:fq1}\\
  & \leq C \Vert \nabla u(t)\Vert _{L^2} \big(\Vert \omega(t)\Vert _{B^0_{\infty,\infty}} + \Vert \theta_0\Vert _{L^1\cap L^\infty}\big),\nonumber
\end{align}
and
\begin{align}\label{es:fq2}
  \Vert [\mathcal{R}_{-1},u\cdot\nabla]\theta(t)\Vert _{L^2} \leq C \Vert \nabla u(t)\Vert _{L^2} \Vert \theta(t)\Vert _{L^2}
   \leq C \Vert \nabla u(t)\Vert _{L^2} \Vert \theta_0\Vert _{L^2}.
\end{align}
Note that for $\Vert \nabla u(t)\Vert _{L^2}$, by performing the classical $L^2$-estimate of vorticity $\omega$, 
we have the following $\varepsilon$-dependent bound (see p.1405 of \cite{LXY25})
\begin{align*}
  \Vert \nabla u \Vert _{L^\infty_T(L^2)} + \frac{1}{\sqrt{\varepsilon}} \Vert \nabla^2 u\Vert _{L^2_T(L^2)}
  \leq C \big(\Vert \omega_0\Vert _{L^2} + \Vert \theta_0\Vert _{L^2} \big)\Big(1+\sqrt{\tfrac{T}{\varepsilon}}\Big).
\end{align*}
Hence, gathering the above estimates leads to
\begin{align*}
  \sum_{q\geq N} 2^q \Vert \Delta_q \Gamma\Vert _{L^1_T(L^2)} 
  & \leq C \varepsilon \sum_{q\geq N} 2^{-q} \Big( \Vert \Delta_q \Gamma_0\Vert _{L^2} 
  + \Vert [\Delta_q, u\cdot\nabla]\Gamma\Vert _{L^1_T(L^2)}
  + \Vert [\mathcal{R}_{-1},u\cdot\nabla]\theta\Vert _{L^1_T(L^2)}\Big)\\
  & \leq 
  C \varepsilon  2^{-N} \Big( \Vert (\omega_0,\theta_0)\Vert _{L^2}
  + \Vert \nabla u\Vert _{L^\infty_T(L^2)} \big(\Vert \omega\Vert _{L^1_T(B^0_{\infty,\infty})} + T \Vert \theta_0\Vert _{L^1\cap L^\infty} \big) \Big) \\
  & \leq C \sqrt{\varepsilon} 2^{-N} \big(1+\sqrt{T} \big) \big(\Vert \omega\Vert _{L^1_T(B^0_{\infty,\infty})} 
  + \Vert \theta_0\Vert _{L^1\cap L^\infty}T\big)  + C \varepsilon 2^{-N}  \Vert (\omega_0,\theta_0)\Vert _{L^2},
\end{align*}
and
\begin{align*}
  \Vert \omega\Vert _{L^1_T(B^0_{\infty,1})} \leq C T + C e^{CT} 2^N 
  + C\sqrt{\varepsilon} 2^{-N} \big(1+\sqrt{T} \big) \big(\Vert \omega\Vert _{L^1_T(B^0_{\infty,\infty})} 
  + T\big)  + C \varepsilon 2^{-N} .
\end{align*}
Choosing $N\in \mathbb{N}$ such that $C (1+\sqrt{T})2^{-N}\approx \frac{1}{2}$ (recalling $\varepsilon\in (0,1]$) 
allows us to get $\Vert \omega\Vert _{L^1_T(B^0_{\infty,1})} \leq C e^{CT}$, 
where $C>0$ depends only on the norms of $(u_0,\theta_0)$ but is independent of $\varepsilon$.
As a consequence of the above estimate and \eqref{eq:R-1-the-bdd},
we infer from the identity $\Gamma =\omega - \mathcal{R}_{-1}\theta$ that
\begin{equation}\label{eq:Gam-B0inf-1}
\begin{split}
  \Vert \Gamma\Vert _{L^1_T(B^0_{\infty,1})} \leq \Vert \omega\Vert _{L^1_T(B^0_{\infty,1})} + \Vert \mathcal{R}_{-1}\theta\Vert _{L^1_T(B^0_{\infty,1})} 
  \leq C e^{CT} + C T \Vert \theta_0\Vert _{L^1\cap L^\infty} \leq C e^{CT}.
\end{split}
\end{equation}

Similarly as above, taking advantages of \eqref{eq:Gamma-est1}, \eqref{es:fq1} and \eqref{es:fq2}, we find that for some $p\in (2,\infty)$,
\begin{align*}
  \Vert \nabla u\Vert _{\widetilde{L}^1_T( B^1_{\infty,\infty})}
  & \leq  C \Vert \Delta_{-1}\omega\Vert _{L^1_T(L^p)}
  + C \sup_{q\in\mathbb{N}} 2^{q}\Vert \Delta_q \omega\Vert _{L^1_T(L^\infty)} \\
  &\leq  C \Vert \Delta_{-1} (\Gamma^{(1)} + \Gamma^{(2)} + \mathcal{R}_{-1}\theta)\Vert _{L^1_T(L^p)}
  + C \sup_{q\in\mathbb{N}} 2^{q} \Big(\Vert \Delta_q \mathcal{R}_{-1}\theta\Vert _{L^1_T(L^\infty)}
  + \Vert \Delta_q \Gamma\Vert _{L^1_T(L^\infty)}\Big) \\
  &\leq C \Vert \Gamma^{(1)}\Vert _{L^1_T(L^2)} + C \Vert \Gamma^{(2)}\Vert _{L^1_T(L^p)} 
  + C \Vert \mathcal{R}_{-1}\theta\Vert _{L^1_T(L^p)}
  + C \sup_{q\in \mathbb{N}} \Vert \Delta_q \theta\Vert _{L^1_T(L^\infty)} \\
  & \quad + C \sup_{q\in\mathbb{N}} 2^{2q} \Vert  \Delta_q \Gamma\Vert _{L^1_T(L^2)} \\
  &\leq C e^{CT} +  \ C \varepsilon \sup_{q\in\mathbb{N}} \Big( \Vert \Delta_q \Gamma_0\Vert _{L^2}
  + \Vert [\Delta_q,u\cdot\nabla]\Gamma \Vert _{L^1_T(L^2)} + \Vert [\mathcal{R}_{-1},u\cdot\nabla]\theta\Vert _{L^1_T(L^2)}\Big) \\
  &\leq C e^{CT} + C \sqrt{\varepsilon} \big(1+\sqrt{T} \big)
  \Vert \omega\Vert _{L^1_T(B^0_{\infty,\infty})} \leq C e^{CT}.
\end{align*}
Therefore, we finish the proof of the desired estimate \eqref{es:nabu-Gam2-2b}.

\underline{\textit{Proof of \eqref{es:nabu-Gam2} when $\alpha=1$}.}  

Multiplying both sides of the $\Gamma$-equation \eqref{eq:Gamma} (i.e. $\partial_t\Gamma+u\cdot\nabla\Gamma-\frac1\varepsilon\Delta\Gamma=[\mathcal R_{-1}, u\cdot\nabla]\theta$) with $|\Gamma|^{p-2}\Gamma(x,t)$
and integrating over the spatial variable, we use integration by parts to get
\begin{align}\label{eq:Gam-Lp-es1a}
  \frac{1}{p} \frac{\dd}{\dd t} \Vert \Gamma(t)\Vert _{L^p}^p + \frac{p-1}{\varepsilon} \int_{\mathbb{R}^2} |\Gamma(x,t)|^{p-2}
  |\nabla\Gamma(x,t)|^2 \dd x \leq \Vert [\mathcal{R}_{-1}, u\cdot\nabla]\theta \Vert _{L^p} \Vert \Gamma(t)\Vert _{L^p}^{p-1},
\end{align}
Thanks to \eqref{eq:comm-es1} and \eqref{eq:fact-R-1the}, we have 
\begin{align*}
  \frac{\dd}{\dd t}\Vert \Gamma(t)\Vert _{L^p} &\leq \Vert [\mathcal{R}_{-1}, u\cdot\nabla]\theta(t)\Vert _{L^p}\nonumber\leq C \Vert \omega(t)\Vert _{L^p}\Vert \theta(t)\Vert _{L^2}\\
  &\leq C \big(\Vert \Gamma(t)\Vert _{L^p} + \Vert \RR_{-1}\theta(t)\Vert _{L^p}\big) \Vert \theta(t)\Vert _{L^2}\leq C \big(\Vert \Gamma(t)\Vert _{L^p} + \Vert \theta_0\Vert _{L^1\cap L^\infty}\big) \Vert \theta_0\Vert _{L^1\cap L^\infty}.
\end{align*}
Gronwall's inequality ensures that $\Vert \Gamma\Vert _{L^\infty_T(L^p)}\leq C e^{CT}$. 
Consequently,
\begin{align}\label{eq:nab-u-Lp-es}
  \Vert \nabla u\Vert _{L^\infty_T(L^p)} \leq C \Vert \omega\Vert _{L^\infty_T(L^p)}
  \leq C \left(\Vert \Gamma\Vert _{L^\infty_T(L^p)} + \Vert \mathcal{R}_{-1} \theta\Vert _{L^\infty_T(L^p)} \right)
  \leq C e^{CT},
\end{align}
where $C>0$ depends only on the norms of $u_0$ and $\theta_0$, but is independent of $\varepsilon$.

Next, we prove the estimate of $\Vert \Gamma\Vert _{\widetilde{L}^1_T(B^{2}_{p,\infty})}$.
Similarly as obtaining \eqref{eq:Gamma-est1} and using the fact that (see \cite{CMZ07})
\begin{align*}
  -\int_{\mathbb{R}^2} \Delta \Delta_q \Gamma(x)\, |\Delta_q \Gamma|^{p-2}
  \Delta_q\Gamma(x) \dd x \geq c\, 2^{2q} \Vert \Delta_q \Gamma\Vert _{L^p}^p,
  \quad q\in\mathbb{N},
\end{align*}
with some $c>0$ independent of $q$, we have
\begin{align}\label{eq:Gamma-est2}
  \Vert \Delta_q\Gamma\Vert _{L^1_t(L^p)} \leq \frac{\varepsilon}{c} 2^{-2 q} \Vert \Delta_q \Gamma_0\Vert _{L^p}
  + \frac{\varepsilon}{c} 2^{-2 q} \Vert f_q\Vert _{L^1_t(L^p)}.
\end{align}
Taking advantage of \eqref{eq:Gamma-est2},  Lemmas \ref{lem:Rbeta-cm} and \ref{lem:HKR} 
and the interpolation inequality
\begin{align*}
  \Vert f\Vert _{L^1_T(B^0_{\infty,\infty})} 
  \leq C\Vert f\Vert _{L^1_T(B^1_{p,\infty})}^{\frac{2}{p}} 
  \Vert f\Vert _{L^1_T (B^0_{p,\infty})}^{1-\frac{2}{p}}\leq C\Vert f\Vert _{\widetilde{L}^1_T(B^2_{p,\infty})}^{\frac{2}{p}} 
  \Vert f\Vert _{L^1_T (L^p)}^{1-\frac{2}{p}},
\end{align*}
with any $\epsilon'\in (0,1)$, we find
\begin{align}
  & \Vert \Gamma\Vert _{\widetilde{L}^1_T(B^{2}_{p,\infty})} \leq \Vert \Delta_{-1} \Gamma\Vert _{L^1_T(L^p)}
  + \sup_{q \in \mathbb{N}} 2^{2q} \Vert \Delta_q \Gamma\Vert _{L^1_T(L^p)} \nonumber \\
  & \leq C \Vert \Delta_{-1}\Gamma\Vert _{L^1_T(L^p)}
  + C \sup_{q\in \mathbb{N}} \Big( \Vert \Gamma_0\Vert _{L^p}
  + \varepsilon \Vert [\Delta_q, u\cdot\nabla]\Gamma\Vert _{L^1_T(L^p)}
  + \varepsilon \Vert \Delta_q([\mathcal{R}_{-1},u\cdot\nabla]\theta)\Vert _{L^1_T(L^p)} \Big) \nonumber \\
  & \leq C \Vert \Gamma\Vert _{L^1_T(L^p)}  + C \Vert \Gamma_0\Vert _{L^p} +  C \varepsilon\Vert \nabla u\Vert _{L^\infty_T(L^p)}
  \Big(\Vert \Gamma\Vert _{L^1_T(B^0_{\infty,\infty})} +  \Vert \theta\Vert _{L^1_T(L^2)} \Big)\nonumber\\
  & \leq C \Vert \Gamma\Vert _{L^1_T(L^p)}  + C \Vert \Gamma_0\Vert _{L^p} +  C \varepsilon\Vert \nabla u\Vert _{L^\infty_T(L^p)}
  \Vert \theta\Vert _{L^1_T(L^2)}+  C \varepsilon\Vert \nabla u\Vert _{L^\infty_T(L^p)}\Vert \Gamma\Vert _{\widetilde{L}^1_T(B^2_{p,\infty})}^{\frac{2}{p}} 
  \Vert \Gamma\Vert _{L^1_T (L^p)}^{1-\frac{2}{p}}
  \nonumber\\
   & \leq C \Vert \Gamma\Vert _{L^1_T(L^p)}  + C 
  +  C \varepsilon\Vert \nabla u\Vert _{L^\infty_T(L^p)}
  T + \tfrac{C }{(\epsilon')^{2/(p-2)}}(\varepsilon\Vert \nabla u\Vert _{L^\infty_T(L^p)})^{\frac{p}{p-2}}
 \Vert \Gamma\Vert _{L^1_T(L^p)} + \epsilon' \Vert \Gamma\Vert _{\widetilde{L}^1_T(B^2_{p,\infty})}. \nonumber 
\end{align}
Collecting the above estimates and letting $\epsilon'>0$ be small enough, together with the continuous embedding 
$\widetilde{L}_T^1 (B^2_{p,\infty}) \hookrightarrow L_T^1(B^{\gamma^\prime}_{\infty,1})$ 
for $\gamma^\prime\in(0,2-\frac{2}{p})$, we deduce that for every 
$\gamma'\in (0,2-\frac{2}{p})$,
\begin{align}\label{eq:Gam-B2p-inf1}
  \Vert \Gamma\Vert _{\widetilde{L}^1_T(B^{2}_{p,\infty})} 
  + \Vert \Gamma\Vert _{L^1_T(B^{\gamma'}_{\infty,1})} \leq C e^{CT}.
\end{align}

Next, we intend to prove that $\nabla u\in L^1_T(B^{1}_{\infty,\infty})$. 
By using the identity
\begin{align}\label{eq:nab-u-rela}
  \nabla u = \nabla \nabla^\perp \Lambda^{-2} \omega 
  = \nabla \nabla^\perp \Lambda^{-2} \Gamma + \nabla \nabla^\perp \partial_1 \Lambda^{-4} \theta,
\end{align}
and \eqref{eq:nab-u-Lp-es}, \eqref{eq:Gam-B2p-inf1}, we have
\begin{equation}\label{eq:nab-u-B1inf}
\begin{split}
  \Vert \nabla u\Vert _{L^1_T(B^1_{\infty,\infty})} & \leq C \Vert \Delta_{-1}\nabla u\Vert _{L^1_T(L^\infty)} +
  C \Big\Vert  \sup_{q\in \mathbb{N}}2^{q} \Vert \Delta_q \nabla u\Vert _{L^\infty} \Big\Vert _{L^1_T} \\
  & \leq C\Vert \nabla u\Vert _{L^1_T(L^p)} + C \Vert \Gamma\Vert _{L^1_T(B^1_{\infty,\infty})} + C \Vert \theta\Vert _{L^1_T(B^0_{\infty,\infty})}\leq C e^{CT}.
\end{split}
\end{equation}
Hence, gathering the above inequalities completes the proof of \eqref{es:nabu-Gam2}
in the case $\alpha=1$.

\underline{Proof of \eqref{eq:u-LQ-es} when $\frac{1}{2}<\alpha < 1$.}

We first claim that, for $\frac12<\alpha\leq1$
and $2\leq p<\infty$,
\begin{equation}\label{eq:u-Lp-es-refined}
  \Vert u\Vert _{L^\infty_T(L^p)}
  \leq
  Ce^{\exp(CT)}
  \bigg(
    1+\int_0^{T/\varepsilon}
    \min\left\{1,r^{-\frac{1-1/p}{\alpha}}\right\}\,\dd r
  \bigg),
\end{equation}
where $C>0$ depends only on $\alpha$, $p$,
and the prescribed initial norms,
but is independent of $\varepsilon$.
Indeed, applying the Leray projector
$\mathbb P:=\mathrm{Id}-\nabla\Delta^{-1}\div$
and writing the velocity equation in mild form gives
\begin{equation}\label{eq:u-mild-Lp}
\begin{aligned}
  u(t)
  ={} e^{-\frac{t}{\varepsilon}\Lambda^{2\alpha}}u_0
  +\frac1\varepsilon\int_0^t
    e^{-\frac{t-\tau}{\varepsilon}\Lambda^{2\alpha}}
    \mathbb P\bigl(\theta(\tau)e_2\bigr)\,\dd\tau
    -\int_0^t
    e^{-\frac{t-\tau}{\varepsilon}\Lambda^{2\alpha}}
    \mathbb P(u\cdot\nabla u)(\tau)\,\dd\tau.
\end{aligned}
\end{equation}
The fractional heat semigroup estimates
(see~\cite{MYZ08}), together with the commutation
of $\mathbb P$ with the semigroup
and its $L^p$ boundedness, give, for $r>0$,
\begin{align*}
  \bigl\Vert e^{-r\Lambda^{2\alpha}}
    \mathbb P(\theta(\tau)e_2)\bigr\Vert _{L^p}
  &\leq C_{\alpha,p}
  \min\left\{
    \Vert \theta(\tau)\Vert _{L^p},
    r^{-\frac{1-1/p}{\alpha}}\Vert \theta(\tau)\Vert _{L^1}
  \right\}
  \\
  &\leq C_{\alpha,p}\Vert \theta_0\Vert _{L^1\cap L^\infty}
    \min\left\{1,r^{-\frac{1-1/p}{\alpha}}\right\},
\end{align*}
where we used conservation of the temperature norms.
Applying this bound with $r=(t-\tau)/\varepsilon$
and using
$\Vert u\cdot\nabla u\Vert _{L^p}
\leq\Vert u\Vert _{L^p}\Vert \nabla u\Vert _{L^\infty}$
in \eqref{eq:u-mild-Lp}, we obtain
\begin{align*}
  \Vert u(t)\Vert _{L^p}
  \leq {}
  \Vert u_0\Vert _{L^p}
  +C_{\alpha,p}\Vert \theta_0\Vert _{L^1\cap L^\infty}
  \int_0^{t/\varepsilon}
    \min\left\{1,r^{-\frac{1-1/p}{\alpha}}\right\}\,\dd r
    +C_p\int_0^t
    \Vert u(\tau)\Vert _{L^p}
    \Vert \nabla u(\tau)\Vert _{L^\infty}\,\dd\tau.
\end{align*}
Gronwall's inequality and the preceding uniform bound on
$\int_0^T\Vert \nabla u(t)\Vert _{L^\infty}\,\dd t$
prove \eqref{eq:u-Lp-es-refined}.

The time factor satisfies
\begin{equation}\label{eq:u-Lp-time-factor}
  \int_0^{T/\varepsilon}
    \min\left\{1,r^{-\frac{1-1/p}{\alpha}}\right\}\,\dd r
  \leq C_{\alpha,p}
  \begin{cases}
    (1+T/\varepsilon)^{1-\frac{1-1/p}{\alpha}},
      &\quad \textrm{for}\;\, 1-1/p<\alpha,\\[1mm]
    \log(e+T/\varepsilon),
      & \quad \textrm{for}\;\, 1-1/p=\alpha,\\[1mm]
    1,
      & \quad \textrm{for}\;\, 1-1/p>\alpha.
  \end{cases}
\end{equation}
Fix
$Q>\max\{p,\frac{1}{1-\alpha}\}$, since
$W^{1,p}(\mathbb{R}^2)\hookrightarrow L^Q(\mathbb{R}^2)$,
we have $u_0\in L^Q(\mathbb{R}^2)$. Applying
\eqref{eq:u-Lp-es-refined} and \eqref{eq:u-Lp-time-factor}
with $Q$ in place of $p$, and observing that
$1-\frac{1}{Q}>\alpha$, we obtain the $\varepsilon$-independent velocity bound in
$L^\infty_T(L^Q)$ as in \eqref{eq:u-LQ-es}.

\medskip
\underline{Proof of \eqref{eq:u-LQ-es} when $\alpha=1$.}

Fix a finite exponent $Q>p$.
Sobolev embedding gives $u_0\in L^Q(\mathbb R^2)$.
Recall that, in this case,
$\int_{\mathbb R^2}\theta_0(x)\,\dd x=0$,
and $\int_{\mathbb R^2}|x|\,|\theta_0(x)|\,\dd x<\infty$.
Since the temperature is transported by a divergence-free
velocity, its integral and Lebesgue norms are conserved.
Thus
\[
  \int_{\mathbb R^2}\theta(t,x)\,\dd x=0,
  \qquad
  \Vert \theta(t)\Vert _{L^s}=\Vert \theta_0\Vert _{L^s},
  \quad 1\leq s\leq\infty.
\]

The two-dimensional heat kernel
$K_r(x)=\frac1{4\pi r}e^{-\frac{|x|^2}{4r}}$, $r>0$,
satisfies
\[
  \Vert \nabla K_r\Vert _{L^Q}
  =r^{-\frac32+\frac1Q}\Vert \nabla K_1\Vert _{L^Q}.
\]
Using the zero-mean condition, we write
\[
  e^{r\Delta}\theta(t,x)
  =
  \int_{\mathbb R^2}
  \bigl(K_r(x-y)-K_r(x)\bigr)\theta(t,y)\,\dd y.
\]
The mean value theorem and Minkowski's inequality give
\[
  \Vert e^{r\Delta}\theta(t)\Vert _{L^Q}
  \leq
  C_Qr^{-\frac32+\frac1Q}
  \int_{\mathbb R^2}|x|\,|\theta(t,x)|\,\dd x.
\]
Combining this with the $L^Q$ contractivity of the heat
semigroup and the $L^Q$ boundedness of the Leray projector $\mathbb P$, we obtain
\begin{equation}\label{eq:alpha-one-semigroup}
\begin{aligned}
  \bigl\Vert 
    e^{r\Delta}\mathbb P(\theta(t)e_2)
  \bigr\Vert _{L^Q}
  \leq {}
  C_Q\min\left\{1,r^{-\frac32+\frac1Q}\right\}
  \left(
    \Vert \theta_0\Vert _{L^Q}
    +\int_{\mathbb R^2}|x|\,|\theta(t,x)|\,\dd x
  \right).
\end{aligned}
\end{equation}
Since $Q>p\geq2$, we have
\begin{equation}\label{eq:alpha-one-time-integrability}
  \int_0^\infty
  \min\left\{1,r^{-\frac32+\frac1Q}\right\}\,\dd r
  <\infty.
\end{equation}

We next control the first moment.
Let $X_t$ denote the measure-preserving flow of $u$
defined by \eqref{eq:flowmap}. Since
\[
  X_t(y)=y+\int_0^t u(s,X_s(y))\,\dd s,
  \qquad
  \theta(t,X_t(y))=\theta_0(y),
\]
H\"older's inequality gives
\begin{equation}\label{eq:theta-first-moment}
\begin{aligned}
  \int_{\mathbb R^2}|x|\,|\theta(t,x)|\,\dd x
  &=
  \int_{\mathbb R^2}|X_t(y)|\,|\theta_0(y)|\,\dd y
  \\
  &\leq
  \int_{\mathbb R^2}|y|\,|\theta_0(y)|\,\dd y
  +\int_0^t\int_{\mathbb R^2}
    |u(s,x)|\,|\theta(s,x)|\,\dd x\,\dd s
  \\
  &\leq
  \int_{\mathbb R^2}|y|\,|\theta_0(y)|\,\dd y
  +\Vert \theta_0\Vert _{L^{Q/(Q-1)}}
    \int_0^t\Vert u(s)\Vert _{L^Q}\,\dd s,
\end{aligned}
\end{equation}
where we used conservation of the
$L^{Q/(Q-1)}$ norm of $\theta$.

Taking $\alpha=1$ in the mild formulation
\eqref{eq:u-mild-Lp}, applying
\eqref{eq:alpha-one-semigroup}, and changing variables
$r=(t-\tau)/\varepsilon$, we find
\begin{align*}
  \Vert u(t)\Vert _{L^Q}
  \leq {}&
  \Vert u_0\Vert _{L^Q}
  +C_Q\int_0^{t/\varepsilon}
    \min\left\{1,r^{-\frac32+\frac1Q}\right\}
    \left(
      \Vert \theta_0\Vert _{L^Q}
      +\int_{\mathbb R^2}
        |x|\,|\theta(t-\varepsilon r,x)|\,\dd x
    \right)\,\dd r
  \\
  &+
  C_Q\int_0^t
    \Vert u(s)\Vert _{L^Q}\Vert \nabla u(s)\Vert _{L^\infty}\,\dd s.
\end{align*}
Using \eqref{eq:theta-first-moment} at time
$t-\varepsilon r\in[0,t]$ and
\eqref{eq:alpha-one-time-integrability}, we obtain
\[
  \Vert u(t)\Vert _{L^Q}
  \leq
  C+C\int_0^t
    \bigl(1+\Vert \nabla u(s)\Vert _{L^\infty}\bigr)
    \Vert u(s)\Vert _{L^Q}\,\dd s.
\]
Here $C$ depends on $Q$, $\Vert u_0\Vert _{L^Q}$,
$\Vert \theta_0\Vert _{L^1\cap L^\infty}$, and
$\int_{\mathbb R^2}|x|\,|\theta_0(x)|\,\dd x$,
but is independent of $\varepsilon$.
Gronwall's inequality and the uniform bound
$\Vert \nabla u\Vert _{L^1_T(L^\infty)}\leq Ce^{CT}$
therefore yield
\[
  \Vert u\Vert _{L^\infty_T(L^Q)}
  \leq
  C\exp\left(
    CT+C\Vert \nabla u\Vert _{L^1_T(L^\infty)}
  \right)
  \leq Ce^{\exp(CT)}.
\]
This proves \eqref{eq:u-LQ-es} when $\alpha=1$.
\end{proof}

We now establish uniqueness in the following class.

\begin{proposition}\label{prop:uniq}
Let $\alpha\in(\tfrac12,1]$,
$\varepsilon\in(0,1]$, and $T>0$.
Assume that
\[
  (u_0,\theta_0)
  \in B^0_{p,\infty}(\mathbb R^2)
  \times B^{-1}_{p,\infty}(\mathbb R^2)
\]
for some $2\leq p<\infty$, with $\div u_0=0$.
Then the 2D Boussinesq system \eqref{BoussEq}
has at most one solution pair $(u,\theta)$
on $[0,T]$ in the class
\[
  \mathcal X_T
  :=
  \bigl(
    L^1_T(B^1_{\infty,1})
    \cap L^\infty_T(W^{1,p})
  \bigr)
  \times
  \bigl(
    L^\infty_T(B^{-1}_{p,\infty})
    \cap L^\infty_T(B^0_{\infty,\infty})
  \bigr).
\]
\end{proposition}

\begin{proof}[Proof of Proposition~\ref{prop:uniq}]
Let $(u^{(1)},\theta^{(1)})$
and $(u^{(2)},\theta^{(2)})$
be two solutions in $\mathcal X_T$
with the same initial data.
Set
\[
  \delta u:=u^{(1)}-u^{(2)},
  \qquad
  \delta\theta:=\theta^{(1)}-\theta^{(2)},
  \qquad
  \delta p:=p^{(1)}-p^{(2)}.
\]
Then
\[
\left\{
\begin{aligned}
  \varepsilon
  \bigl(
    \partial_t\delta u
    +u^{(2)}\cdot\nabla\delta u
  \bigr)
  +\Lambda^{2\alpha}\delta u+\nabla\delta p
  &=
  -\varepsilon\delta u\cdot\nabla u^{(1)}
  +\delta\theta e_2,
  \\
  \partial_t\delta\theta
  +u^{(2)}\cdot\nabla\delta\theta
  &=
  -\delta u\cdot\nabla\theta^{(1)},
  \\
  \div\delta u&=0,
  \\
  (\delta u,\delta\theta)|_{t=0}&=(0,0).
\end{aligned}
\right.
\]

Applying \eqref{es:linear-velocity-eqn-2}
after the time change $t=\varepsilon\tau$,
with $s=0$, $\rho_1=1$, and $\rho_2=\infty$,
we obtain
\begin{equation}\label{eqn-the-thing-to-bound}
\begin{aligned}
  \Vert \delta u\Vert _{L^\infty_t(B^0_{p,\infty})}
  \leq
  Ce^{C\int_0^t
    \Vert u^{(2)}(\tau)\Vert _{W^{1,\infty}}\,\dd\tau}
  \Bigl(
    \Vert \delta u\cdot\nabla u^{(1)}
      \Vert _{L^1_t(B^0_{p,\infty})}
      +(1+t/\varepsilon)
    \Vert \delta\theta\Vert _{L^\infty_t(B^{-1}_{p,\infty})}
  \Bigr).
\end{aligned}
\end{equation}
Here we used $2\alpha>1$ and the embedding
$L^\infty_t(B^{-1}_{p,\infty}) \hookrightarrow L^\infty_t(B^{-2\alpha}_{p,\infty})
= \widetilde L^\infty_t(B^{-2\alpha}_{p,\infty})$.

For the nonlinear velocity term,
H\"older's inequality and the embeddings
$B^0_{p,1}\hookrightarrow L^p
\hookrightarrow B^0_{p,\infty}$ give
\begin{align*}
  \Vert \delta u\cdot\nabla u^{(1)}\Vert _{B^0_{p,\infty}}
  \leq
  C\Vert \nabla u^{(1)}\Vert _{L^\infty}
  \Vert \delta u\Vert _{L^p}
  \leq
  C\Vert \nabla u^{(1)}\Vert _{L^\infty}
  \Vert \delta u\Vert _{B^0_{p,1}}.
\end{align*}

We use the logarithmic interpolation inequality
\[
  \Vert \delta u\Vert _{B^0_{p,1}}
  \leq
  C\Vert \delta u\Vert _{B^0_{p,\infty}}
  \log\left(
    e+
    \frac{\Vert \delta u\Vert _{W^{1,p}}}
         {\Vert \delta u\Vert _{B^0_{p,\infty}}}
  \right),
\]
with the right-hand side interpreted as zero
when $\delta u=0$.

Set
\[
  \mu(r):=r\log(e+1/r)\quad\text{for }r>0,
  \qquad \mu(0):=0.
\]
Using
$\log(e+ab)\leq C\log(e+a)\log(e+b)$
for $a,b\geq0$, we obtain
\[
  \Vert \delta u\Vert _{B^0_{p,1}}
  \leq
  C\mu\bigl(\Vert \delta u\Vert _{B^0_{p,\infty}}\bigr)
  \log\bigl(e+\Vert \delta u\Vert _{W^{1,p}}\bigr).
\]
Consequently,
\begin{equation}\label{eqn-log-interpol}
\begin{aligned}
  \Vert \delta u\cdot\nabla u^{(1)}\Vert _{B^0_{p,\infty}}
  &\leq
  C\Vert \nabla u^{(1)}\Vert _{L^\infty}
  \mu\bigl(\Vert \delta u\Vert _{B^0_{p,\infty}}\bigr)
  \log\bigl(e+\Vert \delta u\Vert _{W^{1,p}}\bigr).
\end{aligned}
\end{equation}

For the temperature difference,
\eqref{eq:TE-prog-es3}
and Lemma~\ref{lem-610-2} with $s=-1$ yield
\begin{equation}\label{eqn-theta-bound}
\begin{aligned}
  \Vert \delta\theta\Vert _{L^\infty_t(B^{-1}_{p,\infty})}
  &\leq
  Ce^{C\Vert u^{(2)}\Vert _{L^1_t(B^1_{\infty,1})}}
  \int_0^t
  \Vert \delta u\cdot\nabla\theta^{(1)}(\tau)
    \Vert _{B^{-1}_{p,\infty}}\,\dd\tau
  \\
  &\leq
  Ce^{C\Vert u^{(2)}\Vert _{L^1_t(B^1_{\infty,1})}}
  \int_0^t
  \Vert \theta^{(1)}(\tau)\Vert _{B^0_{\infty,\infty}}
  \Vert \delta u(\tau)\Vert _{B^0_{p,1}}\,\dd\tau.
\end{aligned}
\end{equation}

Set
\[
  X(t):=
  \Vert \delta u\Vert _{L^\infty_t(B^0_{p,\infty})}
  +
  \Vert \delta\theta\Vert _{L^\infty_t(B^{-1}_{p,\infty})}.
\]
Combining \eqref{eqn-the-thing-to-bound},
\eqref{eqn-log-interpol}, and \eqref{eqn-theta-bound}
with the logarithmic interpolation inequality gives
\[
  X(t)
  \leq
  C_{\varepsilon,T}
  \int_0^t
  \bigl(
    1+\Vert \nabla u^{(1)}(\tau)\Vert _{L^\infty}
  \bigr)
  \mu\bigl(X(\tau)\bigr)\,\dd\tau,
  \qquad 0\leq t\leq T.
\]
Here $C_{\varepsilon,T}$ depends on the fixed parameters
and the $\mathcal X_T$ norms of the two solutions.
We have also used the monotonicity of $\mu$
and the embedding
$B^1_{\infty,1}\hookrightarrow W^{1,\infty}$.
In particular,
\[
  1+\Vert \nabla u^{(1)}(\cdot)\Vert _{L^\infty}
  \in L^1(0,T).
\]
Since
$\int_{0^+}^1\frac{\dd r}{\mu(r)}=\infty$,
Osgood's lemma implies $X(t)=0$ on $[0,T]$.
Thus the two solutions coincide.
\end{proof}

\begin{proof}[Proof of Theorem~\ref{thm:main1}]
We first establish existence.
Fix $\varepsilon\in(0,1]$.
Let $(u_{0,\delta},\theta_{0,\delta})$
be the smooth approximations constructed in
Appendix~\ref{sec:A1}, and let
$(u_\delta,\theta_\delta)$
be the corresponding global smooth solutions
of \eqref{BoussEq}.

By Proposition~\ref{prop:compactness-main1},
for every $T>0$, a subsequence,
still denoted by $\delta$, converges to a pair
$(u,\theta)$ in the senses specified in
\eqref{conv-main1-u-weak}-\eqref{conv-main1-theta-Linf}.
The limit is a weak solution of \eqref{BoussEq},
in the sense of Definition~\ref{defn:weak-Bouss},
with initial data $(u_0,\theta_0)$.
Moreover,
\[
  \theta\in L^\infty([0,T],L^1\cap L^\infty),
  \qquad
  u\in L^\infty([0,T],W^{1,p})
  \cap\widetilde L^1([0,T],B^{2\alpha}_{\infty,\infty}),
\]
and all the estimates in
Proposition~\ref{prop:ap-es-basic}
pass to the limit, with constants independent of
$\varepsilon\in(0,1]$ as stated there.

To prove uniqueness, note that $2\alpha>1$ gives
$\widetilde L^1([0,T],B^{2\alpha}_{\infty,\infty})\hookrightarrow
  L^1([0,T],B^1_{\infty,1})$.
Moreover,
$L^1\cap L^\infty \hookrightarrow L^p \hookrightarrow B^{-1}_{p,\infty}$ and 
$L^\infty\hookrightarrow B^0_{\infty,\infty}$.
Thus $(u,\theta)$ belongs to the uniqueness class
$\mathcal X_T$
in Proposition~\ref{prop:uniq},
and uniqueness follows.

Since $T>0$ is arbitrary, uniqueness gives a global solution.
This completes the proof.
\end{proof}

\section{Persistence of the \texorpdfstring{$C^{1,\gamma}$}{C(1,gamma)}-boundary regularity with \texorpdfstring{$0<\gamma\leq1$}{0 < gamma <= 1}}\label{sec:C1}
In this section, we shall prove Theorem \ref{thm:C1}, 
which concerns the global persistence of $C^{1,\gamma}$-boundary regularity of temperature patches with $0<\gamma\leq 1$. 

We first record the relation between the striated estimates and the
boundary regularity. Restoring the superscript $\varepsilon$, set
$z^\varepsilon(a,t):=X_t^\varepsilon(z_0(a))$. The equation
for $\partial_{a }z^\varepsilon$ and the uniform Lipschitz estimate
for $u^\varepsilon$ give
\begin{align}\label{es:inf-dz}
\inf_{0<\varepsilon\leq1}
\inf_{0\leq t\leq T}
\inf_{a\in\mathbb S^1}
|\partial_{a}z^\varepsilon(a,t)|
\geq
c_0e^{-C_T}>0,
\qquad
c_0:=
\inf_{a\in\mathbb S^1}
|\partial_{a}z_0(a)|.
\end{align}
Moreover, refer to \eqref{eq:boundary-identities}, we only need to show the regularity of $W^\varepsilon$.

Throughout this section, we prove the required {\it{{\it{a priori}}}} estimates at the level of smooth solutions. The construction of smooth approximations and the rigorous passage to the limiting solution are deferred to Appendices \ref{sec:A} and \ref{sec:B2}, respectively.

We divide the proof into four cases according to the values of $\gamma\in (0,1]$ 
and $\alpha\in (\frac{1}{2},1]$.
\vskip1mm

$\bullet$ \textbf{Case I: $\gamma\in(0,2\alpha-1]$, $\alpha\in (\frac{1}{2},1)$.}
\vskip1mm

Noting that $\partial D(t) = X_t(\partial D_0)$ with the flow map $X_t(\cdot)$ given by \eqref{eq:flowmap},
to get \eqref{eq:reg-parD(t)-1-0-main}, it suffices to prove $\nabla u \in L^1_T(C^\gamma)$ uniformly in $\varepsilon$,
since it directly implies $\nabla X_t^{\pm 1}\in L^\infty_T(C^\gamma)$ uniformly in $\varepsilon$. 
This has been addressed in \cite{LXY25}; indeed, one can refer to \eqref{eq:u-Gam-es2} and \eqref{es:nabu-Gam2} 
in Proposition \ref{prop:ap-es-basic} for $\gamma\in(0,2\alpha-1)$ (using the embedding 
$\widetilde{L}^1_T(B^{2\alpha-1}_{\infty,\infty})\hookrightarrow L^1_T(C^\gamma)$), 
and \eqref{es:nabu-Gam2} for $\gamma=2\alpha-1$.
\vskip1mm

$\bullet$ \textbf{Case II: $\gamma\in(0,1)$, $\alpha=1$.}
\vskip1mm

Thanks to \eqref{es:nabu-Gam2-2b} and \eqref{es:nabu-Gam2} in Proposition \ref{prop:ap-es-basic}, 
we have $\nabla u\in L_T^1(C^\gamma)$ uniformly in $\varepsilon$, 
which immediately implies the global uniform-in-$\varepsilon$ persistence of the 
$C^{1,\gamma}$-patch boundary $\partial D(t)$.
\vskip1mm
 
$\bullet$ \textbf{Case III: $\gamma\in(2\alpha-1,1)$, $\alpha\in(\frac{1}{2},1)$.} 
\vskip1mm

In this case $u_0\in W^{1,p}(\mathbb{R}^2)$ with $p>\frac{2}{2\alpha-\gamma}$. 
Noting that $\varphi(\cdot,t)$ is the level-set characterization of the patch boundary 
$\partial D(t)$ and $W=\nabla^\perp \varphi$, 
it suffices to show that $W \in L^\infty([0,T]; C^\gamma(\R^2))$ uniformly in $\varepsilon$. 

Thanks to $W$-equation \eqref{eq:W-eq} and the regularity propagation estimate \eqref{eq:T-BesEs}, we have
\begin{equation}\label{eq:W-Cgam}
\begin{aligned}
  \Vert  W(t)\Vert _{C^\gamma} \leq & C \left\Vert  W_0\right\Vert _{C^\gamma}
  + C \int_0^t\left\Vert \partial_W u(\tau) \right\Vert _{C^\gamma} \mathrm{d} \tau 
  + C \int_0^t \Vert \nabla u(\tau)\Vert _{L^\infty}\Vert  W(\tau)\Vert _{C^\gamma} \dd\tau.
\end{aligned}
\end{equation}
The remaining main task is to estimate $\Vert \partial_W  u\Vert _{L^1_t(C^\gamma)}$. 
The identity \eqref{eq:u-Gam-rela} gives
\begin{align}\label{pWu}
  \Vert \partial_W  u\Vert _{L^1_t(C^\gamma)} 
  \leq \big\Vert \partial_W \nabla \Lambda^{-2} \Gamma\big\Vert _{L^1_t(C^\gamma)}
  + \big\Vert \partial_W  \nabla^2\Lambda^{-2-2\alpha} \theta \big\Vert _{L^1_t(C^\gamma)}.
\end{align}
For the first term, due to that
\begin{equation}\label{es:mDGam}
\begin{aligned}
  \Vert \nabla^2 \Lambda^{-2} \Gamma\Vert _{C^\gamma}
  \leq C \big(\Vert \Delta_{-1}(\nabla^2 \Lambda^{-2} \Gamma)\Vert _{L^\infty} + 
  \Vert \Gamma\Vert _{C^\gamma}\big)
  \leq C \Vert \Gamma\Vert _{L^p\cap C^\gamma},
\end{aligned}
\end{equation}
one can get 
\begin{align*}
  \Vert \partial_W \nabla \Lambda^{-2} \Gamma\Vert _{C^\gamma} 
  \leq C \Vert W\Vert _{C^\gamma} \Vert \nabla^2 \Lambda^{-2} \Gamma\Vert _{C^\gamma}
  \leq C \Vert W\Vert _{C^\gamma} \Vert \Gamma\Vert _{L^p \cap C^\gamma}.
\end{align*}
According to \eqref{es:prod-es6} with $\sigma = -2\alpha\in (-2,-1)$, and using the fact that 
$\p_W(\nabla^2\Lambda^{-2-2\alpha} \theta)=\nabla^2 \Lambda^{-2-2\alpha} 
\p_W\theta-[\nabla^2\Lambda^{-2-2\alpha}, W\cdot\nabla]\theta$, 
we deduce 
\begin{align*}
  &\big\Vert \partial_W \nabla^2 \Lambda^{-2-2\alpha} \theta \big\Vert _{C^\gamma} \\
  & \leq \Vert  \nabla^2 \Lambda^{-2-2\alpha} \p_W\theta\Vert _{C^\gamma} 
  + \Vert [\nabla^2 \Lambda^{-2-2\alpha}, W\cdot\nabla]\theta\Vert _{C^\gamma} \\
  & \leq C \Vert \Delta_{-1} \nabla^2 \Lambda^{-2-2\alpha} \p_W\theta\Vert _{L^\infty}
  + C \sup\limits_{j\in\NN} 2^{j\gamma}\Vert \Delta_j \nabla^2 \Lambda^{-2-2\alpha} 
  \p_W\theta\Vert _{L^\infty} + C \Vert  W\Vert _{C^\gamma}\Vert \theta\Vert _{L^1\cap L^\infty};
\end{align*}
together with the following estimates
\begin{align*}
  \Vert \Delta_{-1} \nabla^2 \Lambda^{-2-2\alpha} \p_W\theta\Vert _{L^\infty}
  & =\Vert \Delta_{-1} \nabla^2 \Lambda^{-2-2\alpha} \mathrm{div}(W\theta)\Vert _{L^\infty}
  \leq C\Vert \p \Lambda^{-2\alpha} (W\theta)\Vert _{L^p}\\
  &\leq C\Vert W\theta\Vert _{L^\frac{2p}{(2\alpha-1)p+2}}\leq C\Vert W\Vert _{L^\infty}\Vert \theta\Vert _{L^\frac{2p}{(2\alpha-1)p+2}},
\end{align*}
and
\begin{align*}
  \sup\limits_{j\in\NN}2^{j\gamma}\Vert \Delta_j \nabla^2 
  \Lambda^{-2-2\alpha} \p_W\theta\Vert _{L^\infty}
  \leq C \sup\limits_{j\in\NN} 2^{j(\gamma-2\alpha)}\Vert \Delta_j \p_W\theta\Vert _{L^\infty}
  \leq C \Vert  \p_W\theta\Vert _{C^{\gamma-2\alpha}},
\end{align*}
we infer that 
\begin{align*}
  \big\Vert \partial_W \nabla^{\perp}\p_1\Lambda^{-2-2\alpha} \theta(t) \big\Vert _{C^\gamma}  
  \leq  C \Vert  \p_W\theta(t)\Vert _{C^{\gamma-2\alpha}} 
  + C \Vert  W(t)\Vert _{C^\gamma} \Vert \theta_0\Vert _{L^1\cap L^\infty}.
\end{align*}
In the above, we also have used the simple fact that for every $r\in[1,\infty]$,\,
\begin{align}\label{eq:the-MP} 
  \Vert \theta(t)\Vert _{L^r}=\Vert \theta_0\Vert _{L^r},\quad \forall t\geq 0.
\end{align}
Since $\partial_W\theta$ solves the transport equation
\begin{align}\label{eq:par-W-the}
  \partial_t (\partial_W \theta) + u\cdot \nabla(\partial_W \theta) =0,\quad 
  \partial_W \theta |_{t=0} = \partial_{W_0} \theta_0,
\end{align}
we apply Lemma \ref{lem:TE-prog-es} and \eqref{es:nabu-Gam2} to show that for every $\gamma\in(2\alpha-1,1)$,
\begin{align}\label{eq:parWthe-Bes2}
  \Vert \partial_W \theta(t)\Vert _{C^{\gamma-2\alpha}} \leq C  e^{C \int_0^t \Vert \nabla u(\tau)\Vert _{L^\infty}
  \dd \tau} \Vert \partial_{W_0}\theta_0\Vert _{C^{\gamma-2\alpha}} \leq C  e^{\exp(C t)}.
\end{align}
Substituting these estimates into \eqref{pWu} and \eqref{eq:W-Cgam}, we obtain
\begin{align*}
  \Vert \partial_W  u\Vert _{L^1_t(C^\gamma)} 
  \leq C \int_0^t \Vert W(\tau)\Vert _{C^\gamma} \big(\Vert \Gamma(\tau)\Vert _{L^p \cap C^\gamma} 
  + 1  \big) \dd \tau 
  + C e^{\exp (Ct)},
\end{align*}
and
\begin{align*}
  \Vert W(t)\Vert _{C^\gamma}  \leq Ce^{\exp(Ct)} + C \int_0^t \Vert W(\tau)\Vert _{C^\gamma} 
  \big(\Vert \Gamma(\tau)\Vert _{L^p\cap C^\gamma} + \Vert \nabla u(\tau)\Vert _{L^\infty} 
  + 1 \big) \mathrm{d} \tau.
\end{align*}
Since $\widetilde{L}^1_T(B^{2\alpha}_{p,\infty})\hookrightarrow L_T^1(L^p\cap C^\gamma)$ for $p>\frac{2}{2\alpha-\gamma}$, 
taking advantage of Gronwall's inequality and the uniform estimate \eqref{es:nabu-Gam2} gives 
\begin{align}\label{es:W-Cgam}
  \Vert W\Vert _{L^\infty([0,T], C^\gamma(\mathbb{R}^2))}\leq C e^{\exp(CT)} \exp\Big\{C T + C\Vert \Gamma\Vert _{L^1_T(L^p\cap C^\gamma)}
  + C \Vert \nabla u\Vert _{L^1_T(L^\infty)}\Big\} 
  \leq C e^{\exp (CT)}.
\end{align}
\vskip1mm

$\bullet$ \textbf{Case IV: $\gamma=1$, $\alpha\in(\frac{1}{2},1]$.} 
\vskip1mm

Recall that in this case $u_0\in W^{1,p}(\mathbb{R}^2)$, $p>\frac{2}{2\alpha-1}$. 
According to \eqref{eq:boundary-identities} and \eqref{es:inf-dz}, it suffices to show that $\big(\partial_W W\big)(\cdot,t) \in L^\infty([0,T]; L^\infty(\R^2))$
uniformly in $\varepsilon$. 
In fact, we shall prove a stronger result, that is, $W\in L^\infty([0,T], W^{1,\infty}(\R^2))$ uniformly in $\varepsilon$.

The maximum principle of the transport equation \eqref{eq:W-eq} gives
\begin{align}\label{eq:W-Linf-es}
  \Vert W(t)\Vert _{L^\infty} \leq \Vert W_0\Vert _{L^\infty} 
  + \int_0^t \Vert \nabla u(\tau)\Vert _{L^\infty} \Vert W(\tau)\Vert _{L^\infty} \dd \tau,
\end{align}
which combined with \eqref{es:nabu-Gam2} implies that
\begin{align}\label{eq:W-Linf-es-2}
  \Vert W(t)\Vert _{L^\infty} \leq \Vert W_0\Vert _{L^\infty} 
  e^{\int_0^t \Vert \nabla u(\tau)\Vert _{L^\infty}\dd \tau}
  \leq C e^{\exp(Ct)}.
\end{align}
Applying the $\nabla$ operator to the $W$-equation \eqref{eq:W-eq} gives
\begin{align}\label{eq:nabW}
  \partial_t \nabla W+u \cdot \nabla(\nabla W)
  =\partial_W \nabla u+\nabla W \cdot \nabla u-\nabla u \cdot \nabla W,
  \quad \nabla W|_{t=0} = \nabla W_0.
\end{align}
Taking advantage of the maximum principle again, we get
\begin{equation}\label{eq:nabW-Linf}
\begin{aligned}
  \Vert \nabla W(t)\Vert _{L^\infty} \leq  \left\Vert \nabla W_0\right\Vert _{L^\infty}
  + \int_0^t\left\Vert \partial_W \nabla u(\tau) \right\Vert _{L^\infty} \mathrm{d} \tau
  + 2\int_0^t\Vert \nabla u(\tau)\Vert _{L^\infty}\Vert \nabla W(\tau)\Vert _{L^\infty} \mathrm{d} \tau .
\end{aligned}
\end{equation}

The main task is to estimate $\partial_W \nabla u$. 
Using the identity \eqref{eq:u-Gam-rela}, we find 
\begin{align}\label{es:parWnabu-Linf}
  \Vert \partial_W \nabla u\Vert _{L^1_t (L^\infty)} 
  \leq \big\Vert \partial_W \nabla^2 \Lambda^{-2} \Gamma\big\Vert _{L^1_t(L^\infty)}
  + \big\Vert \partial_W \nabla^3\Lambda^{-2-2\alpha} \theta \big\Vert _{L^1_t(L^\infty)}.
\end{align}
For the first term on the right-hand side of the above inequality, by virtue of \eqref{es:nabu-Gam2}, \eqref{eq:W-Linf-es-2}
and the embedding $\widetilde{L}^1_t(B^{2\alpha}_{p,\infty})\hookrightarrow L^1_t(L^p\cap B^1_{\infty,1})$, 
we see that
\begin{align}
  \Vert \partial_W\nabla^2 \Lambda^{-2}\Gamma\Vert _{L^1_t(L^\infty)} 
  & \leq C \Vert W\Vert _{L^\infty_t(L^\infty)} 
  \Vert \nabla^3 \Lambda^{-2}\Gamma\Vert _{L^1_t(L^\infty)} \nonumber\\
  & \leq C \Vert W\Vert _{L^\infty_t(L^\infty)} \big(\Vert \Delta_{-1}
  (\nabla^3\Lambda^{-2}\Gamma)\Vert _{L^1_t(L^p)}+ 
  \Vert \Gamma\Vert _{L^1_t(B^1_{\infty,1})}\big) \nonumber\\
  & \leq C \Vert W\Vert _{L^\infty_t(L^\infty)} \Vert \Gamma\Vert _{L^1_t(L^p\cap B^1_{\infty,1})} \nonumber\\
  &\leq Ce^{\exp(Ct)}.\label{es:parWnabu-Linf1}
\end{align}
For the second term, using \eqref{es:prod-es5} with $\sigma= 1-2\alpha \in (-2,0)$ and continuous embedding 
$B^{\mu+2\alpha-2}_{\infty,\infty} \hookrightarrow L^\infty\,(\mu\in(2-2\alpha,1))$
leads to
\begin{align*}
  \big\Vert \partial_W \nabla^3\Lambda^{-2-2\alpha} \theta \big\Vert _{L^1_t(L^\infty)}
  & \leq \Vert \nabla^3\Lambda^{-2-2\alpha} \p_W\theta\Vert _{L_t^1(L^\infty)} 
  + \Vert [\nabla^3\Lambda^{-2-2\alpha}, W\cdot\nabla]\theta\Vert _{L_t^1(L^\infty)} \\
  & \lesssim  \Vert \nabla^3\Lambda^{-2-2\alpha} \p_W\theta\Vert _{L_t^1(B^{\mu+2\alpha-2}_{\infty, \infty})} 
  +  \int_0^t\Vert \nabla W\Vert _{L^\infty}\Vert \theta\Vert _{L^1\cap L^\infty}\dd\tau.
\end{align*}
Noticing that
\begin{align*}
   \Vert \nabla^3\Lambda^{-2-2\alpha} \p_W\theta\Vert _{B^{\mu+2\alpha-2}_{\infty, \infty}} &\leq C \Vert \Delta_{-1}\nabla^3\Lambda^{-2-2\alpha} \mathrm{div}\,(W\theta)\Vert _{L^\infty}  + \sup\limits_{j\geq 0}2^{j(\mu+2\alpha-2)}
  \Vert \Delta_j \nabla^3\Lambda^{-2-2\alpha} \p_W\theta\Vert _{L^\infty} \\
  & \lesssim \Vert  W\theta\Vert _{L^2} 
  + C\sup\limits_{j\geq 0}2^{j(\mu-1)}\Vert \Delta_j \p_W\theta\Vert _{L^\infty} \\
  &\lesssim \Vert W\Vert _{L^\infty}\Vert \theta\Vert _{L^2}+ \Vert \p_W\theta\Vert _{B^{\mu-1}_{\infty,\infty}},
\end{align*}
together with \eqref{eq:the-MP} and \eqref{eq:W-Linf-es-2}, we deduce that
\begin{align}
  \big\Vert \partial_W \nabla^3\Lambda^{-2-2\alpha} \theta \big\Vert _{L^1_t(L^\infty)} 
  \leq Ce^{\exp(Ct)}
  + Ct\Vert \p_W\theta\Vert _{L_t^\infty(B^{\mu-1}_{\infty,\infty})} 
  + C\int_0^t \Vert \nabla W(\tau)\Vert _{L^\infty} \dd\tau.\label{es:parWnabu-Linf2}
\end{align}
In light of the fact that $\partial_{W_0}\theta_0\in C^{\mu-1}$ by Lemma \ref{lem:pWk-theta} 
and our assumption on $\theta_0$, similarly as deriving \eqref{eq:parWthe-Bes2}, we have, for the exponent $\mu$ in the profile assumption,
\begin{align*}
  \Vert \partial_W \theta(t)\Vert _{C^{\mu-1}} \lesssim   e^{C \int_0^t \Vert \nabla u(\tau)\Vert _{L^\infty}
  \dd \tau} \Vert \partial_{W_0}\theta_0\Vert _{C^{\mu-1}} \leq C  e^{\exp(C t)}.
\end{align*}
Inserting the above estimates into \eqref{es:parWnabu-Linf2}, together with \eqref{es:parWnabu-Linf1}, \eqref{es:parWnabu-Linf} and \eqref{eq:nabW-Linf} gives
\begin{align*}
  \Vert \nabla W(t)\Vert _{L^\infty} \leq C e^{\exp(Ct)}
  + C \int_0^t\big(\Vert \nabla u(\tau)\Vert _{L^\infty} + 1\big)\Vert \nabla W(\tau)\Vert _{L^\infty} \mathrm{d} \tau .
\end{align*}
Gronwall's inequality combined with \eqref{eq:W-Linf-es-2} and \eqref{es:nabu-Gam2} ensures 
\begin{align}\label{eq:W-W1inf-bdd}
  \Vert W\Vert _{L^\infty_T(W^{1,\infty})} 
  + \Vert \partial_W \nabla u\Vert _{L^1_T(L^\infty)}
  \leq C e^{\exp(CT)} e^{CT+C\int_0^T\Vert \nabla u(t)\Vert _{L^\infty} \dd t} 
  \leq Ce^{\exp (CT)},
\end{align}
and
\begin{align}\label{eq:nab-paw-u-L1Linf}
  \Vert \nabla \partial_W u\Vert _{L^1_T(L^\infty)}
  \leq \Vert \partial_W \nabla u\Vert _{L^1_T(L^\infty)}
  + \Vert \nabla W\Vert _{L^\infty_T(L^\infty)} \Vert \nabla u\Vert _{L^1_T(L^\infty)}
  \leq C e^{\exp(CT)}.
\end{align}

The preceding \textit{{\it{a priori}}} estimates are uniform with respect to
$\varepsilon\in(0,1]$. The approximation and compactness argument
given in Appendix~\ref{sec:B2} shows that these bounds are inherited
by the solution $(u,\theta)$ constructed in
Theorem~\ref{thm:main1}. In particular, the associated tangent field
$W=\nabla^\perp\varphi$ satisfies the corresponding $C^\gamma$ estimate
when $0<\gamma<1$, and the $W^{1,\infty}$ estimate when $\gamma=1$.
Consequently, the $C^{1,\gamma}$-regularity of $\partial D(t)$ is
preserved on $[0,T]$, with a bound independent of $\varepsilon$.
This completes the proof of Theorem~\ref{thm:C1}.

\section{Persistence of the \texorpdfstring{$C^{2,\gamma}$}{C(2,gamma)}-boundary regularity with \texorpdfstring{$0<\gamma\leq1$}{0 < gamma <= 1}}\label{sec:C2}
The purpose of this section is to provide the proof of Theorem \ref{thm:main2} with any $\alpha\in(\frac{1}{2},1]$ and $\gamma\in (0,1]$. As in the preceding section, all computations below are carried out at the level of smooth solutions. Their rigorous justification for the solution of Theorem \ref{thm:main1} is given in Appendix \ref{sec:B3}.

First, we present the following result, which is crucial for relaxing the regularity assumptions on $\partial_{W_0}u_0$ imposed in \cite{LXY25}.
\begin{lemma}\label{lem:pGamma}
Consider \{\textrm{case\, I}: $\gamma\in(0,2\alpha-1]$, $\alpha\in(\tfrac{1}{2},1)$\} 
and \{\textrm{case II}: $\gamma\in(0,1)$, $\alpha=1$\}.
Assume that $\theta_0(x) = \bar{\theta}_0(x) \mathbf{1}_{D_0}(x)$, $\partial D_0\in C^{2,\gamma}$,  
$\bar\theta_0$ satisfies the corresponding profile assumption
\eqref{eq:bar-the-cd3} or \eqref{eq:bar-the-cd4}, 
and $u_0\in W^{1,q}$ for some $q>\frac{2}{2\alpha-1}$, 
$\partial_{W_0}u_0\in \dot{W}^{1,p}$ for some 
$p > \frac{2}{2\alpha-\gamma}$.
Then we have
\begin{align*}
  \Vert \partial_W \Gamma\Vert _{L^1_T(B^\gamma_{\infty,1})} \leq C e^{\exp (CT)}.
\end{align*}
\end{lemma}

\begin{proof}[Proof of Lemma \ref{lem:pGamma}]
(1) We first address case I.
From the $\Gamma$-equation \eqref{eq:Gamma-intr} and the fact that $\left[\partial_W, \partial_t+u \cdot \nabla\right]=0$,
we see that $\partial_W \Gamma$ solves the following equation:
\begin{equation}\label{eq:parW}
\left\{
\begin{aligned}
  \partial_t\left(\partial_W \Gamma\right)+u \cdot \nabla\left(\partial_W \Gamma\right)
  + \frac{1}{\varepsilon} \Lambda^{2\alpha}\left(\partial_W \Gamma\right)
  & = \frac{1}{\varepsilon} [\Lambda^{2\alpha},W\cdot\nabla]\Gamma
  + \partial_W \big(\left[\mathcal{R}_{1-2\alpha}, u \cdot \nabla\right] \theta\big), \\
  \partial_W \Gamma |_{t=0} & = \partial_{W_0}\Gamma_0.
\end{aligned}
\right.
\end{equation}
Taking advantage of the identities 
$\Gamma_0=\omega_0-\mathcal{R}_{1-2\alpha} \theta_0  =\omega_0-\partial_1 \Lambda^{-2\alpha} \theta_0$
and $\partial_{W_0}\omega_0=\partial_{W_0}\mathrm{curl}\,u_0=\mathrm{curl}\, (\partial_{W_0}u_0)
+\partial_2W_0\cdot\nabla u_{1,0}-\partial_1W_0\cdot\nabla u_{2,0}$ (recalling $u_0= (u_{1,0},u_{2,0})$),
we decompose $\partial_W\Gamma$ as
\begin{equation*}
 \partial_W\Gamma = F^{(1)} + F^{(2)},
\end{equation*}
where $F^{(1)}$ and $F^{(2)}$ respectively solve 
\begin{equation}\label{eq:F1}
\left\{
\begin{aligned}
  & \partial_t F^{(1)} + u\cdot\nabla F^{(1)} + \frac{1}{\varepsilon} \Lambda^{2\alpha} F^{(1)}
  = \frac{1}{\varepsilon} [\Lambda^{2\alpha},W\cdot\nabla]\Gamma
  + \partial_W \big(\left[\mathcal{R}_{1-2\alpha}, u \cdot \nabla\right] \theta\big), \\
  & F^{(1)}\big|_{t=0} = -\partial_{W_0} \mathcal{R}_{1-2\alpha} \theta_0 + \partial_2W_0\cdot\nabla u_{1,0} 
  - \partial_1W_0\cdot\nabla u_{2,0},
\end{aligned}
\right.
\end{equation}
and
\begin{equation}\label{eq:F2}
\begin{aligned}
  \partial_t F^{(2)} + u\cdot\nabla F^{(2)} + \frac{1}{\varepsilon} \Lambda^{2\alpha} F^{(2)} = 0, 
  \quad F^{(2)}\big|_{t=0} = \mathrm{curl} \,(\partial_{W_0}u_0) .
\end{aligned}
\end{equation}

According to Lemma \ref{lem:TD-sm},
there exists a constant $C>0$ independent of $\varepsilon$ so that for all $\gamma'$ such that  $2\alpha-1<\gamma^{\prime}<
\min\{1,4\alpha-2-\frac{2}{q}\}$ (recalling that $0<\frac{2}{q}<2\alpha-1 $), we have 
\begin{align*}
  \sum_{j\in\mathbb{N}} 2^{j\gamma'} \Vert \Delta_j F^{(1)}\Vert _{L^1_t(L^\infty)}
  & \leq C\, \varepsilon\, e^{C\int_0^t\Vert \nabla u(\tau)\Vert _{L^\infty}\dd \tau}
  \bigg(\left\Vert \partial_{W_0} \mathcal{R}_{1-2\alpha}\theta_0\right\Vert _{B_{\infty, 1}^{\gamma^{\prime}-2\alpha}} 
  + \left\Vert \nabla W_0\cdot\nabla u_0\right\Vert _{B_{\infty, 1}^{\gamma^{\prime}-2\alpha}} \\
  & \quad +\frac{1}{\varepsilon} \int_0^t\left\Vert \left[\Lambda^{2\alpha}, W \cdot \nabla\right] 
  \Gamma\right\Vert _{ B_{\infty, 1}^{\gamma^{\prime}-2\alpha}}\dd \tau 
  +\int_0^t\left\Vert \partial_W\left(\left[\mathcal{R}_{1-2\alpha}, u \cdot \nabla\right] 
  \theta\right)\right\Vert _{B_{\infty, 1}^{\gamma^{\prime}-2\alpha}}
  \dd \tau \bigg).
\end{align*}
Taking advantage of the product estimate \eqref{eq:prodBes}, we get
\begin{align*}
  \left\Vert \partial_{W_0} \mathcal{R}_{1-2\alpha} \theta_0\right\Vert _{B_{\infty, 1}^{\gamma^{\prime}-2\alpha}} 
  \leq C \left\Vert W_0\right\Vert _{L^\infty} 
  \left\Vert \nabla \mathcal{R}_{1-2\alpha} \theta_0\right\Vert _{ B_{\infty, 1}^{\gamma^{\prime}-2\alpha}}
  \leq C \left\Vert \varphi_0\right\Vert _{W^{1, \infty}}\left\Vert \theta_0\right\Vert _{L^2\cap L^\infty} <\infty,
\end{align*}
and
\begin{align*}
   \left\Vert \nabla W_0\cdot\nabla u_0\right\Vert _{ B_{\infty, 1}^{\gamma^{\prime}-2\alpha}}  \leq C \left\Vert \nabla W_0\right\Vert _{L^{\infty}}
  \left\Vert \nabla u_0\right\Vert _{ B_{\infty, 1}^{\gamma^{\prime}-2\alpha}}
  \leq C\left\Vert \varphi_0\right\Vert _{W^{2, \infty}}\left\Vert \nabla u_0\right\Vert _{L^q} <\infty,
\end{align*}
where we use the continuous embedding
\begin{align}\label{eq:embed-fact1}
  L^q(\mathbb{R}^2)\hookrightarrow B^{\gamma'-2\alpha}_{\infty,1}(\mathbb{R}^2),\quad
  L^\infty(\mathbb{R}^2)\hookrightarrow B^{\gamma'+2-4\alpha + \frac{2}{q}}_{\infty,1}(\mathbb{R}^2),
\end{align}
for all $\gamma'$ such that  $2\alpha-1<\gamma^{\prime}<
\min\{1,4\alpha-2-\frac{2}{q}\}$. 

For the above choice of $\gamma'$, Lemma~\ref{lem:Rbeta-cm},
together with \eqref{es:nabu-Gam2}, \eqref{eq:the-MP},
and \eqref{eq:embed-fact1}, gives
\begin{align}\label{eq:R-comm-Bes1}
  \left\Vert \nabla\left(\left[\mathcal{R}_{1-2\alpha}, u \cdot \nabla\right] \theta\right)
  \right\Vert _{L^1_t\big(B_{\infty, 1}^{\gamma^{\prime}-2\alpha}\big)}
  &\leq C \left\Vert \left[\mathcal{R}_{1-2\alpha}, u \cdot \nabla\right] \theta
  \right\Vert _{L^1_t\big(B_{q, 1}^{\gamma^{\prime}+1-2\alpha+\frac{2}{q}}\big)} \nonumber \\
  &\leq C\Vert \nabla u\Vert _{L^1_t(L^q)}\Big(
  \Vert \theta\Vert _{L^\infty_t\big(B^{\gamma^{\prime}+ 2-4\alpha+\frac{2}{q}}_{\infty,1}\big)}
  +\Vert \theta\Vert _{L^\infty_t(L^2)}\Big) \nonumber \\
  &\leq C \Vert \nabla u\Vert _{L^1_t(L^q)} \Vert \theta\Vert _{L^\infty\cap L^2} \\
  &\leq C e^{Ct}, 
\end{align}
and together with \eqref{eq:prodBes} and \eqref{eq:W-Linf-es-2} (recalling $-1<\gamma^{\prime}-2\alpha<1-2\alpha<0$),
\begin{align*}
  \left\Vert \partial_W\left(\left[\mathcal{R}_{1-2\alpha}, u \cdot \nabla\right] \theta\right)
  \right\Vert _{L^1_t\big(B_{\infty, 1}^{\gamma^{\prime}-2\alpha}\big)}
  \leq C\Vert W\Vert _{L^\infty_t(L^\infty)} \left\Vert \nabla\left[\mathcal{R}_{1-2\alpha}, u \cdot \nabla\right] \theta\right\Vert _{L^1_t\big(B_{\infty, 1}^{\gamma^{\prime}-2\alpha}\big)} 
  \leq C e^{\exp(Ct)} .
\end{align*}
According to \eqref{eq:commEs} with $\sigma =2\alpha$ and $s = \gamma'-2\alpha$ and the uniform estimates  
\eqref{es:nabu-Gam2}, \eqref{eq:W-W1inf-bdd}, 
we deduce that for all $2\alpha-1<\gamma^{\prime}<1$,
\begin{align*}
  \left\Vert \left[\Lambda^{2\alpha}, W \cdot \nabla\right] \Gamma\right\Vert _{L^1_t(B_{\infty, 1}^{\gamma^{\prime}-2\alpha})}
  \leq C \Vert  W\Vert _{L^\infty_t(W^{1,\infty})} \Vert \Gamma\Vert _{L^1_t(B_{\infty, 1}^{\gamma'})} 
  \leq C e^{\exp(Ct)}.
\end{align*}
Hence, collecting the above estimates gives that for all $2\alpha-1<\gamma^{\prime}< \min\{1,4\alpha-2-\frac{2}{q}\}$,
\begin{align*}
  \big\Vert  \big( 2^{j\gamma'} \Vert \Delta_j F^{(1)}\Vert _{L^1_t(L^\infty)} \big)_{j\in\mathbb{N}} \big\Vert _{\ell^1}
  \leq C e^{\exp(Ct)}.
\end{align*}

We next estimate $F^{(2)}$ by an argument valid
in both cases I and II.
Since $p>2/(2\alpha-\gamma)$, we may choose
\[
  \max\{-1,\gamma+2/p-2\alpha\}<s_0<0.
\]
Using
$L^p\hookrightarrow B^{s_0}_{p,1}$ and
$s_0+2\alpha-2/p>\gamma$,
Bernstein's inequality and Lemma~\ref{lem:TD-sm}
applied to \eqref{eq:F2} yield
\begin{align*}
  \sum_{j\geq0}2^{j\gamma}
    \Vert \Delta_jF^{(2)}\Vert _{L^1_t(L^\infty)}
  &\leq
  C\sum_{j\geq0}2^{j(s_0+2\alpha)}
    \Vert \Delta_jF^{(2)}\Vert _{L^1_t(L^p)}
  \\
  &\leq
  C\varepsilon
  e^{C\int_0^t\Vert \nabla u(\tau)\Vert _{L^\infty}\,\dd\tau}
  \Vert \partial_{W_0}u_0\Vert _{\dot W^{1,p}}.
\end{align*}
For the low-frequency block, Bernstein's inequality
and the $L^p$ maximum principle give
\[
  \Vert \Delta_{-1}F^{(2)}\Vert _{L^1_t(L^\infty)}
  \leq Ct\Vert F^{(2)}(0)\Vert _{L^p}
  \leq Ct\Vert \partial_{W_0}u_0\Vert _{\dot W^{1,p}}.
\]
Combining these estimates, we obtain
\begin{equation}\label{eq:F2-unified}
\begin{aligned}
  \Vert F^{(2)}\Vert _{L^1_t(B^\gamma_{\infty,1})}
  \leq C(t+\varepsilon)
  e^{C\int_0^t\Vert \nabla u(\tau)\Vert _{L^\infty}\,\dd\tau}
  \Vert \partial_{W_0}u_0\Vert _{\dot W^{1,p}}
  \leq Ce^{\exp(Ct)}.
\end{aligned}
\end{equation}

Therefore, gathering the above estimates leads to that for every $\gamma\in (0,2\alpha-1]$,
\begin{align}\label{eq:parW-Gam-L1Bes1}
  \left\Vert \partial_W \Gamma\right\Vert _{L_t^1\big(B_{\infty, 1}^{\gamma}\big)}  
  &\leq \Vert \Delta_{-1} (\partial_W \Gamma)\Vert _{L_t^1(L^\infty)} + \sum_{j\in\mathbb{N}} 2^{j\gamma} 
  \Vert \Delta_j F^{(1)}\Vert _{L^1_t(L^\infty)}
  +\sum_{j\in\mathbb{N}} 2^{j\gamma} \Vert \Delta_j F^{(2)}\Vert _{L^1_t(L^\infty)} \nonumber \\
  &\leq C \Vert W\, \Gamma\Vert _{L_t^1(L^q)} 
  + C \sum_{j\in\mathbb{N}} 2^{j\gamma'} \Vert \Delta_j F^{(1)}\Vert _{L^1_t(L^\infty)}
  + C e^{\exp(Ct)}  \leq C e^{\exp(Ct)}.
\end{align}

(2) Next, we deal with case II. 
In this case, $\partial_W \Gamma$ solves the following equation
\begin{align*}
  \partial_t\left(\partial_W \Gamma\right)+u \cdot \nabla\left(\partial_W \Gamma\right)
  -\frac{1}{\varepsilon}\Delta\left(\partial_W \Gamma\right)
  & =-\frac{1}{\varepsilon}\left[\Delta, \partial_W\right] \Gamma
  +\partial_W\left(\left[\mathcal{R}_{-1}, u \cdot \nabla\right] \theta\right) \\
  & =-\frac{1}{\varepsilon}\Delta W \cdot \nabla \Gamma-\frac{2}{\varepsilon} \nabla W: \nabla^2 \Gamma
  + \partial_W\left(\left[\mathcal{R}_{-1}, u \cdot \nabla\right] \theta\right),
\end{align*}
with initial data $(\partial_W \Gamma)|_{t=0} = \partial_{W_0}\Gamma_0$. 
We also have $\partial_W \Gamma = F^{(1)} + F^{(2)}$ with $F^{(1)}$ and $F^{(2)}$ satisfying \eqref{eq:F1} and \eqref{eq:F2} 
in the case $\alpha =1$, 
respectively. According to Lemma \ref{lem:TD-sm}, together with the product estimate \eqref{eq:prodBes}, 
we find that 
\begin{align*}
  & \sup_{j\in\mathbb{N}} 2^j \Vert \Delta_j F^{(1)}\Vert _{L^1_t(L^\infty)} \\
  & \leq C\, \varepsilon\, e^{C\int_0^t\Vert \nabla u(\tau)\Vert _{B^0_{\infty,1}} \dd \tau}
  \bigg(\left\Vert \partial_{W_0} \mathcal{R}_{-1}\theta_0\right\Vert _{B_{\infty, \infty}^{-1}} 
  + \left\Vert \nabla W_0\cdot\nabla u_0\right\Vert _{B_{\infty, \infty}^{-1}} 
  +\frac{1}{\varepsilon} \int_0^t\left\Vert \Delta W \cdot \nabla\Gamma\right\Vert _{B_{\infty, \infty}^{-1}} \dd \tau \\
  & \quad 
  +\frac{1}{\varepsilon} \int_0^t\left\Vert \nabla W: \nabla^2 \Gamma \right\Vert _{ B_{\infty, \infty}^{-1}}\dd \tau 
  +\int_0^t\left\Vert \partial_W\left(\left[\mathcal{R}_{-1}, u \cdot \nabla\right] 
  \theta\right)\right\Vert _{B_{\infty, \infty}^{-1}}
  \dd \tau \bigg).
\end{align*}
Estimate \eqref{eq:F2-unified} remains valid for $\alpha=1$.
To estimate the initial data for $F^{(1)}$,
fix $\gamma'\in(1,2-2/q)$.
Using \eqref{eq:prodBes}, \eqref{eq:prodBes-endp} 
and the relevant Besov embeddings,
we obtain
\begin{align*}
  \left\Vert \nabla W_0\cdot\nabla u_0\right\Vert _{ B_{\infty, \infty}^{-1}}  
  & \leq C \Vert \nabla W_0\cdot \nabla u_0\Vert _{B^{\gamma'-2}_{\infty,\infty}} \\
  & \leq C \left\Vert \nabla W_0\right\Vert _{L^{\infty}}
  \left\Vert \nabla u_0\right\Vert _{ B_{\infty, \infty}^{\gamma^{\prime}-2}}
  \leq C\left\Vert \varphi_0\right\Vert _{W^{2, \infty}}\left\Vert \nabla u_0\right\Vert _{L^q} <\infty,
\end{align*}
\begin{align*}
  \left\Vert \partial_{W_0} \mathcal{R}_{-1} \theta_0\right\Vert _{B_{\infty,\infty}^{-1}} 
  & \leq C \left\Vert \partial_{W_0} \mathcal{R}_{-1} \theta_0\right\Vert _{B_{\infty, \infty}^{\gamma^{\prime}-2}} \\
  & \leq C \left\Vert W_0\right\Vert _{L^\infty} 
  \left\Vert \nabla \mathcal{R}_{-1} \theta_0\right\Vert _{ B_{\infty, \infty}^{\gamma^{\prime}-2}}
  \leq C \left\Vert \varphi_0\right\Vert _{W^{1, \infty}}\left\Vert \theta_0\right\Vert _{L^2\cap L^\infty} <\infty,
\end{align*}
and
\begin{align*}
  \Vert \Delta W \cdot \nabla \Gamma\Vert _{L^1_t(B^{-1}_{\infty,\infty})}
  \leq C \Vert \Delta W\Vert _{L^\infty_t(B^{-1}_{\infty,\infty})} \Vert \nabla \Gamma\Vert _{L^1_t(B^0_{\infty,1})}
  \leq  C \Vert W\Vert _{L^\infty_t(W^{1,\infty})} \Vert  \Gamma\Vert _{L^1_t(B^1_{\infty,1})},
\end{align*}
\begin{align*}
  \Vert \nabla W:\nabla^2\Gamma\Vert _{L^1_t(B^{-1}_{\infty,\infty})} 
  & = \Vert  \partial_j (\nabla W_j \cdot \nabla\Gamma) \Vert _{L^1_t(B^{-1}_{\infty,\infty})} \\
  & \leq C \Vert \nabla W_j\cdot \nabla \Gamma\Vert _{L^1_t(L^\infty)}
  \leq C \Vert \nabla W\Vert _{L^\infty_t(L^\infty)} \Vert \nabla \Gamma\Vert _{L^1_T(L^\infty)} .
\end{align*}
Thanks to \eqref{eq:R-1cm-es1}, \eqref{es:nabu-Gam2}, \eqref{eq:W-Linf-es-2} 
and the embedding $L^\infty\hookrightarrow B^{\gamma'-2}_{\infty,1}$ for every $\gamma'\in(1,2-\frac{2}{q})$, 
we get
\begin{align*}
  &\varepsilon\left\Vert \partial_W\left(\left[\mathcal{R}_{-1}, u \cdot \nabla\right] \theta\right)
  \right\Vert _{L_t^1\big(B_{\infty, \infty}^{-1}\big)}\\
  & \leq C \varepsilon\,\left\Vert \partial_W\left(\left[\mathcal{R}_{-1}, u \cdot \nabla\right] \theta\right)
  \right\Vert _{L_t^1\big(B_{\infty, \infty}^{\gamma'-2}\big)}\leq  C  \varepsilon\, \Vert W\Vert _{L_t^\infty(L^\infty)} \left\Vert \left[\mathcal{R}_{-1}, u \cdot \nabla\right]
  \theta\right\Vert _{L_t^1 \big(B_{\infty, \infty}^{\gamma^\prime -1}\big)} \nonumber \\
  &\leq  C \varepsilon\, e^{\exp(Ct)}  \Big(\Vert \nabla u\Vert _{L^1_t(L^\infty)}\Vert \theta\Vert _{L^\infty_t\big(B^{\gamma'-2}_{\infty,\infty}\cap L^2\big)}
  + \Vert \nabla u\Vert _{L^1_t(L^q)}\Vert \theta\Vert _{L^\infty_t(L^2)} \Big) \leq C \varepsilon\, e^{\exp(Ct)}.
\end{align*}
Collecting the above estimates and \eqref{es:nabu-Gam2}, \eqref{eq:W-W1inf-bdd}, we find that for every $\gamma\in (0,1)$,
\begin{align*}
  & \left\Vert \partial_W \Gamma\right\Vert _{L_t^1\big(B_{\infty, 1}^{\gamma}\big)}  
  \leq \Vert \Delta_{-1} (\partial_W \Gamma)\Vert _{L_t^1(L^\infty)} + \sum_{j\in\mathbb{N}} 2^{j\gamma} 
  \Vert \Delta_j F^{(1)}\Vert _{L^1_t(L^\infty)}
  +\sum_{j\in\mathbb{N}} 2^{j\gamma} \Vert \Delta_j F^{(2)}\Vert _{L^1_t(L^\infty)} \nonumber \\
  &\leq C \Vert W\, \Gamma\Vert _{L_t^1(L^q)} 
  + C \sup_{j\in\mathbb{N}} 2^j \Vert \Delta_j F^{(1)}\Vert _{L^1_t(L^\infty)} 
  + C\Vert F^{(2)}\Vert _{L^1_t(B^\gamma_{\infty,1})}
  \nonumber \\
  & \leq C e^{\exp(Ct)}.
\end{align*}
\end{proof}

Based on Lemma \ref{lem:pGamma}, we obtain the following proposition: by suitably weakening the assumptions on the initial data, 
one can recover the same conclusions as \cite{LXY25}, 
and moreover, the persistence of the boundary regularity in case $\alpha=1$ is uniform with respect to $\mathrm{Pr}$.
\begin{proposition}\label{prop:pGamma-low}
Assume that $\theta_0(x) =\bar{\theta}_0(x) \mathbf{1}_{D_0}(x)$, where
$\partial D_0\in C^{2,\gamma}$ and $\bar{\theta}_0$ satisfies \eqref{eq:bar-the-cd3}-\eqref{eq:bar-the-cd4},
and $u_0\in W^{1,q}(\mathbb{R}^2)$, $q>\frac{2}{2\alpha-1}$, $\partial_{W_0}u_0\in \dot W^{1,p}(\mathbb{R}^2)$, $p>\frac{2}{2\alpha-\gamma}$.
Consider \{\textrm{case\, I}: $\gamma\in(0,2\alpha-1]$, $\alpha\in(\tfrac{1}{2},1)$\},
and \{\textrm{Case II}: $\gamma\in(0,1)$, $\alpha=1$\}, 
then we have 
\begin{align}\label{nab-W-Cgam-es}
  \Vert W\Vert _{L^\infty_T(C^{1,\gamma})}  
  + \Vert \partial_W \Gamma\Vert _{L^1_T(B^\gamma_{\infty,1})}  
  + \big\Vert \big(\nabla u, \partial_W \nabla u, \nabla (\partial_W u)\big)\big\Vert _{L^1_T (B^\gamma_{\infty,\infty})} 
  \leq C e^{\exp(CT)},
\end{align}
with $C>0$ depending on the initial data but independent of $\varepsilon=\frac{1}{\mathrm{Pr}}\in (0,1]$.
Consequently, it implies the uniform-in-$\varepsilon$ persistence of $C^{2,\gamma}$-boundary regularity of $\partial D(t)$
in case I and case II.
\end{proposition}


\begin{proof}[Proof of Proposition \ref{prop:pGamma-low}]
We deal with these two cases, respectively, in the sequel.
\vskip1mm
\textbf{Case I: $\gamma\in (0,2\alpha-1]$, $\alpha\in (\frac{1}{2},1)$.}
\vskip1mm

From the equation of $\nabla W$ in \eqref{eq:nabW}, using \eqref{eq:T-BesEs0} and the relation \eqref{eq:u-Gam-rela}, 
we find
\begin{equation}\label{eq:nabW-Cgam-1}
  \Vert \nabla W(t)\Vert _{C^\gamma} \leq \Vert \nabla W_0\Vert _{C^\gamma} 
  + C \int_0^t \Vert \partial_W \nabla u(\tau)\Vert _{C^\gamma} \dd \tau 
  + C \int_0^t \Vert \nabla u(\tau)\Vert _{C^\gamma} \Vert \nabla W(\tau)\Vert _{C^\gamma} \dd \tau, 
\end{equation}
and
\begin{equation}\label{eq:parW-nab-u-L1Cgam}
  \int_0^t \Vert \partial_W \nabla u(\tau)\Vert _{C^\gamma} \dd \tau 
  \leq  \int_0^t \Vert \partial_W \nabla^2 \Lambda^{-2} \Gamma(\tau)\Vert _{C^\gamma} \dd \tau 
  +  \int_0^t \Vert \partial_W \nabla^3 \Lambda^{-2-2\alpha} \theta(\tau)\Vert _{C^\gamma} \dd \tau .
\end{equation}
According to \cite[pp. 1416-1417]{LXY25} and \eqref{es:nabu-Gam2}, \eqref{eq:W-W1inf-bdd}, we have
\begin{equation}\label{eq:parW-Gam-L1Cgam1}
\begin{split}
  \Vert \partial_W(\nabla^2 \Lambda^{-2} \Gamma)\Vert _{L^1_t(C^\gamma)}
  & \leq C e^{\exp(Ct)} + 
  C \Vert \partial_W \Gamma\Vert _{L^1_t(C^\gamma)} 
  + C \int_0^t \Vert W(\tau)\Vert _{W^{1,\infty}} \Vert \Gamma(\tau)\Vert _{C^\gamma}\dd \tau \\
  & \leq C e^{\exp(Ct)} +  C \Vert \partial_W \Gamma\Vert _{L^1_t(C^\gamma)} ,
\end{split}
\end{equation}
and
\begin{align}\label{eq:parW-the-L1Cgam1}
  \Vert \partial_W(\nabla^3 \Lambda^{-2-2\alpha}\theta)\Vert _{L^1_t(C^\gamma)}
  & \leq C e^{\exp(Ct)} + C \Vert \partial_W\theta\Vert _{L^1_t(C^{\gamma+1-2\alpha})} + C \int_0^t \Vert W(\tau)\Vert _{W^{1,\infty}}
  \Vert \theta(\tau)\Vert _{L^2\cap L^\infty} \dd \tau \nonumber \\
  & \leq C e^{\exp(Ct)} + C \int_0^t \Vert W(\tau)\Vert _{W^{1,\infty}} \dd \tau \leq C e^{\exp(Ct)} .
\end{align}
In Lemma \ref{lem:pGamma}, we obtain the crucial uniform estimate $\Vert \partial_W \Gamma\Vert _{L^1_t(B^\gamma_{\infty,1})} \leq C e^{\exp(Ct)}$. 
Hence, collecting the above estimates, and using Gronwall's inequality and \eqref{es:nabu-Gam2} gives that
\begin{equation}\label{nab-W-Cgam-es1}
\begin{split}
  \Vert W\Vert _{L^\infty_T(C^{1,\gamma})} + \Vert \partial_W \Gamma\Vert _{L^1_T(B^\gamma_{\infty,1})}  
  + \Vert (\nabla u, \partial_W \nabla u)\Vert _{L^1_T (C^\gamma)} 
  & \leq C e^{\exp (CT)} e^{C \Vert \nabla u\Vert _{L^1_T(C^\gamma)}} \\
  & \leq C e^{\exp(CT)}.
\end{split}
\end{equation}
Thanks to $\nabla (\partial_Wu)=\partial_W\nabla u + \nabla W\cdot\nabla u$, we also get 
\begin{equation}\label{nabpWu}
\begin{aligned}
  \Vert \nabla (\partial_W u)\Vert _{L^1_T (C^\gamma)} 
  & \leq \Vert \partial_W\nabla u \Vert _{L^1_T (C^\gamma)}
  + \Vert \nabla W\cdot\nabla u\Vert _{L^1_T (C^\gamma)} \\
  & \leq C\Vert \partial_W\nabla u \Vert _{L^1_T (C^\gamma)} 
  + C\Vert \nabla W\Vert _{L^\infty_T (C^\gamma)}\Vert \nabla u\Vert _{L^1_T (C^\gamma)} 
  \leq Ce^{\exp (CT)}.
\end{aligned}
\end{equation}
\vskip1mm
\textbf{Case II: $\gamma\in (0,1)$, $\alpha=1$.}
\vskip1mm

We start with \eqref{eq:nabW-Cgam-1} and \eqref{eq:parW-nab-u-L1Cgam} with $\alpha=1$;
and for the right-hand side of \eqref{eq:parW-nab-u-L1Cgam}, both estimates \eqref{eq:parW-Gam-L1Cgam1}
and \eqref{eq:parW-the-L1Cgam1} with $\alpha=1$ also hold. 
The crucial estimate of $L^1_t(B^\gamma_{\infty,1})$-norm of 
$\partial_W \Gamma$ (which controls $\Vert \partial_W\Gamma\Vert _{L^1_t(C^\gamma)}$) is addressed in Lemma \ref{lem:pGamma}.
Hence, these estimates together with Gronwall's inequality and 
\eqref{es:nabu-Gam2} guarantee that for every $\gamma\in (0,1)$,
\begin{align*}
  \Vert W\Vert _{L_T^{\infty}\left(C^{1,\gamma}\right)}
  +\left\Vert \partial_W \Gamma\right\Vert _{L_T^1\left(B^\gamma_{\infty,1}\right)}
  +\left\Vert \partial_W \nabla u\right\Vert _{L_T^1\left(C^\gamma\right)} \leq C e^{\exp(CT)},
  \end{align*}
where $C>0$ is independent of $\varepsilon$. 
Combining this inequality with \eqref{nabpWu} gives \eqref{nab-W-Cgam-es} in case II.
\end{proof}

For the remaining two cases, we need stronger regularity
estimates for $\partial_W\Gamma$.
To this end, we establish the following lemma.

\begin{lemma}\label{lem:pGamma-high}
\begin{enumerate}[(1)]
\item
Consider $\{\textrm{case III}: \gamma\in(2\alpha-1,1), \,\alpha\in(\frac{1}{2},1)\}$.
Assume that
$\theta_0=\bar\theta_0\mathds{1}_{D_0}$,
$\partial D_0\in C^{2,\gamma}$,
$\bar\theta_0\in
C^{1,\gamma+1-2\alpha}(\overline{D_0})$,
and
\[
  u_0\in W^{1,q}(\mathbb R^2),
  \qquad q>\tfrac{2}{2\alpha-1},
  \qquad
  \partial_{W_0}u_0\in\dot W^{1,p}(\mathbb R^2),
  \qquad p>\tfrac{2}{2\alpha-\gamma}.
\]
Then there exists $C>0$, independent of $\varepsilon$,
such that, for every $T>0$,
\begin{equation}\label{es:pGamma-high1}
  \Vert \partial_W\Gamma\Vert _{L^1_T(B^\gamma_{\infty,1})}
  \leq Ce^{\exp(CT)}.
\end{equation}

\item
Consider $\{\textrm{case IV}:\gamma=1, \,\alpha\in(\frac{1}{2},1]\}$.
Assume that
$\theta_0=\bar\theta_0\mathds{1}_{D_0}$,
$\partial D_0\in C^{2,1}$,
$\bar\theta_0\in C^{1,\mu}(\overline{D_0})$
for some $\mu\in(2-2\alpha,1)$, and
\[
  u_0\in W^{1,q}(\mathbb R^2),
  \qquad q>\tfrac{2}{2\alpha-1},
  \qquad
  \partial_{W_0}u_0\in\dot W^{1,p}(\mathbb R^2),
  \qquad p>\tfrac{2}{2\alpha-1}.
\]
Then, for every
\[
  1<\gamma'
  <\min\left\{2\alpha-\tfrac2p,\,2\alpha-\tfrac2q\right\},
\]
there exists $C>0$, which may depend on $\gamma'$
but is independent of $\varepsilon$, such that,
for every $T>0$,
\begin{equation}\label{es:pGamma-high2}
  \Vert \partial_W\Gamma\Vert _{L^1_T(B^{\gamma'}_{\infty,1})}
  \leq Ce^{\exp(CT)}.
\end{equation}
\end{enumerate}
\end{lemma}

\begin{proof}[Proof of Lemma~\ref{lem:pGamma-high}]
The assumptions in both cases allow us to use
uniform estimates
\eqref{eq:W-W1inf-bdd}-\eqref{eq:nab-paw-u-L1Linf}.

\medskip
\noindent
(1) We first consider case III. 
According to equation \eqref{eq:parW} and Lemma \ref{lem:TD-sm}, 
we infer that
\begin{equation}\label{eq:paw-Gam-L1Bes1}
\begin{split}
  \left\Vert \partial_W \Gamma\right\Vert _{L_T^1\big( B_{\infty, 1}^{\gamma}\big)}
  \leq  & \, \big\Vert  \big( 2^{j\gamma} \Vert \Delta_j (\partial_W \Gamma)\Vert _{L^1_T(L^\infty)}
  \big)_{j\in\mathbb{N}} \big\Vert _{\ell^1} +
  C \Vert \Delta_{-1} (\partial_W \Gamma)\Vert _{L^1_T(L^\infty)}  \\
  \leq & \,C\, \varepsilon\, e^{C\int_0^T\Vert \nabla u(t)\Vert _{L^\infty}\dd t}
  \bigg(\left\Vert \partial_{W_0} \Gamma_0\right\Vert _{B_{\infty, 1}^{\gamma-2\alpha}}
  + \frac{1}{\varepsilon} \int_0^T\left\Vert \left[\Lambda^{2\alpha}, W \cdot \nabla\right] 
  \Gamma\right\Vert _{ B_{\infty, 1}^{\gamma-2\alpha}}\dd t  \\
  & +\int_0^T\left\Vert \partial_W\left(\left[\mathcal{R}_{1-2\alpha}, u \cdot \nabla\right] 
  \theta\right)\right\Vert _{B_{\infty, 1}^{\gamma-2\alpha}}
  \dd t\bigg) + C\int_0^T\Vert W(t)\Vert _{L^\infty}\Vert \Gamma(t)\Vert _{L^\infty}\dd t.
\end{split}
\end{equation}

To estimate the initial data, recall the identities
\begin{equation}\label{eq:idd-rela1}
\begin{aligned}
  \Gamma_0
  =\omega_0-\mathcal R_{1-2\alpha}\theta_0,\qquad
  \partial_{W_0}\partial_j u_0
  =\partial_j(\partial_{W_0}u_0)
    -\partial_jW_0\cdot\nabla u_0.
\end{aligned}
\end{equation}
Using the embeddings
$L^p,\ L^q\hookrightarrow  B^{\gamma-2\alpha}_{\infty,1}$,
$B^{\gamma-1}_{\infty,\infty}\hookrightarrow B^{\gamma-2\alpha}_{\infty,1}$,
and \eqref{eq:paWmD-es} with
$\sigma=1-2\alpha$ and $s=\gamma-1$ ($s,s+\sigma\in(-1,1)$), we obtain
\begin{align*}
  \Vert \partial_{W_0}\Gamma_0
    \Vert _{B^{\gamma-2\alpha}_{\infty,1}}
  &\leq C\Bigl(
    \Vert \partial_{W_0}u_0\Vert _{\dot W^{1,p}}
    +\Vert \nabla W_0\Vert _{L^\infty}\Vert \nabla u_0\Vert _{L^q}
  + \Vert \partial_{W_0}\mathcal R_{1-2\alpha}\theta_0
      \Vert _{B^{\gamma-1}_{\infty,\infty}}
  \Bigr)
  \\
  &\leq C\Bigl(
    \Vert \partial_{W_0}u_0\Vert _{\dot W^{1,p}}
    +\Vert W_0\Vert _{W^{1,\infty}}
      \bigl(\Vert \nabla u_0\Vert _{L^q}
            +\Vert \theta_0\Vert _{L^\infty}\bigr)
      +  \Vert \partial_{W_0}\theta_0
      \Vert _{B^{\gamma-2\alpha}_{\infty,\infty}}
  \Bigr)
  \leq C.
\end{align*}
Here, the last term of $\partial_{W_0} \theta_0$
is controlled by Lemma~\ref{lem:pWk-theta}.

The striated commutator estimate \eqref{es:prod-es4},
together with \eqref{es:nabu-Gam2} and
\eqref{eq:W-W1inf-bdd}-\eqref{eq:nab-paw-u-L1Linf},
gives
\begin{align}\label{eq:paw-cm-es1}
  & \left\Vert \partial_W\left(\left[\mathcal{R}_{1-2\alpha}, u \cdot \nabla\right] \theta\right)
  \right\Vert _{L^1_T(B_{\infty, 1}^{\gamma-2\alpha})}
  \leq C \left\Vert \partial_W\left(\left[\mathcal{R}_{1-2\alpha}, u \cdot \nabla\right] \theta\right)
  \right\Vert _{L^1_T(L^\infty)} \nonumber \\
  & \leq C \big(\Vert \nabla W\Vert _{L^\infty_T(L^\infty)} \Vert \nabla u\Vert _{L^1_T(L^\infty)} 
  + \Vert \nabla \partial_W u\Vert _{L^1_T(L^\infty)} \big) 
  \Vert \theta\Vert _{L^\infty_T(L^1\cap L^\infty)} 
  + C \Vert \nabla u\Vert _{L^1_T(L^\infty)} 
  \Vert \partial_W \theta\Vert _{L^\infty_T(L^1\cap L^\infty)} \nonumber \\
  & \leq C e^{\exp(CT)}\big( 1 + \Vert \partial_W \theta\Vert _{L^\infty_T(L^1\cap L^\infty)}
  \big).
\end{align}

At the level of smooth approximations,
$[\partial_W,\partial_t+u\cdot\nabla]=0$ implies
\begin{equation}\label{eq:parWj-the}
  \partial_t(\partial_W^j\theta)
  +u\cdot\nabla(\partial_W^j\theta)=0,
  \qquad
  \partial_W^j\theta|_{t=0}
  =\partial_{W_0}^j\theta_0,
  \qquad j\geq1.
\end{equation}
Hence \eqref{pWj-theta} and conservation of the
temperature striated norms yield, for every $r\in[1,\infty]$,
\begin{equation}\label{eq:paw-the-Lp}
  \Vert \partial_W\theta\Vert _{L^\infty_T(L^r)}
  =\Vert \partial_{W_0}\theta_0\Vert _{L^r}
  \leq C.
\end{equation}

Finally, applying \eqref{eq:commEs} with
$\sigma=2\alpha$ and $s=\gamma-2\alpha$, we obtain
\begin{align*}
  \Vert [\Lambda^{2\alpha},W\cdot\nabla]\Gamma
    \Vert _{L^1_T(B^{\gamma-2\alpha}_{\infty,1})}
  &\leq
  C\Vert W\Vert _{L^\infty_T(W^{1,\infty})}
    \Vert \Gamma\Vert _{L^1_T(B^\gamma_{\infty,1})}
  \leq Ce^{\exp(CT)}.
\end{align*}
Substituting these estimates into
\eqref{eq:paw-Gam-L1Bes1} and using
\eqref{es:nabu-Gam2} and \eqref{eq:W-Linf-es-2}
proves \eqref{es:pGamma-high1}.

\medskip
\noindent
(2) We now consider case IV.
Estimate \eqref{eq:paw-Gam-L1Bes1} remains valid
with $\gamma'$ in place of $\gamma$, since
$-1<\gamma'-2\alpha<0$.

We first estimate the initial data.
Suppose that $\alpha<1$.
Using \eqref{eq:idd-rela1},
\eqref{eq:paWmD-es} with $s=0$ and $\sigma=1-2\alpha$,
and Lemma~\ref{lem:pWk-theta}, we obtain
\begin{align*}
  \Vert \partial_{W_0}\Gamma_0
    \Vert _{B^{\gamma'-2\alpha}_{\infty,1}}
  &\leq C\Bigl(
    \Vert \partial_{W_0}u_0\Vert _{\dot W^{1,p}}
    +\Vert \nabla W_0\Vert _{L^\infty}\Vert \nabla u_0\Vert _{L^q}
  +  \Vert \partial_{W_0}\mathcal R_{1-2\alpha}\theta_0
      \Vert _{B^0_{\infty,\infty}}
  \Bigr)
  \\
  &\leq C\Bigl(
    \Vert \partial_{W_0}u_0\Vert _{\dot W^{1,p}}
    +\Vert W_0\Vert _{W^{1,\infty}}
      \bigl(\Vert \nabla u_0\Vert _{L^q}
            +\Vert \theta_0\Vert _{L^\infty}\bigr)
    +  \Vert \partial_{W_0}\theta_0
      \Vert _{B^{1-2\alpha}_{\infty,\infty}}
  \Bigr)
  \leq C.
\end{align*}
Here we have used
$B^0_{\infty,\infty},\ L^p,\ L^q \hookrightarrow B^{\gamma'-2\alpha}_{\infty,1}$.

At $\alpha=1$, we instead use the
Calder\'on-Zygmund estimate for $\nabla\mathcal R_{-1}$.
Since $\gamma'<2-2/q$, we have
\begin{align*}
  \Vert \partial_{W_0}\mathcal R_{-1}\theta_0
    \Vert _{B^{\gamma'-2}_{\infty,1}}
  \leq
  C\Vert \partial_{W_0}\mathcal R_{-1}\theta_0\Vert _{L^q}
  \leq
  C_q\Vert W_0\Vert _{L^\infty}\Vert \theta_0\Vert _{L^q}.
\end{align*}
Together with \eqref{eq:idd-rela1} and the preceding
embeddings, this again gives
\[
  \Vert \partial_{W_0}\Gamma_0
    \Vert _{B^{\gamma'-2}_{\infty,1}}
  \leq C.
\]

The argument in
\eqref{eq:paw-cm-es1}-\eqref{eq:paw-the-Lp}
also applies in case IV, including $\alpha=1$.
Since $\gamma'-2\alpha<0$, it gives
\begin{align*}
  \Vert \partial_W(
    [\mathcal R_{1-2\alpha},u\cdot\nabla]\theta)
    \Vert _{L^1_T(B^{\gamma'-2\alpha}_{\infty,1})}
  &\leq
  C\Vert \partial_W(
    [\mathcal R_{1-2\alpha},u\cdot\nabla]\theta)
    \Vert _{L^1_T(L^\infty)}
  \leq Ce^{\exp(CT)}.
\end{align*}

To estimate the dissipative commutator, choose
$\gamma'-1<\beta<2\alpha-1$.
The assumptions of Proposition~\ref{prop:pGamma-low}
hold with boundary exponent $\beta$.
Thus \eqref{nab-W-Cgam-es} and the embedding
$C^{1,\beta}\hookrightarrow B^{\gamma'}_{\infty,1}$
yield
\[
  \Vert W\Vert _{L^\infty_T(B^{\gamma'}_{\infty,1})}
  \leq Ce^{\exp(CT)}.
\]
Applying \eqref{eq:commEs2} with
$s=\gamma'-2\alpha$ and $\sigma=2\alpha$, we obtain
\begin{align*}
  &\Vert [\Lambda^{2\alpha},W\cdot\nabla]\Gamma
    \Vert _{L^1_T(B^{\gamma'-2\alpha}_{\infty,1})}
  \\
  &\quad\leq
  C\Vert W\Vert _{L^\infty_T(W^{1,\infty})}
    \Vert \Gamma\Vert _{L^1_T(B^{\gamma'}_{\infty,1})}
  +C\Vert W\Vert _{L^\infty_T(B^{\gamma'}_{\infty,1})}
    \Vert \nabla\Gamma\Vert _{L^1_T(L^\infty)}
  \leq Ce^{\exp(CT)}.
\end{align*}
Here the required bounds for $\Gamma$ follow from
\eqref{es:nabu-Gam2} and $\gamma'<2\alpha-2/q$.
Substituting these estimates into
\eqref{eq:paw-Gam-L1Bes1}, with $\gamma'$ in place
of $\gamma$, proves \eqref{es:pGamma-high2}.
\end{proof}

Next, we present the proof of the main result in this section.
\begin{proof}[Proof of Theorem \ref{thm:main2}]
As in the above section, we also divide the proof into four cases.
Since Proposition \ref{prop:pGamma-low} has addressed case I and case II, 
we only need to consider the remaining two cases.
\vskip1mm

$\bullet$ \textbf{Case III: $\gamma\in (2\alpha-1,1)$, $\alpha\in(\frac{1}{2},1)$.} 
\vskip1mm

To establish the uniform-in-$\varepsilon$ $C^{2,\gamma}$-regularity of $\p D(t)$, according to \eqref{eq:boundary-identities} and \eqref{es:inf-dz}, 
it suffices to show that $\partial_W W \in L^\infty([0,T], C^\gamma(\R^2))$ uniformly in $\varepsilon$. 

Taking advantage of $W$-equation \eqref{eq:W-eq} and the fact that $[\p _W, \p _t+u\cdot\nabla]=0$, we have 
\begin{equation}\label{pWW}
\begin{split}
  \partial_t (\p_W W) + u\cdot\nabla (\p_W W) =  \partial_W^2 u , \quad  (\p_W W)|_{t=0}  = \p_{W_0} W_0.
\end{split}
\end{equation}
Estimate \eqref{eq:T-BesEs0} in Lemma~\ref{lem:TE-prog-es} directly implies
\begin{equation}
\begin{aligned}
  \Vert \p_W  W\Vert _{C^\gamma} \leq & C \left\Vert \p_{W_0}  W_0\right\Vert _{C^\gamma}
  + C \int_0^t\left\Vert \partial_W^2 u(\tau) \right\Vert _{C^\gamma} \mathrm{d} \tau
  + C \int_0^t \Vert \nabla u(\tau)\Vert _{L^\infty} \Vert \p_W  W(\tau)\Vert _{C^\gamma}  \mathrm{d} \tau .
\end{aligned}
\end{equation}
The main task below is to estimate the $L^1_t(C^\gamma)$-norm of $\partial_W^2 u$. Using the identity 
\eqref{eq:u-Gam-rela} gives
\begin{align}\label{eq:parWu-Cgam-0}
  \Vert \partial_W^2 u\Vert _{L^1_t(C^\gamma)} 
  \leq \big\Vert \partial_W^2  \nabla^{\perp} \Lambda^{-2} \Gamma\big\Vert _{L^1_t(C^\gamma)}
  + \big\Vert \partial_W^2 \nabla^2 \Lambda^{-2-2\alpha} \theta \big\Vert _{L^1_t(C^\gamma)}.
\end{align}
For the first term on the right-hand side of \eqref{eq:parWu-Cgam-0}, due to that 
\begin{align*}
  \partial_W^2  \nabla^{\perp} \Lambda^{-2} \Gamma
  = \partial_W(W\cdot\nabla  \nabla^{\perp} \Lambda^{-2} \Gamma)
  = (\partial_W W) \cdot \nabla\nabla^{\perp} \Lambda^{-2} \Gamma 
  + W\cdot\partial_W(\nabla \nabla^{\perp} \Lambda^{-2} \Gamma),
\end{align*}
we get
\begin{align*}
  \Vert \partial_W^2  \nabla^{\perp} \Lambda^{-2} \Gamma\Vert _{L^1_t(C^\gamma)}
  \leq \int_0^t\Vert \partial_W W (\tau)\Vert _{C^\gamma}\Vert \nabla^2 \Lambda^{-2} \Gamma(\tau)\Vert _{C^\gamma} \dd \tau
  + \Vert W\Vert _{L^\infty_t(C^\gamma)} \Vert \partial_W\nabla^2\Lambda^{-2} \Gamma\Vert _{L^1_t(C^\gamma)}.
\end{align*}
Observing that referring to \eqref{eq:paWmD-es}, \eqref{es:nabu-Gam2}, \eqref{eq:W-W1inf-bdd}, 
and \eqref{es:pGamma-high1} in Proposition \ref{lem:pGamma-high}, one can get
\begin{align*}
  \Vert \partial_W\nabla^2 \Lambda^{-2} \Gamma\Vert _{L^1_t(C^\gamma)} 
  & \leq C\left( \left\Vert  \p_W \Gamma\right\Vert _{L^1_t(C^\gamma)} + \Vert W\Vert _{L^\infty_t(W^{1,\infty})} 
  \left\Vert \Gamma\right\Vert _{L^1_t(C^\gamma)} \right) \leq C e^{\exp(Ct)}.
\end{align*}
In combination with \eqref{es:mDGam}, we have 
\begin{align}\label{eq:parW2-Gam}
  \Vert \partial_W^2  \nabla^{\perp} \Lambda^{-2} \Gamma\Vert _{L^1_t(C^\gamma)}
  \leq C \int_0^t \Vert \partial_W W(\tau) \Vert _{C^\gamma} \Vert \Gamma(\tau)\Vert _{L^q\cap C^\gamma} \dd \tau
  + C e^{\exp(Ct)}.
\end{align}
In light of the splitting
\begin{equation*}
\begin{split}
  \p_W^2\big(\nabla^2\Lambda^{-2-2\alpha} \theta \big) 
  = \nabla^2\Lambda^{-2-2\alpha} \p_W^2\theta
  -\p_W\big([\nabla^2\Lambda^{-2-2\alpha}, W\cdot\nabla]\theta\big)
  - [\nabla^2\Lambda^{-2-2\alpha}, W\cdot\nabla]\p_W\theta,
\end{split}
\end{equation*} 
it follows that 
\begin{align*}
  \big\Vert \partial^2_W \nabla^2 \Lambda^{-2-2\alpha} \theta(t) \big\Vert _{C^\gamma} 
  & \leq \Vert \nabla^2\Lambda^{-2-2\alpha} \p_W^2\theta(t)\Vert _{C^\gamma} 
  + \big\Vert \p_W\big([\nabla^2 \Lambda^{-2-2\alpha}, W\cdot\nabla]\theta\big)(t)\big\Vert _{C^\gamma}\\
  & \quad + \Vert [\nabla^2 \Lambda^{-2-2\alpha}, W\cdot\nabla]\p_W\theta (t)\Vert _{C^\gamma}.
\end{align*}
The high-low frequency decomposition and Bernstein's inequality ensure that
\begin{align*}
  \Vert \nabla^2 \Lambda^{-2-2\alpha} \p^2_W\theta(t)\Vert _{C^\gamma}
  & \leq C \Vert \Delta_{-1} \nabla^2 \Lambda^{-2-2\alpha} \p^2_W\theta (t)\Vert _{L^\infty} 
  + \sup\limits_{j\in\NN}2^{j\gamma}\Vert \Delta_j \nabla^2\Lambda^{-2-2\alpha} \p^2_W\theta (t)\Vert _{L^\infty}\\
  &\leq C \Vert \Delta_{-1} \nabla^2 \Lambda^{-2-2\alpha}\mathrm{div}(W \p_W\theta) (t)\Vert _{L^\infty} 
  + C \sup\limits_{j\in\NN}2^{j(\gamma-2\alpha)}\Vert \Delta_j \p^2_W\theta(t)\Vert _{L^\infty}  \\
  & \leq C \Vert W(t)\Vert _{L^\infty}\Vert  \p_W\theta(t)\Vert _{L^1\cap L^\infty} + C \Vert \p^2_W\theta(t)\Vert _{C^{\gamma-2\alpha}} .
\end{align*}
Since $\partial_W^j\theta$ for every $j\in\mathbb{N}$ solves the transport equation \eqref{eq:parWj-the},
we use \eqref{eq:T-BesEs}, \eqref{es:nabu-Gam2} and Lemma \ref{lem:pWk-theta} to deduce that for every $\gamma\in(2\alpha-1,1)$,
\begin{equation*}
  \Vert \partial_W^2\theta\Vert _{L_t^\infty(C^{\gamma-2\alpha})}
  \leq Ce^{C\Vert \nabla u\Vert _{L_t^1(L^\infty)}}\Vert \partial_{W_0}^2\theta_0\Vert _{C^{\gamma-2\alpha}}\leq Ce^{\exp (Ct)},
\end{equation*}
and 
\begin{equation}\label{eq:the-uni-bdd2}
\begin{aligned}
  \Vert (\theta,\,\partial_W\theta)\Vert _{L_t^\infty(L^1\cap L^\infty)} 
  = \Vert (\theta_0,\,\partial_{W_0}\theta_0)\Vert _{L^1\cap L^\infty} \leq C.
\end{aligned}
\end{equation}
Thanks to commutator estimates \eqref{es:prod-es6} (with $\sigma=-2\alpha\in (-2,-1)$) 
and \eqref{es:prod-es7} in Lemma \ref{lem:pW-Cgam}, we obtain 
\begin{align*}
  \Vert [\nabla^2\Lambda^{-2-2\alpha}, W\cdot\nabla]\p_W\theta(t)\Vert _{C^\gamma}
  \leq C\Vert W(t)\Vert _{C^\gamma}\Vert \p_W\theta(t)\Vert _{L^1\cap L^\infty},
\end{align*}
and
\begin{align*}
  & \big\Vert \p_W\big([\nabla^2 \Lambda^{-2-2\alpha}, W\cdot\nabla]\theta\big)(t)\big\Vert _{C^\gamma} \\
  & \leq C\left(\Vert \nabla W(t)\Vert _{L^\infty}^2 + \Vert \partial_W W(t)\Vert _{C^\gamma}\right)\Vert \theta(t)\Vert _{L^1\cap L^\infty} 
  + C \Vert \nabla W(t)\Vert _{L^\infty}\Vert \partial_W\theta(t)\Vert _{L^1\cap L^\infty}.
\end{align*}
Hence, gathering the above estimates with the uniform bounds \eqref{eq:W-W1inf-bdd} leads to 
\begin{align}\label{eq:parW2-the}
  \big\Vert \partial_W^2 \nabla^{\perp}\p_1\Lambda^{-2-2\alpha} \theta \big\Vert _{L^1_t(C^\gamma)} 
  \leq C e^{\exp(Ct)}  + C \int_0^t \Vert \partial_W W(\tau)\Vert _{C^\gamma}\dd \tau.
\end{align}
Substituting these estimates \eqref{eq:parW2-Gam} and \eqref{eq:parW2-the} into \eqref{eq:parWu-Cgam-0} gives
\begin{align*}
  \Vert  \p_WW(t)\Vert _{C^\gamma} \leq C e^{\exp(Ct)}
  + C \int_0^t \big(\Vert \nabla u(\tau)\Vert _{L^\infty} 
  + \Vert \Gamma(\tau)\Vert _{L^q\cap C^\gamma}+ 1\big) \Vert  \p_W W(\tau)\Vert _{C^\gamma} \dd\tau .
\end{align*}
Taking advantage of Gronwall's inequality and \eqref{es:nabu-Gam2}, we conclude that
\begin{align}\label{es:pWW-Cgam}
  \Vert  \p_W W\Vert _{L^\infty_T(C^\gamma)}\leq C e^{\exp (CT)}.
\end{align}
\vskip1mm
$\bullet$ \textbf{Case IV: $\gamma=1$, $\alpha\in(\frac{1}{2},1]$.} 
\vskip1mm

In view of \eqref{eq:boundary-identities}, uniform
$C^{2,1}$ boundary persistence reduces to a uniform bound
for $\partial_W^2W$ in $L^\infty([0,T],L^\infty)$.
We prove the stronger result 
$\p_W W\in L^\infty([0,T], W^{1,\infty}(\R^2))$ uniformly in 
$\varepsilon$.

In light of \eqref{eq:nabW} and the fact that $[\p _W, \p _t+u\cdot\nabla]=0,$ one can get 
\begin{equation}\label{pWnaW}
\left\{
\begin{aligned}
  \partial_t (\p_W \nabla W) + u\cdot\nabla (\p_W\nabla W) 
  & =  \partial_W^2 \nabla u + \p_W(\nabla W \cdot \nabla u) - \p_W(\nabla u \cdot \nabla W), \\
  \quad (\p_W\nabla W)|_{t=0} & = \p_{W_0}\nabla W_0,
\end{aligned}
\right.
\end{equation}
Taking advantage of the maximum principle and the uniform bound \eqref{eq:W-W1inf-bdd} gives
\begin{equation}\label{eq:pWnabW-L}
\begin{aligned}
  \Vert \p_W \nabla W(t)\Vert _{L^\infty} & \leq C \left\Vert \p_{W_0}\nabla W_0\right\Vert _{L^\infty}
  + C \int_0^t\left\Vert \partial_W^2 \nabla u(\tau) \right\Vert _{L^\infty} \mathrm{d} \tau \\
  & \quad + C \int_0^t \Vert \p_W\nabla W(\tau)\Vert _{L^\infty} \Vert  \nabla u(\tau)\Vert _{L^\infty} \dd \tau
  + \Vert \p_W\nabla u\Vert _{L^1_t(L^\infty)} \Vert \nabla W\Vert _{L^\infty_t(L^\infty)}  \\
  & \leq C e^{\exp(Ct)} +  C \int_0^t\left\Vert \partial_W^2 \nabla u(\tau) \right\Vert _{L^\infty} \mathrm{d} \tau
  + \int_0^t \Vert \p_W\nabla W(\tau)\Vert _{L^\infty} \Vert  \nabla u(\tau)\Vert _{L^\infty} \dd \tau.
\end{aligned}
\end{equation}
The main task remained is to estimate the $L^1_t(L^\infty)$-norm of $\partial_W^2 \nabla u$. Using the identity 
\eqref{eq:subcritical}, we find that
\begin{align}\label{eq:parWu-Cgam-IV}
  \Vert \partial_W^2 \nabla u\Vert _{L^1_t (L^\infty)} 
  \leq \big\Vert \partial_W^2 \nabla^2 \Lambda^{-2} \Gamma\big\Vert _{L^1_t(L^\infty)}
  + \big\Vert \partial_W^2 \nabla^3\Lambda^{-2-2\alpha} \theta \big\Vert _{L^1_t(L^\infty)}.
\end{align}
For the first term on the right-hand side of \eqref{eq:parWu-Cgam-IV}, 
using  \eqref{eq:W-W1inf-bdd} we have
\begin{align*}
  \Vert \partial_W^2\nabla^2\Lambda^{-2}\Gamma\Vert _{L^1_t(L^\infty)}
  & = \Vert \p_W( W \cdot \nabla^3 \Lambda^{-2}\Gamma)\Vert _{L^1_t(L^\infty)} \\
  & \leq \Vert W\Vert _{L^\infty_t(L^\infty)} 
  \Vert \p_W \nabla^3\Lambda^{-2} \Gamma\Vert _{L^1_t(L^\infty)} 
  + \Vert \p_W W\Vert _{L^\infty_t(L^\infty)} 
  \Vert \nabla^3 \Lambda^{-2} \Gamma\Vert _{L^1_t(L^\infty)} \\
  & \leq C e^{\exp(Ct)} \big(1+ \Vert \p_W \nabla^3\Lambda^{-2} \Gamma\Vert _{L^1_t(L^\infty)}  \big),
\end{align*}
where in the last line we also used the following estimate
\begin{align}\label{es:mDGamma}
  \Vert \nabla^3\Lambda^{-2} \Gamma\Vert _{L^1_t(L^\infty)} 
  \leq \Vert \nabla^3\Lambda^{-2} \Gamma\Vert _{L^1_t(B^0_{\infty,1})}
  \leq C\big(\Vert \Gamma\Vert _{L^1_t(L^q)}+ \Vert \Gamma\Vert _{L^1_t(B^1_{\infty,1})}\big).
\end{align}
Thanks to \eqref{eq:paWmD-es2} with $\sigma=1$ and $s=0$, one can see that
\begin{align*}
  & \left\Vert \p_W\nabla^3 \Lambda^{-2} \Gamma\right\Vert _{L^1_t(L^\infty)} 
  \leq \left\Vert \p_W\nabla^3\Lambda^{-2} \Gamma\right\Vert _{L^1_t(B^0_{\infty,1})} \\
  & \leq C\left\Vert  \p_W \Gamma\right\Vert _{L^1_t(B^1_{\infty,1})} 
  + C \Vert W\Vert _{L^\infty_t(W^{1,\infty})} \left\Vert \Gamma\right\Vert _{L^1_t(B^1_{\infty,1})} 
  + C \left\Vert   W\right\Vert _{L^\infty_t(B^1_{\infty,1})} \left\Vert \nabla\Gamma\right\Vert _{L^1_t(L^\infty)},
\end{align*}
which together with the uniform estimates \eqref{es:nabu-Gam2}, \eqref{nab-W-Cgam-es} 
and \eqref{es:pGamma-high2} implies
\begin{align*}
  \Vert \partial_W \nabla^3 \Lambda^{-2}\Gamma\Vert _{L^1_t(L^\infty)} 
  \leq C e^{\exp(Ct)}.
\end{align*}
For the second term on the right-hand side of \eqref{eq:parWu-Cgam-IV}, in light of the decomposition that  
\begin{align*}
  \p_W^2(\nabla^3\Lambda^{-2-2\alpha} \theta) 
  = \nabla^3\Lambda^{-2-2\alpha} \p_W^2\theta
  -\p_W \big([\nabla^3\Lambda^{-2-2\alpha}, W\cdot\nabla]\theta\big) 
  - [\nabla^3\Lambda^{-2-2\alpha}, W\cdot\nabla]\p_W\theta,
\end{align*} 
one can get 
\begin{align*}
  \big\Vert \partial^2_W \nabla^3\Lambda^{-2-2\alpha} \theta \big\Vert _{L^1_t(L^\infty)}
  & \leq \Vert \nabla^3\Lambda^{-2-2\alpha} \p_W^2\theta\Vert _{L_t^1(L^\infty)} 
  + \big\Vert \p_W\big([\nabla^3\Lambda^{-2-2\alpha}, W\cdot\nabla]\theta\big)\big\Vert _{L_t^1(L^\infty)}\\
  &\quad+\Vert [\nabla^3\Lambda^{-2-2\alpha}, W\cdot\nabla]\p_W\theta\Vert _{L_t^1(L^\infty)}.
\end{align*}
Recalling that $\mu\in (2-2\alpha,1)$ ($\mu$ appears in the assumption of $\bar{\theta}_0$), we have
\begin{align*}
  \Vert \nabla^3 \Lambda^{-2-2\alpha} \p^2_W & \theta \Vert _{L^1_t(L^\infty)} 
  \leq \Vert \nabla^3 \Lambda^{-2-2\alpha} \partial_W^2 
  \theta\Vert _{L^1_t(B^0_{\infty,1})} \\
  & \leq \Vert \Delta_{-1}\nabla^3\Lambda^{-2-2\alpha} \p^2_W\theta\Vert _{L^1_t(L^\infty)} 
  + \sum_{j\in\mathbb{N}} \Vert \Delta_j\nabla^3\Lambda^{-2-2\alpha} \p^2_W\theta\Vert _{L^1_t(L^\infty)} \\
  & \leq C\Vert \Delta_{-1}\nabla^3\Lambda^{-2-2\alpha} \div (W\p_W\theta)\Vert _{L^1_t(L^2)} 
  + C\sum\limits_{j\in\mathbb{N}} 2^{j(2-2\alpha - \mu)} 2^{j(\mu-1)} \Vert \Delta_j \p^2_W\theta\Vert _{L^1_t(L^\infty)} \\
  & \leq C\Vert W\Vert _{L^\infty_t(L^\infty)} \Vert \p_W\theta\Vert _{L^1_t(L^2)} 
  + C\Vert \p^2_W\theta\Vert _{L^1_t(B^{\mu-1}_{\infty,\infty})}.
\end{align*}
Due to that $\partial_W^j\theta$ for all $j\in\mathbb{N}$ solves the equation \eqref{eq:parWj-the}, 
we apply \eqref{eq:T-BesEs}, \eqref{es:nabu-Gam2} and Lemma \ref{lem:pWk-theta} 
to infer that \eqref{eq:the-uni-bdd2} holds and for every $\mu\in(2-2\alpha,1)$,
\begin{equation*}
  \Vert \partial_W^2\theta\Vert _{L_t^\infty(C^{\mu-1})}
  \leq Ce^{C\Vert \nabla u\Vert _{L_t^1(L^\infty)}}\Vert \partial_{W_0}^2\theta_0\Vert _{C^{\mu-1}}\leq Ce^{\exp (Ct)}.
\end{equation*}
Thanks to \eqref{es:prod-es4} and \eqref{es:prod-es5} with $\sigma =1-2\alpha \in (-2, 0)$ respectively, we deduce that
\begin{align*}
  \big\Vert \p_W\big([\nabla^3\Lambda^{-2-2\alpha}, W\cdot\nabla]\theta\big)\big\Vert _{L^1_t(L^\infty)} 
  &\leq C \Vert \nabla W\Vert _{L^\infty_t(L^\infty)}^2 \Vert \theta\Vert _{L^1_t(L^1\cap L^\infty)}   
  + C \Vert \nabla W\Vert _{L^\infty_t(L^\infty)} \Vert \partial_W\theta\Vert _{L^1_t(L^1\cap L^\infty)} \\
  & \quad + C \int_0^t \Vert \nabla (\partial_W W)(\tau)\Vert _{L^\infty} \Vert \theta(\tau)\Vert _{L^1\cap L^\infty} \dd \tau ,
\end{align*}
and
\begin{align*}
  \Vert [\nabla^3\Lambda^{-2-2\alpha}, W\cdot\nabla]\p_W\theta\Vert _{L_t^1(L^\infty)}
  \leq C\Vert \nabla W\Vert _{L^\infty_t(L^\infty)} \Vert \p_W\theta\Vert _{L^1_t(L^1\cap L^\infty)}.
\end{align*}
Consequently, in combination with \eqref{eq:W-W1inf-bdd}, the above estimates ensure that
\begin{align*}
  \big\Vert \partial_W^2 \nabla^3 \Lambda^{-2-2\alpha} \theta \big\Vert _{L^1_t(L^\infty)} \leq C e^{\exp(Ct)}
  + C\int_0^t\Vert \nabla (\partial_W W)(\tau)\Vert _{ L^\infty} \dd\tau.
\end{align*}

Substituting the above estimates of \eqref{eq:parWu-Cgam-IV} into \eqref{eq:pWnabW-L}, we obtain
\begin{align*}
  & \Vert  \p_W\nabla W(t)\Vert _{L^\infty} + \Vert \nabla \p_WW(t)\Vert _{L^\infty}
  \leq 2\Vert  \p_W\nabla W(t)\Vert _{L^\infty} + \Vert  \nabla W(t)\Vert ^2_{L^\infty} \\
  & \leq C e^{\exp (Ct)} 
  + C \int_0^t \big(\Vert \nabla u(\tau)\Vert _{L^\infty}+1\big) \big(\Vert  \p_W\nabla W\Vert _{L^\infty} 
  + \Vert \nabla \p_W W\Vert _{L^\infty}\big) \dd\tau .
\end{align*}
Gronwall's inequality and \eqref{es:nabu-Gam2} guarantee that 
\begin{align}\label{eq:es-k=2-IV}
  \Vert  \p_W(\nabla W)\Vert _{L^\infty_T(L^\infty)} + \Vert \nabla \p_W W\Vert _{L^\infty_T(L^\infty)} 
  + \Vert \partial_W^2 \nabla u\Vert _{L^1_T(L^\infty)}  \leq C e^{\exp(CT)} ,
\end{align}
which implies $\p^2_W W \in L^\infty_T(L^{\infty})$.

The preceding striated estimates are uniform with respect to
$\varepsilon\in(0,1]$. The approximation and compactness argument
carried out in Appendix~\ref{sec:B3} shows that these bounds are
inherited by the solution $(u,\theta)$ constructed in
Theorem~\ref{thm:main1}. In particular, when $0<\gamma<1$, the
associated tangent field $W$ and its striated derivative
$\partial_W W$ satisfy the required uniform $C^\gamma$ estimates.
At the endpoint $\gamma=1$, the corresponding $W^{1,\infty}$
estimates yield the $C^{2,1}=W^{3,\infty}$ boundary regularity.
Consequently, the $C^{2,\gamma}$-regularity of $\partial D(t)$ is
preserved on $[0,T]$, with a bound independent of $\varepsilon$.
This completes the proof of Theorem~\ref{thm:main2}.
\end{proof}



\section{Infinite Prandtl number limit}\label{sec:limit}

This section proves Theorem~\ref{thm:infinite-Prandtl}
on the infinite Prandtl number limit.

\subsection{A general convergence result}

We first establish a general convergence result
for initial temperatures uniformly bounded
in $L^1\cap L^\infty$, without assuming a patch structure.
At the endpoint $\alpha=1$, we additionally require
their supports to lie in a fixed ball.

\begin{proposition}\label{prop:weak-limit-R2}
Let $\alpha\in(\frac12,1]$, $T>0$, and $2<p<\infty$.
For each $\varepsilon=1/\mathrm{Pr}\in(0,1]$,
let $u_0^\varepsilon$ be divergence-free, and assume that
\begin{equation}\label{eq:inidata-eps-assum}
\begin{gathered}
  \sup_{0<\varepsilon\leq1}
  \Vert u_0^\varepsilon\Vert _{W^{1,p}}
  +
  \sup_{0<\varepsilon\leq1}
  \Vert \theta_0^\varepsilon\Vert _{L^1\cap L^\infty}
  <\infty,
  \\
  \theta_0^\varepsilon\longrightarrow\theta_0
  \qquad\text{strongly in }L^2(\mathbb R^2).
\end{gathered}
\end{equation}
If $\alpha=1$, assume in addition that there exists
$R_0>0$, independent of $\varepsilon$, such that
\begin{equation}\label{eq:the-eps-supp-assum}
  \operatorname{supp}\theta_0^\varepsilon
  \subset B_{R_0},
  \qquad 0<\varepsilon\leq1.
\end{equation}

Let $(u^\varepsilon,\theta^\varepsilon)$ be the unique
global solution of \eqref{BoussEq} on $\mathbb R^2$
with initial data $(u_0^\varepsilon,\theta_0^\varepsilon)$,
provided by Theorem~\ref{thm:main1}.

As $\varepsilon\to0$,
the following statements hold for the whole family.
\begin{enumerate}[(i)]
\item
If $\frac12<\alpha<1$, or if $\alpha=1$ and
\begin{equation}\label{eq:zero-mean-family}
  \int_{\mathbb R^2}\theta_0^\varepsilon(x)\,\dd x=0,
  \qquad 0<\varepsilon\leq1,
\end{equation}
the original solutions $(u^\varepsilon,\theta^\varepsilon)$
converge to $(u,\theta)$.
Set $a^\varepsilon\equiv0$ and
$(U^\varepsilon,\Theta^\varepsilon)
:=(u^\varepsilon,\theta^\varepsilon)$.

\item
If $\alpha=1$ and \eqref{eq:zero-mean-family} is not assumed,
there exist absolutely continuous paths
$a^\varepsilon:[0,T]\to\mathbb R^2$,
with $a^\varepsilon(0)=0$, such that the transformed fields
$(U^\varepsilon,\Theta^\varepsilon)$ defined in
\eqref{def:Theta-U-eps} converge to $(u,\theta)$.
\end{enumerate}
In both cases, $(u,\theta)$ is the unique normalized global
weak solution of the fractional Stokes-transport system
\eqref{eq:ST}, with
\[
  \theta(0)=\theta_0,
  \qquad u=G_\alpha*\theta,
\]
where $G_\alpha$ is given by
\eqref{def:G-alp-frac}-\eqref{def:G-alp=1}.
More precisely, for every $R>0$,
\begin{equation}\label{eq:U-The-eps-conv-gen}
\begin{aligned}
  U^\varepsilon
  &\rightharpoonup^\ast u
  &&\text{weakly-$*$ in }\;
  L^\infty\bigl(0,T;W^{1,p}(B_R)\bigr),
  \\
  \Theta^\varepsilon
  &\longrightarrow\theta
  &&\text{strongly in }\;
  C\bigl([0,T];H^{-1/2}(B_R)\bigr).
\end{aligned}
\end{equation}
Moreover, if $\alpha=1$, in either case there exists $R_T>0$ such that
\begin{equation}\label{eq:theta-supp}
  \operatorname{supp}\theta(t)\subset B_{R_T},
  \qquad 0\leq t\leq T.
\end{equation}
\end{proposition}

No well-preparedness assumption is imposed on the initial
velocities: only the stated uniform bounds on
$u_0^\varepsilon$ are required.
The weak-$*$ convergence of the velocities does not
imply convergence of their initial traces
and does not exclude an initial velocity layer.
The temperature convergence, however, holds up to $t=0$.

Before proving Proposition~\ref{prop:weak-limit-R2},
we collect two auxiliary results.
The first is a version of Carath\'eodory's theorem,
used to construct the translation paths in case~(ii);
see, for instance,~\cite[Chapter~2]{CL55}.

\begin{lemma}\label{lem:Caratheodory-Lipschitz}
Let $T>0$ and
$b:[0,T]\times\mathbb R^d\to\mathbb R^d$.
Assume that:
\begin{enumerate}
\item
For every $z\in\mathbb R^d$,
the map $t\mapsto b(t,z)$ is measurable.

\item
For every $M>0$, there exists a nonnegative
$L_M\in L^1(0,T)$ such that
$|b(t,z_1)-b(t,z_2)| \leq L_M(t)|z_1-z_2|$
for almost every $t\in(0,T)$
and all $z_1,z_2\in B_M$.

\item
There exists a nonnegative $m\in L^1(0,T)$ such that
$|b(t,z)|\leq m(t)(1+|z|)$ 
for almost every $t\in(0,T)$
and every $z\in\mathbb R^d$.
\end{enumerate}

Then, for every $z_0\in\mathbb R^d$,
there exists a unique absolutely continuous function
$a:[0,T]\to\mathbb R^d$ satisfying
$a(t)=z_0+\int_0^t b(s,a(s))\,\dd s$ on $t\in[0,T]$.
\end{lemma}

The second establishes uniqueness for the normalized
Stokes-transport system in the solution class
arising in the convergence argument.
We give a self-contained proof covering both
the fractional and classical cases.

\begin{lemma}\label{lem:ST-uniqueness}
Let $\frac12<\alpha\leq1$, $T>0$, and
$\theta_0\in L^1\cap L^\infty(\mathbb R^2)$.
Suppose that $(u_i,\theta_i)$, $i=1,2$, are weak solutions
of \eqref{eq:ST} with the same initial temperature
$\theta_0$, satisfying
\[
  \theta_i\in
  L^\infty([0,T],L^1\cap L^\infty(\mathbb R^2)),
  \qquad
  u_i=G_\alpha*\theta_i.
\]
The initial condition is attained in the sense of
distributions:
$\theta_i(t)\to\theta_0$ in $\mathcal D'(\mathbb R^2)$
as $t\rightarrow0$. When $\alpha=1$, assume additionally that there exists
$R_T<\infty$ such that
\[
  \operatorname{supp}\theta_i(t)
  \subset\overline B_{R_T},
  \qquad i=1,2,
\]
for almost every $t\in(0,T)$.
Then $\theta_1=\theta_2$ and $u_1=u_2$
almost everywhere on $(0,T)\times\mathbb R^2$.
\end{lemma}

\begin{proof}
We compare the Lagrangian flows;
see~\cite[Section~4.1]{MS24} for related stability arguments for three-dimensional Stokes-transport system.

The explicit kernel formulas
\eqref{def:G-alp-frac}-\eqref{def:G-alp=1} give
\[
  |\nabla G_\alpha(z)|
  \leq C_\alpha|z|^{2\alpha-3},
  \qquad z\neq0.
\]
Since $\alpha>\frac12$, splitting the convolution
at unit distance gives
\begin{equation}\label{es:u:lip}
  \Vert \nabla u_i(t)\Vert _{L^\infty}
  \leq C_\alpha
  \Vert \theta_i(t)\Vert _{L^1\cap L^\infty}.
\end{equation}
For $\alpha<1$, the same splitting gives
$u_i\in L^\infty([0,T],L^\infty)$.
For $\alpha=1$, the support assumption gives
$|u_i(t,x)|\leq C_T(1+\log(2+|x|))$;
see~\cite[Lemma~1]{Gray23}.
Thus both velocities have at most linear growth.
Together with \eqref{es:u:lip} and
$\operatorname{div}u_i=0$, this ensures that each $u_i$
generates a global measure-preserving bi-Lipschitz flow
$X_{t,i}$ satisfying \eqref{eq:flowmap} with $u=u_i$.
For each fixed velocity $u_i$, uniqueness for the
linear transport equation gives
\begin{equation}\label{id:theta-i}
  \theta_i(t,X_{t,i}(a))=\theta_0(a),
  \qquad 0\leq t\leq T,
\end{equation}
for almost every $a\in\mathbb R^2$,
after choosing the transported representative of $\theta_i$.
In particular, the $L^1$ and $L^\infty$ norms
are conserved.

Set $K=\mathbb R^2$ when $\alpha<1$.
When $\alpha=1$, the support assumption and the
distributional initial trace imply
$\operatorname{supp}\theta_0\subset\overline B_{R_T}$;
we therefore take $K=\overline B_{R_T}$.
Define
\[
  D(t):=\sup_{a\in K}
  |X_{t,1}(a)-X_{t,2}(a)|.
\]
For $\alpha<1$,
\[
  D(t)\leq
  \int_0^t
  \bigl(\Vert u_1(s)\Vert _{L^\infty}
       +\Vert u_2(s)\Vert _{L^\infty}\bigr)\,\dd s<\infty.
\]
Moreover, for $0\leq s\leq t\leq T$,
$|D(t)-D(s)|
  \leq
  \int_s^t
  \bigl(\Vert u_1(\tau)\Vert _{L^\infty}
       +\Vert u_2(\tau)\Vert _{L^\infty}\bigr)\,\dd\tau$,
which gives continuity.
For $\alpha=1$, finiteness and continuity follow
from the joint continuity of the flows
on the compact set $[0,T]\times K$.
Thus $D\in C([0,T])$ and $D(0)=0$.

We next record the kernel difference estimate
\begin{equation}\label{es:diff}
  |G_\alpha(z)-G_\alpha(z')|
  \leq C_\alpha|z-z'|
  \bigl(|z|^{2\alpha-3}+|z'|^{2\alpha-3}\bigr),
  \qquad z,z'\neq0.
\end{equation}
For $\alpha<1$, this follows from
\[
  |G_\alpha(z)|\leq C_\alpha|z|^{2\alpha-2},
  \qquad
  |\nabla G_\alpha(z)|
  \leq C_\alpha|z|^{2\alpha-3}.
\]
Indeed, if
$|z-z'|\leq\frac12\min\{|z|,|z'|\}$,
the gradient estimate applies along the segment
joining $z$ and $z'$.
Otherwise, $\max\{|z|,|z'|\}<3|z-z'|$,
and the bound on $G_\alpha$ gives \eqref{es:diff}.

For $\alpha=1$, the explicit formula has the form
\[
  G_1(z)
  =-\frac1{4\pi}\log|z|\,e_2
   +H\left(\frac{z}{|z|}\right),
  \qquad
  H\in C^\infty(\mathbb S^1;\mathbb R^2).
\]
The bounds
\begin{align*}
  \bigl|\log|z|-\log|z'|\bigr|
  \leq
  |z-z'|\left(\frac1{|z|}+\frac1{|z'|}\right),\qquad
  \left|\frac{z}{|z|}-\frac{z'}{|z'|}\right|
  \leq
  2|z-z'|\left(\frac1{|z|}+\frac1{|z'|}\right),
\end{align*}
together with the Lipschitz continuity of $H$,
prove \eqref{es:diff} at $\alpha=1$.

By \eqref{id:theta-i} and the measure-preserving property,
\[
  u_i(t,x)
  =
  \int_K
  G_\alpha\bigl(x-X_{t,i}(a)\bigr)
  \theta_0(a)\,\dd a.
\]
Using \eqref{es:diff} and changing variables
$y=X_{t,j}(a)$, we obtain
\begin{align*}
  |u_1(t,x)-u_2(t,x)|
  &\leq C_\alpha D(t)
  \sum_{j=1}^2
  \int_K
  |x-X_{t,j}(a)|^{2\alpha-3}
  |\theta_0(a)|\,\dd a
  \\
  &=C_\alpha D(t)
  \sum_{j=1}^2
  \int_{\mathbb R^2}
  |x-y|^{2\alpha-3}|\theta_j(t,y)|\,\dd y
  \leq
  C_\alpha\Vert \theta_0\Vert _{L^1\cap L^\infty}D(t).
\end{align*}
Consequently, for almost every $t\in(0,T)$,
\[
  \Vert u_1(t)-u_2(t)\Vert _{L^\infty}
  \leq C_\alpha
  \Vert \theta_0\Vert _{L^1\cap L^\infty}D(t).
\]

Comparing the flow equations \eqref{eq:flowmap}
for $X_{t,1}$ and $X_{t,2}$
and taking the supremum over $a\in K$, we find
\begin{align*}
  D(t)
  \leq
  \int_0^t
  \Bigl(
    \Vert \nabla u_1(s)\Vert _{L^\infty}D(s)
    +\Vert u_1(s)-u_2(s)\Vert _{L^\infty}
  \Bigr)\,\dd s
  \leq
  C_\alpha\Vert \theta_0\Vert _{L^1\cap L^\infty}
  \int_0^tD(s)\,\dd s.
\end{align*}
Since $D(0)=0$, Gronwall's inequality gives
$D\equiv0$ on $[0,T]$.
Thus the two flows coincide on $K$.
By the transport representation and
the measure-preserving property, for every
$\varphi\in C_c^\infty(\mathbb R^2)$,
\[
  \int_{\mathbb R^2}\varphi(x)\theta_i(t,x)\,\dd x
  =
  \int_K
  \varphi(X_{t,i}(a))\theta_0(a)\,\dd a,
\]
where we used that $\theta_0$ vanishes outside $K$.
The right-hand sides agree, so $\theta_1=\theta_2$.
The normalization $u_i=G_\alpha*\theta_i$
then gives $u_1=u_2$.
\end{proof}

We now turn to the proof of
Proposition~\ref{prop:weak-limit-R2}.

\begin{proof}[Proof of Proposition~\ref{prop:weak-limit-R2}]
We prove both cases within a common five-step argument,
using the original frame in case~(i) and an averaged
normalization in case~(ii).
After deriving common local bounds, we establish compactness
and pass to the limit in the equations.
Global $L^Q$ integrability identifies the limiting velocity
in case~(i), and the averaged normalization does so in case~(ii).

Throughout the proof, constants without an $\varepsilon$ subscript
may depend on $T$, $\alpha$, $p$, the uniform initial bounds,
the exponent $Q$ in case~(i), $R_0$ when $\alpha=1$,
and the fixed averaging function $\eta$ in case~(ii),
but not on $\varepsilon$.

\medskip
\noindent\textbf{Step 1. Choice of reference frame.}

In case~(i), choose a finite exponent
$Q>\max\{p,1/(1-\alpha)\}$ when $\alpha<1$,
or $Q>p$ when $\alpha=1$.
In the latter case, \eqref{eq:the-eps-supp-assum} gives
$\sup_{0<\varepsilon\leq1}
  \int_{\mathbb R^2}|x|\,|\theta_0^\varepsilon(x)|\,\dd x
  \leq R_0\sup_{0<\varepsilon\leq1}
      \Vert \theta_0^\varepsilon\Vert _{L^1}<\infty$.
Thus \eqref{eq:u-LQ-es}, together with
\eqref{eq:zero-mean-family} when $\alpha=1$, gives
\begin{equation}\label{eq:limit-u-LQ-bound}
  \sup_{0<\varepsilon\leq1}
  \Vert u^\varepsilon\Vert _{L^\infty([0,T],L^Q(\mathbb R^2))}
  \leq C_T.
\end{equation}
We set $a^\varepsilon\equiv0$, so that
$(U^\varepsilon,\Theta^\varepsilon)
=(u^\varepsilon,\theta^\varepsilon)$.

For the remainder of this construction, consider case~(ii).
We define the moving frame as follows.
Fix a nonnegative function
$\eta\in C_c^\infty(B_1)$ with
$\int_{\mathbb R^2}\eta\,\dd x=1$, and define
\[
  b^\varepsilon(t,z)
  :=\int_{\mathbb R^2}\eta(x)
  \bigl(u^\varepsilon-G_1*\theta^\varepsilon\bigr)
  (t,x+z)\,\dd x.
\]
For each fixed $\varepsilon$, we verify the hypotheses of
Lemma~\ref{lem:Caratheodory-Lipschitz}.

By Theorem~\ref{thm:main1} and the embedding
$W^{1,p}(\mathbb R^2)\hookrightarrow L^\infty(\mathbb R^2)$
for $p>2$,
\[
  \Vert u^\varepsilon\Vert _{L^\infty([0,T],L^\infty)}
  \leq C_{\varepsilon,T}.
\]
The flow representation and
\eqref{eq:the-eps-supp-assum} therefore imply
\[
  \operatorname{supp}\theta^\varepsilon(t)
  \subset B_{R_{\varepsilon,T}},
  \qquad 0\leq t\leq T,
\]
for some finite $R_{\varepsilon,T}$.
The logarithmic kernel estimate
in~\cite[Lemma~1]{Gray23}, together with conservation
of the temperature norms, gives
\begin{equation}\label{es:log-grow}
  |(G_1*\theta^\varepsilon)(t,\xi)|
  \leq
  C\bigl(1+\log(2+|\xi|+R_{\varepsilon,T})\bigr)
  \Vert \theta_0^\varepsilon\Vert _{L^1\cap L^\infty},
  \qquad \xi\in\mathbb R^2.
\end{equation}
In particular,
$u^\varepsilon-G_1*\theta^\varepsilon
\in L^\infty([0,T],L^1_{\mathrm{loc}})$.
Writing
\[
  b^\varepsilon(t,z)
  =\int_{\mathbb R^2}\eta(y-z)
  \bigl(u^\varepsilon-G_1*\theta^\varepsilon\bigr)
  (t,y)\,\dd y
\]
shows that $b^\varepsilon(t,z)$ is measurable in $t$
for each fixed $z$ and smooth in $z$ for almost every $t$.
Moreover, for every $M>0$,
\[
  \sup_{|z|\leq M}|\nabla_z b^\varepsilon(t,z)|
  \leq
  \Vert \nabla\eta\Vert _{L^\infty}
  \Vert u^\varepsilon(t)-G_1*\theta^\varepsilon(t)
    \Vert _{L^1(B_{M+1})}
  \leq C_{\varepsilon,T,M}.
\]
Thus the required local Lipschitz condition holds.

Since $\operatorname{supp}\eta\subset B_1$,
the velocity bound and \eqref{es:log-grow} also give
\[
  |b^\varepsilon(t,z)|
  \leq C_{\varepsilon,T}
       \bigl(1+\log(3+|z|)\bigr)
  \leq C_{\varepsilon,T}(1+|z|).
\]
Hence the linear growth condition is satisfied.
Lemma~\ref{lem:Caratheodory-Lipschitz} therefore gives
a unique absolutely continuous function
$a^\varepsilon:[0,T]\to\mathbb R^2$ satisfying
\begin{equation}\label{def:a-eps}
  \dot a^\varepsilon(t)
  =b^\varepsilon(t,a^\varepsilon(t)),
  \qquad a^\varepsilon(0)=0,
\end{equation}
for almost every $t\in(0,T)$.
The growth bound and Gronwall's inequality further imply
$a^\varepsilon\in W^{1,\infty}(0,T;\mathbb R^2)$,
with a bound that may depend on $\varepsilon$.

Still assuming $\alpha=1$, introduce
$(U^\varepsilon,\Theta^\varepsilon)$ as in
\eqref{def:Theta-U-eps}.
Since convolution with $G_1$ commutes with translations,
\[
  (G_1*\Theta^\varepsilon)(t,x)
  =(G_1*\theta^\varepsilon)
    \bigl(t,x+a^\varepsilon(t)\bigr).
\]

It follows directly from the definitions of $a^\varepsilon$ and $b^\varepsilon$ that for almost every $t\in[0,T]$,
\begin{equation}\label{eq:averaged-normalization}
\begin{split}
  & \int_{\mathbb{R}^2} \eta(x) \bigl( U^\varepsilon(t,x) - G_1*\Theta^\varepsilon(t,x) \bigr)\,\dd x \\
  & = \int_{\mathbb{R}^2} \eta(x) \big(u^\varepsilon (t, x+a^\varepsilon(t))  
  - G_1*\theta^\varepsilon(t,x+a^\varepsilon(t)) \big) \dd x 
  -b^\varepsilon(t,a^\varepsilon(t)) = 0.
\end{split}
\end{equation}

For either choice of frame, the transformed fields satisfy
\[
  \partial_t\Theta^\varepsilon
  +U^\varepsilon\cdot\nabla\Theta^\varepsilon=0,
  \qquad
  \operatorname{div}U^\varepsilon=0,
  \qquad
  \Theta^\varepsilon|_{t=0}=\theta_0^\varepsilon.
\]
Moreover, for every divergence-free test function
$\varphi\in C_c^\infty((0,T)\times\mathbb R^2;\mathbb R^2)$,
the momentum equation takes the form
\begin{equation}\label{eq:weak-momentum-translated}
\begin{aligned}
  &-\varepsilon\int_0^T\!\int_{\mathbb R^2}
    U^\varepsilon\cdot\partial_t\varphi\,\dd x\,\dd t
  -\varepsilon\int_0^T\!\int_{\mathbb R^2}
    (U^\varepsilon\otimes U^\varepsilon):\nabla\varphi
    \,\dd x\,\dd t
  \\
  &\qquad
  +\int_0^T\!\int_{\mathbb R^2}
    U^\varepsilon\cdot\Lambda^{2\alpha}\varphi
    \,\dd x\,\dd t
  =
  \int_0^T\!\int_{\mathbb R^2}
    \Theta^\varepsilon e_2\cdot\varphi\,\dd x\,\dd t.
\end{aligned}
\end{equation}
In case~(i), these are the original equations.
In case~(ii), they follow by using translated test
functions in the original weak formulations.
In particular, the momentum identity follows by testing
with $\varphi(t,y-a^\varepsilon(t))$,
justified by smooth approximation of $a^\varepsilon$ in time.
The additional terms cancel using
$\operatorname{div}U^\varepsilon=0$ and
\[
  \int_{\mathbb R^2}\varphi_i(t,x)\,\dd x
  =
  \int_{\mathbb R^2}
    \operatorname{div}\bigl(x_i\varphi(t,x)\bigr)\,\dd x
  =0,
  \qquad i=1,2.
\]
Since $a^\varepsilon\in W^{1,\infty}(0,T;\mathbb R^2)$
for each fixed $\varepsilon$, this argument requires
no classical second derivative of $a^\varepsilon$.

\medskip
\noindent\textbf{Step 2. Uniform velocity and support bounds.}
Set $\sigma=1-2/p\in(0,1)$.

Let $X_t^\varepsilon$ be the flow map generated by $u^\varepsilon$, and
define the relative flow $Y_t^\varepsilon(x) := X_t^\varepsilon(x)-a^\varepsilon(t)$.
Then 
\begin{align*}
  \partial_tY_t^\varepsilon(x) = U^\varepsilon\bigl(t,Y_t^\varepsilon(x)\bigr),
  \qquad Y_0^\varepsilon(x)=x,
\end{align*}
and the transport equation gives 
\begin{align}\label{eq:Theta-eps-exp}
  \Theta^\varepsilon(t,x) = \theta_0^\varepsilon \bigl((Y_t^\varepsilon)^{-1}(x)\bigr) 
  = \theta_0^\varepsilon \big((X_t^\varepsilon)^{-1} (x + a^\varepsilon(t)) \big) = \theta^\varepsilon \big(t, x + a^\varepsilon(t) \big).
\end{align}

We first treat case~(i).
By \eqref{eq:limit-u-LQ-bound}, H\"older's inequality,
and the uniform gradient estimate,
\[
  \Vert u^\varepsilon\Vert _{L^\infty([0,T],L^p(B_R))}
  +\Vert \nabla u^\varepsilon\Vert _{-L^p)}
  \leq C_{T,R}.
\]
Averaging Morrey's inequality over $B_1$ gives,
for almost every $t\in(0,T)$ and every $x\in\mathbb R^2$,
\begin{align*}
  |U^\varepsilon(t,x)|
  =|u^\varepsilon(t,x)|
  &\leq C\Vert u^\varepsilon(t)\Vert _{L^Q}
    +C(1+|x|^\sigma)
      \Vert \nabla u^\varepsilon(t)\Vert _{L^p}
  \\
  &\leq C_T(1+|x|^\sigma).
\end{align*}
If $\alpha=1$ in case~(i), the initial temperatures
are supported in $B_{R_0}$ and $Y_t^\varepsilon=X_t^\varepsilon$.
The preceding growth bound and Gronwall's inequality
therefore keep their supports in a common ball on $[0,T]$.

For later use, we also record a convolution estimate
for the fractional Stokes kernel.
For $\frac12<\alpha<1$, the kernel $G_\alpha$
defined by \eqref{def:G-alp-frac} satisfies
\[
  |G_\alpha(x)|\leq C_\alpha|x|^{2\alpha-2},
  \qquad x\neq0.
\]
For every $f\in L^1\cap L^\infty$,
splitting the convolution at $|x-y|=1$ gives
\begin{equation}\label{eq:Galp-conv-Linf-bdd}
\begin{aligned}
  \Vert G_\alpha*f\Vert _{L^\infty}
  \leq C_\alpha
  \bigg(
    \Vert f\Vert _{L^\infty}
    \int_{|z|\leq1}|z|^{2\alpha-2}\,\dd z
    +\Vert f\Vert _{L^1}
  \bigg)
  \leq C_\alpha\Vert f\Vert _{L^1\cap L^\infty}.
\end{aligned}
\end{equation}
Conservation of the temperature norms
and \eqref{eq:inidata-eps-assum} therefore yield
\begin{equation}\label{es:Galp*the-Linf}
  \Vert G_\alpha*\theta^\varepsilon(t)\Vert _{L^\infty}
  \leq C_\alpha
    \Vert \theta_0^\varepsilon\Vert _{L^1\cap L^\infty}
  \leq C,
\end{equation}
uniformly in $\varepsilon\in(0,1]$ and $t\in[0,T]$.

We next treat case~(ii), where $\alpha=1$. Define
\[
  R^\varepsilon(t)
  :=\sup_{x\in\overline B_{R_0}}|Y_t^\varepsilon(x)|.
\]
Then $R^\varepsilon(t)<\infty$, $R^\varepsilon(0)=R_0$,
and $\operatorname{supp}\Theta^\varepsilon(t)
\subset\overline B_{R^\varepsilon(t)}$.

To estimate the translated velocity, we use the definition of $a^\varepsilon$ 
in \eqref{def:a-eps} to write
\begin{align}\label{eq:translated-velocity-decomposition}
  U^\varepsilon(t,x) = u^\varepsilon(t,x+a^\varepsilon(t))
  - \int_{\mathbb{R}^2} \eta(z)u^\varepsilon(t,z+a^\varepsilon(t))\,\dd z +
  \int_{\mathbb{R}^2} \eta(z) (G_1*\Theta^\varepsilon)(t,z)\,\dd z.
\end{align}
By the homogeneous Morrey estimate
$[u^\varepsilon(t)]_{C^{0,\sigma}} we have
\leq C\Vert \nabla u^\varepsilon(t)\Vert _{L^p}$,
\begin{equation}\label{eq:Morrey-translated}
\begin{split}
  &\left| u^\varepsilon(t,x+a^\varepsilon(t)) - \int_{\mathbb{R}^2} \eta(z)u^\varepsilon(t,z+a^\varepsilon(t))\,\dd z \right| \\
  &\leq \int_{\mathbb{R}^2} \eta(z) \left| u^\varepsilon(t,x+a^\varepsilon(t)) 
  - u^\varepsilon(t,z+a^\varepsilon(t)) \right|\dd z 
  \leq C(1+|x|^\sigma) \Vert \nabla u^\varepsilon(t)\Vert _{L^p}.
\end{split}
\end{equation}

The logarithmic kernel estimate gives
\[
  \left|\int_{\mathbb R^2}
    \eta(z)(G_1*\Theta^\varepsilon)(t,z)\,\dd z\right|
  \leq C\bigl(1+\log(2+R^\varepsilon(t))\bigr).
\]
Since $\log(2+r)\leq C(1+\sqrt r)$ for $r\geq0$,
we obtain
\begin{equation}\label{eq:global-growth-Ueps}
  |U^\varepsilon(t,x)|
  \leq C_T\left(1+|x|^\sigma+\sqrt{R^\varepsilon(t)}\right).
\end{equation}
The constants are independent of $\varepsilon$.
We next close the support estimate in this endpoint case.

For every $x\in\overline B_{R_0}$,
the relative trajectory satisfies the integral identity
\begin{align*}
  Y_t^\varepsilon(x) = x + \int_0^t U^\varepsilon\bigl(s,Y_s^\varepsilon(x)\bigr)\,\dd s.
\end{align*}
Taking the supremum over $\overline B_{R_0}$
and using \eqref{eq:global-growth-Ueps}, we obtain
\begin{align*}
  R^\varepsilon(t) \leq R_0 + C\int_0^t \left( 1+(R^\varepsilon(s))^\sigma +\sqrt{R^\varepsilon(s)} \right)\,\dd s \leq R_0 + C\int_0^t \bigl(1+R^\varepsilon(s)\bigr)\,\dd s.
\end{align*}
Gronwall's inequality gives
$\sup_{0\leq t\leq T}R^\varepsilon(t)\leq C_T$,
uniformly in $\varepsilon$.
Combining this with the support bound obtained in case~(i)
and enlarging $R_T$ if necessary, we conclude that,
whenever $\alpha=1$,
\begin{align}\label{eq:The-eps-supp}
  \operatorname{supp}\Theta^\varepsilon(t)\subset B_{R_T},
  \qquad 0\leq t\leq T.
\end{align}

Substituting this support bound into
\eqref{eq:global-growth-Ueps} in case~(ii), and combining
it with the estimate obtained in case~(i), we find that, for all
$\alpha\in(\frac12,1]$, almost every $t\in(0,T)$,
and every $x\in\mathbb R^2$,
\begin{align}\label{eq:U-eps-uppbdd}
  |U^\varepsilon(t,x)|\leq C_T(1+|x|^\sigma),
\end{align}
with $\sigma = 1- \frac{2}{p}$.
Moreover, from $\nabla U^\varepsilon(t,x) = \nabla u^\varepsilon(t,x+a^\varepsilon(t))$, it follows that, for every $R>0$,
\begin{align}\label{eq:local-W1p-Ueps}
  \Vert U^\varepsilon\Vert _{L^\infty([0,T],W^{1,p}(B_R))} \leq C_{T,R}.
\end{align}

\medskip
\noindent\textbf{Step 3. Compactness.}


By the Banach-Alaoglu theorem and a diagonal argument,
there exist a subsequence and a vector field $u$ such that,
for every $R>0$,
\begin{equation}\label{eq:local-weak-Ueps}
  U^\varepsilon\rightharpoonup^\ast u
  \quad\text{in }
  L^\infty([0,T],W^{1,p}(B_R)).
\end{equation}
In case~(i), \eqref{eq:limit-u-LQ-bound}
allows a further extraction such that
\begin{equation}\label{eq:limit-u-global-LQ}
  U^\varepsilon=u^\varepsilon\rightharpoonup^\ast u
  \quad\text{in }
  L^\infty([0,T],L^Q(\mathbb R^2)).
\end{equation}
The global and local limits agree by uniqueness
of the distributional limit.

Weak-$\ast$ lower semicontinuity gives, for every $R>0$,
\[
  \Vert u\Vert _{L^\infty([0,T],W^{1,p}(B_R))}
  \leq\liminf_{\varepsilon\to0}
    \Vert U^\varepsilon\Vert _{L^\infty([0,T],W^{1,p}(B_R))}
  \leq C_{T,R}.
\]
The growth bound \eqref{eq:U-eps-uppbdd}
also passes to the limit.
Indeed, for every
$\psi\in C_c^\infty((0,T)\times\mathbb R^2;\mathbb R^2)$,
\eqref{eq:local-weak-Ueps} gives
\begin{align*}
  \left|
    \int_0^T\!\int_{\mathbb R^2}
    u\cdot\psi\,\dd x\,\dd t
  \right|
  =
  \lim_{\varepsilon\to0}
  \left|
    \int_0^T\!\int_{\mathbb R^2}
    U^\varepsilon\cdot\psi\,\dd x\,\dd t
  \right|
  \leq
  C_T\int_0^T\!\int_{\mathbb R^2}
    (1+|x|^\sigma)|\psi|\,\dd x\,\dd t.
\end{align*}
By density and $L^1$-$L^\infty$ duality
with respect to the measure
$(1+|x|^\sigma)\,\dd x\,\dd t$, we obtain
\[\Big\Vert 
    \frac{u}{1+|x|^\sigma}
  \Big\Vert _{L^\infty((0,T)\times\mathbb R^2)}
  \leq C_T.
\]
Thus $|u(t,x)|\leq C_T(1+|x|^\sigma)$
for almost every $(t,x)\in(0,T)\times\mathbb R^2$.
Since $p>2$, Morrey's embedding provides a spatially
continuous representative of $u(t,\cdot)$
for almost every $t\in(0,T)$.
For this representative,
\begin{equation}\label{eq:u-upper-bdd}
  |u(t,x)|\leq C_T(1+|x|^\sigma),
  \qquad x\in\mathbb R^2,
\end{equation}
for almost every $t\in(0,T)$.

Next, we establish compactness of $\Theta^\varepsilon$.
By \eqref{eq:Theta-eps-exp}, the measure-preserving
property of the flow, and \eqref{eq:inidata-eps-assum},
\begin{equation}\label{eq:The-eps-L1Linf-bdd}
  \sup_{0<\varepsilon\leq1}
  \Vert \Theta^\varepsilon\Vert _{L^\infty([0,T],L^1\cap L^\infty)}
  \leq C.
\end{equation}
Since
$\partial_t\Theta^\varepsilon
=-\div(U^\varepsilon\Theta^\varepsilon)$,
the uniform local velocity bound
\eqref{eq:local-W1p-Ueps} gives
\[
  \Vert \partial_t\Theta^\varepsilon
    \Vert _{L^\infty([0,T],H^{-1}(B_R))}
  \leq C_{T,R}.
\]
Moreover, $\Theta^\varepsilon$ is uniformly bounded
in $L^\infty([0,T],L^2(B_R))$.
Using the compact embedding
$L^2(B_R)\hookrightarrow H^{-1/2}(B_R)$
and the continuous embedding
$H^{-1/2}(B_R)\hookrightarrow H^{-1}(B_R)$,
the Aubin-Lions lemma
(see e.g.~\cite[Lemma~2.10]{LXY25})
and a diagonal extraction yield, for every $R>0$,
\begin{equation}\label{eq:strong-local-theta}
  \Theta^\varepsilon\longrightarrow\theta
  \quad\text{strongly in }
  C\bigl([0,T];H^{-1/2}(B_R)\bigr).
\end{equation}

For every fixed $t\in[0,T]$,
the uniform local $L^2$ bound and
\eqref{eq:strong-local-theta} imply
\[
  \Theta^\varepsilon(t)\rightharpoonup\theta(t)
  \quad\text{weakly in }L^2(B_R),
\]
and hence also weakly in $L^1(B_R)$.
By weak lower semicontinuity,
\[
  \Vert \theta(t)\Vert _{L^1(B_R)}
  \leq\liminf_{\varepsilon\to0}
    \Vert \Theta^\varepsilon(t)\Vert _{L^1(B_R)}
  \leq C.
\]
The uniform $L^\infty$ bound also passes to the limit.
Since these bounds are independent of $R$ and $t$,
letting $R\to\infty$ gives
\begin{equation}\label{eq:the-L1Linf-bdd}
  \theta\in
  L^\infty([0,T],L^1\cap L^\infty(\mathbb R^2)).
\end{equation}

At $\alpha=1$, the uniform support property
\eqref{eq:The-eps-supp} also passes to the limit.
Indeed, for every
$\varphi\in C_c^\infty(
\mathbb R^2\setminus\overline B_{R_T})$,
\eqref{eq:strong-local-theta} gives
$\langle\theta(t),\varphi\rangle = \lim_{\varepsilon\to0}
  \langle\Theta^\varepsilon(t),\varphi\rangle =0$,  
$t\in[0,T]$.
Hence
\begin{equation}\label{eq:the-supp-alp=1}
  \operatorname{supp}\theta(t)
  \subset\overline B_{R_T},
  \qquad 0\leq t\leq T.
\end{equation}

\medskip
\noindent\textbf{Step 4. Passage to the limit in the equations.}


We now pass to the limit in the transport equation.
Let $\zeta\in C_c^\infty([0,T)\times\mathbb R^2)$,
and choose $R>0$ such that
$\operatorname{supp}\zeta\subset[0,T)\times B_R$.
The translated transport equation has the weak formulation
\begin{equation}\label{eq:Theta-eps-weak-form}
  \int_0^T\!\int_{\mathbb R^2}
    \Theta^\varepsilon
    \bigl(\partial_t\zeta
      +U^\varepsilon\cdot\nabla\zeta\bigr)
    \,\dd x\,\dd t
  +\int_{\mathbb R^2}
    \theta_0^\varepsilon(x)\zeta(0,x)\,\dd x
  =0.
\end{equation}
The term involving $\partial_t\zeta$ converges by
\eqref{eq:strong-local-theta}, since
$\partial_t\zeta\in L^1([0,T],H_0^{1/2}(B_R))$.

For the nonlinear term, write $\Theta^\varepsilon U^\varepsilon-\theta u
  =
  \theta(U^\varepsilon-u)
  +(\Theta^\varepsilon-\theta)U^\varepsilon$. Since $\theta\in L^\infty((0,T)\times B_R)$,
we have
$\theta\nabla\zeta\in L^1([0,T],L^{p'}(B_R))$.
Thus \eqref{eq:local-weak-Ueps} gives
\[
  \int_0^T\!\int_{B_R}
    (U^\varepsilon-u)\cdot\theta\nabla\zeta
    \,\dd x\,\dd t
  \longrightarrow0.
\]
Moreover, $p>2$ and the compact spatial support
of $\nabla\zeta$ imply
$U^\varepsilon\cdot\nabla\zeta\in H_0^1(B_R)$
for almost every $t$, with
$\Vert U^\varepsilon\cdot\nabla\zeta
    \Vert _{L^1([0,T],H_0^{1/2}(B_R))}
  \leq
  C_{\zeta,R}
  \Vert U^\varepsilon\Vert _{L^1([0,T],W^{1,p}(B_R))}
  \leq C_{T,\zeta,R}$.
Consequently,
\begin{align*}
  \left|
    \int_0^T
    \left\langle
      \Theta^\varepsilon-\theta,\,
      U^\varepsilon\cdot\nabla\zeta
    \right\rangle_{H^{-1/2},H_0^{1/2}}
    \,\dd t
  \right|
  \leq
  C_{T,\zeta,R}
  \Vert \Theta^\varepsilon-\theta
    \Vert _{C([0,T];H^{-1/2}(B_R))}
  \longrightarrow0.
\end{align*}

Finally, the strong $L^2$ convergence of
$\theta_0^\varepsilon$ gives convergence of the initial term.
The uniform $L^1$ and $L^\infty$ bounds also imply
$\theta_0\in L^1\cap L^\infty$.
Passing to the limit in
\eqref{eq:Theta-eps-weak-form}, we obtain
\[
  \int_0^T\!\int_{\mathbb R^2}
    \theta\bigl(\partial_t\zeta+u\cdot\nabla\zeta\bigr)
    \,\dd x\,\dd t
  +\int_{\mathbb R^2}\theta_0(x)\zeta(0,x)\,\dd x
  =0.
\]

We next pass to the limit in the momentum equation.
Let
$\varphi\in C_c^\infty((0,T)\times\mathbb R^2;\mathbb R^2)$
satisfy $\div\varphi=0$.
The translated momentum equation has the weak formulation \eqref{eq:weak-momentum-translated}.

Choose $R>0$ such that
$\operatorname{supp}\varphi\subset(0,T)\times B_R$.
Since $p>2$, \eqref{eq:local-W1p-Ueps}
and Sobolev embedding give
\[
  \sup_{0<\varepsilon\leq1}
  \Vert U^\varepsilon\Vert _{L^\infty((0,T)\times B_R)}
  \leq C_{T,R}.
\]
Hence both inertial terms vanish:
\begin{align*}
  \varepsilon
  \left|
    \int_0^T\!\int_{\mathbb R^2}
    U^\varepsilon\cdot\partial_t\varphi\,\dd x\,\dd t
  \right|
  +
  \varepsilon
  \left|
    \int_0^T\!\int_{\mathbb R^2}
    (U^\varepsilon\otimes U^\varepsilon):\nabla\varphi
    \,\dd x\,\dd t
  \right|
  \leq C_{T,\varphi,R}\,\varepsilon
  \longrightarrow0.
\end{align*}
The forcing term converges by
\eqref{eq:strong-local-theta}, since
$e_2\cdot\varphi\in L^1([0,T],H_0^{1/2}(B_R))$.

For the dissipative term, we distinguish two cases.
If $\alpha=1$, then
$\Lambda^2\varphi=-\Delta\varphi$
is smooth and compactly supported,
so \eqref{eq:local-weak-Ueps} gives the required convergence.

If $\frac12<\alpha<1$, the fractional Laplacian satisfies
\begin{equation}\label{eq:fractional-test-decay}
  |\Lambda^{2\alpha}\varphi(t,x)|
  \leq C_{\alpha,\varphi}(1+|x|)^{-2-2\alpha}.
\end{equation}
In particular,
$\Lambda^{2\alpha}\varphi
\in L^1([0,T],L^{Q'}(\mathbb R^2))$,
where $Q'=Q/(Q-1)$.
The global weak-$\ast$ convergence
\eqref{eq:limit-u-global-LQ} therefore gives $\int_0^T\!\int_{\mathbb R^2}
    (U^\varepsilon-u)\cdot\Lambda^{2\alpha}\varphi
    \,\dd x\,\dd t
  \longrightarrow0$. 

Thus, in both cases, passing to the limit in
\eqref{eq:weak-momentum-translated} gives
\begin{equation}\label{eq:fracStokes}
  \int_0^T\!\int_{\mathbb R^2}
    u\cdot\Lambda^{2\alpha}\varphi\,\dd x\,\dd t
  =
  \int_0^T\!\int_{\mathbb R^2}
    \theta e_2\cdot\varphi\,\dd x\,\dd t.
\end{equation}

Finally, the identity $\operatorname{div}u=0$ follows directly
from $\operatorname{div}U^\varepsilon=0$
and local distributional convergence.


\medskip
\noindent\textbf{Step 5. Identification of the normalized limiting velocity.}


We first show that the limiting (fractional) Stokes relation
and the sublinear growth bound determine the velocity
up to a spatially constant vector.
Set $w:=u-G_\alpha*\theta$, 
where $G_\alpha$ is defined by
\eqref{def:G-alp-frac}-\eqref{def:G-alp=1}.
Since both $u$ and $G_\alpha*\theta$ are divergence-free
and satisfy \eqref{eq:fracStokes}, we have,
for almost every $t\in(0,T)$,
\begin{equation}\label{eq:w}
  \langle w(t),\Lambda^{2\alpha}\Phi\rangle=0,
  \qquad
  \langle w(t),\nabla\Psi\rangle=0,
\end{equation}
for every divergence-free
$\Phi\in C_c^\infty(\mathbb R^2;\mathbb R^2)$
and every $\Psi\in C_c^\infty(\mathbb R^2)$.

For $\frac12<\alpha<1$,
\eqref{eq:Galp-conv-Linf-bdd}
and \eqref{eq:the-L1Linf-bdd} give $|(G_\alpha*\theta)(t,x)|\leq C_T$.

For $\alpha=1$, the kernel estimate
in~\cite[Lemma~1]{Gray23}, together with
\eqref{eq:the-L1Linf-bdd}
and \eqref{eq:the-supp-alp=1}, gives
\[
  |(G_1*\theta)(t,x)|
  \leq C_T\bigl(1+\log(2+|x|)\bigr).
\]
Combining these estimates with \eqref{eq:u-upper-bdd},
we obtain
\[
  |w(t,x)|\leq C_T(1+|x|^\sigma),
  \qquad \sigma=1-\frac2p\in(0,1).
\]
Thus $w(t,\cdot)$ has sublinear growth
and defines a tempered distribution.

Choose $\rho\in\mathcal S(\mathbb R^2)$ with
$\widehat\rho\in C_c^\infty(\mathbb R^2\setminus\{0\})$.
To justify testing \eqref{eq:w} with
$\Phi=\nabla^\perp\rho$, take
$\chi\in C_c^\infty(\mathbb R^2)$ equal to one
near the origin and set $\Phi_R:=\nabla^\perp\bigl(\chi(\cdot/R)\rho\bigr)$. 
Then $\Phi_R$ is compactly supported and divergence-free,
and $\Phi_R\to\nabla^\perp\rho$ in $\mathcal S$.

For $0<\alpha<1$ and $0\leq\sigma<2\alpha$,
the singular integral representation of $\Lambda^{2\alpha}$
gives
\[
  \int_{\mathbb R^2}(1+|x|^\sigma)
    |\Lambda^{2\alpha}f(x)|\,\dd x
  \leq C_{\alpha,\sigma}
  \sum_{|\beta|\leq2}
  \int_{\mathbb R^2}(1+|x|^\sigma)
    |\partial^\beta f(x)|\,\dd x,
  \qquad f\in\mathcal S(\mathbb R^2).
\]
Indeed, this follows by splitting the kernel integral
at unit distance, using second differences near the origin
and $\sigma<2\alpha$ at infinity.
Applying this estimate componentwise gives
\[
  \int_{\mathbb R^2}(1+|x|^\sigma)
  \left|
    \Lambda^{2\alpha}
    (\Phi_R-\nabla^\perp\rho)(x)
  \right|\,\dd x
  \longrightarrow0,\qquad\textrm{as}\;\; R\rightarrow \infty.
\]
For $\alpha=1$, the same convergence follows directly
from $\Lambda^2=-\Delta$ and convergence in $\mathcal S$.
The growth bound on $w$ therefore allows us to pass
to the limit in \eqref{eq:w}, giving $\langle w(t),\Lambda^{2\alpha}\nabla^\perp\rho\rangle=0$. 

Fourier duality, applied away from the origin
where the multiplier is smooth, gives
\[
  |\xi|^{2\alpha}\xi^\perp\cdot\widehat w(t,\xi)=0
  \quad\text{in }
  \mathcal D'(\mathbb R^2\setminus\{0\}),
\]
where $\xi^\perp=(\xi_2,-\xi_1)$.
The divergence-free condition also gives $\xi\cdot\widehat w(t,\xi)=0$ in $\mathcal S'(\mathbb R^2)$. 

Since $|\xi|^{2\alpha}$ is smooth and nonvanishing
away from the origin, these identities imply
\[
  |\xi|^2\widehat w
  =
  \xi(\xi\cdot\widehat w)
  +\xi^\perp(\xi^\perp\cdot\widehat w)
  =0
  \quad\text{in }
  \mathcal D'(\mathbb R^2\setminus\{0\}).
\]
Multiplying by $|\xi|^{-2}$ on this set gives
$\widehat w=0$ away from the origin.
Thus $\operatorname{supp}\widehat w(t)\subset\{0\}$.
The structure theorem for distributions supported
at a point shows that $w(t,\cdot)$ is
a vector-valued polynomial.
Its sublinear growth forces it to be constant.
Consequently,
\begin{equation}\label{eq:velocity-up-to-constant}
  u(t,\cdot)=G_\alpha*\theta(t,\cdot)+c(t)
  \quad\text{in }\mathcal D'(\mathbb R^2)
\end{equation}
for almost every $t\in(0,T)$,
where $c(t)$ is a spatially constant vector.

If $\frac12<\alpha<1$, then
$G_\alpha\mathds{1}_{B_1}\in L^1$ and
$G_\alpha\mathds{1}_{\mathbb R^2\setminus B_1}\in L^Q$,
since $Q>1/(1-\alpha)$.
Young's inequality and \eqref{eq:the-L1Linf-bdd} give
\[
  \Vert G_\alpha*\theta(t)\Vert _{L^Q}
  \leq C_{\alpha,Q}
  \bigl(\Vert \theta(t)\Vert _{L^Q}+\Vert \theta(t)\Vert _{L^1}\bigr)
  \leq C_T.
\]
If $\alpha=1$ in case~(i), conservation of the total
temperature, the common support bound
\eqref{eq:The-eps-supp}, and local distributional convergence
give $\int_{\mathbb R^2}\theta(t,x)\,\dd x=0$.
Enlarge $R_T$ so that $R_T\geq1$.
Since $|\nabla G_1(x)|\leq C|x|^{-1}$ for $x\ne0$,
zero mean and \eqref{eq:the-supp-alp=1} imply, for $|x|>2R_T$,
\[
  |(G_1*\theta)(t,x)|
  =\left|\int_{B_{R_T}}
    \bigl(G_1(x-y)-G_1(x)\bigr)\theta(t,y)\,\dd y\right|
  \leq \frac{C}{|x|}
    \int_{B_{R_T}}|y|\,|\theta(t,y)|\,\dd y
  \leq \frac{C_T}{|x|}.
\]
On $B_{2R_T}$, local integrability of $G_1$ and the
uniform $L^\infty$ bound on $\theta$ give a uniform bound
for the convolution. Hence,
\[
  \Vert G_1*\theta\Vert _{L^\infty([0,T],L^Q(\mathbb R^2))}
  \leq C_{T,Q},\qquad Q>2.
\]
Thus, throughout case~(i), both $u$ and $G_\alpha*\theta$
belong to $L^\infty([0,T],L^Q(\mathbb R^2))$
for the admissible finite exponent $Q$ fixed above.
By \eqref{eq:velocity-up-to-constant},
their difference $c(t)$ is spatially constant and belongs
to $L^Q(\mathbb R^2)$, so $c(t)=0$ for almost every $t$.

In case~(ii), we instead pass to the limit
in the averaged normalization.
Let $\chi\in C_c^\infty(0,T)$.
Multiplying \eqref{eq:averaged-normalization}
by $\chi(t)$ and integrating in time gives
\begin{equation}\label{eq:normalization-before-limit}
  \int_0^T\chi(t)\int_{\mathbb R^2}
  \eta(x)\bigl(
    U^\varepsilon-G_1*\Theta^\varepsilon
  \bigr)(t,x)\,\dd x\,\dd t=0.
\end{equation}
Since $\eta$ is compactly supported,
the local weak-$*$ convergence
\eqref{eq:local-weak-Ueps} gives
\begin{equation}\label{eq:limit-normalization-U}
  \int_0^T\chi(t)\int_{\mathbb R^2}
  \eta(x)U^\varepsilon(t,x)\,\dd x\,\dd t
  \longrightarrow
  \int_0^T\chi(t)\int_{\mathbb R^2}
  \eta(x)u(t,x)\,\dd x\,\dd t.
\end{equation}

For the convolution term, Fubini's theorem gives
\[
  \int_{\mathbb R^2}
  \eta(x)(G_1*\Theta^\varepsilon)(t,x)\,\dd x
  =
  \int_{\mathbb R^2}
  K_{1,\eta}(y)\Theta^\varepsilon(t,y)\,\dd y,
\]
where
\[
  K_{1,\eta}(y)
  :=
  \int_{\mathbb R^2}\eta(x)G_1(x-y)\,\dd x
  =
  \int_{\mathbb R^2}\eta(y+z)G_1(z)\,\dd z.
\]
Since $G_1\in L^1_{\mathrm{loc}}$ and
$\eta\in C_c^\infty(B_1)$,
differentiating $\eta$ under the integral signe shows that
$K_{1,\eta}\in C^\infty(\mathbb R^2)$.

Fix $R>\max\{1,R_T\}$ and choose
$\rho_R\in C_c^\infty(B_{2R})$
with $\rho_R=1$ on $B_R$.
By \eqref{eq:The-eps-supp}
and \eqref{eq:the-supp-alp=1},
\[
  \int_{\mathbb R^2}
  K_{1,\eta}(\Theta^\varepsilon-\theta)\,\dd y
  =
  \int_{\mathbb R^2}
  \rho_R K_{1,\eta}(\Theta^\varepsilon-\theta)\,\dd y.
\]
Since $\rho_R K_{1,\eta}\in C_c^\infty(B_{2R})$,
\eqref{eq:strong-local-theta} implies
\begin{align*}
  &\left|
    \int_0^T\chi(t)\int_{\mathbb R^2}
    K_{1,\eta}(y)
    (\Theta^\varepsilon-\theta)(t,y)\,\dd y\,\dd t
  \right|
  \\
  &\qquad\leq
  \Vert \chi\Vert _{L^1(0,T)}
  \Vert \rho_RK_{1,\eta}\Vert _{H_0^{1/2}(B_{2R})}
  \Vert \Theta^\varepsilon-\theta\Vert _
    {L^\infty([0,T],H^{-1/2}(B_{2R}))}
  \longrightarrow0.
\end{align*}
Consequently,
\begin{equation}\label{eq:limit-normalization-convolution}
\begin{aligned}
  \int_0^T\chi(t)\int_{\mathbb R^2}
  \eta(x)(G_1*\Theta^\varepsilon)(t,x)\,\dd x\,\dd t
  \longrightarrow
  \int_0^T\chi(t)\int_{\mathbb R^2}
  \eta(x)(G_1*\theta)(t,x)\,\dd x\,\dd t.
\end{aligned}
\end{equation}

Passing to the limite in
\eqref{eq:normalization-before-limit}, we obtain
\[
  \int_0^T\chi(t)\int_{\mathbb R^2}
  \eta(x)(u-G_1*\theta)(t,x)\,\dd x\,\dd t=0.
\]
By \eqref{eq:velocity-up-to-constant}
and $\int_{\mathbb R^2}\eta\,\dd x=1$,
this reduces to
$\int_0^T\chi(t)c(t)\,\dd t=0$.
Since $\chi$ is arbitrary, $c(t)=0$
for almost every $t\in(0,T)$.
Thus, in both cases,
\[
  u(t,\cdot)=G_\alpha*\theta(t,\cdot)
  \quad\text{in }\mathcal D'(\mathbb R^2)
\]
for almost every $t\in(0,T)$.

The normalization is fixed throughout the familly:
$a^\varepsilon\equiv0$ in case~(i), while in case~(ii)
the fixed averaging function $\eta$ and \eqref{def:a-eps}
uniquely determine $a^\varepsilon$.

Every subsequential limit constructed in Steps~3-5
belongs to the class of Lemma~\ref{lem:ST-uniqueness},
with the support condition at $\alpha=1$
ensured by \eqref{eq:the-supp-alp=1}.
Thus all such limits coincide, and relative compactness
implies that the stated convergences hold
for the whole familly as $\varepsilon\to0$.
Since $T>0$ is arbitrary, uniqueness makes the limits
on different finite time intervals compatible,
yielding a global normalized solution.
This completes the proof.
\end{proof}

We next strengthen the temperature convergence to strong
convergence in Lebesgue spaces, uniformly in time.
\begin{corollary}\label{cor:strong-limit-R2}
Under the assumptions of
Proposition~\ref{prop:weak-limit-R2},
as $\varepsilon\rightarrow 0$,
\begin{equation}\label{eq:Thet-eps-the-CLs-conv}
  \Theta^\varepsilon\longrightarrow\theta
  \quad\text{strongly in }
  C\bigl([0,T];L^s(\mathbb R^2)\bigr),
  \qquad 1<s<\infty.
\end{equation}
For $\alpha=1$, the convergence also holds for $s=1$.
In case~(i) of Proposition~\ref{prop:weak-limit-R2},
these convergences concern the original temperatures
$\Theta^\varepsilon=\theta^\varepsilon$.
\end{corollary}

\begin{proof}[Proof of Corollary~\ref{cor:strong-limit-R2}]
We first prove that
\begin{equation}\label{eq:The-eps-conv-CL2}
  \Theta^\varepsilon\longrightarrow\theta
  \quad\text{strongly in }
  C\bigl([0,T];L^2(\mathbb R^2)\bigr).
\end{equation}
The flow representation established in the proof of
Lemma~\ref{lem:ST-uniqueness}, together with the
corresponding representation of $\Theta^\varepsilon$,
gives
$\Theta^\varepsilon,\theta\in C([0,T];L^2(\mathbb R^2))$.
Since both flows preserve Lebesgue measure,
\begin{equation}\label{eq:The-eps-the-L2-cons}
  \Vert \Theta^\varepsilon(t)\Vert _{L^2}
  =\Vert \theta_0^\varepsilon\Vert _{L^2},
  \qquad
  \Vert \theta(t)\Vert _{L^2}=\Vert \theta_0\Vert _{L^2},
  \quad 0\leq t\leq T.
\end{equation}

Since $\theta\in C([0,T];L^2(\mathbb R^2))$,
the set
$\mathcal K:=\{\theta(t):0\leq t\leq T\}$
is compact in $L^2(\mathbb R^2)$
as the continuous image of the compact interval $[0,T]$.
Thus, for every $\delta>0$, there exist
$g_1,\ldots,g_N\in C_c^\infty(\mathbb R^2)$ such that
$\sup_{0\leq t\leq T}\min_{1\leq i\leq N}\Vert \theta(t)-g_i\Vert _{L^2}<\delta$.
Choose $R_\delta>0$ so that all the $g_i$
are compactly supported in $B_{R_\delta}$.
Approximating $\theta(t)$ by one of these functions
and using the uniform $L^2$ bounds, we obtain
\begin{align*}
  \sup_{0\leq t\leq T}
  \left|
    \int_{\mathbb R^2}
    (\Theta^\varepsilon(t)-\theta(t))\theta(t)\,\dd x
  \right|
  \leq
  \Vert \Theta^\varepsilon-\theta\Vert _
    {C([0,T];H^{-1/2}(B_{R_\delta}))}
  \max_{1\leq i\leq N}
  \Vert g_i\Vert _{H_0^{1/2}(B_{R_\delta})}
  +C\delta.
\end{align*}
First letting $\varepsilon\to0$ using
\eqref{eq:strong-local-theta}, and then
$\delta\rightarrow0$, gives
\begin{equation}\label{eq:The-eps-conv-mid1}
  \sup_{0\leq t\leq T}
  \left|
    \int_{\mathbb R^2}
    (\Theta^\varepsilon(t)-\theta(t))\theta(t)\,\dd x
  \right|
  \longrightarrow0.
\end{equation}

By \eqref{eq:The-eps-the-L2-cons},
\begin{align*}
  \Vert \Theta^\varepsilon(t)-\theta(t)\Vert _{L^2}^2
  =
  \Vert \theta_0^\varepsilon\Vert _{L^2}^2
  -\Vert \theta_0\Vert _{L^2}^2
  -2\int_{\mathbb R^2}
  (\Theta^\varepsilon(t)-\theta(t))\theta(t)\,\dd x.
\end{align*}
Taking the supremum over $t\in[0,T]$,
the strong $L^2$ convergence of the initial temperatures
and \eqref{eq:The-eps-conv-mid1}
prove \eqref{eq:The-eps-conv-CL2}.

The uniform bounds
\eqref{eq:The-eps-L1Linf-bdd}
and \eqref{eq:the-L1Linf-bdd}
allow interpolation with $L^1$ and $L^\infty$:
\[
  \Vert \Theta^\varepsilon(t)-\theta(t)\Vert _{L^s}
  \leq
  \begin{cases}
    C_T\Vert \Theta^\varepsilon(t)-\theta(t)\Vert _{L^2}^
      {2(1-1/s)},
      &1<s<2,\\[1mm]
    C_T\Vert \Theta^\varepsilon(t)-\theta(t)\Vert _{L^2}^{2/s},
      &2\leq s<\infty.
  \end{cases}
\]
Time continuity in these spaces follows from
the same interpolation inequalities.
Taking the supremum in time and using
\eqref{eq:The-eps-conv-CL2}
therefore proves \eqref{eq:Thet-eps-the-CLs-conv}.

Finally, when $\alpha=1$,
\eqref{eq:The-eps-supp}
and \eqref{eq:the-supp-alp=1}
place both temperature supports in a common ball
$B_{R_T}$, after enlarging $R_T$ if necessary.
Hence,
$\Vert \Theta^\varepsilon(t)-\theta(t)\Vert _{L^1}\leq |B_{R_T}|^{1/2}
  \Vert \Theta^\varepsilon(t)-\theta(t)\Vert _{L^2}$.
Together with \eqref{eq:The-eps-conv-CL2},
this gives convergence in
$C([0,T];L^1(\mathbb R^2))$.
\end{proof}

\subsection{Proof of Theorem \ref{thm:infinite-Prandtl}}

Fix $T>0$.
Apply Proposition~\ref{prop:weak-limit-R2}
and Corollary~\ref{cor:strong-limit-R2}
with $\theta_0^\varepsilon=\theta_0$
and velocity exponent $p$ when $k=1$,
or $q$ when $k=2$.
Take $a^\varepsilon\equiv0$ when $\alpha<1$,
or when $\alpha=1$ and
$\int_{\mathbb R^2}\theta_0\,\dd x=0$.
In the latter case, compact support ensures
the finite first absolute moment required
in \eqref{eq:u-LQ-es}.
For nonzero total temperature at $\alpha=1$,
choose $a^\varepsilon$ by \eqref{def:a-eps}.
As $\varepsilon\to0$, the whole familly
$(U^\varepsilon,\Theta^\varepsilon)$ converges
in the senses of \eqref{eq:U-The-eps-conv-gen}
and \eqref{eq:The-eps-conv-CL2}
to the unique normalized solution $(u,\theta)$
with $u=G_\alpha*\theta$.

Since $\theta_0$ is compactly supported,
\eqref{eq:The-eps-supp} gives uniform support control
when $\alpha=1$.
For $\alpha<1$, the growth bound
\eqref{eq:U-eps-uppbdd} and Gronwall's inequality
applied to $Y_t^\varepsilon=X_t^\varepsilon$
give the same conclusion.
This support bound passes to the limit, so
there exists $R_T>0$, independent of $\varepsilon$,
such that
\begin{equation}\label{eq:limit-temperature-support}
  \operatorname{supp}\Theta^\varepsilon(t)
  \cup\operatorname{supp}\theta(t)
  \subset B_{R_T},
  \qquad 0\leq t\leq T.
\end{equation}
Together with \eqref{eq:The-eps-conv-CL2},
this gives strong convergence in
$C([0,T];L^1(\mathbb R^2))$.
Interpolation with the uniform $L^\infty$ bound
then gives strong convergence in
$C([0,T];L^r(\mathbb R^2))$
for every $1\leq r<\infty$.

Let $X_t$ denote the flow map generated by
the limiting velocity $u=G_\alpha*\theta$.
As shown in the proof of
Corollary~\ref{cor:strong-limit-R2},
$\nabla u\in L^\infty([0,T],L^\infty)$
and $u$ has sublinear growth.
Together with $\operatorname{div}u=0$,
these properties ensure that $X_t$ is a globally defined,
measure-preserving bi-Lipschitz flow on $[0,T]$.
Uniqueness for the transport equation then gives $\theta(t,x)=\theta_0(X_t^{-1}(x))$. 
Since $\theta_0=\overline\theta_0\mathds{1}_{D_0}$,
the limiting temperature retains the patch structure:
\begin{equation}\label{eq:limit-temperature-front}
  \theta(t,x)
  =
  \overline\theta_0\bigl(X_t^{-1}(x)\bigr)
  \mathds{1}_{D(t)}(x),
  \qquad D(t):=X_t(D_0).
\end{equation}

It remains to transfer the uniform boundary estimates
to the limit.
We first consider $C^{1,\gamma}$ boundary persistence.

Define the translated tangent field by  $\mathcal W^\varepsilon(t,x)
  :=W^\varepsilon(t,x+a^\varepsilon(t))$. 
It satisfies
\begin{equation}\label{eq:translated-W-base}
  \partial_t\mathcal W^\varepsilon
  +U^\varepsilon\cdot\nabla\mathcal W^\varepsilon
  =
  \mathcal W^\varepsilon\cdot\nabla U^\varepsilon,
  \qquad
  \mathcal W^\varepsilon|_{t=0}=W_0,
\end{equation}
with
$\operatorname{div}U^\varepsilon
=\operatorname{div}\mathcal W^\varepsilon=0$.
Throughout this argument,
$C^{0,1}$ is identified with $W^{1,\infty}$.

If $\frac12<\alpha<1$ and $0<\gamma\leq2\alpha-1$,
or if $\alpha=1$ and $0<\gamma<1$,
the preceding velocity estimates give $\sup\limits_{0<\varepsilon\leq1}
  \Vert \nabla U^\varepsilon\Vert _{L^1([0,T],C^{0,\gamma})}
  \leq C_T$. 
The H\"older transport estimate \eqref{eq:T-BesEs} applied to
\eqref{eq:translated-W-base}, together with
$\Vert \mathcal W^\varepsilon\cdot\nabla U^\varepsilon  \Vert _{C^{0,\gamma}}\leq C \Vert \mathcal W^\varepsilon\Vert _{C^{0,\gamma}} \Vert \nabla U^\varepsilon\Vert _{C^{0,\gamma}}$
and Gronwall's inequality gives
\begin{equation}\label{eq:W-uniform-Cgamma}
  \sup_{0<\varepsilon\leq1}
  \Vert \mathcal W^\varepsilon
    \Vert _{L^\infty([0,T],C^{0,\gamma}(\mathbb R^2))}
  \leq C_T.
\end{equation}
For the remaining cases,
$\frac12<\alpha<1$ with $2\alpha-1<\gamma\leq1$
and $\alpha=1$ with $\gamma=1$,
the uniform tangent-field estimates established
in \eqref{es:W-Cgam} and \eqref{eq:W-W1inf-bdd}
give \eqref{eq:W-uniform-Cgamma} directly,
since spatial translations preserve these norms.

By the Banach-Alaoglu theorem,
after extracting a subsequence,
there exists $\mathcal W$ such that
\[
  \mathcal W^\varepsilon\rightharpoonup^\ast\mathcal W
  \quad\text{in }
  L^\infty((0,T)\times\mathbb R^2).
\]
The uniform spatial H\"older bounds,
including the Lipschitz bound at $\gamma=1$,
pass to the limit under distributional convergence.
Thus, for every $0<\gamma\leq1$ under consideration,
\[
  \Vert \mathcal W\Vert _
    {L^\infty([0,T],C^{0,\gamma}(\mathbb R^2))}
  \leq
  \liminf_{\varepsilon\to0}
  \Vert \mathcal W^\varepsilon\Vert _
    {L^\infty([0,T],C^{0,\gamma}(\mathbb R^2))}
  \leq C_T.
\]
At $\gamma=1$, we may also extract the subsequence so that
\[
  \nabla\mathcal W^\varepsilon
  \rightharpoonup^\ast\nabla\mathcal W
  \quad\text{in }
  L^\infty((0,T)\times\mathbb R^2),
\]
with the limit identified by distributional differentiation.

We next establish strong local compactness.
Since $U^\varepsilon$ and $\mathcal W^\varepsilon$
are divergence-free,
\eqref{eq:translated-W-base} can be written as $\partial_t\mathcal W^\varepsilon
  =
  \operatorname{div}\bigl(
    U^\varepsilon\otimes\mathcal W^\varepsilon
    -\mathcal W^\varepsilon\otimes U^\varepsilon
  \bigr)$. 
For every $R>0$, H\"older's inequality,
\eqref{eq:local-W1p-Ueps}, and \eqref{eq:W-uniform-Cgamma}
therefore give
\begin{equation}\label{eq:W-time-bound}
\begin{aligned}
  \Vert \partial_t\mathcal W^\varepsilon
    \Vert _{L^\infty([0,T],W^{-1,p}(B_R))}
  \leq
  C\Vert U^\varepsilon\Vert _{L^\infty([0,T],L^p(B_R))}
  \Vert \mathcal W^\varepsilon
    \Vert _{L^\infty((0,T)\times B_R)}
  \leq C_{T,R},
\end{aligned}
\end{equation}
uniformly in $\varepsilon$.

For every $0<\gamma'<\gamma\leq1$, we have
$C^{0,\gamma}(\overline{B_R})
  \hookrightarrow\hookrightarrow
  C^{0,\gamma'}(\overline{B_R})
  \hookrightarrow W^{-1,p}(B_R)$.
The Aubin-Lions lemma,
applied with \eqref{eq:W-uniform-Cgamma}
and \eqref{eq:W-time-bound}, therefore gives,
after extracting a further subsequence,
\begin{equation}\label{eq:W-strong-local}
  \mathcal W^\varepsilon\longrightarrow\mathcal W
  \quad\text{strongly in }
  C\bigl([0,T];C^{0,\gamma'}(\overline{B_R})\bigr).
\end{equation}
The limit agrees with the preceding weak-$*$ limit
by uniqueness of the distributional limit.
A diagonal argument over integer radii
and a sequence of exponents increasing to $\gamma$
gives a common subsequence for every $R>0$
and every $0<\gamma'<\gamma$.

We now identify the equation satisfied by the limit.
For every
$\psi\in C_c^\infty([0,T)\times\mathbb R^2;\mathbb R^2)$,
\eqref{eq:translated-W-base} gives
\begin{align*}
  \int_0^T\int_{\mathbb R^2}
  \mathcal W^\varepsilon\cdot\partial_t\psi\,\dd x\,\dd t
  +
  \int_0^T\int_{\mathbb R^2}
  \bigl(
    \mathcal W^\varepsilon\otimes U^\varepsilon
    -U^\varepsilon\otimes\mathcal W^\varepsilon
  \bigr):\nabla\psi\,\dd x\,\dd t
  +\int_{\mathbb R^2}W_0\cdot\psi(0)\,\dd x=0.
\end{align*}
The strong local convergence \eqref{eq:W-strong-local}
and the weak-$*$ convergence \eqref{eq:local-weak-Ueps}
allow passage to the limit in the product terms.
Thus
\begin{equation}\label{eq:limit-W-base}
  \partial_t\mathcal W+u\cdot\nabla\mathcal W
  =\mathcal W\cdot\nabla u,
  \qquad
  \mathcal W|_{t=0}=W_0,
\end{equation}
in the sense of distributions.
The divergence-free condition also passes to the limit,
so $\operatorname{div}\mathcal W=0$.

The kernel estimates imply that $\nabla u(t,\cdot)$
is continuous for almost every $t$.
Together with the preceding gradient and growth bounds,
this ensures that $X_t$ is continuously differentiable
in space.
Uniqueness for \eqref{eq:limit-W-base} along the flow gives $\mathcal W(t,X_t(x))=DX_t(x)W_0(x)$, see \cite[Lemma 1.4]{MB02}.
For $z(a,t):=X_t(z_0(a))$, with
$\partial_a z_0=W_0\circ z_0$, it follows that $\partial_a z(a,t)=\mathcal W(t,z(a,t))$. 

Since
$\partial_t\partial_a z(a,t) =\nabla u(t,z(a,t))\,\partial_a z(a,t)$, we obtain
\[
  |\partial_a z(a,t)|
  \geq |\partial_a z_0(a)|
  \exp\left(
    -\int_0^t\Vert \nabla u(\tau)\Vert _{L^\infty}\,\dd\tau
  \right),
\]
so the boundary tangent remains uniformly nondegenerate
on $[0,T]$.

The uniform $C^{0,\gamma}$ bound on $\mathcal W$,
together with the flow bounds, now gives
\begin{equation}\label{eq:limit-C1gamma}
  \sup_{0\leq t\leq T}
  \Vert z(\cdot,t)\Vert _{C^{1,\gamma}}
  \leq C_T.
\end{equation}
This proves $C^{1,\gamma}$ boundary persistence
for the limiting patch, including
$W^{2,\infty}$ persistence when $\gamma=1$.

We next prove $C^{2,\gamma}$ boundary persistence.
The commutation identity
$[\partial_t+U^\varepsilon\cdot\nabla, \partial_{\mathcal W^\varepsilon}]=0$
and \eqref{eq:translated-W-base} give
\begin{equation}\label{eq:Z2-evolution}
  \partial_t
    (\partial_{\mathcal W^\varepsilon}\mathcal W^\varepsilon)
  +U^\varepsilon\cdot\nabla
    (\partial_{\mathcal W^\varepsilon}\mathcal W^\varepsilon)
  =
  \partial_{\mathcal W^\varepsilon}^{\,2}U^\varepsilon.
\end{equation}
The uniform striated estimates established
in the proof of Theorem~\ref{thm:main2} yield
\begin{equation}\label{eq:Z2-uniform}
  \sup_{0<\varepsilon\leq1}
  \Vert \partial_{\mathcal W^\varepsilon}\mathcal W^\varepsilon
    \Vert _{L^\infty([0,T],C^{0,\gamma}(\mathbb R^2))}
  \leq C_T.
\end{equation}
Indeed,
$\partial_{\mathcal W^\varepsilon}\mathcal W^\varepsilon(t,x)=
  (\partial_{W^\varepsilon}W^\varepsilon) (t,x+a^\varepsilon(t))$, 
and spatial translations preserve these norms.

To establish time compactness, we use
$\operatorname{div}\mathcal W^\varepsilon=0$ to write
$\partial_{\mathcal W^\varepsilon}^{\,2}U^\varepsilon
  =
  \operatorname{div}\bigl(
    (\partial_{\mathcal W^\varepsilon}U^\varepsilon)
    \otimes\mathcal W^\varepsilon
  \bigr)$.
Since $\operatorname{div}U^\varepsilon=0$,
\eqref{eq:Z2-evolution} becomes
\[
  \partial_t
    (\partial_{\mathcal W^\varepsilon}\mathcal W^\varepsilon)
  =
  \operatorname{div}\Bigl(
    (\partial_{\mathcal W^\varepsilon}U^\varepsilon)
      \otimes\mathcal W^\varepsilon
    -
    (\partial_{\mathcal W^\varepsilon}\mathcal W^\varepsilon)
      \otimes U^\varepsilon
  \Bigr).
\]
For every $R>0$, it follows that
\begin{align*}
  \bigl\Vert 
    \partial_t
    (\partial_{\mathcal W^\varepsilon}\mathcal W^\varepsilon)
  \bigr\Vert _{W^{-1,q}(B_R)}
  \leq
  C\Vert \mathcal W^\varepsilon\Vert _{L^\infty(B_R)}^2
  \Vert \nabla U^\varepsilon\Vert _{L^q(B_R)}
  +  C\Vert \partial_{\mathcal W^\varepsilon}\mathcal W^\varepsilon
    \Vert _{L^\infty(B_R)}
  \Vert U^\varepsilon\Vert _{L^q(B_R)}.
\end{align*}
The uniform local $W^{1,q}$ bounds for $U^\varepsilon$,
together with \eqref{eq:W-uniform-Cgamma}
and \eqref{eq:Z2-uniform}, therefore give
\begin{equation}\label{eq:Z2-time}
  \sup_{0<\varepsilon\leq1}
  \bigl\Vert 
    \partial_t
    (\partial_{\mathcal W^\varepsilon}\mathcal W^\varepsilon)
  \bigr\Vert _{L^\infty([0,T],W^{-1,q}(B_R))}
  \leq C_{T,R}.
\end{equation}

For $0<\gamma_2<\gamma\leq1$, the embeddings
$C^{0,\gamma}(\overline{B_R})
  \hookrightarrow\hookrightarrow
  C^{0,\gamma_2}(\overline{B_R})
  \hookrightarrow W^{-1,q}(B_R)$,
together with the spatial and time-derivative bounds,
allow us to apply the Aubin-Lions lemma.
After a diagonal extraction,
\begin{equation}\label{eq:Z2-strong}
  \partial_{\mathcal W^\varepsilon}\mathcal W^\varepsilon
  \longrightarrow\mathcal Z
  \quad\text{strongly in }
  C\bigl([0,T];C^{0,\gamma_2}(\overline{B_R})\bigr)
\end{equation}
for every $R>0$ and $0<\gamma_2<\gamma$.

Since $\operatorname{div}\mathcal W^\varepsilon=0$, we have
$ \partial_{\mathcal W^\varepsilon}\mathcal W^\varepsilon
  =
  \operatorname{div}
  (\mathcal W^\varepsilon\otimes\mathcal W^\varepsilon)$.
Passing to the limit using
\eqref{eq:W-strong-local} and \eqref{eq:Z2-strong} gives $\mathcal Z
  =
  \operatorname{div}(\mathcal W\otimes\mathcal W)
  =
  \partial_{\mathcal W}\mathcal W$ 
in the sense of distributions.
Lower semicontinuity of the spatial H\"older norm gives $\Vert \partial_{\mathcal W}\mathcal W
    \Vert _{L^\infty([0,T],C^{0,\gamma}(\mathbb R^2))}
  \leq C_T$. 
At $\gamma=1$, this is the corresponding
$W^{1,\infty}$ bound, which can also be obtained
by weak-$*$ convergence of the first derivatives.

Finally, arguing as in the derivation of~\eqref{eq:boundary-identities},  the uniform bounds for $\mathcal W$ and
$\partial_{\mathcal W}\mathcal W$,
together with the previously established tangent
nondegeneracy, yield $C^{2,\gamma}$ boundary persistence,
including the $W^{3,\infty}$ endpoint when $\gamma=1$.
This completes the proof of
Theorem~\ref{thm:infinite-Prandtl}.

\appendix
\addtocontents{toc}{\protect\setcounter{tocdepth}{1}}

\section{Regularization of the initial data}\label{sec:A}

This appendix constructs smooth approximations
of the initial data compatible with the {\it{a priori}} estimates
in Sections~\ref{sec:well-pose}-\ref{sec:C2}.
We first recall the approximation scheme used
in the existence proof of Theorem~\ref{thm:main1}.
We then establish the additional uniform striated bounds
required for boundary regularity persistence.

\subsection{Smooth approximation of the initial data
for the existence theory}\label{sec:A1}

Let $\theta_0\in L^1\cap L^\infty(\mathbb R^2)$
and $u_0\in W^{1,p}(\mathbb R^2)$,
where $2\leq p<\infty$ and $\operatorname{div}u_0=0$.
Choose a nonnegative mollifier
$\rho\in C_c^\infty(\mathbb R^2)$ with
$\int_{\mathbb R^2}\rho\,\dd x=1$, and set $\rho_\delta(x):=\delta^{-2}\rho(x/\delta)$. 
Let $\chi\in C_c^\infty(\mathbb R^2)$ satisfy
$0\leq\chi\leq1$, $\chi=1$ on $B_1$, and
$\operatorname{supp}\chi\subset B_2$,
and define $\chi_\delta(x):=\chi(\delta x)$.
Denoting the Leray projector by $\mathbb P$, set $\theta_{0,\delta}:=\rho_\delta*\theta_0,\,u_{0,\delta}:=\rho_\delta*\mathbb P(\chi_\delta u_0)$. 
Then $\operatorname{div}u_{0,\delta}=0$, and,
for every fixed $\delta>0$,
\[
  u_{0,\delta},\theta_{0,\delta}
  \in H^m(\mathbb R^2),
  \qquad m\geq0.
\]
Indeed, $\chi_\delta u_0$ is compactly supported
and belongs to $L^2$, while $\mathbb P$ is bounded
on $L^2$.
Also, $\theta_0\in L^2$ because
$\theta_0\in L^1\cap L^\infty$.
The asserted regularity therefore follows
from the smoothing property of convolution.

As $\delta\to0$, we have
\[
  u_{0,\delta}\longrightarrow u_0
  \quad\text{strongly in }W^{1,p}(\mathbb R^2),
  \qquad
  \theta_{0,\delta}\longrightarrow\theta_0
  \quad\text{strongly in }L^r(\mathbb R^2),
  \quad 1\leq r<\infty.
\]
For the velocity, since $\mathbb P u_0=u_0$
and $\mathbb P$ commutes with convolution
and is bounded on $W^{1,p}$,
\begin{align*}
  \Vert u_{0,\delta}-u_0\Vert _{W^{1,p}}
  &\leq
  C\Vert \rho_\delta*(\chi_\delta u_0)-u_0\Vert _{W^{1,p}}
  \\
  &\leq
  C\Vert \chi_\delta u_0-u_0\Vert _{W^{1,p}}
  +C\Vert \rho_\delta*u_0-u_0\Vert _{W^{1,p}}
  \longrightarrow0.
\end{align*}
Here, $\chi_\delta u_0\to u_0$ in $W^{1,p}$
follows from dominated convergence and
$\Vert \nabla\chi_\delta\,u_0\Vert _{L^p}\leq C\delta\Vert u_0\Vert _{L^p}\longrightarrow 0$.
The temperature convergence follows from
$\theta_0\in L^r$ for every $1\leq r<\infty$
and the approximation property of mollifiers.

Moreover,
\begin{align*}
  \sup_{0<\delta\leq1}
  \Vert u_{0,\delta}\Vert _{W^{1,p}}
  \leq C\Vert u_0\Vert _{W^{1,p}},\qquad
  \sup_{0<\delta\leq1}
  \Vert \theta_{0,\delta}\Vert _{L^1\cap L^\infty}
  \leq \Vert \theta_0\Vert _{L^1\cap L^\infty}.
\end{align*}
The temperature mollification also preserves its integral
and controls its first absolute moment:
\[
  \int_{\mathbb R^2}\theta_{0,\delta}\,\dd x
  =\int_{\mathbb R^2}\theta_0\,\dd x,
  \qquad
  \int_{\mathbb R^2}|x|\,|\theta_{0,\delta}(x)|\,\dd x
  \leq \int_{\mathbb R^2}|x|\,|\theta_0(x)|\,\dd x
       +C_\rho\delta\Vert \theta_0\Vert _{L^1},
\]
whenever the initial first absolute moment is finite,
where $C_\rho:=\int_{\mathbb R^2}|y|\rho(y)\,\dd y$.
Thus the additional hypotheses of \eqref{eq:u-LQ-es}
at $\alpha=1$ are preserved uniformly in $\delta$.
These are the approximations and uniform bounds
used in the proof of Theorem~\ref{thm:main1}.
The velocity convergence and uniform bound also hold
with $p$ replaced by any $q\in(1,\infty)$
for which $u_0\in W^{1,q}(\mathbb R^2)$.

\subsection{Regularization of the initial temperature
and tangent vector field}\label{sec:A2}

Recall that, in Section~\ref{subsec:intr-patch-bdr},
we chose the initial level-set function $\varphi_0$
to be constant outside a sufficiently large ball
before imposing the striated regularity assumptions
on the initial velocity.
Set
\[
  W_0:=\nabla^\perp\varphi_0,
  \qquad K:=\operatorname{supp}W_0\Subset\mathbb R^2.
\]
For the boundary regularity argument, we also define $\varphi_{0,\delta}:=\rho_\delta*\varphi_0,\,W_{0,\delta}:=\nabla^\perp\varphi_{0,\delta}
  =\rho_\delta*W_0$. 
Then $\theta_{0,\delta}$, $\varphi_{0,\delta}$,
and $W_{0,\delta}$ are smooth, and
 $\operatorname{div}W_{0,\delta}=0,\,W_{0,\delta}\cdot\nabla\varphi_{0,\delta}=0$. 
The uniform continuity of $W_0$, supplied by
the boundary regularity assumptions, also gives
$W_{0,\delta}\to W_0$ uniformly on $\mathbb R^2$.

We begin with a commutator estimate for mollification
that will be used repeatedly below.

\begin{lemma}\label{lem:mollified-striated}
Let $\rho_\delta$ be a standard mollifier,
$W\in L^\infty(\mathbb R^2;\mathbb R^2)$
a divergence-free vector field, and
$f\in L^1_{\mathrm{loc}}(\mathbb R^2)$.
Set $W_\delta:=\rho_\delta*W,\,f_\delta:=\rho_\delta*f$, 
and define $\partial_Wf:=\operatorname{div}(Wf)$.
Then
\[
  \partial_{W_\delta}f_\delta
  =
  (\partial_Wf)*\rho_\delta+\mathcal R_\delta(f,W)
\]
in the sense of distributions, where
\[
  \mathcal R_\delta(f,W)(x)
  :=
  \int_{\mathbb R^2}
  f(x-y)\bigl(W_\delta(x)-W(x-y)\bigr)
  \cdot\nabla\rho_\delta(y)\,\dd y.
\]
The following estimates hold uniformly for
$0<\delta\leq1$:
\begin{enumerate}
\item
If $W\in W^{1,\infty}$ and $f\in L^p$,
with $1\leq p\leq\infty$, then
\begin{equation}\label{eq:Rdelta-comm1}
  \Vert \mathcal R_\delta(f,W)\Vert _{L^p}
  \leq C\Vert \nabla W\Vert _{L^\infty}\Vert f\Vert _{L^p}.
\end{equation}

\item
If $W\in C^{0,\beta}$ and $f\in L^\infty$,
with $0<\beta<1$, then
\begin{equation}\label{eq:Rdelta-comm8}
  \Vert \mathcal R_\delta(f,W)\Vert _{C^{\beta-1}}
  \leq C\Vert \ W\Vert _{C^{0,\beta}}\Vert f\Vert _{L^\infty},
\end{equation}
where $C^{\beta-1}:=B^{\beta-1}_{\infty,\infty}$.
\end{enumerate}
\end{lemma}

\begin{proof}
Since $\operatorname{div}W=\operatorname{div}W_\delta=0$, one can get
\begin{align}\label{id:comm-pWf}
  \partial_{W_\delta}f_\delta
  -(\partial_Wf)*\rho_\delta
  =
  \operatorname{div}\bigl(
    W_\delta f_\delta-\rho_\delta*(Wf)
  \bigr)
  =\mathcal R_\delta(f,W).
\end{align}
If $W\in W^{1,\infty}$, then $|W_\delta(x)-W(x-y)|
  \leq C\delta\Vert \nabla W\Vert _{L^\infty}$ for $y\in\operatorname{supp}\nabla\rho_\delta$. 
Together with
$\Vert \nabla\rho_\delta\Vert _{L^1}\leq C\delta^{-1}$,
Young's inequality gives the first estimate \eqref{eq:Rdelta-comm1}.

For the second estimate, set $F_\delta:=W_\delta f_\delta-\rho_\delta*(Wf)$. 
The representation
\[
  F_\delta(x)
  =
  \int_{\mathbb R^2}
  \rho_\delta(y)
  \bigl(W_\delta(x)-W(x-y)\bigr)f(x-y)\,\dd y
\]
and the standard mollification estimates yield
\[
  \Vert F_\delta\Vert _{L^\infty}
  +\delta\Vert \nabla F_\delta\Vert _{L^\infty}
  \leq
  C\delta^\beta
  \Vert W\Vert _{C^{0,\beta}}\Vert f\Vert _{L^\infty}.
\]
H\"older interpolation therefore gives
\begin{align*}
  \Vert F_\delta\Vert _{C^{0,\beta}}
  \leq
  C\left(
    \Vert F_\delta\Vert _{L^\infty}
    +
    \Vert F_\delta\Vert _{L^\infty}^{1-\beta}
    \Vert \nabla F_\delta\Vert _{L^\infty}^{\beta}
  \right)
  \leq
  C\Vert W\Vert _{C^{0,\beta}}\Vert f\Vert _{L^\infty},
\end{align*}
where we used $0<\delta\leq1$. Thanks to \eqref{id:comm-pWf}, one can get $\mathcal R_\delta(f,W)=\operatorname{div}F_\delta$, which gives
\eqref{eq:Rdelta-comm8}.
\end{proof}

We now verify that the initial striated norms
of the regularized data are bounded uniformly in $\delta$.

Let $k=2$, $\frac12<\alpha<1$, and
$2\alpha-1<\gamma<1$.
Applying Lemma~\ref{lem:mollified-striated}
with $f=\theta_0$ and $W=W_0$, and using
the boundedness of mollification on
$C^{\gamma-2\alpha}
=B^{\gamma-2\alpha}_{\infty,\infty}$, we obtain
\[
  \Vert \partial_{W_{0,\delta}}\theta_{0,\delta}
    \Vert _{C^{\gamma-2\alpha}}
  \leq
  C\Vert \partial_{W_0}\theta_0\Vert _{C^{\gamma-2\alpha}}
  +C\Vert W_0\Vert _{C^{0,\gamma}}\Vert \theta_0\Vert _{L^\infty}.
\]
Here we used
$C^{\gamma-1}\hookrightarrow C^{\gamma-2\alpha}$,
since $2\alpha>1$.

At $\gamma=1$, let $\frac12<\alpha\leq1$
and $2-2\alpha<\mu<1$.
Since $W_0\in W^{1,\infty}$,
the first estimate in Lemma~\ref{lem:mollified-striated}
and the embedding $L^\infty\hookrightarrow C^{\mu-1}$
give
\[
  \Vert \partial_{W_{0,\delta}}\theta_{0,\delta}
    \Vert _{C^{\mu-1}}
  \leq
  C\Vert \partial_{W_0}\theta_0\Vert _{C^{\mu-1}}
  +C\Vert \nabla W_0\Vert _{L^\infty}\Vert \theta_0\Vert _{L^\infty}.
\]
Both estimates are uniform in $\delta$.

We next consider $k=3$.
Suppose that $\frac12<\alpha<1$,
$2\alpha-1<\gamma<1$, and $W_0\in C^{1,\gamma}$.
Applying Lemma~\ref{lem:mollified-striated}
with $p=1$ and $p=\infty$ gives
\[
  \Vert \partial_{W_{0,\delta}}\theta_{0,\delta}
    \Vert _{L^1\cap L^\infty}
  \leq
  C\Bigl(
    \Vert \partial_{W_0}\theta_0\Vert _{L^1\cap L^\infty}
    +
    \Vert \nabla W_0\Vert _{L^\infty}
    \Vert \theta_0\Vert _{L^1\cap L^\infty}
  \Bigr).
\]

To estimate the second tangential derivative,
apply $\partial_{W_{0,\delta}}$ to the identity
in Lemma~\ref{lem:mollified-striated}.
Using the same lemma with $f=\partial_{W_0}\theta_0$,
we have
\[
  \partial_{W_{0,\delta}}
  \bigl((\partial_{W_0}\theta_0)*\rho_\delta\bigr)
  =
  (\partial_{W_0}^2\theta_0)*\rho_\delta
  +\mathcal R_\delta(\partial_{W_0}\theta_0,W_0).
\]
Differentiating $\mathcal R_\delta(\theta_0,W_0)$
along $W_{0,\delta}$ and integrating by parts,
using $\operatorname{div}W_0=0$, gives
\begin{align*}
  \partial_{W_{0,\delta}}^2\theta_{0,\delta}(x)
  ={}&
  \bigl((\partial_{W_0}^2\theta_0)*\rho_\delta\bigr)(x)
  \\
  &+
  2\int_{\mathbb R^2}
  \partial_{W_0}\theta_0(x-y)
  \bigl(W_{0,\delta}(x)-W_0(x-y)\bigr)
  \cdot\nabla\rho_\delta(y)\,\dd y
  \\
  &+
  \int_{\mathbb R^2}
  \theta_0(x-y)
  \bigl(
    \partial_{W_{0,\delta}}W_{0,\delta}(x)
    -\partial_{W_0}W_0(x-y)
  \bigr)
  \cdot\nabla\rho_\delta(y)\,\dd y
  \\
  &+
  \sum_{i,j=1}^2
  \int_{\mathbb R^2}
  \theta_0(x-y)
  \bigl(W_{0,\delta}^i(x)-W_0^i(x-y)\bigr)
  \bigl(W_{0,\delta}^j(x)-W_0^j(x-y)\bigr)
  \partial_{ij}\rho_\delta(y)\,\dd y
  \\
  :={}&D_1+D_2+D_3+D_4.
\end{align*}

The boundedness of mollification controls $D_1$,
while Lemma~\ref{lem:mollified-striated} controls
$D_2=2\mathcal R_\delta(\partial_{W_0}\theta_0,W_0)$.
For $D_4$, we use
\[
  |W_{0,\delta}(x)-W_0(x-y)|
  \leq C\delta\Vert \nabla W_0\Vert _{L^\infty}
  \qquad
  \text{for }y\in\operatorname{supp}\nabla^2\rho_\delta,
\]
together with
$\Vert \nabla^2\rho_\delta\Vert _{L^1}\leq C\delta^{-2}$.
Consequently,
\begin{align*}
  \Vert D_1\Vert _{C^{\gamma-2\alpha}}
  +\Vert D_2\Vert _{L^\infty}
  +\Vert D_4\Vert _{L^\infty}
  \leq
  C\Bigl(
    \Vert \partial_{W_0}^2\theta_0\Vert _{C^{\gamma-2\alpha}}
    +
    \Vert \nabla W_0\Vert _{L^\infty}
    \Vert \partial_{W_0}\theta_0\Vert _{L^\infty}
    +
    \Vert \nabla W_0\Vert _{L^\infty}^2
    \Vert \theta_0\Vert _{L^\infty}
  \Bigr).
\end{align*}

It remains to estimate $D_3$.
We write
\begin{align*}
  D_3
  ={}
  \operatorname{div}\Bigl[
    \partial_{W_{0,\delta}}W_{0,\delta}\,
      \theta_{0,\delta}
    -
    \rho_\delta*
    \bigl((\partial_{W_0}W_0)\theta_0\bigr)
  \Bigr]
  -
  \operatorname{div}
  (\partial_{W_{0,\delta}}W_{0,\delta})
  \,\theta_{0,\delta}.
\end{align*}
Since $W_0\in C^{1,\gamma}$, for
$y\in\operatorname{supp}\rho_\delta$ we have
\begin{align*}
  \bigl|
    \partial_{W_{0,\delta}}W_{0,\delta}(x)
    -\partial_{W_0}W_0(x-y)
  \bigr|
  &\leq
  |W_{0,\delta}(x)-W_0(x-y)|
  \Vert \nabla W_{0,\delta}\Vert _{L^\infty}
  \\
  &\quad+
  |W_0(x-y)|
  |\nabla W_{0,\delta}(x)-\nabla W_0(x-y)|
  \\
  &\leq
  C\delta^\gamma\Vert W_0\Vert _{C^{1,\gamma}}^2.
\end{align*}

The cancellation
$\int_{\mathbb R^2}\nabla\rho_\delta\,\dd y=0$
also gives
\[
  \Vert \nabla^2W_{0,\delta}\Vert _{L^\infty}
  \leq
  [\nabla W_0]_{C^{0,\gamma}}
  \int_{\mathbb R^2}
  |\nabla\rho_\delta(y)|\,|y|^\gamma\,\dd y
  \leq
  C\delta^{\gamma-1}\Vert W_0\Vert _{C^{1,\gamma}}.
\]
Consequently,
\begin{align*}
  \Vert \nabla(\partial_{W_{0,\delta}}W_{0,\delta})
    \Vert _{L^\infty}
  \leq
  C\Vert \nabla W_{0,\delta}\Vert _{L^\infty}^2
  +
  C\Vert W_{0,\delta}\Vert _{L^\infty}
  \Vert \nabla^2W_{0,\delta}\Vert _{L^\infty}
  \leq
  C\delta^{\gamma-1}\Vert W_0\Vert _{C^{1,\gamma}}^2,
\end{align*}
where we used $0<\delta\leq1$ and $\gamma<1$.

Differentiating the convolutions gives, component wise,
\begin{align*}
  &\partial_i\Bigl[
    \partial_{W_{0,\delta}}W_{0,\delta}\,
      \theta_{0,\delta}
    -
    \rho_\delta*
    \bigl((\partial_{W_0}W_0)\theta_0\bigr)
  \Bigr] \\
  & =
  \partial_i(\partial_{W_{0,\delta}}W_{0,\delta})
  \,\theta_{0,\delta}
  +
  \int_{\mathbb R^2}
  \partial_i\rho_\delta(y)\theta_0(x-y)
  \bigl(
    \partial_{W_{0,\delta}}W_{0,\delta}(x)
    -\partial_{W_0}W_0(x-y)
  \bigr)\,\dd y.
\end{align*}
The preceding estimates therefore imply
\begin{align*}
  &\Bigl\Vert 
    \partial_{W_{0,\delta}}W_{0,\delta}\,
      \theta_{0,\delta}
    -
    \rho_\delta*
    \bigl((\partial_{W_0}W_0)\theta_0\bigr)
  \Bigr\Vert _{L^\infty}
  \\
  &+
  \delta\Bigl\Vert 
    \nabla\Bigl[
      \partial_{W_{0,\delta}}W_{0,\delta}\,
        \theta_{0,\delta}
      -
      \rho_\delta*
      \bigl((\partial_{W_0}W_0)\theta_0\bigr)
    \Bigr]
  \Bigr\Vert _{L^\infty}
  \leq
  C\delta^\gamma
  \Vert W_0\Vert _{C^{1,\gamma}}^2
  \Vert \theta_0\Vert _{L^\infty}.
\end{align*}
H\"older interpolation, together with the
$L^\infty$ bound, then gives
\begin{align*}
  \Bigl\Vert 
    \operatorname{div}\Bigl[
      \partial_{W_{0,\delta}}W_{0,\delta}\,
        \theta_{0,\delta}
      -
      \rho_\delta*
      \bigl((\partial_{W_0}W_0)\theta_0\bigr)
    \Bigr]
  \Bigr\Vert _{C^{\gamma-1}}
  \leq
  C\Vert W_0\Vert _{C^{1,\gamma}}^2
  \Vert \theta_0\Vert _{L^\infty}.
\end{align*}

For the remaining term, the divergence-free condition give
$\operatorname{div}
  (\partial_{W_{0,\delta}}W_{0,\delta})
  =
  \sum_{i,j=1}^2
  \partial_iW_{0,\delta}^j
  \partial_jW_{0,\delta}^i$.
Hence
\[
  \bigl\Vert 
    \operatorname{div}
    (\partial_{W_{0,\delta}}W_{0,\delta})
    \,\theta_{0,\delta}
  \bigr\Vert _{L^\infty}
  \leq
  C\Vert \nabla W_0\Vert _{L^\infty}^2
  \Vert \theta_0\Vert _{L^\infty}.
\]
Since $L^\infty\hookrightarrow C^{\gamma-1}$,
we conclude that  $\Vert D_3\Vert _{C^{\gamma-1}}
  \leq
  C\Vert W_0\Vert _{C^{1,\gamma}}^2
  \Vert \theta_0\Vert _{L^\infty}$. 

Combining this with the estimates for $D_1$, $D_2$,
and $D_4$, and using the embeddings
$L^\infty\hookrightarrow C^{\gamma-2\alpha}$
and $C^{\gamma-1}\hookrightarrow C^{\gamma-2\alpha}$,
we obtain
\begin{align*}
  \sup_{0<\delta\leq1}
  \Vert \partial_{W_{0,\delta}}^2\theta_{0,\delta}
    \Vert _{C^{\gamma-2\alpha}}
  \leq
  C\Bigl(
    \Vert \partial_{W_0}^2\theta_0\Vert _{C^{\gamma-2\alpha}}
    +
    \Vert \nabla W_0\Vert _{L^\infty}
    \Vert \partial_{W_0}\theta_0\Vert _{L^\infty}
    +
    \Vert W_0\Vert _{C^{1,\gamma}}^2
    \Vert \theta_0\Vert _{L^\infty}
  \Bigr).
\end{align*}

Finally, let $\gamma=1$, $\frac12<\alpha\leq1$,
$2-2\alpha<\mu<1$, and $W_0\in W^{2,\infty}$.
The estimates for $D_2$ and $D_4$ remain unchanged,
while $D_1$ is controlled by the boundedness
of mollification on $C^{\mu-1}$.
Moreover,
\[
  \bigl|
    \partial_{W_{0,\delta}}W_{0,\delta}(x)
    -\partial_{W_0}W_0(x-y)
  \bigr|
  \leq C\delta\Vert W_0\Vert _{W^{2,\infty}}^2
\]
for $y\in\operatorname{supp}\rho_\delta$.
Thus the integral representation of $D_3$ gives $\Vert D_3\Vert _{L^\infty}
  \leq C\Vert W_0\Vert _{W^{2,\infty}}^2
  \Vert \theta_0\Vert _{L^\infty}$. 
Using $L^\infty\hookrightarrow C^{\mu-1}$,
we obtain
\begin{align*}
  \sup_{0<\delta\leq1}
  \Vert \partial_{W_{0,\delta}}^2\theta_{0,\delta}
    \Vert _{C^{\mu-1}}
  \leq
  C\Bigl(
    \Vert \partial_{W_0}^2\theta_0\Vert _{C^{\mu-1}}
    +
    \Vert \nabla W_0\Vert _{L^\infty}
    \Vert \partial_{W_0}\theta_0\Vert _{L^\infty}
    +
    \Vert W_0\Vert _{W^{2,\infty}}^2
    \Vert \theta_0\Vert _{L^\infty}
  \Bigr).
\end{align*}

\subsection{Regularization of the striated derivative
of the initial velocity}
\label{sec:A3}

Recall the approximations introduced in
Sections~\ref{sec:A1}-\ref{sec:A2}.
We take $\operatorname{supp}\rho\subset B_1$.
Under the assumptions of Section~\ref{sec:C2},
$W_0\in W^{1,\infty}(\mathbb R^2)$
and $K:=\operatorname{supp}W_0$ is compact.
Set $v_{0,\delta}:=\mathbb P(\chi_\delta u_0),\,u_{0,\delta}:=\rho_\delta*v_{0,\delta},\, r:=\min\{p,q\}$.
We prove that
\begin{equation}\label{eq:velocity-regularization-uniform}
  \sup_{0<\delta\leq1}
  \left(
    \Vert u_{0,\delta}\Vert _{W^{1,q}}
    +
    \Vert \partial_{W_{0,\delta}}u_{0,\delta}
      \Vert _{\dot W^{1,r}}
  \right)
  \leq
  C\left(
    \Vert u_0\Vert _{W^{1,q}}
    +
    \Vert \partial_{W_0}u_0\Vert _{\dot W^{1,p}}
  \right),
\end{equation}
where $C$ may depend on $p,q,\rho,\chi,K$,
and $\Vert W_0\Vert _{W^{1,\infty}}$,
but is independent of $\delta$.

The uniform $W^{1,q}$ bound, smoothness of $u_{0,\delta}$,
and strong convergence $u_{0,\delta}\to u_0$ in $W^{1,q}$
follow from Section~\ref{sec:A1}.
It remains to estimate the striated derivative.

Since $\partial_{W_0}u_0$ and its distributional gradient
are supported in $K$, H\"older's inequality gives
\[
  \Vert \partial_{W_0}u_0\Vert _{\dot W^{1,r}}
  \leq C_K
  \Vert \partial_{W_0}u_0\Vert _{\dot W^{1,p}}.
\]
We choose $\delta_0\in(0,1)$ sufficiently small
so that $\chi_\delta=1$ on a fixed neighborhood
of $K+B_1$ for $0<\delta\leq\delta_0$.
On this neighborhood,
\[
  v_{0,\delta}-u_0
  =
  -\nabla\Delta^{-1}
  (\nabla\chi_\delta\cdot u_0),
\]
because $\operatorname{div}u_0=0$.
The scaling of $\chi_\delta$ and H\"older's inequality give $\Vert \nabla\chi_\delta\cdot u_0\Vert _{L^1}
  \leq
  C\delta^{-1+2/q}\Vert u_0\Vert _{L^q}$. 

Let $G(x):=(2\pi)^{-1}\log|x|$ be the fundamental
solution of the Laplacian on $\mathbb R^2$.
For $m=1,2$,
\[
  |\nabla^{m+1}G(z)|
  \leq C_m|z|^{-m-1},
  \qquad z\neq0.
\]
After decreasing $\delta_0$ if necessary, $|x-y|\geq c\delta^{-1}$ 
whenever $x\in K+B_1$,
$y\in\operatorname{supp}\nabla\chi_\delta$,
and $0<\delta\leq\delta_0$.
Thus
\begin{align*}
  |\nabla^m(v_{0,\delta}-u_0)(x)|
  =
  \left|
    \int_{\mathbb R^2}
    \nabla^{m+1}G(x-y)
    (\nabla\chi_\delta\cdot u_0)(y)\,\dd y
  \right|
  \leq
  C\delta^{m+1}
  \Vert \nabla\chi_\delta\cdot u_0\Vert _{L^1}.
\end{align*}
Consequently,
\begin{equation}\label{eq:far-field}
  \Vert \nabla^m(v_{0,\delta}-u_0)
    \Vert _{L^\infty(K+B_1)}
  \leq
  C\delta^{m+2/q}\Vert u_0\Vert _{L^q},
  \qquad m=1,2.
\end{equation}
Using $\operatorname{supp}W_0\subset K$, we obtain $\Vert \partial_{W_0}v_{0,\delta}\Vert _{\dot W^{1,r}}
  \leq
  C\left(
    \Vert \partial_{W_0}u_0\Vert _{\dot W^{1,p}}
    +
    \Vert W_0\Vert _{W^{1,\infty}}\Vert u_0\Vert _{L^q}
  \right)$. 

Lemma~\ref{lem:mollified-striated} gives $\partial_{W_{0,\delta}}u_{0,\delta}
  =
  \rho_\delta*(\partial_{W_0}v_{0,\delta})
  +\mathcal R_\delta(v_{0,\delta},W_0)$. 
Since $v_{0,\delta}\in W^{1,q}$,
the remainder can also be written as
\begin{align*}
  \mathcal R_\delta(v_{0,\delta},W_0)(x)
  =
  \int_{\mathbb R^2}\int_{\mathbb R^2}
  \rho_\delta(y)\rho_\delta(z)
  \bigl(W_0(x-y)-W_0(x-z)\bigr)
  \cdot\nabla v_{0,\delta}(x-z)\,\dd y\,\dd z.
\end{align*}
Differentiating in distributions and integrating by parts
in $z$, we find, for $k=1,2$,
\begin{align*}
  \partial_k\mathcal R_\delta(v_{0,\delta},W_0)(x)
  ={}&
  \int_{\mathbb R^2}\int_{\mathbb R^2}
  \rho_\delta(y)\rho_\delta(z)
  \partial_kW_0(x-y)\cdot\nabla v_{0,\delta}(x-z)
  \,\dd y\,\dd z
  \\
  &+
  \int_{\mathbb R^2}\int_{\mathbb R^2}
  \rho_\delta(y)\partial_k\rho_\delta(z)
  \bigl(W_0(x-y)-W_0(x-z)\bigr)
  \cdot\nabla v_{0,\delta}(x-z)
  \,\dd y\,\dd z.
\end{align*}
For $y,z\in\operatorname{supp}\rho_\delta$,
$|W_0(x-y)-W_0(x-z)|\leq C\delta\Vert \nabla W_0\Vert _{L^\infty}$.
Together with
$\Vert \nabla\rho_\delta\Vert _{L^1}\leq C\delta^{-1}$,
Young's inequality gives $\Vert \nabla\mathcal R_\delta(v_{0,\delta},W_0)\Vert _{L^q}
  \leq
  C\Vert \nabla W_0\Vert _{L^\infty}
  \Vert \nabla v_{0,\delta}\Vert _{L^q}$. 
The boundedness of $\mathbb P$ on $W^{1,q}$ gives
$\Vert \nabla v_{0,\delta}\Vert _{L^q}
  \leq C\Vert \chi_\delta u_0\Vert _{W^{1,q}}
  \leq C\Vert u_0\Vert _{W^{1,q}}$.
Moreover, $\mathcal R_\delta(v_{0,\delta},W_0)$
is supported in the fixed bounded set $K+B_1$.
Since $r\leq q$, H\"older's inequality therefore gives
\[
  \Vert \mathcal R_\delta(v_{0,\delta},W_0)\Vert _{\dot W^{1,r}}
  \leq
  C\Vert \nabla W_0\Vert _{L^\infty}\Vert u_0\Vert _{W^{1,q}}.
\]
Combining the preceding estimates, we conclude that
\[
  \Vert \partial_{W_{0,\delta}}u_{0,\delta}\Vert _{\dot W^{1,r}}
  \leq
  C\left(
    \Vert \partial_{W_0}u_0\Vert _{\dot W^{1,p}}
    +
    \Vert W_0\Vert _{W^{1,\infty}}\Vert u_0\Vert _{W^{1,q}}
  \right)
\]
for $0<\delta\leq\delta_0$.

For $\delta_0\leq\delta\leq1$, Young's inequality gives
$\Vert \nabla u_{0,\delta}\Vert _{L^q}+\Vert \nabla^2u_{0,\delta}\Vert _{L^q}
  \leq C_{\delta_0}\Vert u_0\Vert _{W^{1,q}}$.
Since $W_{0,\delta}$ is uniformly bounded
in $W^{1,\infty}$ and supported in $K+B_1$,
and $r\leq q$, it follows that
\begin{align*}
  \Vert \partial_{W_{0,\delta}}u_{0,\delta}\Vert _{\dot W^{1,r}}
  &\leq
  C\Vert W_0\Vert _{W^{1,\infty}}
  \left(
    \Vert \nabla u_{0,\delta}\Vert _{L^q}
    +\Vert \nabla^2u_{0,\delta}\Vert _{L^q}
  \right)
  \leq
  C_{\delta_0}\Vert W_0\Vert _{W^{1,\infty}}
  \Vert u_0\Vert _{W^{1,q}}.
\end{align*}
This proves \eqref{eq:velocity-regularization-uniform}
for all $0<\delta\leq1$.

Since $p>\frac{2}{2\alpha-\gamma}$ and
$q>\frac{2}{2\alpha-1}\geq\frac{2}{2\alpha-\gamma}$,
the exponent $r=\min\{p,q\}$ satisfies
$r>\frac{2}{2\alpha-\gamma}$.
Thus, the regularized data satisfy the assumptions required for the
estimates in Section~\ref{sec:C2}, with $r$ as the striated
exponent and $q$ as the velocity exponent.

\section{Compactness and passage to the limit in the
regularization procedure}\label{sec:B}

\subsection{Compactness of the regularized solutions and proof of
existence in Theorem~\ref{thm:main1}}
\label{sec:B1}

In this subsection, we justify the compactness argument used in
the proof of Theorem~\ref{thm:main1}.
Throughout the regularization procedure, $\varepsilon\in(0,1]$
is fixed.

\begin{proposition}\label{prop:compactness-main1}
Under the assumptions of Theorem~\ref{thm:main1}, let
$(u_{0,\delta},\theta_{0,\delta})$ be the smooth initial data
constructed in Appendix~\ref{sec:A1}, and let
$(u_\delta,\theta_\delta)$ be the corresponding global smooth
solutions to \eqref{BoussEq}.
Then, for every $T>0$, there exist a subsequence, still denoted
by $\delta$, and functions $(u,\theta)$ such that
\begin{align}
  u_\delta
  &\rightharpoonup^\ast u
  \qquad\qquad\text{in }L^\infty([0,T],W^{1,p}(\mathbb R^2)),
  \label{conv-main1-u-weak}\\
  u_\delta
  &\longrightarrow u
  \qquad\text{strongly in }L^2([0,T],L^2(B_R))
  \quad\text{for every }R>0,
  \label{cov-u-loc}\\
  \theta_\delta
  &\rightharpoonup^\ast\theta
  \qquad\qquad\text{in }L^\infty([0,T],L^2(\mathbb R^2)),
  \label{conv-main1-theta-L2}\\
  \theta_\delta
  &\rightharpoonup^\ast\theta
  \qquad\qquad\text{in }L^\infty((0,T)\times\mathbb R^2).
  \label{conv-main1-theta-Linf}
\end{align}
The limiting pair $(u,\theta)$ is a weak solution of
\eqref{BoussEq}, in the sense of
Definition~\ref{defn:weak-Bouss}, with initial data
$(u_0,\theta_0)$.
Moreover,
\[
  \theta\in L^\infty([0,T],L^1\cap L^\infty),
  \qquad
  u\in L^\infty([0,T],W^{1,p})
  \cap\widetilde L^1([0,T],B_{\infty,\infty}^{2\alpha}),
\]
and all the estimates in Proposition~\ref{prop:ap-es-basic}
pass to the limite.
\end{proposition}

\begin{proof}
By Appendix~\ref{sec:A1}, the approximate initial data satisfy
\begin{align}
  \operatorname{div}u_{0,\delta}&=0,
  \qquad
  u_{0,\delta},\theta_{0,\delta}\in H^m(\mathbb R^2)
  \quad\text{for every }m\geq0,
  \label{regularized-main1-smooth}\\
  u_{0,\delta}&\longrightarrow u_0
  \quad\text{strongly in }W^{1,p}(\mathbb R^2),
  \label{regularized-main1-u-conv}\\
  \theta_{0,\delta}&\longrightarrow\theta_0
  \quad\text{strongly in }L^r(\mathbb R^2),
  \qquad 1\leq r<\infty,
  \label{regularized-main1-theta-conv}
\end{align}
as $\delta\to0$.
In addition,
\[
  \sup_{0<\delta\leq1}
  \Vert u_{0,\delta}\Vert _{W^{1,p}}
  \leq C\Vert u_0\Vert _{W^{1,p}},
  \qquad
  \sup_{0<\delta\leq1}
  \Vert \theta_{0,\delta}\Vert _{L^1\cap L^\infty}
  \leq\Vert \theta_0\Vert _{L^1\cap L^\infty}.
\]

For each fixed $\delta>0$, local smooth solvability,
the {\it{a priori}} estimates in
Proposition~\ref{prop:ap-es-basic}, and the standard
continuation argument yield a unique global smooth solution
$(u_\delta,\theta_\delta)$.
Indeed, the bound on
$\int_0^T\Vert \nabla u_\delta\Vert _{L^\infty}\,\dd t$
controls $\Vert \nabla\theta_\delta\Vert _{L^\infty}$ through the
transport equation and prevents blowup of the higher
Sobolev norms on finite time intervals.

The maximum principle,
Proposition~\ref{prop:ap-es-basic}, and
the velocity bound \eqref{eq:u-Lp-es-refined} give
\begin{equation}\label{regularized-main1-basic-bounds}
\begin{aligned}
  \sup_{0<\delta\leq1}\Big(
    \Vert \theta_\delta\Vert _{L^\infty([0,T],L^1\cap L^\infty)}
    +\Vert u_\delta\Vert _{L^\infty([0,T],W^{1,p})}
    +\Vert u_\delta\Vert _{\widetilde L^1
      (0,T;B_{\infty,\infty}^{2\alpha})}
  \Big)
  \leq C_{T,\varepsilon}.
\end{aligned}
\end{equation}
Here and below, the constants are independent of $\delta$.
The estimates in Proposition~\ref{prop:ap-es-basic}
that are uniform in $\varepsilon$ are also uniform in $\delta$.

Applying the Leray projector to the velocity equation gives
\[
  \varepsilon\partial_tu_\delta
  =
  -\varepsilon\mathbb P\operatorname{div}
    (u_\delta\otimes u_\delta)
  -\Lambda^{2\alpha}u_\delta
  +\mathbb P(\theta_\delta e_2).
\]
Since $W^{1,p}(\mathbb R^2)\hookrightarrow L^{2p}(\mathbb R^2)$
for $2\leq p<\infty$, the boundedness of $\mathbb P$ on
$L^p$ and the mapping properties of $\Lambda^{2\alpha}$ give
\[
  \Vert \partial_tu_\delta\Vert _{W^{-2,p}}
  \leq C_{\alpha,p}\left(
    \Vert u_\delta\Vert _{L^{2p}}^2
    +\frac1\varepsilon\Vert u_\delta\Vert _{L^p}
    +\frac1\varepsilon\Vert \theta_\delta\Vert _{L^p}
  \right).
\]
Consequently, for every $R>0$ and any sufficiently large
integer $m$,
\begin{equation}\label{regularized-main1-time-bound}
  \sup_{0<\delta\leq1}
  \Vert \partial_tu_\delta\Vert _{L^1([0,T],H^{-m}(B_R))}
  \leq C_{T,R,\varepsilon}.
\end{equation}

The Banach-Alaoglu theorem gives
\eqref{conv-main1-u-weak}.
Moreover, since
$W^{1,p}(B_R)\Subset L^2(B_R)
  \hookrightarrow H^{-m}(B_R)$,
the Aubin-Lions lemma and a diagonal argument yield
\eqref{cov-u-loc}.
The uniform temperature bounds similarly give
\eqref{conv-main1-theta-L2} and
\eqref{conv-main1-theta-Linf}, after a further extraction.
The limits agree by uniqueness of the distributional limit.

The strong local convergence of $u_\delta$ implies
\begin{align}
  u_\delta\otimes u_\delta
  &\longrightarrow u\otimes u
  \qquad\text{strongly in }
  L^1([0,T],L^1_{\mathrm{loc}}(\mathbb R^2)),
  \label{conv-main1-quadratic-u}\\
  \theta_\delta u_\delta
  &\longrightarrow\theta u
  \qquad\text{in }\mathcal D'((0,T)\times\mathbb R^2).
  \label{conv-main1-theta-u}
\end{align}
For the second convergence, we use
$\theta_\delta u_\delta-\theta u
  =
  \theta_\delta(u_\delta-u)+(\theta_\delta-\theta)u$,
together with the uniform $L^\infty$ bound and weak-$\ast$
convergence of $\theta_\delta$.

For the dissipative term, every
$\varphi\in C_c^\infty([0,T)\times\mathbb R^2;\mathbb R^2)$
satisfies
$\Lambda^{2\alpha}\varphi
\in L^1([0,T],L^{p'}(\mathbb R^2))$.
Thus \eqref{conv-main1-u-weak} gives
\[
  \int_0^T\!\int_{\mathbb R^2}
  u_\delta\cdot\Lambda^{2\alpha}\varphi\,\dd x\,\dd t
  \longrightarrow
  \int_0^T\!\int_{\mathbb R^2}
  u\cdot\Lambda^{2\alpha}\varphi\,\dd x\,\dd t.
\]
Using these convergences and
\eqref{regularized-main1-u-conv}-\eqref{regularized-main1-theta-conv}, we may pass to the limit
in the weak formulations of the momentum and transport equations.
The divergence-free condition also passes to the limit.
Hence $(u,\theta)$ satisfies
Definition~\ref{defn:weak-Bouss} with initial data
$(u_0,\theta_0)$.

For every $R>0$, weak-$\ast$ lower semicontinuity on $B_R$
and the maximum principle give
\[
  \Vert \theta\Vert _{L^\infty([0,T],L^1(B_R))}
  \leq
  \liminf_{\delta\to0}
  \Vert \theta_\delta\Vert _{L^\infty([0,T],L^1(B_R))}
  \leq\Vert \theta_0\Vert _{L^1}.
\]
Letting $R\to\infty$ and using
\eqref{conv-main1-theta-Linf}, we obtain
\[
  \Vert \theta\Vert _{L^\infty([0,T],L^1)}
  \leq\Vert \theta_0\Vert _{L^1},
  \qquad
  \Vert \theta\Vert _{L^\infty([0,T],L^\infty)}
  \leq\Vert \theta_0\Vert _{L^\infty}.
\]
Likewise, \eqref{conv-main1-u-weak} and weak-$\ast$ lower
semicontinuity yield
$u\in L^\infty([0,T],W^{1,p}(\mathbb R^2))$.

Under the assumptions of part~(4) of
Proposition~\ref{prop:ap-es-basic}, Appendix~\ref{sec:A1}
and \eqref{eq:u-LQ-es} give a uniform bound for
$u_\delta$ in $L^\infty([0,T],L^Q)$, also when $\alpha=1$
with zero mean and finite first absolute moment.
A further weak-$\ast$ extraction, identified by
\eqref{cov-u-loc}, and lower semicontinuity pass this
$\varepsilon$-independent bound to $u$.

It remains to pass the
$\widetilde L^1([0,T],B_{\infty,\infty}^{2\alpha})$ estimate
to the limit.
After extracting a further subsequence,
\eqref{cov-u-loc} gives $u_\delta(t)\rightarrow u(t)$ in  $L^2_{\mathrm{loc}}(\mathbb R^2)$ 
for almost every $t\in(0,T)$.
Together with the uniform global $L^p$ bound, this implies
convergence in $\mathcal S'(\mathbb R^2)$ for almost every $t$.
Therefore, for every fixed $j\geq-1, \Delta_j u_\delta(t)\longrightarrow\Delta_j u(t)$ in $\mathcal S'(\mathbb R^2)$
for almost every $t\in(0,T)$.

We claim that
\begin{equation}\label{es:dya-u-Linf}
  \Vert \Delta_j u(t)\Vert _{L^\infty}
  \leq
  \liminf_{\delta\to0}
  \Vert \Delta_j u_\delta(t)\Vert _{L^\infty}.
\end{equation}
If the liminf is finite, a subsequence realizing it admits
a weak-$\ast$ convergent subsequence in $L^\infty$.
Distributional convergence identifies its limit with
$\Delta_j u(t)$, proving \eqref{es:dya-u-Linf}.
If the liminf is infinite, the inequality is immediate.

Fatou's lemma now gives
$\Vert \Delta_j u\Vert _{L^1([0,T],L^\infty)} \leq
  \liminf_{\delta\to0}
  \Vert \Delta_j u_\delta\Vert _{L^1([0,T],L^\infty)}$.
Taking the supremum over $j\geq-1$, we obtain
\begin{align*}
  \Vert u\Vert _{\widetilde L^1
    (0,T;B_{\infty,\infty}^{2\alpha})}
  =
  \sup_{j\geq-1}
  2^{2\alpha j}
  \Vert \Delta_j u\Vert _{L^1([0,T],L^\infty)}
  \leq
  \liminf_{\delta\to0}
  \Vert u_\delta\Vert _{\widetilde L^1
    (0,T;B_{\infty,\infty}^{2\alpha})}
  \leq C_{T,\varepsilon}.
\end{align*}

The remaining estimates in
Proposition~\ref{prop:ap-es-basic} pass to the limit similarly,
by distributional convergence and lower semicontinuity
of the corresponding Sobolev and dyadic norms.
Their constants retain the asserted uniformity in
$\varepsilon\in(0,1]$.
This completes the proof.
\end{proof}

\subsection{Passage to the limit in the first-order striated estimates}
\label{sec:B2}

We now transfer the first-order striated estimates from the
regularized solutions to the limiting solution.
As in Appendix~\ref{sec:B1}, $\varepsilon\in(0,1]$ is fixed
throughout the compactness argument.

We first consider case III.
The uniform estimate \eqref{es:W-Cgam} and the
Banach-Alaoglu theorem give, after extracting a subsequence,
\[
  W_\delta\rightharpoonup^\ast\widetilde W
  \quad\text{in }
  L^\infty((0,T)\times\mathbb R^2).
\]
Since $\operatorname{div}u_\delta
=\operatorname{div}W_\delta=0$, the tangent equation reads
\[
  \partial_tW_\delta
  +\operatorname{div}
  \bigl(W_\delta\otimes u_\delta-u_\delta\otimes W_\delta\bigr)
  =0.
\]
Consequently, for every $R>0$,
\begin{align*}
  \Vert \partial_tW_\delta\Vert _
    {L^\infty([0,T],W^{-1,p}(B_R))}
  \leq
  C\Vert W_\delta\Vert _{L^\infty((0,T)\times\mathbb R^2)}
  \Vert u_\delta\Vert _{L^\infty([0,T],L^p(\mathbb R^2))}
  \leq C_{T,\varepsilon},
\end{align*}
uniformly in $\delta$.

For every $0<\gamma'<\gamma$, we have
$C^\gamma(\overline{B_R})
  \hookrightarrow C^{\gamma'}(\overline{B_R})
  \hookrightarrow W^{-1,p}(B_R)$ where the first embedding is compact.
The Aubin-Lions lemma therefore gives
\begin{equation}\label{eq:Wdelta-strong-C2}
  W_\delta\longrightarrow\widetilde W
  \quad\text{strongly in }
  C\bigl([0,T];C^{\gamma'}(\overline{B_R})\bigr).
\end{equation}
A diagonal argument allows the same subsequence to be chosen
for every $R>0$ and $0<\gamma'<\gamma$.

We next identify the limiting equation.
For every
$\zeta\in C_c^\infty([0,T)\times\mathbb R^2;\mathbb R^2)$,
the weak formulation reads
\begin{align*}
  \int_0^T\!\int_{\mathbb R^2}
  \Bigl[
    W_\delta\cdot\partial_t\zeta
    +\bigl(
      W_\delta\otimes u_\delta
      -u_\delta\otimes W_\delta
    \bigr):\nabla\zeta
  \Bigr]\,\dd x\,\dd t
  +
  \int_{\mathbb R^2}
  W_{0,\delta}(x)\cdot\zeta(0,x)\,\dd x=0.
\end{align*}
By Proposition~\ref{prop:compactness-main1},
\[
  u_\delta\longrightarrow u
  \quad\text{strongly in }
  L^2([0,T],L^2(B_R))
\]
for every $R>0$.
Together with \eqref{eq:Wdelta-strong-C2}, this implies
\[
  W_\delta\otimes u_\delta
  \longrightarrow\widetilde W\otimes u,
  \qquad
  u_\delta\otimes W_\delta
  \longrightarrow u\otimes\widetilde W,
\]
strongly in $L^2([0,T],L^2(B_R))$.

Since $W_{0,\delta}\to W_0$ in
$C^{\gamma'}_{\mathrm{loc}}(\mathbb R^2)$,
we may pass to the limit in the weak formulation.
The divergence-free conditions also pass to the limit,
giving $\operatorname{div}u=\operatorname{div}\widetilde W=0$.
Thus
\[
  \partial_t\widetilde W+u\cdot\nabla\widetilde W
  =\widetilde W\cdot\nabla u,
  \qquad
  \widetilde W|_{t=0}=W_0,
\]
in the sense of distributions.

In case III, the assumptions
$\gamma>2\alpha-1$ and $p>2/(2\alpha-\gamma)$ imply $p>2$.
The velocity estimates give
$\nabla u\in L^1([0,T],L^\infty)$.
Thus the tangent equation has at most one bounded solution,
by Gronwall's inequality along the flow.
On the other hand, setting
\[
  \varphi(t,x):=\varphi_0(X_t^{-1}(x)),
  \qquad W:=\nabla^\perp\varphi,
\]
and differentiating the transport equation for $\varphi$
shows that $W$ is a bounded distributional solution of
the same tangent equation with initial data $W_0$.
Consequently, $\widetilde W=W$.

The local uniform convergence
\eqref{eq:Wdelta-strong-C2} and lower semicontinuity
of the H\"older norm now yield
\[
  \Vert W\Vert _{L^\infty([0,T],C^\gamma(\mathbb R^2))}
  \leq
  \liminf_{\delta\to0}
  \Vert W_\delta\Vert _{L^\infty([0,T],C^\gamma(\mathbb R^2))}
  \leq Ce^{\exp(CT)},
\]
where $C$ is independent of $\varepsilon\in(0,1]$.
Moreover, the tangent equation along the flow gives
\[
  \inf_{x\in\partial D(t)}|W(t,x)|
  \geq
  \inf_{x\in\partial D_0}|W_0(x)|
  \exp\left(
    -\int_0^t\Vert \nabla u(\tau)\Vert _{L^\infty}\,\dd\tau
  \right)>0.
\]
Here the spatial continuity of $W$ extends the flow identity
from almost every initial point to the boundary.
The level-set representation and this nondegeneracy bound
therefore yield the uniform-in-$\varepsilon$ persistence
of $C^{1,\gamma}$ boundary regularity in case III.

We finally consider case IV.
The uniform estimate \eqref{eq:W-W1inf-bdd} gives
\[
  \sup_{0<\delta\leq1}
  \Vert W_\delta\Vert _{L^\infty([0,T],W^{1,\infty}(\mathbb R^2))}
  \leq Ce^{\exp(CT)},
\]
with $C$ independent of $\varepsilon$.
Using
$W^{1,\infty}(B_R)\Subset C^{\gamma'}(\overline{B_R})$
for $0<\gamma'<1$, the preceding compactness and
identification arguments again yield, after extraction,
\[
  W_\delta\longrightarrow W=\nabla^\perp\varphi
  \quad\text{strongly in }
  C\bigl([0,T];C^{\gamma'}(\overline{B_R})\bigr)
\]
for every $R>0$ and $0<\gamma'<1$.

By the Banach-Alaoglu theorem and distributional
differentiation, after a further extraction,
\[
  \nabla W_\delta\rightharpoonup^\ast\nabla W
  \quad\text{in }L^\infty((0,T)\times\mathbb R^2).
\]
Weak-$\ast$ lower semicontinuity therefore gives
\[
  \Vert W\Vert _{L^\infty([0,T],W^{1,\infty}(\mathbb R^2))}
  \leq
  \liminf_{\delta\to0}
  \Vert W_\delta\Vert _{L^\infty([0,T],W^{1,\infty}(\mathbb R^2))}
  \leq Ce^{\exp(CT)}.
\]
Together with the preceding tangent nondegeneracy bound,
this proves $W^{2,\infty}$ boundary persistence
in case IV, with bounds uniform in $\varepsilon$.


\subsection{Passage to the limit in the second-order striated estimates}
\label{sec:B3}

Fix $T>0$ and $\varepsilon\in(0,1]$.
We first consider cases I and II.
Using the uniform $L^\infty([0,T],C^{1,\gamma})$ bound
in place of the $L^\infty([0,T],C^\gamma)$ bound,
the compactness argument in Section~\ref{sec:B2} gives,
after a diagonal extraction,
\[
  W_\delta\longrightarrow W
  \quad\text{strongly in }
  C\bigl([0,T];C^{1,\gamma'}(\overline{B_R})\bigr)
\]
for every $R>0$ and $0<\gamma'<\gamma$.
Lower semicontinuity therefore gives
\[
  \Vert W\Vert _{L^\infty([0,T],C^{1,\gamma}(\mathbb R^2))}
  \leq
  \liminf_{\delta\to0}
  \Vert W_\delta\Vert _{L^\infty([0,T],C^{1,\gamma}(\mathbb R^2))}
  \leq Ce^{\exp(CT)},
\]
where $C$ is independent of $\varepsilon$.

It remains to consider cases III and IV.
We use the uniform striated estimates
\eqref{es:pWW-Cgam} and \eqref{eq:es-k=2-IV}
established in Section~\ref{sec:C2}.

\medskip
\noindent\textbf{Case III.}
Since $\operatorname{div}W_\delta=0$, one can get $\partial_{W_\delta}W_\delta
  =\operatorname{div}(W_\delta\otimes W_\delta)$. 
The strong local convergence \eqref{eq:Wdelta-strong-C2}
therefore implies
\[
  \partial_{W_\delta}W_\delta
  \longrightarrow
  \operatorname{div}(W\otimes W)
  :=\partial_WW
  \quad\text{in }
  \mathcal D'((0,T)\times\mathbb R^2).
\]

We next establish compactness in time.
Using
$\partial_{W_\delta}^{\,2}u_\delta
=\operatorname{div}
((\partial_{W_\delta}u_\delta)\otimes W_\delta)$,
we rewrite \eqref{pWW} as
\[
  \partial_t(\partial_{W_\delta}W_\delta)
  =
  \operatorname{div}\Bigl(
    (\partial_{W_\delta}u_\delta)\otimes W_\delta
    -(\partial_{W_\delta}W_\delta)\otimes u_\delta
  \Bigr).
\]
Consequently, for every $R>0$,
\begin{align*}
  \Vert \partial_t(\partial_{W_\delta}W_\delta)
    \Vert _{H^{-1}(B_R)}
  \leq
  C\Vert W_\delta\Vert _{L^\infty(B_R)}^2
    \Vert \nabla u_\delta\Vert _{L^2(B_R)}
  +
  C\Vert \partial_{W_\delta}W_\delta\Vert _{L^\infty(B_R)}
    \Vert u_\delta\Vert _{L^2(B_R)}.
\end{align*}
The uniform striated estimates and the velocity bounds
from Appendix~\ref{sec:B1} thus give
\[
  \sup_{0<\delta\leq1}
  \Vert \partial_t(\partial_{W_\delta}W_\delta)\Vert _
    {L^\infty([0,T],H^{-1}(B_R))}
  \leq C_{T,R,\varepsilon}.
\]
Using
$C^\gamma(\overline{B_R})
  \Subset C^{\gamma'}(\overline{B_R})
  \hookrightarrow H^{-1}(B_R)$,
$0<\gamma'<\gamma$,
the Aubin-Lions lemma and a diagonal argument yield,
after extraction,
\[
  \partial_{W_\delta}W_\delta
  \longrightarrow\partial_WW
  \quad\text{strongly in }
  C\bigl([0,T];C^{\gamma'}(\overline{B_R})\bigr)
\]
for every $R>0$ and $0<\gamma'<\gamma$.
The limit is identified by the preceding distributional
convergence.

Lower semicontinuity of the H\"older norm gives
\[
  \Vert \partial_WW\Vert _{L^\infty([0,T],C^\gamma(\mathbb R^2))}
  \leq
  \liminf_{\delta\to0}
  \Vert \partial_{W_\delta}W_\delta\Vert _
    {L^\infty([0,T],C^\gamma(\mathbb R^2))}
  \leq Ce^{\exp(CT)},
\]
with $C$ independent of $\varepsilon$.

\medskip
\noindent\textbf{Case IV.}
The same divergence-form argument gives the uniform
$L^\infty([0,T],H^{-1}(B_R))$ bound for
$\partial_t(\partial_{W_\delta}W_\delta)$.
Combining this with \eqref{eq:es-k=2-IV} and  $W^{1,\infty}(B_R)
  \Subset C^\beta(\overline{B_R})
  \hookrightarrow H^{-1}(B_R)$ for $0<\beta<1$, 
we obtain, after a diagonal extraction,
\[
  \partial_{W_\delta}W_\delta
  \longrightarrow\partial_WW
  \quad\text{strongly in }
  C\bigl([0,T];C^\beta(\overline{B_R})\bigr)
\]
for every $R>0$ and $0<\beta<1$,
with the limit identified as in case III.

By the Banach-Alaoglu theorem and 
differentiation in the sense of distribution, after a further extraction if necessary,
\[
  \nabla(\partial_{W_\delta}W_\delta)
  \rightharpoonup^\ast\nabla(\partial_WW)
  \quad\text{in }
  L^\infty((0,T)\times\mathbb R^2).
\]
Weak-$\ast$ lower semicontinuity therefore gives
\[
  \Vert \partial_WW\Vert _{L^\infty([0,T],W^{1,\infty}(\mathbb R^2))}
  \leq
  \liminf_{\delta\to0}
  \Vert \partial_{W_\delta}W_\delta\Vert _
    {L^\infty([0,T],W^{1,\infty}(\mathbb R^2))}
  \leq Ce^{\exp(CT)},
\]
with $C$ independent of $\varepsilon$.
For the representative obtained by local convergence
in time, the spatial Lipschitz bound holds
for every $t\in[0,T]$.

Finally, combining~\eqref{eq:boundary-identities} with the
estimates above and the uniform tangent nondegeneracy established
in Section~\ref{sec:B2}, we obtain the desired persistence of
$C^{2,\gamma}$ boundary regularity in cases I-III and of
$W^{3,\infty}$ boundary regularity in case IV, with bounds
uniform in $\varepsilon\in(0,1]$.

\addtocontents{toc}{\protect\setcounter{tocdepth}{2}}

\end{document}